\documentclass[reqno, 12pt]{amsart}

\usepackage{mathtools}
\usepackage[T1]{fontenc}
\usepackage{amsmath, amsthm, amssymb}
\usepackage[ansinew]{inputenc}
\usepackage{xcolor}
\usepackage{bbm}
\usepackage{mathrsfs}
\usepackage{cite}
\usepackage{xparse}
\usepackage{fullpage}
\usepackage{pdflscape}
\usepackage{stmaryrd}
\usepackage{multirow}
\usepackage{lscape}
\usepackage{dsfont}
\usepackage{amsfonts}
\usepackage{scrextend} 
\usepackage{array} 
\usepackage{verbatim}
\usepackage{subfig}
\usepackage{yfonts}
\usepackage[colorlinks=true,allcolors=red]{hyperref}
\usepackage{graphicx} 
\usepackage{accents}

\newcommand\smallO{
  \mathchoice
    {{\scriptstyle\mathcal{O}}}
    {{\scriptstyle\mathcal{O}}}
    {{\scriptscriptstyle\mathcal{O}}}
    {\scalebox{.7}{$\scriptscriptstyle\mathcal{O}$}}
  }
  
\newcommand{\Hquad}{\hspace{0.5em}} 
 
\DeclareSymbolFont{sfletters}{OML}{cmbrm}{m}{it}
\DeclareMathSymbol{\somega}{\mathord}{sfletters}{"21}
\DeclareSymbolFont{Letters} {U}{zeur}{m}{n} 
\DeclareMathSymbol{\ssomega}{\mathalpha}{Letters}{"21} 
\DeclareSymbolFont{Letters} {U}{zeur}{m}{n} 
\DeclareMathSymbol{\pPhi}{\mathalpha}{Letters}{"1E}

\makeatletter
\newcommand*\bigcdot{\mathpalette\bigcdot@{.5}}
\newcommand*\bigcdot@[2]{\mathbin{\vcenter{\hbox{\scalebox{#2}{$\m@th#1\bullet$}}}}}
\makeatother

\newtheorem{theorem}{Theorem}[section]
\newtheorem{lemma}[theorem]{Lemma}
\newtheorem{prop}[theorem]{Proposition}

\newtheorem{Assump}[theorem]{Assumptions}

\theoremstyle{definition}

\theoremstyle{remark}
\newtheorem{remark}[theorem]{Remark}

\allowdisplaybreaks[4]

\numberwithin{equation}{section}

\usepackage{scalerel,stackengine}
\stackMath
\newcommand\reallywidehat[1]{
\savestack{\tmpbox}{\stretchto{
  \scaleto{
    \scalerel*[\widthof{\ensuremath{#1}}]{\kern.1pt\mathchar"0362\kern.1pt}
    {\rule{0ex}{\textheight}}
  }{\textheight}
}{2.4ex}}
\stackon[-6.9pt]{#1}{\tmpbox}
}
\makeatletter
\newcommand*{\rom}[1]{\expandafter\@slowromancap\romannumeral #1@}
\makeatother

\begin{document}

\title{non-linear periodic waves on the Einstein cylinder}

\author{Athanasios Chatzikaleas}

\address{Westf{\"a}lische Wilhelms-Universit{\"a}t M{\"u}nster, Mathematical Institute, Einsteinstrasse 62, 48149 Munster, Germany} 
\email{achatzik@uni-muenster.de}

\author{Jacques Smulevici}

\address{Laboratory Jacques-Louis Lions (LJLL), University Pierre and Marie Curie (Paris 6), 4 place Jussieu, 75252 Paris, France}
\email{jacques.smulevici@ljll.math.upmc.fr}

\dedicatory{}

\begin{abstract}
Motivated by the study of small amplitudes non-linear waves in the Anti-de-Sitter spacetime and in particular the conjectured existence of periodic in time solutions to the Einstein equations, we construct families of arbitrary small time-periodic solutions to the conformal cubic wave equation and the spherically-symmetric Yang--Mills equations on the Einstein cylinder $\mathbb{R}\times \mathbb{S}^3$. For the conformal cubic wave equation, we consider both spherically-symmetric solutions and complexed-valued aspherical solutions with an ansatz relying on the Hopf fibration of the $3$-sphere. In all three cases, the equations reduce to $1$+$1$ semi-linear wave equations. \\ 
Our proof relies on a theorem of Bambusi--Paleari for which the main assumption is the existence of a seed solution, given by a non-degenerate zero of a non-linear operator associated with the resonant system. For the problems that we consider, such seed solutions are simply given by the mode solutions of the linearized equations. Provided that the Fourier coefficients of the systems can be computed, the non-degeneracy conditions then amount to solving infinite dimensional linear systems. Since the eigenfunctions for all three cases studied are given by Jacobi polynomials, we derive the different Fourier and resonant systems using linearization and connection formulas as well as integral transformation of Jacobi polynomials. \\
In the Yang--Mills case, the original version of the theorem of Bambusi--Paleari is not applicable because the non-linearity of smallest degree is nonresonant. The resonant terms are then provided by the next order non-linear terms with an extra correction due to backreaction terms of the smallest degree non-linearity and we prove an analogous theorem in this setting. \\
Finally, we emphasize that, in view of the KAM type method used, the time periodic solutions that we construct exist only when the small perturbative parameter belongs to a Cantor--like set.
\end{abstract}

\maketitle
\tableofcontents 
\addtocontents{toc}{\protect\setcounter{tocdepth}{1}} 
 
\noindent

\section{Introduction}
\subsection{Stability/instability of the Anti-de-Sitter spacetime}
The Anti-de-Sitter spacetime (AdS) is the maximally symmetric solution to the vacuum Einstein equations with a negative cosmological constant
\begin{equation} \label{eve}
\mathrm{Ric}(g)=- \Lambda g, \quad \Lambda < 0. 
\end{equation}
Given $\Lambda < 0$, this is the simplest or trivial solution to \eqref{eve}, in a similar sense that the Minkowski or de-Sitter spacetimes are the trivial solutions to the vacuum Einstein equation with $\Lambda=0$ or $\Lambda >0$. While the stability properties of the Minkowski or de-Sitter spacetimes are now well-understood \cite{MR1316662, MR868737}, the study of pertubations of AdS is still an active subject of research. One important aspect is that, AdS or, more generally, spacetimes which are asymptotically AdS are not globally hyperbolic. Hence, any evolution problem for these solutions is only well-posed after boundary conditions are imposed at the conformal boundary. Two naturally opposite classes of boundary conditions are the fully reflective and dissipative boundary conditions. 
In the reflective case, we expect, as originally conjectured by Dafermos-Holzegel in \cite{DafHol} and independently by Anderson in \cite{MR2271154}, that AdS is unstable. Strong evidence for this instability was first presented by Bizo\'n-Rostworowski \cite{Bizon:2011gg}, which pioneered the study of the spherically-symmetric Einstein-Klein-Gordon system using numerical and Fourier based analysis and proposed weak turbulence as the non-linear source of instability. Recently, a proof of instability for the spherically Einstein-Vlasov\footnote{Moschidis has further announced similar results for the spherically-symmetric Einstein-scalar-field system \cite{gm:esf}.} system has been obtained in the work of Moschidis \cite{gm:evads,MR4150259}, based on a physical space mechanism. 
 In the dissipative case, one has strong decay of solutions for the linearized Einstein equations \cite{hlsw:aplfads} and this should lead to stability even at the non-linear level.  

\subsection{The time-periodic solutions of Rostworowski--Maliborski}
In the reflective case, parallel to the study of the instability conjecture, an interesting class of solutions was introduced by Rostworowski--Maliborski in \cite{Maliborski:2013jca}, where they constructed pertubatively and numerically small data, time-periodic solutions of the spherically-symmetric Einstein-scalar field system. They furthermore suggested, based on their numerical analysis, that these solutions should enjoy stronger stability properties than the original AdS space. The present paper is directly motivated by this work. We indeed prove here the existence of arbitrary small time-periodic solutions for various toy models, which mimic certain properties of non-linear waves in the AdS space.

\subsection{The conformal wave and the Yang--Mills equations} \label{TheToyModels}
More precisely, we study the conformal wave and the Yang-Mills equations on the Einstein cylinder $\mathbb{R}\times \mathbb{S}^3$. The AdS spacetime is conformal to half of the Einstein cylinder, so that solutions to the conformal wave and the Yang-Mill equations on the AdS spacetime can be mapped to solutions on the whole Einstein cylinder with a certain reflection symmetry at the equator. 
The conformal cubic wave equation on the Einstein cylinder can be written as 
	\begin{align}\label{Model1CW} 
	- \partial_{t}^2 \phi(t,\omega)+ \Delta_{\mathbb{S}^3}\phi(t,\omega) - \phi(t,\omega) =\left| \phi(t,\omega) \right|^2 \phi(t,\omega),
	\end{align}
	for a scalar field $\phi:\mathbb{R} \times \mathbb{S}^3 \longrightarrow \mathbb{C}$ with $\phi=\phi(t,\omega)$. We will consider perturbations around the static solution $\phi_0=0$ and for simplicity with zero initial velocity. 
\\ \\In the spherically symmetric case, the initial value problem for \eqref{Model1CW} reduces to 
\begin{align}\label{Model1WSphericalSymmetry}
\begin{dcases}
	\left( \partial_{t}^2 + L \right) u = f(u), \Hquad (t,x)\in \mathbb{R} \times (0,\pi), \\
	u_{0}(x)= u(0,x),  \\
	u_{1}(x)=\partial_{t} u(0,x) =0, 
\end{dcases}
\end{align}
for a scalar field $u:\mathbb{R} \times (0,\pi) \longrightarrow \mathbb{R}$ with $u = u(t,x)$ and  
\begin{align}\label{DefinitionLinearOperatorModel1WSphericalSymmetry}
    Lu: = -\Delta^{ss}_{\mathbb{S}^3 } u +  u, \Hquad f(u):= - u^3,
\end{align}
where
\begin{align*}
	-\Delta^{ss}_{\mathbb{S}^3 }  u := - \frac{1}{\sin^2 (x) } \partial_{x} \left(
	\sin^2 (x) \partial_{x} u
	\right)
\end{align*}
stands for the spherically symmetric Laplace--Beltrami operator on $\mathbb{S}^3$. 
\\ \\
When the spherical symmetry assumption is removed \cite{2021arXiv210409797E, MR3450568}, we use an ansatz based on \textit{Hopf coordinates}\footnote{We would like to thank Oleg Evnin who suggested the Hopf coordinate ansatz.} $(\eta,\xi_{1},\xi_{2}) \in  \left[0,\frac{\pi }{2} \right] \times \left[0,2\pi  \right) \times \left[0,2\pi  \right)$ rather than the standard spherical coordinates.
The Laplace--Beltrami operator on $\mathbb{S}^3$ in these coordinates reads as
\begin{align*}
	\Delta_{(\eta,\xi_1,\xi_2)}^{\mathbb{S}^3} \chi= \partial_{\eta}^2 \chi + \left(\frac{\cos \eta}{\sin \eta} - \frac{\sin \eta}{\cos \eta} \right)\partial_{\eta}\chi +\frac{1}{\sin^2 \eta} \partial_{\xi_1}^2 \chi +\frac{1}{\cos^2 \eta} \partial_{\xi_2}^2 \chi.
\end{align*}
While in principle, the Fourier expansion with respect to $\xi_1$ and $\xi_2$ of a solution $\chi(t,\eta,\xi_1,\xi_2)$ to \eqref{Model1CW} may include all possible admissible frequencies, we will pick a fixed pair of frequencies $(\mu_1,\mu_2) $ and force the Fourier expansion to excite only this particular pair by implementing the ansatz
\begin{align}\label{AnsatzHof}
	\chi(t,\eta,\xi_1,\xi_2)=u(t,\eta) e^{i \mu_1 \xi_1}  e^{i \mu_2 \xi_2}.
\end{align}  
This leads us to consider the following initial value problem
\begin{align}\label{Model1CWOUTSphericalSymmetryHOPFanstaz}
\begin{dcases}
	\left( \partial_{t}^2 + \mathsf{L}^{(\mu_1,\mu_2)} \right) u = \mathsf{f}(u), \Hquad (t,\eta)\in \mathbb{R} \times \left(0, \pi/2 \right) , \\
	u_{0}(\eta)= u(0,\eta), \\
	u_{1}(\eta)=\partial_{t} u(0,\eta) =0, \\ 
\end{dcases}
\end{align}
for a scalar field $u:\mathbb{R} \times \left( 0, \pi/2 \right) \longrightarrow \mathbb{R}$ with $u = u(t,\eta)$ and
\begin{align}\label{DefinitionLinearOperatorModel1CWOUTSphericalSymmetryHOPF} 
\mathsf{L}^{(\mu_1,\mu_2)}u:=-\partial_{\eta}^2 u - \left(\frac{\cos \eta}{\sin \eta} - \frac{\sin \eta}{\cos \eta} \right) \partial_{\eta} u + \left( 
	\frac{\mu_1^2}{\sin^2 \eta} + \frac{\mu_2^2}{\cos^2 \eta}+1
	\right)u, \Hquad \mathsf{f}(u)= -u^3.
\end{align}
Finally, we consider the spherically symmetric equivariant Yang--Mills equation on the Einstein cylinder for the $SU(2)$ connection, which reduces \cite{MR1622539} to the study of
 \begin{align}
 	- \partial_{t}^2 \phi(t,x) + \partial_{xx}^2\phi (t,x) + \frac{\phi (t,x)  }{\sin^2(x)}    
 	=  \frac{\phi ^3 (t,x) }{\sin^2(x)} ,
 \end{align}
 for a scalar field $\phi:\mathbb{R} \times (0,\pi) \longrightarrow \mathbb{R}$ with $\phi=\phi(t,x)$. We will study perturbations of the static solution $\phi_0=1$ to the above equation. Writing 
 $$
 \phi(t,x)=1+ \sin^2 (x) u(t,x),
 $$
 we are led to the following initial value problem   
%
%
%
\begin{align}\label{Model2YM}
\begin{dcases}
	\left( \partial_{t}^2 + \mathfrak{L} \right) u = \mathfrak{f}(x,u), \Hquad (t,\omega)\in \mathbb{R} \times (0,\pi) , \\
	u_{0}(x)= u(0,x), \Hquad \omega \in (0,\pi)  \\
	u_{1}(x)=\partial_{t} u(0,x) =0, \Hquad \omega \in  (0,\pi)
\end{dcases}
\end{align} 
where
\begin{align}\label{DefinitionLinearOperatorModel2YM}
	\mathfrak{L}u:=- \frac{1}{\sin^4 x}\partial_x \left( \sin^4 x \partial_x u \right)+4u, \Hquad \mathfrak{f}(x,u):= -3 u^2 -\sin^2(x) u^3.
\end{align} 
 
\subsection{Main results and general strategy}
In the following, we use the abbreviations:
\begin{itemize}
	\item CW: conformal cubic wave equation in spherical symmetry, that is \eqref{Model1WSphericalSymmetry}--\eqref{DefinitionLinearOperatorModel1WSphericalSymmetry},
	\item  CH: conformal cubic wave equation out of spherical symmetry in Hopf coordinates according to the ansatz \eqref{AnsatzHof}, that is \eqref{Model1CWOUTSphericalSymmetryHOPFanstaz}--\eqref{DefinitionLinearOperatorModel1CWOUTSphericalSymmetryHOPF},
	\item  YM: Yang--Mills equation in spherical symmetry, that is \eqref{Model2YM}--\eqref{DefinitionLinearOperatorModel2YM},
\end{itemize}

and study the evolution of the perturbations  
\begin{align*}
	u:\mathbb{R} \times I \longrightarrow \mathbb{R}, \Hquad u = u(t,x), \Hquad  I =
	\begin{dcases}
		(0,\pi), & \text{ for CW} \\
		\left(0,\pi/2 \right), & \text{ for CH} \\
		(0,\pi), & \text{ for YM}
	\end{dcases} 
\end{align*}
driven by the partial differential equations
\begin{align}\label{AllModelsinOnePDE}
	\left( \partial_{t}^2 + \mathbf{L} \right) u = \mathbf{f}(x,u), \Hquad (t,x)\in \mathbb{R} \times I, 
\end{align}
subject to the initial data with zero initial velocity, 
\begin{align*}
\begin{dcases}
	u_{0}(x)= u(0,x), \Hquad x \in I,  \\
	u_{1}(x)=\partial_{t} u(0,x) =0, \Hquad x \in  I.
\end{dcases}
\end{align*} 
Here, the linear operators and the non--linearities are given respectively by
\begin{align}
	\mathbf{L}u&=  
	\begin{dcases}
		- \frac{1}{\sin^2 (x) } \partial_{x} \left(
	\sin^2 (x) \partial_{x} u
	\right) +  u, & \text{ for CW}  \\
	-\partial_{x}^2 u - \left(\frac{\cos x}{\sin x} - \frac{\sin x}{\cos x} \right) \partial_{x} u + \left( 
	\frac{\mu_1^2}{\sin^2 x} + \frac{\mu_2^2}{\cos^2 x}+1
	\right)u, & \text{ for CH} \\
	- \frac{1}{\sin^4 x}\partial_x \left( \sin^4 x \partial_x u \right)+4u,& \text{ for YM}
	\end{dcases}\label{DefinitionLinearOperator}   \\
	\mathbf{f}(x,u) & = 
	\begin{dcases}
	 - u^3	, & \text{ for CW}  \\
	  - u^3	, & \text{ for CH} \\
	   -3 u^2 -\sin^2(x) u^3 ,& \text{ for YM}
	\end{dcases}\label{Definitionnon-linearily}
\end{align}
Associated to the linear operators given by \eqref{DefinitionLinearOperator}, one can introduce natural Hilbert spaces and with suitable definitions for their domains (Section \ref{SectionEigenvalueproblems}), the linear operators are then all self--adjoint operators with compact resolvent. 
In order to simplify the presentation below, we denote by $\{e_n(x):n\ge 0 \}$ the set of eigenfunctions of any of these operators\footnote{Of course, the eigenfunctions are different for the different operators, so that this is just a generic name.} and by $\{\omega_n: n \geq 0 \}$ the set of corresponding eigenvalues. Recall the facts that, in all three models considered, the sequences $\{\omega_n : n \geq 0\}$ are all strictly increasing with $\omega_n \sim n$ as $n \rightarrow + \infty$. \\ \\
We also denote by $\Phi^t(\xi)$ the flow associated to any of the linearized equations with initial data $\left(u_{t=0},\partial_t u_{t=0}\right)=\left( \xi,0\right)$. If $\xi_n$ denote the Fourier coefficients of $\xi$ associated to the eigenbasis $\{e_n(x):n\ge 0 \}$, then 
\begin{align} \label{eq:lf}
\Phi^t(\xi)= \sum_{n=0}^\infty \cos(t \omega_n) \xi_n e_n(x).
\end{align}
To state our result, we need to introduce a set of frequencies verifying a certain Diophantine condition \cite{MR1819863}. Given $0<\alpha < 1/3$, define
\begin{equation}\label{def:Wa}
\mathcal{W}_\alpha:=\left\{ \omega \in \mathbb{R}: \Hquad 	   |\omega \cdot l -\omega_{j}| \geq \frac{\alpha}{l}, \Hquad \forall (l,j) \in \mathbb{N}^2,\Hquad l\geq 1, \Hquad \omega_{j} \neq l \,\right\}.
\end{equation}
Consider any of the problems \text{CW}, \text{CH} or \text{YM} and let $e_\gamma$ be one of the eigenfunction to the corresponding linear operator. On top of $\alpha$, the statements of our results depend on the constant $\gamma\in \mathbb{N}\cup \{0\}$, the index of the eigenfunction, and $s>0$, which defines the Sobolev space\footnote{The definition of the $H^s$ spaces is adapted to each problem, see Section \ref{SectionMethodBambusiPaleari}.} $H^s$ where the solutions will belong to. Our assumptions are slightly different depending on the problems addressed. 

\begin{Assump}\label{AssumptionsRef}
Specifically, we make the following assumptions.

\begin{itemize}
	\item CW:  We take $\gamma\in \{0,1,2,\dots,\}$ and $s \in \mathbb{N}$ with $s \geq 2$. 
	\item CH: We take $\gamma\in \{0,1,2,3,4,5\}$ and $s \in \mathbb{N}$ with $s \geq 2$. Moreover, we assume that the parameters $\mu_1$ and $\mu_2$ appearing in \eqref{Model1CWOUTSphericalSymmetryHOPFanstaz} verify $\mu_1=\mu_2=\mu$ with $\mu$ either sufficiently large, or $\mu\in\{0,1,2,3,4,5\}$.
	\item YM: We take $\gamma\in \{0,1,2,3,4,5\}$ and $s \in \mathbb{N}$ with $s \geq 3$.  
\end{itemize}
\end{Assump}

\begin{remark}[Range of $\gamma$]\label{RemarkRangeGamma}
 We note that our proof is based on closed formulas for the Fourier coefficients, integrals that quantify the mode couplings. Although we derive these formulas uniformly with respect to $\gamma$ (see Section \ref{SectionFourierConstants}), we also need to check the validity of particular non--linear conditions depending on the Fourier coefficients. On the one hand, for the CW model, the Fourier coefficients have a relatively simple closed formula. Hence, there is no need to restrict the range of $\gamma$ and we establish the validity of the conditions needed uniformly with respect to $\gamma$. On the other hand, for the CH and YM models, the complexity of the Fourier coefficients requires to restrict the range of $\gamma$ to any finite set. Since the smaller the range is, the easier one can verify our computations, we fix $\gamma \in \{0,1,2,3,4,5\}$ solely for the purpose of computing and verifying all computations in the manuscript by hand. However, we believe that our result stated below holds true also for larger values of $\gamma $. For the convenience of the reader, we attach Mathematica files that can help the reader to both easily verify our computations for small $\gamma$ as well as derive and verify the analogous closed formulas for larger values of $\gamma$.       
\end{remark}

Under Assumptions \ref{AssumptionsRef}, we prove the following result. 
	\begin{theorem}[Main result 1: Existence of time--periodic solutions to all three models bifurcating from various 1--modes]\label{MainTheorem1} 
		 Let $(\gamma, s) \in (\mathbb{N}\cup\{0\}) \times \mathbb{R}$ verify the Assumptions \ref{AssumptionsRef} and let $e_\gamma$ be the eigenfunction to the corresponding linear operator. 
Also, let $0< \alpha< 1/3$ and $\mathcal{W}_\alpha$ be the corresponding set of frequencies, defined in \eqref{def:Wa}. Then, there exists a family $\{ u_\epsilon:\epsilon \in \mathcal{E}_\alpha\}$ of time-periodic solutions to either CW, CH or YM where $\mathcal{E}_{\alpha}$ is an uncountable set that has zero as an accumulation point. In addition, each element $u_{\epsilon}$ has the following properties:
		 \begin{enumerate}
		 	\item $ u_{\epsilon}$ has period $T_{\epsilon}=2\pi / \omega_\epsilon  $ with $\omega_\epsilon \in \mathcal{W}_\alpha$ being $\epsilon-$close to one,  
		 	\item $u_\epsilon \in H^1 \left( \left[0,T_{\epsilon} \right]; H^s \right)$. 
		 	\item $u_\epsilon$ stays for all times close to the solution to the linearized equation with initial data $\left(u_{t=0},\partial_t u_{t=0} \right)=\left( \epsilon \kappa_\gamma e_\gamma,  0\right)$,
		 	\begin{align*}
		 		\sup_{t\in \mathbb{R}} \left \|u_{\epsilon}(t,\cdot)-\Phi^{t \omega_\epsilon }\left(\epsilon \kappa_\gamma e_\gamma  \right ) \right \|_{H^s} \lesssim \epsilon^2,
		 	\end{align*} 
	 	where $\kappa_\gamma$ is a normalization constant. 
		 \end{enumerate}
	 \end{theorem}
For the CW and CH models, the results above are proven using a theorem of Bambusi and Paleari \cite{MR1819863}, while for the YM model, the original version of their theorem (stated at Theorem \ref{OrigivalversionTheoremBambusi} below) is not applicable and we will adapt their work. To explain this, we follow \cite{MR1819863} and consider any of the models above in the Fourier space by projecting the equations on the eigenbasis $\{e_n : n \geq 0 \}$, so that, schematically, the equations take the form
\begin{equation} \label{eq:fb}
\ddot{ u}^j (t) + (Au(t))^j = (f(u))^j, \Hquad j \geq 0,
\end{equation}
where $u=\{ u^j:j\ge 0 \}$ denotes the array of the Fourier coefficients, $A$ is a multiplication operator defined by $(Au)^j= \omega_j^2 u^j$ and $(f(u))^j$ are the coefficients of the non-linearity written in Fourier space, which takes the form of a polynomial in the $u^j$. In addition, we assume that
\begin{align*}
f(u)=f^{(0)}(u)+ f^{(1)}(u),
\end{align*}
where $f^{(0)}$ is a homogeneous polynomial of degree $r \geq 2$ and $f^{(1)}$ is a polynomial of degree at least $r+1$. Then, one looks for solutions $u(t)$ to \eqref{eq:fb} where $u(t)$ belongs to the Hilbert space 
\begin{align*}
	l_s^2 := \left\{u=\{u^j: j\geq 0\}: \Hquad | u |^2_s<\infty \right\}
\end{align*}
associated with the norm
\begin{align*}
	| u |^2_s := \sum_{j= 0}^{\infty} j^{2s} |u_j|^2.
\end{align*}
Besides some regularity considerations, the main theorem in \cite{MR1819863} asserts that given any non-degenerate zero of the operator
\begin{align*}
	\mathcal{M} \xi = A \xi + \langle  f^{(0)} \rangle(\xi), \quad \xi= \left\{ \xi^j  :j \ge 0 \right\}\in l_s^2,
\end{align*}
where $\langle  f^{(0)} \rangle(\xi)$ denotes the average in time of the non-linearity $f^{(0)}$ along the linearized flow, one can construct a family of small data periodic in time solutions, with properties similar to those stated in Theorem \ref{MainTheorem1}. 
The operator $\mathcal{M}$ is in fact linked to the resonant system associated to the original equation.  If $u(t)$ is periodic in time with period $\omega$, let $q$ be defined by $u(t)=q(\omega t)$ and let $L_\omega$ be the operator 
$$
L_\omega q = \omega^2 \frac{d}{dt^2}q +Aq. 
$$
The proof of \cite{MR1819863} is based on a Lyapunov-Schmidt decomposition $q=q_\perp+v$, with $v \in \ker L_1$ and $q_\perp \in (\ker L_1)^\perp$, together with the projections of the equations onto the range and kernel of $L_1$, leading to the so-called $P$-equation 
\begin{align}
	L_{\ssomega} q_{\perp} &= 
	P  f(v+q_{\perp}) , \label{PEquation}
	\end{align}
	and $Q$-equation
	\begin{align}
(1-\omega^2) A v &= 
	Q  f(v+q_{\perp}). \label{QEquation} 
	\end{align}
 The Diophantine condition \eqref{def:Wa} is then used to solve the $P$-equation, while the non-degeneracy assumption and an implicit function argument is used to solve the $Q$-equation. \\ \\
 For the CW and CH models, one easily verifies that the eigenfunctions $\kappa_{n}e_n$, where $\kappa_{n}$ is an appropriate rescaling, are all zeroes of $M$, so that the main difficulty is to establish the non-degeneracy condition, i.e., to prove that the kernel of $d\mathcal{M}(\kappa_{n} e_n)$ is trivial. In the YM case, however, the non-linearity contains both quadratic and cubic terms, so that a priori, only the quadratic terms would contribute to the definition of the operator $\mathcal{M}$. On the other hand, it turns out that, the average along the flow of the quadratic terms actually vanishes identically, leading to a degenerate, linear operator $\mathcal{M}$. Thus, we introduce a replacement for the operator $\mathcal{M}$, that takes into account also the cubic terms. However, the quadratic terms still play a role in this modified operator. Indeed, the solution to the $P-$equation roughly takes the form 
\begin{align*}
	q_{\perp}(v)=q_{\perp,\text{quadratic}}(v) +q_{\perp,\text{cubic}}(v) + \dots 
\end{align*}
 where the term $q_{\perp,\text{quadratic}}(v)$ arises from the quadratic non-linearity, and after substituting $q_{\perp}(v)$ into the $Q-$equation, these terms will generate new additional cubic terms. Thus, in some sense, the backreaction of the quadratic terms into the $Q-$equation must also be taken into account. This type of difficulty, where the contribution of the lowest degree part of the non-linearity is non-resonant, has been treated in some situations, see for instance \cite{MR2021909}, Section 4.2 and \cite{MR2248834}, Section 1.2.3, that consider equations of the form $-\partial_{tt}u+\partial_{xx}u= u^{2p} +\mathcal{O}(u^{2p+1})$. 
Here, we prove a modified abstract version of the Bambusi-Paleari's theorem which we  then apply to the YM model.

\begin{theorem}[Main result 2: Modification of the Bambusi-Paleari theorem for the YM model]\label{MainTheorem2}
	Consider the partial differential equation in the Fourier space
\begin{align}\label{BambusiMainPDEINTRO}
	 \ddot{ u}^j (t) + (\mathfrak{A}u(t))^j = \left( \mathfrak{f}\left(u(t) \right) \right)^j, \Hquad j \geq 0,
\end{align}  
 where the dots stand for derivatives with respect to time and $\mathfrak{A}$ is a positive multiplication self--adjoint operator with pure point and resonant spectrum $\{\ssomega_{j}^2 > 0 : j \geq 0\}$ with $\ssomega_{j} \sim j$, as $j \rightarrow \infty$, defined by
\begin{align*}
	\mathfrak{A}:\mathcal{D} (\mathfrak{A}) \simeq l_{s+2}^2 \longrightarrow l_{s}^2, \Hquad  (\mathfrak{A}u)^{j}:=\ssomega_{j}^2 u_{j},
\end{align*}
with $\mathcal{D} (\mathfrak{A})$ being its maximal domain of definition\footnote{Later, we will take $l_{s+1}^2$ instead of $l_{s}^2$ as our base Hilbert space, so that we will consider $\mathfrak{A}$ as an operator from $l_{s+3}^2 \longrightarrow l_{s+1}^2$. This allows for the construction of classical solutions, instead of solutions defined via the Duhamel formula or duality.}. In addition, assume that the non--linearity is given by
\begin{align*}  
 	\mathfrak{f}(u)= \mathfrak{f}^{(2)} (u)+ \mathfrak{f}^{(3)} (u),
\end{align*}
where each $ \mathfrak{f}^{(k)} $ is a homogeneous polynomial of order $k$ which is well-defined and smooth in $l_s^2$, with $ \mathfrak{f}^{(2)} $ being \textit{non--resonant}, that is
\begin{align}\label{ConditionNonResonantf2INTRO}
	\langle  \mathfrak{f}^{(2)}  \rangle(x):=\frac{1}{2\pi } \int_{0}^{2\pi } \pPhi^{t} \left(  \mathfrak{f}^{(2)}   \left( \pPhi^{t}(x)  \right)\right) dt = 0,
\end{align}
for all initial data $x$, where $\pPhi^{t}(x)$ denotes the solution to the linearized equation with initial data $(x,0)$. Furthermore, define the \textit{modified} operator
	\begin{align*}
	 \textswab{M}_{\pm}(\xi) &=  \pm  \mathfrak{A} \xi +     \langle \mathfrak{f}^{(3)}  \rangle( \xi)+\mathfrak{F}_{0}( \xi) ,
\end{align*} 
where $\mathfrak{F}_{0}( \xi) $ is given by Lemma \ref{LemmaDefinitionMathfrakF} 
and let $\xi_0 \in l_{s+3}^2 $ be initial data such that
\begin{itemize}
	\item $\xi_0$ is a \textit{zero} of $ \textswab{M}_{\pm}$,  
		 \begin{align*}
		 \textswab{M}_{\pm}\left(\xi _0 \right)=0,
	\end{align*}
	\item and $\xi_0$ satisfies the following \textit{non--degeneracy condition}
\begin{align*}
	\ker \left(d  \textswab{M}_{\pm}\left(\xi _0 \right)\right)=\{ 0 \}.
\end{align*}
\end{itemize}
Also, for some $0<\alpha <1/ 3$, define $\mathcal{W}_{\alpha}$ according to \eqref{def:Wa}. Then, there exists a family $\{u_{\epsilon}(t,\cdot):\epsilon \in \mathcal{E}_{\alpha}\}$ of time--periodic solutions to \eqref{BambusiMainPDEINTRO} where $\mathcal{E}_{\alpha} $ is an uncountable set that has zero as an accumulation point. In addition, each element $u_{\epsilon} $ has the following properties:
\begin{enumerate}
		 	\item $ u_{\epsilon}$ has period $T_{\epsilon}=2\pi / \omega_\epsilon  $ with $\omega_\epsilon \in \mathcal{W}_\alpha$ being $\epsilon-$close to one,  
		 	\item $u_\epsilon \in H^1 \left( \left[0,T_{\epsilon} \right]; l_{s}^2\right)$. 
		 	\item $u_\epsilon$ stays for all times close to the solution to the linearized equation with initial data $\left(u_{t=0},\partial_t u_{t=0} \right)=\left( \epsilon \kappa_\gamma e_\gamma,  0\right)$,
		 	\begin{align*}
		 		\sup_{t\in \mathbb{R}} \left |u_{\epsilon}(t,\cdot)-\Phi^{t \omega_\epsilon }\left(\epsilon \kappa_\gamma e_\gamma  \right ) \right |_{s} \lesssim \epsilon^2,
		 	\end{align*} 
	\end{enumerate}
\end{theorem}
 
 \subsection{Remarks}
  We note the following remarks.  \\ \\
 (Minimal periods) Theorems \ref{MainTheorem1} and \ref{MainTheorem2} give no information on the minimal periods $T_{\epsilon}$ of the time--periodic solutions $u_{\epsilon}(t,\cdot)$. However, one can relate $T_{\epsilon}$ to the minimal period $T_0$ of the solutions to the linearized system with 1--mode initial data, see Theorem 5.3 in \cite{MR1819863}. \\ \\ 
  (Cantor--like set) We emphasize that the time periodic solutions we construct exist only when the small perturbative parameter belongs to a Cantor--like set (of $0$ measure). This set together with the Diophantine condition introduced in Theorem \ref{MainTheorem1} are closely related to the presence of small divisors in the perturbation series around equilibrium points, a classical topic in the context of KAM theory in infinite dimensions. Although this type of conditions is essential in proving the existence of time--periodic solutions as in \cite{MR1819863}, it can be removed in some very special cases, see for example \cite{MR4173810}).  On the other hand, we note that the numerical constructions \cite{Maliborski:2013jca, PhysRevLett.121.021103 , PhysRevD.92.025036} do not seem to see the small divisors obstructions. \\ \\ 
 (Proof) The proof of Theorem \ref{MainTheorem2} follows the general strategy of \cite{MR1819863}, the main and essential difference being the backscattering contribution of the quadratic terms. An alternative approach to ours would be to find a normal form, in the spirit of \cite{MR803256}, that allows us to eliminate the quadratic terms and then apply the original result of \cite{MR1819863}. \\ \\ 
(The works of Berti--Bolle) In the works \cite{MR2038735, MR2248834, MR2021909}, a different strategy, based on variational methods, was introduced to solve the $Q$-equation \eqref{QEquation}, instead of the implicit function theorem as in \cite{MR1819863}. This in particular leads to a strong improvement in \cite{MR2248834} concerning the size of the frequency set, using an extra Nash-Moser iteration. We have not implemented this here for simplicity and leave a possible implementation of this improvement for future works. The works \cite{MR2038735,MR2248834} also treat the case of non-resonant quadratic terms, as we have here in the Yang-Mills case, in the specific case of the wave equation on an interval with Dirichlet boundary conditions. \\ \\ 
 (Regularity of the solutions) The solutions constructed here are $H^1$ in time with values in $H^s$, with $s$ arbitrarily large but fixed a priori. A posteriori, one can then use the equation to obtain additional regularity properties of the solutions. For instance, one easily has $\partial_t^2 u \in L_t^2 H_x^{s-1}$. Since some of the estimates depend a priori on the value of $s$, we cannot directly take $s=\infty$ but it is likely that a refinement of the methods presented would lead to such an improvement. \\ \\ 
  (Jacobi polynomials) One of the difficulties to proving Theorem \ref{MainTheorem1} comes from the fact that the eigenfunctions $e_n(x)$ of the linearized operators are given by Jacobi polynomials, instead of simpler explicit functions. This fact is not specific to our model problem and is a general feature of non-linear wave equations on AdS-like background. In particular, in the CH and YM models\footnote{The eigenfunctions for the CW case are given by Chebyshev polynomials of the second kind. The derivation of the resonant system in this case had been previously addressed in \cite{MR3652487}.}, the computation and the analysis of the Fourier coefficients associated to the resonant terms are non--trivial and constitute one of the contributions of this paper. To this end, we use linearization and connection formulas as well as particular Mellin transforms for the Jacobi polynomials. On the one hand, a linearization formula (also called addition formula) represents a product of two orthogonal polynomials with some parameters as a linear combination of orthogonal polynomials of the same kind with the same parameters. On the other hand, a connection formula represents a single orthogonal polynomial with some parameters as a linear combination of orthogonal polynomials of different kind with new parameters. In both cases, these are computationally efficient only in the case where the coefficients in the expansions are known in closed formulas. These computations also motivate our choice of $\mu_1=\mu_2=\mu$ for the CH model, since in this case, the eigenfunctions are reduced to Gegenbauer polynomials, a special class of Jacobi polynomials with additional algebraic properties that lead to closed formulas for the linearization and connection coefficients described above. Moreover, we also use particular Mellin transforms of Gegenbauer polynomials. These are integral transforms that may be regarded as the multiplicative version of the Laplace transform.   \\ \\
   (Mathematica files) For the CH and YM models, Theorem \ref{MainTheorem1} ensures the existence of time--periodic solutions bifurcating only from finitely many 1--mode initial data. As stated in Remark \ref{RemarkRangeGamma}, this is solely for the purpose of computing and verifying all computations in the manuscript by hand. Furthermore, one can use Mathematica \cite{Mathematica} to verify that our result still holds true also for larger values of $\gamma$.   For the convenience of the reader, we attach the Mathematica files needed that can help the reader to both easily verify our computations for small $\gamma$ as well as derive and verify the analogous computations for larger values of $\gamma$.

  \subsection{Previous works}
  The conformal wave equation \eqref{Model1CW} was introduced as a toy problem for the study of non-linear waves in confined geometries in \cite{MR3652487} and has since been the study of several works, see \cite{MR4173810, MR3948554, MR4026950}.  In particular, in \cite{MR4173810}, it was proven that solutions emanating from the first mode $e_0$, stay proportional to $e_0$ for all times and are periodic in time. The fact these data do not excite further modes is however specific to the first mode and to this equation. \\ \\
Concerning the well-posedness theories for the different models, since we do not focus here on low regularity solutions, we will simply recall that global well-posedness holds for the conformal cubic wave equation in the energy space, while the Yang-Mills equations in curved geometry have been shown to be globally well-posed in $H^2 \times H^1$ in \cite{MR1604914, MR722506} and on AdS with reflective boundary conditions in \cite{MR1028380}. We were motivated for the study of the Yang-Mills model by \cite{bizontalk}.   \\ \\
 Since the pioneering work \cite{Maliborski:2013jca}, there has been many investigations of time-periodic solutions for non-linear equations with completely resonant spectrum \cite{MR2038735, MR1877691, MR2021909, MR2248834}. 
 For the conformal wave equations, there exist also several constructions of time-periodic, weak solutions, via the variational techniques first introduced by Rabinowitz \cite{MR467823, MR470378}, see \cite{MR831036, MR2635200}.

\subsection{Acknowledgments}
 The authors would like to thank Andrzej Rostworowski (Jagiellonian University) for useful discussions, Oleg Evnin (Chulalongkorn University and Vrije Universiteit Brussel) who proposed the Hopf coordinates as a suitable coordinate system for the study of aspherical solutions to the conformal wave equation, Piotr Bizo\'{n} (Jagiellonian University) who introduced us to the work \cite{MR1819863}, during the BIRS workshop "Time-like Boundaries in General Relativistic Evolution Problems", held at Oaxaca, Mexico, in July 2019, as well as Riccardo Montalto (University of Milan) for useful discussions concerning KAM theory. This research was concluded at the Mathematisches Forschungsinstitut Oberwolfach during the conference ``Mathematical Aspects of General Relativity'' from August 29 to September 4, 2021. The authors gratefully acknowledge the kind hospitality and the comfortable working environment at Oberwolfach. Finally, the authors gratefully acknowledge the support of the ERC grant 714408 GEOWAKI, under the European Union's Horizon 2020 research and innovation program.

 \subsection{Organization of the paper}
We split the paper into the following sections:
\begin{itemize}
  \item Section \ref{SectionMethodBambusiPaleari}: we describe the methods we are about to use. For CW and CH, we will use the original version \cite{MR1819863} of Bambusi--Paleari's theorem (Theorem \ref{OrigivalversionTheoremBambusi}). However, for YM, as explained above, we need to revise the original version and establish an extension of their result (Theorem \ref{TheoremModificationofBambusiPaleari}) as stated in Theorem \ref{MainTheorem2}. In particular, we define the operators $\mathcal{M}$ and $\textswab{M}_{\pm}$ that determine the ``special'' initial data leading to time--periodic solutions.
	\item Section \ref{SectionEigenvalueproblems}: we study the linear eigenvalue problems where the linearized operators are given by \eqref{DefinitionLinearOperator}. As it turns out, the associated eigenfunctions are given by Jacobi polynomials and this is a common feature with the Einstein-Klein-Gordon system in spherical symmetry \cite{Maliborski:2013jca}. 
	\item  Section \ref{SectionPDEsinFourierSpace}: we express the partial differential equations \eqref{AllModelsinOnePDE} in the Fourier space and obtain infinite  dimensional systems of coupled harmonic oscillators. 
	\item Section \ref{SectionFourierConstants}: we define and study the mode couplings given by the Fourier coefficients. Specifically, we derive explicit closed formulas for all the Fourier coefficients on resonant indices.
	\item Section \ref{Section1modeinitialdata}: we study 1--mode initial data. In particular, we show that these modes satisfy the resonant systems (are zeros of the operators $\mathcal{M}$ and $\textswab{M}_{\pm}$ defined in Section \ref{SectionMethodBambusiPaleari}). In addition, we derive their differentials at these modes.
	\item Section \ref{SectionNonDegeneracyCondition}: firstly, we derive the crucial non--degeneracy conditions for the 1--mode initial data. As it turns out, these are non--linear conditions for the Fourier coefficients. Then, we use the analysis on the Fourier coefficients from Section \ref{SectionFourierConstants} to rigorously establish these conditions and prove the existence of time--periodic solutions as stated in Theorem \ref{MainTheorem1}. 
\end{itemize}

 \subsection{Notation}
We use different notation for each of the models we consider. Specifically, we summarize the notation that will be used in the table below. 
\begin{table}[h!]
\centering
 \begin{tabular}{|c |c|c|c|} 
 \hline
  Model & CW & CH & YM   \\ [0.5ex] 
 \hline 
Equation & 
\eqref{Model1WSphericalSymmetry}--\eqref{DefinitionLinearOperatorModel1WSphericalSymmetry} & 
\eqref{Model1CWOUTSphericalSymmetryHOPFanstaz}--\eqref{DefinitionLinearOperatorModel1CWOUTSphericalSymmetryHOPF} &
\eqref{Model2YM}--\eqref{DefinitionLinearOperatorModel2YM}     \\ [0.5ex]  \hline 
Font & 
Standard & 
Serif & 
Fraktur   \\ [0.5ex] \hline 
Linearized operator & 
$L $ & 
$\mathsf{L}^{(\mu_1,\mu_2)} $ &
$\mathfrak{L} $   \\ [0.5ex]  \hline
 Eigenvalues & 
 $\omega_n$ & 
$\somega_n^{(\mu_1,\mu_2)}$ &
$ \ssomega_n$    \\ [0.5ex]  \hline 
Eigenfunctions & 
$e_n $ & 
$\mathsf{e}_n^{(\mu_1,\mu_2)} $ &
$\mathfrak{e}_n $  \\ [0.5ex]  \hline
Inner product & 
$ (\cdot|\cdot) $ & 
$ \langle\cdot | \cdot\rangle$ &
$ [\cdot|\cdot]$   \\ [0.5ex]  \hline
Linear flow & 
$\Phi^t$ & 
$   \mathsf{\Phi}^t $ &
$\pPhi^t $  \\ [0.5ex]  \hline
Fourier coefficients & 
$C$ & 
$\mathsf{C}^{(\mu_1,\mu_2)}$ &
$ \mathfrak{C}, \overline{ \mathfrak{C}}$   \\ [0.5ex]  \hline
 \end{tabular} 
\end{table}

 \section{The method of Bambusi--Paleari revisited}\label{SectionMethodBambusiPaleari}

 In order to establish Theorem \ref{MainTheorem1} and construct our time--periodic solutions we rely on the method Bambusi--Paleari \cite{MR1819863}. This is an effective method to construct families of small amplitude time--periodic solutions close to a resonant equilibrium point for semi--linear partial differential equations. 
\subsection{The original version of the theorem}
Let $s \geq 0$ be a real number and define the Hilbert space $l_s^2$ to be the space of sequences such that
\begin{align*}
	u = \{u^j: j\geq 0 \}, \Hquad |u|_{s}^2:=\sum_{j = 0}^{\infty} | j^s u^j |^2   <\infty. 
\end{align*}
We endow $l_s^2$ with the natural inner product associated with the norm $| \cdot |_s$
%
%
and consider the following differential equation in $l_s^2$,
\begin{align}\label{BambusiMainPDE}
	\ddot{u}^j + (\mathfrak{A}u)^j = \left( \mathfrak{f}\left(u \right) \right)^j, \Hquad j \geq 0,
\end{align}  
 where the dots denote derivatives with respect to time. Here, $\mathfrak{A}$ is a positive multiplication self--adjoint operator with pure point and resonant spectrum $\{\ssomega_{j}^2 : j \geq 0 \}$ defined by
\begin{align*}
	\mathfrak{A}:\mathcal{D} (\mathfrak{A})  \longrightarrow l_{s}^2, \Hquad  (\mathfrak{A}u)^{j}:=\ssomega_{j}^2 u_{j}
\end{align*} 
and $\mathcal{D} (\mathfrak{A})$ stands for its maximal domain of definition endowed with the norm
\begin{align*}
	\| u \|_{\mathcal{D}(\mathfrak{A})}^2 := 
	|u|_{s}^2 + |\mathfrak{A} u|_{s}^2=
	\sum_{j=0}^{\infty} j^{2s} \left| \xi^j \right|^2+
	\sum_{j=0}^{\infty} j^{2s} \left|\ssomega_{j}^2 \xi^j \right|^2 .
\end{align*}
Moreover, we assume that $\mathfrak{A}$ and $f$ verify the following conditions:
\begin{enumerate}
	\item \label{H:compact} The injection of $\left( \mathcal{D} (\mathfrak{A}), 	\|\,.\,\|_{\mathcal{D} (\mathfrak{A})} \right)$ into $l_s^2$ is compact.
	\item \label{H:NL}The non-linearity $\mathfrak{f}(u)$ can be decomposed into 
	\begin{align}\label{SplitingfBaumbusi}
		\mathfrak{f}(u)=\mathfrak{f}^{(0)}(u)+\mathfrak{f}^{(1)}(u).
	\end{align} 
	\item \label{HCondition3} The lowest degree term $\mathfrak{f}^{(0)}(u)$ is a homogeneous polynomial of order $r\geq 2$ and is bounded operator from $\mathcal{D} (\mathfrak{A})$ to $\mathcal{D} (\mathfrak{A})$ with the domain $\mathcal{D} (\mathfrak{A})$ being invariant under $\mathfrak{f}^{(0)}$.
	\item \label{HCondition4} The highest degree term $\mathfrak{f}^{(1)}(u)$ (treated perturbatively as an error term) has a zero of order $r+1$ at zero, is differentiable in $l^2_s$, its differential is Lipschitz and verify the estimate
	\begin{align*}
		\left | d \mathfrak{f}^{(1)}(u_1)-d \mathfrak{f}^{(1)}(u_2) \right |_{s}  \leq C \epsilon^{r-1} \left | u_1 -u_2 \right |_{s},
	\end{align*} 
	for all $u_1,u_2 \in l_s^2$ with $ |u_1|_s\leq \epsilon$ and $|u_2|_s\leq \epsilon $.
	\end{enumerate}
\begin{remark}
In our case, the conditions above are obtained by starting from any of the equations \eqref{Model1WSphericalSymmetry}, \eqref{Model1CWOUTSphericalSymmetryHOPFanstaz}, \eqref{Model2YM} and projecting them on the eigenfunctions to the corresponding linear operator. Specifically, condition \eqref{H:compact} follows automatically from the fact  $\ssomega_j \sim j$, as $j \rightarrow \infty$, while conditions \eqref{HCondition3} and \eqref{HCondition4} that refer to the non-linearities in the Fourier space essentially follow from the fact the original non-linearities are smooth and that the Sobolev spaces of sufficiently high regularity form an algebra, see Section \ref{SubsectionPreliminaries}. 
\end{remark}
Let $\mathfrak{e}_{n}$ be the eigenfunctions to the associated linearized operators. On the Fourier side, these can be identified with $\mathfrak{e}_{n}= \{\delta^i_n: i  \geq 0 \} \in l^2_s$. Then, for any initial data, 
\begin{align*}
 \xi =\{ \xi ^{n} : n \geq 0 \}=\sum_{n=0}^{\infty} 	\xi^{n}  \mathfrak{e}_{n}  ,\Hquad
\end{align*}
we denote by  
\begin{align*}
  \pPhi ^{t} (\xi) = \left \{	  \xi^{n} \cos(\ssomega_{n} t) : n \geq 0   \right \}
\end{align*}
its linear flow, that is the solution to the initial value problem 
\begin{align*}
\begin{dcases}
	\ddot{u}^{n}(t) +  \ssomega_{n}  ^2 u^{n}(t) =0,\Hquad  t\in \mathbb{R}
	 \\
	 u^n(0)=\xi^n,  \Hquad  \dot{u}^n(0) =0.
\end{dcases}
\end{align*} 
We note that $\pPhi ^{t} (\xi) =\pPhi ^{-t} (\xi) $. Moreover, we define the operator 
\begin{align*}
	\mathcal{M}\left(\xi  \right) : =   \mathfrak{A} \xi  + \langle \mathfrak{f}^{(0)} \rangle(\xi),
\end{align*}
where
\begin{align*}
	\langle \mathfrak{f}^{(0)} \rangle(\xi):=  \frac{1}{2\pi } \int_{0}^{2\pi } \pPhi ^{-t} \left[\mathfrak{f}^{(0)}\left(\pPhi ^{t} (\xi ) \right) 
	\right]  dt
\end{align*}
is the average of $\mathfrak{f}^{(0)}$ along the linear flow. Notice that the highest degree term $\mathfrak{f}^{(1)}$ does not contribute to the definition of the operator $\mathcal{M}$. In \cite{MR1819863}, Bambusi--Paleari used a Lyapunov--Schmidt decomposition together with averaging theory and established the existence of a family of small amplitude time--periodic solutions with frequencies that satisfy the strong Diophantine condition  
\begin{align}\label{DefinitionSetWGamma}
	\ssomega \in  \mathcal{W}_{\alpha}:= \left \{
	 \ssomega\in \mathbb{R}: \Hquad |\ssomega \cdot l -\ssomega_{j}| \geq \frac{\alpha}{l}, \Hquad \forall (l,j) \in \mathbb{N}^2,\Hquad l\geq 1, \Hquad \ssomega_{j} \neq l
	 \right\}.
\end{align}
 
\begin{remark}[Accumulation to one]\label{Acculuatesto1}
For $0< \alpha<1/3$, the set $ \mathcal{W}_{\alpha}$ is an uncountable Cantor--like set that accumulates to one from above and below (Remark 2.4 \cite{MR1819863}). 
\end{remark}

\begin{remark}[Connection to Hurwitz's theorem] 
	According to Hurwitz's theorem, for every irrational number $\ssomega $ there are infinitely many relatively prime integers $\ssomega_{j}$ and $ l$ such that
	\begin{align*}
		\left|\ssomega \cdot l -  \ssomega_{j} \right| < \frac{1}{\sqrt{5}} \frac{1}{l}
	\end{align*}
	and moreover the constant $\sqrt{5}$ is optimal. Consequently, $\mathcal{W}_\alpha = \emptyset$, for $\alpha \ge \frac{1}{\sqrt{5}}$. In this note, we pick a suitable $\alpha$ with $0< \alpha<1/3$.
\end{remark}
The main result of \cite{MR1819863} reads as follows.
\begin{theorem}[Original version of Bambusi--Paleari's theorem \cite{MR1819863}]\label{OrigivalversionTheoremBambusi}
	For $0<\alpha <1/3$, define $\mathcal{W}_{\alpha}$ according to \eqref{DefinitionSetWGamma} and consider the operator
	\begin{align*}
	 \mathcal{M} (\xi) &=   \mathfrak{A} \xi +     \langle \mathfrak{f}^{(0)}  \rangle( \xi).
\end{align*} 
Assume that conditions \eqref{H:compact} and \eqref{H:NL} are verified. Moreover, let $\xi_0 $ be such that 
\begin{itemize}
	\item $\xi_0$ is a \textit{zero} of $ \mathcal{M} $,  
		 \begin{align*}
		\mathcal{M} \left(\xi_0  \right)=0,
	\end{align*}
	\item and $\xi_0$ satisfies the following \textit{non--degeneracy condition}
\begin{align*}
	\ker \left(d \mathcal{M} \left(\xi_0  \right)\right)=\{ 0 \}.
\end{align*}
\end{itemize}
Then, there exists a family $\{u_{\epsilon}:\epsilon \in \mathcal{E}_{\alpha}\} $ of time--periodic solutions to \eqref{BambusiMainPDE}--\eqref{SplitingfBaumbusi} where $\mathcal{E}_{\alpha} $ is an uncountable set that has zero as an accumulation point. In addition, each element $u_{\epsilon}$ has the following properties:
\begin{enumerate}
	\item $u_\epsilon$ has period $T_{\epsilon}=2\pi / \ssomega_{\epsilon} $ and there exists $\ssomega_{\star}>0$ such that the map $\epsilon\in \mathcal{E}_{\alpha}  \longmapsto \ssomega_{\epsilon}\in \mathcal{W}_{\alpha}\cap [1,1+\ssomega_{\star}) $ is a monotone, one-to-one map that stays close to one, $|1-\ssomega_{\epsilon}| \lesssim \epsilon^{r-1}$,
	\item $u_\epsilon \in H^1 ([0,T_{\epsilon}];l_s^2)$,
	\item $u_{\epsilon}$ stays for all times close to the solution to the linearized equation with initial data $\left(u_{t=0},\partial_t u_{t=0} \right)=\left( \epsilon \xi_0,  0\right)$,
	\begin{align*}
			\sup_{t\in \mathbb{R}} \left |u_{\epsilon}(t,\cdot)-\pPhi^{t \ssomega_{\epsilon}} \left(\epsilon \xi_0 \right ) \right |_{s} \lesssim \epsilon^2.
	\end{align*} 
	\end{enumerate}
\end{theorem}

\subsection{A modified Bambusi--Paleari theorem}
 
As we will see in Section \ref{SubsetionYMSysteminFourier}, in the case of the YM model, the non--linearity is given by
\begin{align} \label{Splitingf}
 	\mathfrak{f}(u)= \mathfrak{f}^{(2)} (u)+ \mathfrak{f}^{(3)} (u),
\end{align}
where
\begin{itemize}
	\item the lowest degree term $ \mathfrak{f}^{(2)} $ is a homogeneous polynomial of order $2$,
	\begin{align} \label{def:f2F}
	\left(  \mathfrak{f}^{(2)} (\{u^j(t):j\geq 0\}) \right)^m&:= -3\sum_{i,j=0}^{\infty}\overline{\mathfrak{C}}_{ijm}  u^{i}(t)u^{j}(t) , 
\end{align}
	\item  the highest degree term $ \mathfrak{f}^{(3)} $ is a homogeneous polynomial of order $3$,
	\begin{align} \label{def:f3F}
\left(  \mathfrak{f}^{(3)} (\{u^j(t):j\geq 0\}) \right)^m &:= - \sum_{i,j,k=0}^{\infty}  \mathfrak{C}_{ijkm} u^{i}(t)u^{j}(t)u^{k}(t) ,
\end{align}
\end{itemize}
Thus, according to the original version of Bambusi--Paleari's theorem (Theorem \ref{OrigivalversionTheoremBambusi}), one may argue that  $\mathfrak{f}^{(2)}$ is the main non-linearity and $ \mathfrak{f}^{(3)} $ can be treated perturbatively. However, in this setting, the original version of Bambusi--Paleari's theorem would not be applicable, because $ \mathfrak{f}^{(2)} $ is \textit{non--resonant} (Lemma \ref{LemmaNonResoantnf2}), that is
\begin{align}\label{ConditionNonResonantf2}
	\langle  \mathfrak{f}^{(2)}  \rangle(\xi):=\frac{1}{2\pi } \int_{0}^{2\pi } \pPhi^{t} \left(  \mathfrak{f}^{(2)}   \left( \pPhi^{t}(\xi)  \right)\right) dt = 0,
\end{align}
for all initial data $\xi$, and therefore
	\begin{align*}
	 \mathcal{M} (\xi) &=   \mathfrak{A} \xi  
\end{align*} 
leading to trivial zeros of the operator $\mathcal{M}$. Consequently, we need to revisit the theorem of Bambusi--Paleari in this context. Specifically, we consider \eqref{BambusiMainPDE}--\eqref{Splitingf}--\eqref{ConditionNonResonantf2}, replace $\mathfrak{f}^{(0)}(u)$ by $ \mathfrak{f}^{(2)} (u)+ \mathfrak{f}^{(3)} (u)$ and establish the following theorem.  
\begin{theorem}[Modification of Bambusi--Paleari's theorem]\label{TheoremModificationofBambusiPaleari}
Let $0<\alpha <1/ 3$ and define $\mathcal{W}_{\alpha}$ according to \eqref{DefinitionSetWGamma}. Let $\mathfrak{A}$ be a positive multiplication self--adjoint operator with pure point and resonant spectrum $\{\ssomega_{j}^2 > 0 : j \geq 0\}$ with $\ssomega_{j} \sim j$, as $j \rightarrow \infty$, defined by
\begin{align*}
	\mathfrak{A}:\mathcal{D} (\mathfrak{A}) \simeq l_{s+2}^2 \longrightarrow l_{s}^2, \Hquad  (\mathfrak{A}u)^{j}:=\ssomega_{j}^2 u_{j},
\end{align*}
with $\mathcal{D} (\mathfrak{A})$ being its maximal domain of definition. Assume that $\mathfrak{f}=\mathfrak{f}^{(2)}+ \mathfrak{f}^{(3)}$, where $\mathfrak{f}^{(2)}$ and $\mathfrak{f}^{(3)}$ admit the representations \eqref{def:f2F} and \eqref{def:f3F} respectively. Moreover, assume that both $\mathfrak{f}^{(2)}$ and $\mathfrak{f}^{(3)}$ are differentiable, with Lipschitz differentials and define the \textit{modified} operator
	\begin{align*}
	 \textswab{M}_{\pm}(\xi) &=  \pm \mathfrak{A} \xi +     \langle \mathfrak{f}^{(3)}  \rangle( \xi)+\mathfrak{F}_{0}( \xi) ,
\end{align*} 
where $\mathfrak{F}_{0}( \xi)= \{\left( \mathfrak{F}_{0}( \xi) \right)^m : m\geq 0 \}$ is a bounded map on $l_s^2$ that is given by
\begin{align*}
		\left( \mathfrak{F}_{0}( \xi) \right)^m=& \frac{9}{4} \sum_{\kappa,\nu \geq 0} \overline{\mathfrak{C}}_{\kappa \nu m} \sum_{ \substack{i,j \geq 0 \\ \ssomega_{i}-\ssomega_{j} \neq \pm \ssomega_{\nu}  }  } \frac{\overline{\mathfrak{C}}_{ij\nu }}{\ssomega_{\nu}^2-(\ssomega_{i}-\ssomega_{j})^2  } \xi^i \xi^j \xi^{\kappa} \sum_{\pm} \mathds{1}(\ssomega_{i}-\ssomega_{j}\pm \ssomega_{\kappa} \pm \ssomega_{m}=0) \\
		+ &\frac{9}{4} \sum_{\kappa,\nu \geq 0} \overline{\mathfrak{C}}_{\kappa \nu m} \sum_{ \substack{i,j \geq 0 \\ \ssomega_{i}+\ssomega_{j} \neq \pm \ssomega_{\nu}  }  } \frac{\overline{\mathfrak{C}}_{ij\nu }}{\ssomega_{\nu}^2-(\ssomega_{i}+\ssomega_{j})^2  } \xi^i \xi^j \xi^{\kappa} \sum_{\pm} \mathds{1}(\ssomega_{i}+\ssomega_{j}\pm \ssomega_{\kappa} \pm \ssomega_{m}=0).
	\end{align*}
Also, let $\xi_0 \in l_{s+3}^2$ be such that
\begin{itemize}
	\item $\xi_0$ is a \textit{zero} of $ \textswab{M}_{\pm}$,  
		 \begin{align*}
		 \textswab{M}_{\pm}\left(\xi _0 \right)=0,
	\end{align*}
	\item and $\xi_0$ satisfies the following \textit{non--degeneracy condition}
\begin{align*}
	\ker \left(d  \textswab{M}_{\pm}\left(\xi _0 \right)\right)=\{ 0 \}.
\end{align*}
\end{itemize}
Then, there exists a family $\{u_{\epsilon}(t,\cdot):\epsilon \in \mathcal{E}_{\alpha}\}$ of time--periodic solutions to \eqref{BambusiMainPDE}--\eqref{Splitingf}--\eqref{ConditionNonResonantf2} where $\mathcal{E}_{\alpha} $ is an uncountable set that has zero as an accumulation point. In addition, each element $u_{\epsilon} $ has the following properties: 
\begin{enumerate}
	\item $u_{\epsilon} $ has period  $T_{\epsilon}=2\pi / \ssomega_{\epsilon} $ where there exists $\ssomega_{\star}>0$ such that the maps 
	\begin{itemize}
		\item $\epsilon \longmapsto \ssomega_{\epsilon}\in \mathcal{W}_{\alpha}\cap [1,1+\ssomega_{\star}) $, for $ \textswab{M}_{+}$
		\item $\epsilon \longmapsto \ssomega_{\epsilon}\in \mathcal{W}_{\alpha}\cap (1-\ssomega_{\star},1] $, for $ \textswab{M}_{-}$
	\end{itemize}
	 are monotone, one-to-one map that stay close to one, $|1-\ssomega_{\epsilon}| \lesssim \epsilon$,
	 \item $u_{\epsilon} \in H^{1} ([0,T_{\epsilon}];l_s^2])$,
	\item $u_{\epsilon} $ stays close to the solution to the linearized equation with the same initial data as above and zero initial velocity, 
	\begin{align*}
			\sup_{t\in \mathbb{R}} \left |u_{\epsilon}(t,\cdot)-\pPhi^{t \ssomega_{\epsilon}} \left(\epsilon \xi_0 \right ) \right |_{s} \lesssim \epsilon^2.
	\end{align*} 
	\end{enumerate}
\end{theorem}

The rest of this section is devoted to the proof of the theorem above.  
\subsection{Preliminaries} \label{SubsectionPreliminaries}
The core of the proof follows that of \cite{MR1819863}. Let $0<\alpha< 1/ 3$ and pick a frequency $\ssomega \in \mathcal{W}_{\alpha}$. We are looking for a solution to \eqref{BambusiMainPDE} with frequency $\ssomega$, that is
\begin{align} \label{def:q}
	u(t):=q(\ssomega t).
\end{align}
For any integer $k \geq 0$, we define the Banach space 
	\begin{align*}
	\mathcal{H}^{k}_s   \subseteq H^{k}([0,2\pi];l_s^2)
\end{align*} 
consisting of spacetime functions
\begin{align*}
	q(t)= \sum_{j=0}^{\infty} q^{j}(t) e_{j}=
	 \sum_{j=0}^{\infty} \left(\sum_{l=0}^{\infty}q^{lj} \cos(lt) \right) e_{j}
\end{align*} 
such that the norm
\begin{align*}
	\| q\|_{\mathcal{H}^{k}_s}^2 := \sum_{j=0}^{\infty} j^{2s} \left(2 |q^{0j}|^2  + \sum_{l=1}^{\infty} |q^{lj}|^2 (1+l^2)^k  \right)
\end{align*}
is finite. In particular, we aim towards constructing $q$ in the Hilbert space $\mathcal{H}^{1}_s $. To do so, we substitute \eqref{def:q} into \eqref{BambusiMainPDE} and obtain the non--linear equation
\begin{align}\label{BambusiMainPDE2}
	L_{\ssomega}q=f(q) ,
\end{align}
where
\begin{align*}
L_{\ssomega}:\mathcal{D}(L_{\ssomega})\subset \mathcal{H}^{1}_s \longrightarrow \mathcal{H}^{1}_s, \Hquad  	L_{\ssomega}q:=\ssomega^2 \frac{d^2}{dt^2} q  +\mathfrak{A}q.
\end{align*}
Now, we are looking for a solution with frequency close to one. For this reason, we split $\mathcal{H}^{1}_s$ into
\begin{align*}
	\mathcal{H}^{1}_s=K \oplus R, \Hquad  K:=\ker(L_1), \Hquad R:=K^{\perp},
\end{align*}
and write
\begin{align*}
	q \in \mathcal{H}^{1}_s, \Hquad q=v+q_{\perp}, \Hquad v\in K,\Hquad q_{\perp} \in R.
\end{align*} 
Taking into account the fact that $K$ is generated by $\{\cos(\ssomega_{j}t): j\geq 0 \}$, since
\begin{align*}
	v\in K \Longleftrightarrow v(t)=\{v^{j}(t)=c ^{j} \cos(\ssomega_{j}t),  \Hquad j\geq 0 \},
\end{align*}
for some constants $c ^{j} $ , the latter simply means that we split $q=\{q^j : j \geq 0\} \in \mathcal{H}^{1}_s$ into
\begin{align*}
	  q^{j}(t)=v^{j}(t)+q_{\perp}^{j}(t), \Hquad v^{j}(t)=c ^{j}  \cos(\ssomega_{j}t), \Hquad q_{\perp}^{j}(t)=\sum_{l \neq \ssomega_{j}} d^{jl} \cos(l  t),
\end{align*}
for some constants $c ^{j} $ and $ d^{jl} $. In addition, we define the associated projections
\begin{align*}
        &  P:\mathcal{H}^{1}_s \longrightarrow R, \Hquad P(q) =P(v+q_{\perp}):=q_{\perp},\\
	  &   Q:\mathcal{H}^{1}_s \longrightarrow K, \Hquad Q(q ) =Q(v+q_{\perp}):=v,
\end{align*}
and project \eqref{BambusiMainPDE2} onto $R$ and $K$ respectively. We obtain the following coupled non--linear system
\begin{align}
	L_{\ssomega} q_{\perp} &= 
	 P  f(v+q_{\perp}) , \label{PEquation} \\
	-2\beta A v &= 
	 Q  f(v+q_{\perp}) , \label{QEquation} 
\end{align} 
where we also set
\begin{align}\label{OmegaEquation}
	\ssomega^2=1+2\beta.
\end{align}
As is usual in this setting, we refer to \eqref{PEquation} and \eqref{QEquation} as the $P-$equation and $Q-$equation respectively.
\subsection{Solution to the P--equation}
As we will now see, the Diophantine condition $\ssomega \in \mathcal{W}_{\alpha}$ guarantees the existence of a solution to the $P-$equation.

\begin{lemma}[Solution to the $P-$equation, Lemma 4.6 in \cite{MR1819863}]\label{LemmaSolutionPequation}
	Let $0< \alpha < 1/ 3$ and pick $\ssomega \in \mathcal{W}_{\alpha}$. Then, the operator $L_{\ssomega}$ restricted to $R$ admits a bounded inverse 
	\begin{align*}
		L_{\ssomega}^{-1}: \mathcal{H}^1_s \cap R \longrightarrow \mathcal{H}^1_s \cap R
	\end{align*}
	such that
	\begin{align*}
		\| L_{\ssomega}^{-1} \| \leq c_0 \alpha^{-1},
	\end{align*}
	for some positive constant $c_0$. Moreover, there exists $\rho=\rho(\alpha)>0$  and a $C^1-$function 
	\begin{align*}
		q_{\perp}:B_{\rho} \longrightarrow R, \Hquad  
		v \longmapsto q_{\perp}(v)
	\end{align*} 
	that solves the $P-$equation, where $B_{\rho}$ denotes the ball of radius $\rho$ in $K$ centered at zero. Furthermore, we have the estimates
	\begin{align*} 
		\|q_{\perp}(v) \|_{\mathcal{H}^1_s}  \lesssim_{\alpha} \| v \|_{\mathcal{H}^1_s}^{2}, \Hquad 
		\|q_{\perp}(v) - L_{\ssomega}^{-1} P \mathfrak{f}^{(2)} (v) \|_{\mathcal{H}^1_s}  \lesssim_{\alpha} \| v \|_{\mathcal{H}^1_s}^{3}. 
	\end{align*} 
\end{lemma}
\begin{proof}
Apart from the $C^1$ regularity of $q_\perp$ (which is stated only as Lipschitz in \cite{MR1819863}), the proof coincides with the one of Lemma 4.6 in \cite{MR1819863}, where the $f^{(0)}$ in \cite{MR1819863} is replaced by $\mathfrak{f}^{(2)}$. Once a Lipschitz solution $q_\perp$ has been found, one can read--off the $C^1$ regularity of $q_\perp$ based on the regularity of $\mathfrak{f}$. However, for the convenience of the reader, we give a proof below of the construction of $q_\perp$. Let $0< \alpha < 1/ 3$ and pick $\ssomega \in \mathcal{W}_{\alpha}$. The eigenvalues of $L_{\ssomega}$ are given by
	\begin{align}\label{EigenvaluesofLomega}
		\lambda_{jl}=\ssomega_{j}^2-l^2 \ssomega^2 = (\ssomega_{j}-l\ssomega)(\ssomega_{j}+l\ssomega).
	\end{align}
	Then, for all $(l,j) \in \mathbb{N}^2$ with $l\geq 1$ and $l\neq \ssomega_{j}$, we have
	\begin{align*}
		|\lambda_{jl}|\geq \frac{\alpha}{l} (\ssomega_{j}+l\ssomega) \geq   \alpha  \ssomega 
		 \geq \frac{\alpha}{2}.
	\end{align*}
	Therefore, $L_{\ssomega}|_{R}$ has a bounded inverse and there exists a positive constant $c_0$ such that
	\begin{align*}
		\| L_{\ssomega}^{-1} \| \leq c_0 \alpha^{-1}.
	\end{align*} 
	 In addition, we let $\epsilon>0$ be sufficiently small, $\| v\|_{\mathcal{H}^1_s } \leq \epsilon$, $\delta >0$ sufficiently large, define the closed ball of radius $\delta \|v \|_{\mathcal{H}^1_s}^{3}$ centered at $L_{\ssomega}^{-1} P \mathfrak{f}^{(2)} (v)$, that is
	\begin{align*}
		B:=\{w \in \mathcal{H}^1_s: \Hquad \|w- L_{\ssomega}^{-1} P \mathfrak{f}^{(2)} (v)\|_{\mathcal{H}^1_s} \leq \delta \| v\|_{\mathcal{H}^1_s}^{3} \},
	\end{align*}
	and rewrite the the $P-$equation in the fixed point formulation as follows
	\begin{align*}
		 q_{\perp} &=\mathcal{F}(q_{\perp}):=  L_{\ssomega}^{-1}\left[ P \mathfrak{f}^{(2)} (v)+P\left(  \mathfrak{f}^{(2)} (v+q_{\perp})- \mathfrak{f}^{(2)} (v) \right)+P  \mathfrak{f}^{(3)} (v+q_{\perp})  \right]  .
	\end{align*}
	Next, we show that $\mathcal{F}$ maps the closed ball to itself. Indeed, for all $w \in B$, we have
	\begin{align*}
		\|w \|_{\mathcal{H}^1_s } & \leq \|w - L_{\ssomega}^{-1} P \mathfrak{f}^{(2)} (v) \|_{\mathcal{H}^1_s }+\|L_{\ssomega}^{-1} P \mathfrak{f}^{(2)} (v) \|_{\mathcal{H}^1_s } \\
		& \leq \delta \| v\|_{\mathcal{H}^1_s }^{3} +\|L_{\ssomega}^{-1}  \| \| \mathfrak{f}^{(2)} (v) \|_{\mathcal{H}^1_s } \\
		& \leq \delta \| v\|_{\mathcal{H}^1_s }^{3} + c_0 \alpha^{-1}k_s \|v \|_{\mathcal{H}^1_s }^2 \\
		& \leq c_1 \|v \|_{\mathcal{H}^1_s }^2
	\end{align*}
	and Lemma \ref{LipschitzModel3} implies
\begin{align*}
	\|   \mathfrak{f}^{(2)} (v+w)- \mathfrak{f}^{(2)} (v) \|_{\mathcal{H}^1_s } & \leq k_s \left(\|v+w \|_{\mathcal{H}^1_s }+\|v\|_{\mathcal{H}^1_s }\right) \|w\|_{\mathcal{H}^1_s } \\
	& \leq k_s \left(\|w \|_{\mathcal{H}^1_s }+2\|v\|_{\mathcal{H}^1_s }\right) \|w\|_{\mathcal{H}^1_s } \\
	& \leq c_1  k_s \left(c_1 \|v \|_{\mathcal{H}^1_s }^2+2\|v\|_{\mathcal{H}^1_s }\right) \|v \|_{\mathcal{H}^1_s }^2 \\
	& \leq c_2   \|v \|_{\mathcal{H}^1_s }^3, \\
	\|   \mathfrak{f}^{(3)} (v+w)  \|_{\mathcal{H}^1_s } & \leq k_s \|v+w \|_{\mathcal{H}^1_s }^3 \\
	& \lesssim k_s \left( \|v \|_{\mathcal{H}^1_s }^3 + c_1^3 \|v \|_{\mathcal{H}^1_s }^6  \right) \\
	& \leq c_3 \|v \|_{\mathcal{H}^1_s }^3.
\end{align*}
	Hence, we infer
	\begin{align*}
		\| \mathcal{F}(w)-L_{\ssomega}^{-1} P \mathfrak{f}^{(2)} (v)\|_{\mathcal{H}^1_s } & =
		\| L_{\ssomega}^{-1} \left[  P\left(  \mathfrak{f}^{(2)} (v+w)- \mathfrak{f}^{(2)} (v) \right)+P  \mathfrak{f}^{(3)} (v+w)  \right]  \|_{\mathcal{H}^1_s } \\
		&\leq \| L_{\ssomega}^{-1}\| \left[ \|   \mathfrak{f}^{(2)} (v+w)- \mathfrak{f}^{(2)} (v) \|_{\mathcal{H}^1_s } + \|  \mathfrak{f}^{(3)} (v+w)  \|_{\mathcal{H}^1_s } \right] \\
		& \leq  c_0 \alpha^{-1} \left[ c_2 \|v \|_{\mathcal{H}^1_s }^3 + c_3 \|v \|_{\mathcal{H}^1_s }^3  \right] \\
		& \leq \delta \|v \|_{\mathcal{H}^1_s }^3,
	\end{align*}
	by choosing $\delta$ sufficiently large. The contraction property follows similarly. For the $C^1$ regularity, we set $\mathfrak{F}^{(2)}:= L_{\ssomega}^{-1} P \mathfrak{f}^{(2)}$ and, for $v, v+h \in B_{\rho}$, we have
	\begin{align*}
		q_\perp(v)= \mathfrak{F}^{(2)}(v+q_\perp(v)), \Hquad q_\perp(v+h)= \mathfrak{F}^{(2)}(v+h+q_\perp(v+h)),
		\end{align*}
	so that 
	\begin{align*}
		q_\perp(v+h) &= \mathfrak{F}^{(2)}(v+q_\perp(v))+ d\mathfrak{F}^{(2)}_{v+q_\perp(v)} \left( h+q_\perp(v+h)-q_\perp(v)\right)+\smallO (h+q_\perp(v+h)-q_\perp(v)) \\
		&=q_\perp(v)+d\mathfrak{F}^{(2)}_{v+q_\perp(v)} \left( h+q_\perp(v+h)-q_\perp(v)\right)+ \smallO (h) ,
		\end{align*}
			using that $q_{\perp}$ is Lipschitz. Assuming that $\rho$ is small enough, we can ensure that $ \| d\mathfrak{F}^{(2)}_{v+q_\perp(v)} \|_{\mathcal{H}^1_s } \leq c \| v\|_{\mathcal{H}^1_s } <  1-\delta$, uniformly in $v$, and hence
\begin{align*}
	q_\perp(v+h)	= 	q_\perp(v)+(\text{Id}- d\mathfrak{F}^{(2)}_{v+q_\perp(v)})^{-1} d\mathfrak{F}^{(2)}_{v+q_\perp(v)} \left( h\right)+\smallO (h),
\end{align*}
so that $q_\perp(v)$ is $C^1$ with differential $(\text{Id}- d\mathfrak{F}^{(2)}_{v+q_\perp(v)})^{-1} d\mathfrak{F}^{(2)}_{v+q_\perp(v)}$. 
	
\end{proof}
 
\subsection{Solution to the Q--equation}\label{SubsectionSolutionQEquation}
Next, we turn our attention to the existence of a solution to the $Q-$equation. Firstly, we define the following two Banach spaces of initial data,
\begin{align*}
	\mathcal{Q}:= \left\{\xi = \sum_{j=0}^{\infty} \xi^j e_j: \Hquad  \|  \xi \|_{\mathcal{Q}}^2< \infty 	\right\}\simeq l_{s+1}^2  \subseteq l_s^2,
\end{align*}
endowed with the norm
\begin{align*}
	\|  \xi \|_{\mathcal{Q}}^2 := 
	\sum_{j=0}^{\infty} j^{2s} \left| \xi^j \right|^2+
	\sum_{j=0}^{\infty} j^{2s} \left|\ssomega_{j} \xi^j \right|^2
	  \simeq
		\sum_{j=0}^{\infty} j^{2(s+1)} \left| \xi^j \right|^2  = |\xi|^2_{s+1},
\end{align*}
and 
\begin{align*}
	\mathcal{D}(\mathfrak{A}):= \left\{\xi = \sum_{j=0}^{\infty} \xi^j e_j: \Hquad  \|  \xi \|_{\mathcal{D}(\mathfrak{A})}^2< \infty 	\right\}\simeq l_{s+2}^2  \subseteq l_s^2,
\end{align*}
endowed with the norm
\begin{align*}
	\|  \xi \|_{\mathcal{D}(\mathfrak{A})}^2 := 
	|\xi|_{s}^2 + |\mathfrak{A} \xi|_{s}^2=
	\sum_{j=0}^{\infty} j^{2s} \left| \xi^j \right|^2+
	\sum_{j=0}^{\infty} j^{2s} \left|\ssomega_{j}^2 \xi^j \right|^2
	  \simeq
		\sum_{j=0}^{\infty} j^{2(s+2)} \left| \xi^j \right|^2  = |\xi|^2_{s+2},
\end{align*}
since $\ssomega_{j} \sim j$, as $j \rightarrow \infty$. We call the Hilbert space $(\mathcal{Q},\|\cdot \|_{\mathcal{Q}})$ the configuration space.  In fact, $\mathcal{Q}$ is isomorphic to $K:=\ker (L_{1})$ and the isomorphism is given by the linear flow 
	\begin{align*}
		I:\mathcal{Q}\longrightarrow K,\Hquad \Hquad (I(x))(t):=\pPhi^{t}(x).
	\end{align*} 
Also, recall the Banach space the Banach space of spacetime functions 
\begin{align*}
	\mathcal{H}^{k}_s := \left\{q(t)= \sum_{j=0}^{\infty} q^{j}(t) e_{j}=
	 \sum_{j=0}^{\infty} \left(\sum_{l=0}^{\infty}q^{lj} \cos(lt) \right) e_{j}: \Hquad \| q\|_{\mathcal{H}^{k}_s}^2 <\infty \right\} \subseteq H^{k}([0,2\pi];l_s^2)
\end{align*}
endowed with the norm
\begin{align*}
	\| q\|_{\mathcal{H}^{k}_s}^2 := \sum_{j=0}^{\infty} j^{2s} \left(2 |q^{0j}|^2  + \sum_{l=1}^{\infty} |q^{lj}|^2 (1+l^2)^k  \right).
\end{align*}
Notice that, since
\begin{align*}
	I(\xi)(t)=\sum_{j=0}^{\infty} (I(\xi)(t))^j e_j
	=\sum_{j=0}^{\infty} (\Phi^{t}(\xi))^j e_j = \sum_{j=0}^{\infty} \xi^j \cos(\ssomega_{j }t) e_j 
\end{align*}
and $\ssomega_j  \neq 0$, for all integers $j \geq 0$, we have that
\begin{align}
	\| I(\xi) \|_{\mathcal{H}^{0}_s }^2 &=\sum_{j=0}^{\infty} j^{2s}  |\xi^{j}|^2   =    | \xi |_{s}^2, \nonumber  \\
	\| I(\xi) \|_{\mathcal{H}^{1}_s }^2 &=\sum_{j=0}^{\infty} j^{2s}  |\xi^{j}|^2 (1+\ssomega_{j}^2) =  \sum_{j=0}^{\infty} j^{2s} \left( |\xi^{j}|^2+ |\ssomega_{j} \xi^{j}|^2   \right)  \simeq | \xi |_{s+1}^2  \simeq  \| \xi \|_{\mathcal{Q}}^2,\label{HkNormWithQ} \\
	\| I(\xi) \|_{\mathcal{H}^{2}_s }^2 &=\sum_{j=0}^{\infty} j^{2s}  |\xi^{j}|^2 (1+\ssomega_{j}^2)^2 \simeq  \sum_{j=0}^{\infty} j^{2s} \left( |\xi^{j}|^2+ |\ssomega_{j}^2 \xi^{j}|^2   \right)  \simeq | \xi |_{s+2}^2  \simeq  \| \xi \|_{\mathcal{D}(\mathfrak{A})}^2.\label{HkNormWithDA}
\end{align}

Secondly, we prove the following averaging identity that generalizes the one in Lemma 4.7 in Bambusi--Paleari \cite{MR1819863} from vector fields $F: \mathcal{Q}  \rightarrow \mathcal{Q}  $ to $F: \mathcal{H}^k_s \rightarrow \mathcal{H}^k_s $.
\begin{lemma}[Averaging Identity]\label{LemmaIdentity}
	Let $F: \mathcal{H}^k_s \rightarrow \mathcal{H}^k_s $ be any vector field. Then, for all $x\in l_s^2$, we have
	\begin{align*}
	      \langle F \rangle(x):=\frac{1}{2\pi }\int_{0}^{2\pi }   \Phi^{t}  \left[F \left(w \right)\right]    dt = \frac{1}{2}I^{-1} Q \left[F(I(x)) \right].
	\end{align*}
\end{lemma}
\begin{proof}
	Let $F: \mathcal{H}^k_s \rightarrow\mathcal{H}^k_s$ be a vector field in $\mathcal{H}^k_s$ (not necessarily in $\mathcal{Q}$), pick any $x \in l_s^2$ and set $w=I(x)$. By the definition of the Banach space $ \mathcal{H}^k_s$, we have
\begin{align*} 
	F \left(w \right)&= \sum_{m=0}^{\infty}\left(F \left(w \right) \right)^m e_m  =  \sum_{m=0}^{\infty}\left( \sum_{l=0}^{\infty} \left(F \left(w \right) \right)_{l}^m \cos(lt)\right)e_m ,
\end{align*}	
	where
\begin{align*}
	 \left(F \left(w \right) \right)_{l}^m  = \frac{1}{\pi } \int_{0}^{2\pi }  \left(F \left(w \right) \right)^m \cos(l t) dt.
\end{align*}
Then, the definition of the linear flow together with the definition of the projection $Q$ yield
\begin{align*}	
Q \left[	F \left(w \right) \right] &= \sum_{m=0}^{\infty}\left( Q \left[F \left(w \right)\right] \right)^m 
e_m  = \sum_{m=0}^{\infty}   \left(F \left(w \right)\right)_{\omega_{m}}^{m} \cos(\ssomega_{m}t)   e_m ,\\  
	 I^{-1}Q \left[	F \left(w \right) \right] &=\sum_{m=0}^{\infty}   \left(F \left(w \right)\right)_{\ssomega_{m}}^{m}   e_m,\\
	 \frac{1}{2} I^{-1}Q \left[	F \left(w \right) \right] &=\frac{1}{2}\sum_{m=0}^{\infty}    \left(F \left(w \right)\right)_{\ssomega_{m}}^{m}  e_m \\ 
&=\frac{1}{2}\sum_{m=0}^{\infty}   \left(  \frac{1}{\pi }\int_{0}^{2\pi } \left(F \left(w \right)\right)^{m} \cos(\ssomega_{m} t) dt \right)  e_m \\
&= \sum_{m=0}^{\infty}   \left(  \frac{1}{2\pi }\int_{0}^{2\pi } \left(F \left(w \right)\right)^{m} \cos(\ssomega_{m} t) dt \right)  e_m \\
&= \sum_{m=0}^{\infty}   \left(  \frac{1}{2\pi }\int_{0}^{2\pi } \left( \Phi^{t}  \left[F \left(w \right)\right]  \right)^m dt \right)  e_m \\
&=  \frac{1}{2\pi }\int_{0}^{2\pi } \sum_{m=0}^{\infty}   \left( \left( \Phi^{t}  \left[F \left(w \right)\right]  \right)^m  e_m  \right) dt  \\
&=  \frac{1}{2\pi }\int_{0}^{2\pi }   \Phi^{t}  \left[F \left(w \right)\right]    dt :=\langle F \rangle(x),
\end{align*}
that competes the proof.
\end{proof}
Then, we express the $Q-$equation in the configuration space introduced above.
\begin{lemma}[The $Q-$equation in the configuration space]
	Let $\rho >0$ and $q_{\perp}: B_\rho  \subset K \rightarrow R$ be the solution map to the P--equation derived in Lemma \ref{LemmaSolutionPequation}. Also, let $x \in l_{s+2}^2$ and set $v = I(x) \in B_\rho$. Then, the $Q-$equation \eqref{QEquation} for $v$ is equivalent to
	\begin{align}\label{NewQequation}
		\beta \mathfrak{A} x + \langle  \mathfrak{f} \rangle(x)= - \frac{1}{2} I^{-1} Q \left[ 
		 \mathfrak{f}  \left( I(x)+ q_{\perp}(I(x)) \right)- \mathfrak{f}\left(I(x) \right) \right].
	\end{align}
\end{lemma}
\begin{proof}
Let $\rho >0$ and $q_{\perp}: B_\rho  \subset K \rightarrow R$ be the solution map to the $P-$equation derived in Lemma \ref{LemmaSolutionPequation}. Also, let $x \in \mathcal{Q}$ and set $v = I(x) \in B_\rho$. Then, we rewrite the $Q-$equation given in \eqref{QEquation}, that is
\begin{align*}
	-2\beta A v =  Q  f(v+q_{\perp})  ,
\end{align*}
as follows
\begin{align*}
	- 2\beta \mathfrak{A} I(x)  - Q  \mathfrak{f} ( I(x)  )  =   Q  \mathfrak{f} (I(x) +q_{\perp}(I(x) ))  - Q  \mathfrak{f}(I(x) )  .
\end{align*}
Since $\mathfrak{A} I(x)= I \mathfrak{A} x$, the latter is equivalent to
\begin{align*}
	 - 2\beta I  \mathfrak{A} x - Q  \mathfrak{f}( I(x)  ) =   Q \left[ \mathfrak{f}( I(x) +q_{\perp}( I(x) ))  -   \mathfrak{f}( I(x) ) \right] 
\end{align*}
and, by applying $-(1/2)I^{-1}$ to both sides, we get	
\begin{align*}
	  \beta    \mathfrak{A} x + \frac{1}{2} I^{-1} Q  \mathfrak{f} I(x) )  = -\frac{1}{2} I^{-1}  Q \left[ \mathfrak{f}( I(x) +q_{\perp}( I(x) ))  -   \mathfrak{f}( I(x) ) \right] .
\end{align*}
Now, the claim follows by the averaging identity due to Lemma \ref{LemmaIdentity}.
\end{proof}
It remains to show that there exists a solution to \eqref{NewQequation}. To this end, we define
 \begin{align}\label{XiandBetaIntermsofe}
 	x=\epsilon \xi, \Hquad |\beta|=\epsilon^2
 \end{align}
and \eqref{NewQequation} becomes
\begin{align*}
	& \pm \epsilon^2 \mathfrak{A} (\epsilon \xi) +  \langle \mathfrak{f}  \rangle(\epsilon \xi)  
	= -\frac{1}{2} I^{-1} Q \left[ 
		 \mathfrak{f}  \left( I(\epsilon \xi)+ q_{\perp}(I(\epsilon \xi)) \right)- \mathfrak{f}\left(I(\epsilon \xi) \right) \right] .
\end{align*}
On the one hand, \eqref{Splitingf} and \eqref{ConditionNonResonantf2} yield
\begin{align*}
	\pm \epsilon^2 \mathfrak{A} (\epsilon \xi) +  \langle \mathfrak{f}  \rangle(\epsilon \xi) =
	\pm \epsilon^3 \mathfrak{A} \xi + \epsilon^2 \langle \mathfrak{f}^{(2)}  \rangle( \xi)+ \epsilon^3 \langle \mathfrak{f}^{(3)}  \rangle( \xi) =
	\epsilon^3 \left( \pm  \mathfrak{A} \xi +     \langle \mathfrak{f}^{(3)}  \rangle( \xi) \right).
\end{align*}
On the other hand, \eqref{Splitingf}, \eqref{ConditionNonResonantf2} and the averaging identity from Lemma \ref{LemmaIdentity} yield
\begin{align*}
	& \frac{1}{2} I^{-1} Q \left[ 
		 \mathfrak{f}  \left( I(\epsilon \xi)+ q_{\perp}(I(\epsilon \xi)) \right)- \mathfrak{f}\left(I(\epsilon \xi) \right) \right] = \\
	&	 \frac{1}{2} I^{-1} Q \left[ 
		 \mathfrak{f}^{(2)}  \left( I(\epsilon \xi)+ q_{\perp}(I(\epsilon \xi)) \right)- \mathfrak{f}^{(2)} \left(I(\epsilon \xi) \right) \right] +
		  \frac{1}{2} I^{-1} Q \left[ 
		 \mathfrak{f}^{(3)}  \left( I(\epsilon \xi)+ q_{\perp}(I(\epsilon \xi)) \right)- \mathfrak{f}^{(3)} \left(I(\epsilon \xi) \right) \right]= \\
&	 \frac{1}{2} I^{-1} Q \left[ 
		 \mathfrak{f}^{(2)}  \left( I(\epsilon \xi)+ q_{\perp}(I(\epsilon \xi)) \right)  \right]  -\epsilon ^2 \langle \mathfrak{f}^{(2)} \rangle(\xi)  +
		  \frac{1}{2} I^{-1} Q \left[ 
		 \mathfrak{f}^{(3)}  \left( I(\epsilon \xi)+ q_{\perp}(I(\epsilon \xi)) \right) - \mathfrak{f}^{(3)}(I(\epsilon \xi)) \right]  = \\
 &	 \frac{1}{2} I^{-1} Q \left[ 
		 \mathfrak{f}^{(2)}  \left( I(\epsilon \xi)+ q_{\perp}(I(\epsilon \xi)) \right)  \right]  +
		  \frac{1}{2} I^{-1} Q \left[ 
		 \mathfrak{f}^{(3)}  \left( I(\epsilon \xi)+ q_{\perp}(I(\epsilon \xi)) \right) - \mathfrak{f}^{(3)}(I(\epsilon \xi)) \right] = \\		 
	& 	 \frac{1}{2} I^{-1} Q \left[ 
		 \mathfrak{f}^{(2)}  \left( I(\epsilon \xi)+ L_{\ssomega}^{-1} P \mathfrak{f}^{(2)} (I(\epsilon \xi)) \right)  \right]+  \epsilon^3 \mathfrak{G}_{\epsilon}(\xi),
\end{align*} 
where we set
\begin{align*}
	\mathfrak{G}_{\epsilon}(\xi) &:= \epsilon^{-3}\Bigg[ \frac{1}{2} I^{-1} Q \left[ 
		 \mathfrak{f}^{(2)}  \left( I(\epsilon \xi)+ q_{\perp}(I(\epsilon \xi)) \right)   -
		 \mathfrak{f}^{(2)}  \left( I(\epsilon \xi)+ L_{\ssomega}^{-1} P \mathfrak{f}^{(2)} (I(\epsilon \xi)) \right)  \right]\nonumber \\
		 &   +
		  \frac{1}{2} I^{-1} Q \left[ 
		 \mathfrak{f}^{(3)}  \left( I(\epsilon \xi)+ q_{\perp}(I(\epsilon \xi)) \right) - \mathfrak{f}^{(3)}(I(\epsilon \xi)) \right]  \Bigg]. \nonumber
\end{align*}
In addition, we apply the averaging identity and the notation from Lemma \ref{LemmaIdentity} to the map $v  \rightarrow \mathfrak{f}^{(2)}  \left( v+ L_{1}^{-1} P \mathfrak{f}^{(2)}(v) \right)$ that is a vector field from $\mathcal{H}^1_s $ to $\mathcal{H}^1_s $ to obtain
\begin{align*}
	& \frac{1}{2} I^{-1} Q \left[ 
		 \mathfrak{f}  \left( I(\epsilon \xi)+ q_{\perp}(I(\epsilon \xi)) \right)- \mathfrak{f}\left(I(\epsilon \xi) \right) \right] = \\
		 &  \frac{1}{2} I^{-1} Q \left[ 
		 \mathfrak{f}^{(2)}  \left( I(\epsilon \xi)+ L_{\ssomega}^{-1} P \mathfrak{f}^{(2)} (I(\epsilon \xi)) \right)  \right]+  \epsilon^3 \mathfrak{G}_{\epsilon}(\xi) = \\
		 &  \frac{1}{2} I^{-1} Q \left[ 
		 \mathfrak{f}^{(2)}  \left( I(\epsilon \xi)+ L_{1}^{-1} P \mathfrak{f}^{(2)} (I(\epsilon \xi)) \right)  \right]+  \epsilon^3 \mathfrak{G}_{\epsilon}(\xi) + \epsilon^3  \mathfrak{R}_{\epsilon}(\xi,\ssomega) = \\
		 &\frac{1}{2} \int_0^{2 \pi} \Phi^{-t} \left(   \mathfrak{f}^{(2)}  \left( \Phi^t(\epsilon \xi)+ L_{1}^{-1} P \mathfrak{f}^{(2)}(\Phi^t(\epsilon \xi)) \right)\right)dt +  \epsilon^3 \mathfrak{G}_{\epsilon}(\xi) +  \epsilon^3 \mathfrak{R}_{\epsilon}(\xi,\ssomega) = \\
		 &\langle  
\mathfrak{f}^{(2)}  \left( ( \cdot ) + L_{1}^{-1} P \mathfrak{f}^{(2)} ( \cdot ) \right)   \rangle(\epsilon \xi) +  \epsilon^3 \mathfrak{G}_{\epsilon}(\xi) +  \epsilon^3 \mathfrak{R}_{\epsilon}(\xi,\ssomega)  = \\
&   \epsilon^3 \left(\mathfrak{F}_{\epsilon}( \xi)+   \mathfrak{G}_{\epsilon}(\xi) + \mathfrak{R}_{\epsilon}(\xi,\ssomega) \right) ,
\end{align*}
where we set
\begin{align*}
	 \mathfrak{R}_{\epsilon}(\xi,\ssomega)  &:=  \epsilon^{-3} \frac{1}{2} I^{-1} Q \left[ 
		 \mathfrak{f}^{(2)}  \left( I(\epsilon \xi)+ L_{\ssomega}^{-1} P \mathfrak{f}^{(2)} (I(\epsilon \xi)) \right)  -  
		 \mathfrak{f}^{(2)}  \left( I(\epsilon \xi)+ L_{1}^{-1} P \mathfrak{f}^{(2)} (I(\epsilon \xi)) \right)  \right], \\
		 \mathfrak{F}_{\epsilon}( \xi) &:= 
		 \epsilon^{-3}	 \langle  
		 \mathfrak{f}^{(2)}  \left( ( \cdot ) + L_{1}^{-1} P \mathfrak{f}^{(2)} ( \cdot ) \right)   
		 \rangle(\epsilon \xi) .
\end{align*}
In conclusion, the Q--equation \eqref{NewQequation} can be written equivalently, for $\epsilon >0$ sufficiently small, as
\begin{align}\label{NewQequationConfiguration}
	 \pm  \mathfrak{A} \xi +    \langle \mathfrak{f}^{(3)}  \rangle( \xi)  =  -\left( \mathfrak{F}_{\epsilon}( \xi)+   \mathfrak{G}_{\epsilon}(\xi) + \mathfrak{R}_{\epsilon}(\xi,\ssomega)\right).
\end{align}
However, instead of \eqref{NewQequationConfiguration}, we focus on a modified version, namely 
\begin{align}\label{ModifiedNewQequationConfiguration}
	 \pm  \mathfrak{A} \xi +    \langle \mathfrak{f}^{(3)}  \rangle( \xi)  =  -\left( \mathfrak{F}_{\epsilon}( \xi)+   \mathfrak{G}_{\epsilon}(\xi) \pm  \frac{2\epsilon^2}{\ssomega^2-1}\mathfrak{R}_{\epsilon}(\xi,\ssomega)\right).
\end{align}
Notice that \eqref{NewQequationConfiguration} coincides with \eqref{ModifiedNewQequationConfiguration} provided that $\ssomega^2-1= \pm2\epsilon^2$.
\begin{remark}
	Since $\mathfrak{f}^{(2)}$ is differentiable, quadratic and $\langle \mathfrak{f}^{(2)}\rangle(\xi) =0$, for all initial data $\xi\in l_{s+3}^2$, it follows that $\mathfrak{F}_{\epsilon}( \xi)$ is differentiable and $\| \mathfrak{F}_{\epsilon}(\xi) \|_{\mathcal{Q}}  \lesssim \epsilon$. Later, (Section \ref{se:fe}), we compute the exact expressions of $\mathfrak{F}_{0}( \xi)$ for general initial data (Lemma \ref{LemmaDefinitionMathfrakF}),
	 $\mathfrak{F}_{\epsilon}( \xi)$ for small $\epsilon$ close to zero and 1--mode initial data (Lemma \ref{LemmaComputemathfrakFforsfmallepsilon}),
	as well as the differential $d\mathfrak{F}_{0}( \xi)$ at the 1--mode initial data (Lemma \ref{LemmaDifferentialMathFrakF0LongComputation}).
\end{remark}

In the following, we estimate the error terms. To begin with, we estimate $\mathfrak{R}_{\epsilon}(\xi,\ssomega)$.
\begin{lemma}[Estimate for $\mathfrak{R}_{\epsilon}(\xi,\ssomega)$ and $d_{\xi}\mathfrak{R}_{\epsilon}(\xi,\ssomega)$]\label{LemmaEstimateR}
	Let $0< \alpha<1/ 3$ and pick any $\ssomega \in \mathcal{W}_{\alpha}$. Also, let $\xi \in l_{s+3}^2$ be any initial data. Then, we have
	\begin{align*}
		\| \mathfrak{R}_{\epsilon}(\xi,\ssomega) \|_{\mathcal{Q}}  \lesssim  |\ssomega^2-1|, \Hquad \|d_{\xi} \mathfrak{R}_{\epsilon}(\xi,\ssomega) \|_{\mathcal{Q}}  \lesssim  |\ssomega^2-1|.
	\end{align*}
\end{lemma}
\begin{proof}
Let $0< \alpha<1/ 3$ and pick any $\ssomega \in \mathcal{W}_{\alpha}$. Also, let $\xi \in l_{s+3}^2$ be any initial data. Firstly, we pick any $\epsilon>0$, set $v=I(\epsilon \xi) $ and compute 
\begin{align*}
	\mathfrak{f}^{(2)}(v)&=\sum_{j=0}^{\infty} (\mathfrak{f}^{(2)}(v))^j e_j = \sum_{j=0}^{\infty}\left( \sum_{l=0}^{\infty} (\mathfrak{f}^{(2)}(v))_{l}^{j} \cos(lt) \right)e_j, \\ 
	P\mathfrak{f}^{(2)}(v)&=\sum_{j=0}^{\infty}\left( \sum_{ \substack{ l=0\\l \neq \ssomega_{j}  }}^{\infty}  (\mathfrak{f}^{(2)}(v))_{l}^{j} \cos(lt) \right)e_j, \\
	L_{\ssomega}^{-1} P\mathfrak{f}^{(2)}(v)&=\sum_{j=0}^{\infty}\left( \sum_{ \substack{ l=0\\l \neq \ssomega_{j}  }}^{\infty} \frac{1}{\ssomega_{j}^2-l^2 \ssomega^2} (\mathfrak{f}^{(2)}(v))_{l}^{j} \cos(lt) \right)e_j, \\
   \left( L_{\ssomega}^{-1}  -L_{1}^{-1} \right) P\mathfrak{f}^{(2)}(v)    & = 
	 \sum_{j=0}^{\infty}\left(  \sum_{ \substack{ l=0\\l \neq \ssomega_{j}  }}^{\infty} \left( \frac{1}{\ssomega_{j}^2-l^2 \ssomega^2}- \frac{1}{\ssomega_{j}^2-l^2  }\right) (\mathfrak{f}^{(2)}(v))_{l}^{j} \cos(lt) \right)e_j  \\
	&=\sum_{j=0}^{\infty}\left(  \sum_{ \substack{ l=0\\l \neq \ssomega_{j}  }}^{\infty}  \frac{l^2 (\ssomega^2-1)}{(\ssomega_{j}^2-l^2 \ssomega^2)(\ssomega_{j}^2-l^2)}  (\mathfrak{f}^{(2)}(v))_{l}^{j} \cos(lt) \right)e_j.
\end{align*}
Secondly, we note that 
\begin{align}\label{AuxilirlyBoundsNo2forLemma211}
	\| L_{\ssomega}^{-1} P\mathfrak{f}^{(2)}(I(\epsilon \xi)) \|_{\mathcal{H}^{1}_s }\lesssim \epsilon^2  ,\Hquad \| L_{1}^{-1} P\mathfrak{f}^{(2)}(I(\epsilon \xi)) \|_{\mathcal{H}^{1}_s } \lesssim \epsilon^2  .
\end{align}
These can be easily proved by using the Diophantine condition, the elementary inequality $|\ssomega_{j}^2-l^2| \geq 1$ (since $\ssomega \in \mathcal{W}_{\alpha}$, both $\ssomega_{j}^2 \geq 1$ and $l^2 \geq 0$ are integers with $\ssomega_{j} \neq l$), the Lipschitz estimate 
\begin{align*}
	\|  \mathfrak{f}^{(2)} (u)  \|_{\mathcal{H}^k_s} \lesssim_{s}  \|u \|_{\mathcal{H}^k_s}^2,
\end{align*}
for all $u \in \mathcal{H}^k_s$ with $\|u \|_{\mathcal{H}^k_s} \leq \epsilon$ that follows from Lemma \ref{LipschitzModel3}, together with \eqref{HkNormWithQ}. Indeed, we infer
\begin{align*}
	\|L_{\ssomega}^{-1} P\mathfrak{f}^{(2)}(I(\epsilon \xi )) \|_{\mathcal{H}^{1}_s}^2 &= \sum_{j=0}^{\infty} j^{2s}   \sum_{l=0}^{\infty} \left| \frac{1}{\ssomega_{j}^2-l^2 \ssomega^2} (\mathfrak{f}^{(2)}(I(\epsilon \xi )))_{l}^{j} \right|^2 (1+l^2)  \\
	& \lesssim_{\alpha}  \sum_{j=0}^{\infty} j^{2s}   \sum_{l=0}^{\infty} \left|   (\mathfrak{f}^{(2)}(I(\epsilon \xi )))_{l}^{j} \right|^2 (1+l^2)  \\
	& \leq \| \mathfrak{f}^{(2)}(I(\epsilon \xi )) \|_{\mathcal{H}^{1}_s}^2 \lesssim   \|  I(\epsilon \xi )  \|_{\mathcal{H}^{1}_s}^4  = \epsilon^4 \|  I( \xi )  \|_{\mathcal{H}^{1}_s}^4  \\
	& = \epsilon^4 \| \xi \|_{\mathcal{Q}}^4 \leq \epsilon^4 | \xi |_{s+1}^4\leq \epsilon^4 | \xi |_{s+3}^4 , \\
	\|L_{1}^{-1} P\mathfrak{f}^{(2)}(I(\epsilon \xi )) \|_{\mathcal{H}^{1}_s}^2 &= \sum_{j=0}^{\infty} j^{2s}   \sum_{l=0}^{\infty} \left| \frac{1}{\ssomega_{j}^2-l^2  } (\mathfrak{f}^{(2)}(I(\epsilon \xi )))_{l}^{j} \right|^2 (1+l^2)  \\
	& \lesssim   \sum_{j=0}^{\infty} j^{2s}   \sum_{l=0}^{\infty} \left|  (\mathfrak{f}^{(2)}(I(\epsilon \xi )))_{l}^{j} \right|^2 (1+l^2) \\
	& \leq \| \mathfrak{f}^{(2)}(I(\epsilon \xi )) \|_{\mathcal{H}^{1}_s}^2 \lesssim   \|  I(\epsilon \xi )  \|_{\mathcal{H}^{1}_s}^4  = \epsilon^4 \|  I( \xi )  \|_{\mathcal{H}^{1}_s}^4  \\
	& = \epsilon^4 \| \xi \|_{\mathcal{Q}}^4 \leq \epsilon^4 | \xi |_{s+1}^4\leq \epsilon^4 | \xi |_{s+3}^4 .
\end{align*}   
Next, we use the above together with the Lipschitz estimate for $\mathfrak{f}^{(2)}$ (Lemma \ref{LipschitzModel3}) and the fact that $I^{-1}: \mathcal{H}^1_s \longrightarrow \mathcal{Q}$ to obtain
\begin{align*}
& \left\|	\frac{1}{2} I^{-1} Q \left[\mathfrak{f}^{(2)}(I(\epsilon \xi)+L_{\ssomega}^{-1} P \mathfrak{f}^{(2)}(I(\epsilon \xi))  ) \right]-\frac{1}{2} I^{-1} Q \left[\mathfrak{f}^{(2)}(I(\epsilon \xi)+L_{1}^{-1} P \mathfrak{f}^{(2)}(I(\epsilon \xi))  ) \right] \right\|_{\mathcal{Q} } =  \\
& \left\|	\frac{1}{2} I^{-1} Q \left[\mathfrak{f}^{(2)}(I(\epsilon \xi)+L_{\ssomega}^{-1} P \mathfrak{f}^{(2)}(I(\epsilon \xi))  )  - \mathfrak{f}^{(2)}(I(\epsilon \xi)+L_{1}^{-1} P \mathfrak{f}^{(2)}(I(\epsilon \xi))  ) \right] \right\|_{\mathcal{Q} }\lesssim \\
& \left\| \mathfrak{f}^{(2)}( I(\epsilon \xi)+L_{\ssomega}^{-1} P \mathfrak{f}^{(2)}(I(\epsilon \xi))  )  - \mathfrak{f}^{(2)}( I(\epsilon \xi)+L_{1}^{-1} P \mathfrak{f}^{(2)}(I(\epsilon \xi))  )  \right\|_{\mathcal{H}^{1}_s } \lesssim \\
& \left( \left\| I(\epsilon \xi)+L_{\ssomega}^{-1} P \mathfrak{f}^{(2)}(I(\epsilon \xi))\right\|_{\mathcal{H}^{1}_s } + \left\| I(\epsilon \xi)+L_{1}^{-1} P \mathfrak{f}^{(2)}(I(\epsilon \xi))\right\|_{\mathcal{H}^{1}_s } \right)\cdot \\
& \cdot  \left\| L_{\ssomega}^{-1} P\mathfrak{f}^{(2)}(I(\epsilon \xi))-L_{1}^{-1} P\mathfrak{f}^{(2)}(I(\epsilon \xi)) \right\|_{\mathcal{H}^{1}_s } \lesssim
 \epsilon \left\| \left( L_{\ssomega}^{-1}  -L_{1}^{-1} \right) P\mathfrak{f}^{(2)}(I(\epsilon \xi)) \right\|_{\mathcal{H}^{1}_s }.
\end{align*}
Once again, the Diophantine condition, the elementary inequality $|\ssomega_{j}^2-l^2| \geq 1$, the Lipschitz estimate 
\begin{align*}
	\|  \mathfrak{f}^{(2)} (u)  \|_{\mathcal{H}^k_s} \lesssim_{s}  \|u \|_{\mathcal{H}^k_s}^2,
\end{align*}
for all $u \in \mathcal{H}^k_s$ with $\|u \|_{\mathcal{H}^k_s} \leq \epsilon$ that follows from Lemma \ref{LipschitzModel3}, together with \eqref{HkNormWithQ}, imply that
\begin{align*}
   \left\| \left( L_{\ssomega}^{-1}  -L_{1}^{-1} \right) P\mathfrak{f}^{(2)}(I(\epsilon \xi)) \right\|_{\mathcal{H}^{1}_s } ^2 &=|\ssomega^2-1|^2 \sum_{j=0}^{\infty} j^{2s}   \sum_{l=1}^{\infty} \left| \frac{l^2  }{(\ssomega_{j}^2-l^2 \ssomega^2)(\ssomega_{j}^2-l^2)}  (\mathfrak{f}^{(2)}(I(\epsilon \xi)))_{l}^{j} \right|^2 (1+l^2)  \\
	& \lesssim_{\alpha}|\ssomega^2-1|^2  \sum_{j=0}^{\infty} j^{2s}   \sum_{l=1}^{\infty} \left|  l^2    (\mathfrak{f}^{(2)}(I(\epsilon\xi)))_{l}^{j} \right|^2 (1+l^2)  \\
	& \lesssim  |\ssomega^2-1|^2 \sum_{j=0}^{\infty} j^{2s}   \sum_{l=1}^{\infty} \left|      (\mathfrak{f}^{(2)}(I( \epsilon \xi)))_{l}^{j} \right|^2 (1+l^2)^{3}  \\
	& \leq |\ssomega^2-1|^2 \| \mathfrak{f}^{(2)}(I( \epsilon \xi)) \|_{\mathcal{H}^{3}_s}^2 
	 \lesssim  |\ssomega^2-1|^2 \| I(\epsilon \xi) \|_{\mathcal{H}^{3}_s}^4  \\
	 &=|\ssomega^2-1|^2 \epsilon^4 \| I( \xi) \|_{\mathcal{H}^{3}_s}^4 \lesssim |\ssomega^2-1|^2 \epsilon^4 |\xi |_{s+3}^2 
\end{align*}   
which in turn yields
\begin{align*}
   \left\| \left( L_{\ssomega}^{-1}  -L_{1}^{-1} \right) P\mathfrak{f}^{(2)}(I(\epsilon \xi)) \right\|_{\mathcal{H}^{1}_s }  \lesssim |\ssomega^2-1|\epsilon^2,
  \end{align*}
  due to \eqref{OmegaEquation} and \eqref{XiandBetaIntermsofe} respectively. Finally, putting all together yields
\begin{align*}
	 & \left\| \mathfrak{R}_{\epsilon}(\xi,\ssomega) \right\|_{\mathcal{Q}  }=  \\   & \epsilon^{-3}\left\| \frac{1}{2} I^{-1} Q \left[ 
		 \mathfrak{f}^{(2)}  \left( I(\epsilon \xi)+ L_{\ssomega}^{-1} P \mathfrak{f}^{(2)} (I(\epsilon \xi)) \right)  -  
		 \mathfrak{f}^{(2)}  \left( I(\epsilon \xi)+ L_{1}^{-1} P \mathfrak{f}^{(2)} (I(\epsilon \xi)) \right)  \right]  
		 \rangle(\epsilon \xi) \right\|_{\mathcal{Q}  } \lesssim \\
		 &\epsilon^{-3} \epsilon \left\| \left( L_{\ssomega}^{-1}  -L_{1}^{-1} \right) P\mathfrak{f}^{(2)}(I(\epsilon \xi)) \right\|_{\mathcal{H}^{1}_s } \lesssim \epsilon^{-3} \epsilon |\ssomega^2-1|\epsilon^2  = |\ssomega^2-1|,
\end{align*}
that completes the proof. The estimate for the differential follows similarly.
\end{proof}
Next, we estimate $\mathfrak{G}_{\epsilon}(\xi)$ and its differential.

\begin{lemma}[Estimate for $\mathfrak{G}_{\epsilon}(\xi)$ and $ d_\xi \mathfrak{G}_{\epsilon}(\xi)$]\label{LemmaEstimateG}
	Let $\xi \in l_{s+3}^2$ be any initial data. Then, $\mathfrak{G}_{\epsilon}(\xi)$ is continuously differentiable with respect to $\xi$ and we have
	\begin{align*}
		\| \mathfrak{G}_{\epsilon}(\xi) \|_{\mathcal{Q}}  \lesssim \epsilon, \Hquad \| d_\xi \mathfrak{G}_{\epsilon}(\xi) \|_{\mathcal{Q}}  \lesssim \epsilon.
			\end{align*} 
\end{lemma}
\begin{proof}
	Let $\xi \in  l_{s+3}^2$ be any initial data and recall that 
\begin{align*}
	\mathfrak{G}_{\epsilon}(\xi) &:= \epsilon^{-3}\Bigg[ \frac{1}{2} I^{-1} Q \left[ 
		 \mathfrak{f}^{(2)}  \left( I(\epsilon \xi)+ q_{\perp}(I(\epsilon \xi)) \right)   -
		 \mathfrak{f}^{(2)}  \left( I(\epsilon \xi)+ L_{\ssomega}^{-1} P \mathfrak{f}^{(2)} (I(\epsilon \xi)) \right)  \right]\nonumber \\
		 &   +
		  \frac{1}{2} I^{-1} Q \left[ 
		 \mathfrak{f}^{(3)}  \left( I(\epsilon \xi)+ q_{\perp}(I(\epsilon \xi)) \right) - \mathfrak{f}^{(3)}(I(\epsilon \xi)) \right]  \Bigg].
\end{align*}	
The claim follows by Lemmata \ref{LipschitzModel3} and \ref{LemmaSolutionPequation} together with \eqref{HkNormWithQ} and the fact that $I^{-1}: \mathcal{H}^1_s \longrightarrow \mathcal{Q}$. Indeed, since
	\begin{align}\label{AuxiliralyBoundsforLemma211}
 \left \| I(\epsilon \xi)   \right \|_{\mathcal{H}^1_s} \lesssim \epsilon \left \| \xi  \right \|_{\mathcal{Q}}  ,  \Hquad 	\left \| q_{\perp}(I(\epsilon \xi))  \right \|_{\mathcal{H}^1_s} \lesssim \left \| I(\epsilon \xi)   \right \|_{\mathcal{H}^1_s}^2 \lesssim \epsilon^2\left \| \xi  \right \|_{\mathcal{Q}}^2 ,
	\end{align}
 we can estimate
	\begin{align*}
	&	\left \|\frac{1}{2} I^{-1} Q \left[ 
		 \mathfrak{f}^{(2)}  \left( I(\epsilon \xi)+ q_{\perp}(I(\epsilon \xi)) \right)   -
		 \mathfrak{f}^{(2)}  \left( I(\epsilon \xi)+ L_{\ssomega}^{-1} P \mathfrak{f}^{(2)} (I(\epsilon \xi)) \right)  \right] \right \|_{\mathcal{Q} } \lesssim \\
	&	 \left \|  
		 \mathfrak{f}^{(2)}  \left( I(\epsilon \xi)+ q_{\perp}(I(\epsilon \xi)) \right)   -
		 \mathfrak{f}^{(2)}  \left( I(\epsilon \xi)+ L_{\ssomega}^{-1} P \mathfrak{f}^{(2)} (I(\epsilon \xi)) \right)    \right \|_{\mathcal{H}^1_s} \lesssim \\
   &	\epsilon \left \|  
		   q_{\perp}(I(\epsilon \xi))     -
		  L_{\ssomega}^{-1} P \mathfrak{f}^{(2)} (I(\epsilon \xi))    \right \|_{\mathcal{H}^1_s} \lesssim  
		\epsilon \left \|  
		   I(\epsilon \xi)     \right \|_{\mathcal{H}^1_s}^3 \lesssim \epsilon^4
	\end{align*}
	and
	\begin{align*}
	 & \left\|	\frac{1}{2} I^{-1} Q \left[ 
		 \mathfrak{f}^{(3)}  \left( I(\epsilon \xi)+ q_{\perp}(I(\epsilon \xi)) \right) - \mathfrak{f}^{(3)}(I(\epsilon \xi)) \right]   \right \|_{\mathcal{Q}} \lesssim  \\
		& \left\|	 
		 \mathfrak{f}^{(3)}  \left( I(\epsilon \xi)+ q_{\perp}(I(\epsilon \xi)) \right) - \mathfrak{f}^{(3)}(I(\epsilon \xi))    \right \|_{\mathcal{H}^1_s} \lesssim  \\
		 & \left[ \left\|	 
		 I(\epsilon \xi)+ q_{\perp}(I(\epsilon \xi))  \right \|_{\mathcal{H}^1_s} ^2 +
		 \left\|	 
		 I(\epsilon \xi)    \right \|_{\mathcal{H}^1_s}^2 \right] 
		  \left\|	 
	 q_{\perp}(I(\epsilon \xi))     \right \|_{\mathcal{H}^1_s}  \lesssim  \\
	 & \left[  
		 \left\|	 
		 I(\epsilon \xi)    \right \|_{\mathcal{H}^1_s}^2 +\left\|	 
		 q_{\perp}(I(\epsilon \xi))  \right \|_{\mathcal{H}^1_s} ^2 \right] 
		  \left\|	 
	 q_{\perp}(I(\epsilon \xi))     \right \|_{\mathcal{H}^1_s}   \lesssim \epsilon^4.
	\end{align*}
The estimate for the differential follows similarly. Indeed, we pick any $\xi \in l_{s+3}^2$ and $\overline{\xi}=\xi+\overline{\epsilon}h\in  l_{s+3}^2$ for some $\overline{\epsilon}>0$ and  $h \in l_{s+3}^2$. We have that
\begin{align*}
	&\epsilon^{-3}\frac{1}{2} I^{-1} Q \left[ 
		 \mathfrak{f}^{(2)}  \left( I(\epsilon \overline{\xi})+ q_{\perp}(I(\epsilon \overline{\xi})) \right) \right] = \\
		& \epsilon^{-3}\frac{1}{2} I^{-1} Q \left[ 
		 \mathfrak{f}^{(2)}  \left( I(\epsilon \xi +  \overline{\epsilon}\cdot \epsilon h)+ q_{\perp}(I(\epsilon \xi +\overline{\epsilon}\cdot  \epsilon  h)) \right) \right]= \\
		 &\epsilon^{-3}\frac{1}{2} I^{-1} Q \left[ 
		 \mathfrak{f}^{(2)}  \left(  I(\epsilon\xi) +  \overline{\epsilon}I(\epsilon h)+ q_{\perp}(I(\epsilon \xi )+ \overline{\epsilon}I(\epsilon  h)) \right) \right] = \\
		 &\epsilon^{-3}\frac{1}{2} I^{-1} Q \left[ 
		 \mathfrak{f}^{(2)}  \left(  I(\epsilon\xi) +  \overline{\epsilon}I(\epsilon h)+ q_{\perp}(I(\epsilon \xi )) + \overline{\epsilon}(dq_{\perp})_{I(\epsilon \xi )}[I(\epsilon  h)]  +\mathcal{O}(\overline{\epsilon}^2) \right)\right] = \\
		  &\epsilon^{-3}\frac{1}{2} I^{-1} Q \left[ 
		 \mathfrak{f}^{(2)}  \left(  
		 I(\epsilon\xi)+q_{\perp}(I(\epsilon \xi ))
		 +  \overline{\epsilon} \left( 
		 I(\epsilon h)+(dq_{\perp})_{I(\epsilon \xi )}[I(\epsilon  h)] 
		 \right)
		 +\mathcal{O}(\overline{\epsilon}^2) \right) \right] = \\
		 &\epsilon^{-3}\frac{1}{2} I^{-1} Q \left[ 
		 \mathfrak{f}^{(2)}  \left(  
		 I(\epsilon\xi)+q_{\perp}(I(\epsilon \xi ))
		 \right)
		 +  \overline{\epsilon} 
		 (d  \mathfrak{f}^{(2)}   )_{ I(\epsilon\xi)+q_{\perp}(I(\epsilon \xi )) }
		 \left[ 
		 I(\epsilon h)+(dq_{\perp})_{I(\epsilon \xi )}[I(\epsilon  h)] 
		 \right]
		 +\mathcal{O}(\overline{\epsilon}^2)  \right].
\end{align*}
Similarly, we set $\mathfrak{F}^{(2)}:=L_{\ssomega}^{-1} P  \mathfrak{f}^{(2)}$ and obtain
\begin{align*}
	& \epsilon^{-3}\frac{1}{2} I^{-1} Q \left[ 
		 \mathfrak{f}^{(2)}  \left( I(\epsilon \overline{\xi})+ L_{\ssomega}^{-1} P  \mathfrak{f}^{(2)}  (I(\epsilon \overline{\xi})) \right) \right] = \\
		 & \epsilon^{-3}\frac{1}{2} I^{-1} Q \left[ 
		 \mathfrak{f}^{(2)}  \left( I(\epsilon \overline{\xi})+ \mathfrak{F}^{(2)} (I(\epsilon \overline{\xi})) \right) \right] = \\
		 & \epsilon^{-3}\frac{1}{2} I^{-1} Q \left[ 
		 \mathfrak{f}^{(2)}  \left(  
		 I(\epsilon\xi)+\mathfrak{F}^{(2)} (I(\epsilon \xi ))
		 \right)
		 +  \overline{\epsilon} 
		 (d  \mathfrak{f}^{(2)}   )_{ I(\epsilon\xi)+\mathfrak{F}^{(2)}(I(\epsilon \xi )) }
		 \left[ 
		 I(\epsilon h)+(d \mathfrak{F}^{(2)})_{I(\epsilon \xi )}[I(\epsilon  h)] 
		 \right]
		 +\mathcal{O}(\overline{\epsilon}^2)  \right]
\end{align*}
and
\begin{align*}
	& \epsilon^{-3} \frac{1}{2} I^{-1} Q \left[ 
		 \mathfrak{f}^{(3)}  \left( I(\epsilon \overline{\xi})+ q_{\perp}(I(\epsilon \overline{\xi})) \right) \right] = \\
		 &\epsilon^{-3}\frac{1}{2} I^{-1} Q \left[ 
		 \mathfrak{f}^{(3)}  \left(  
		 I(\epsilon\xi)+q_{\perp}(I(\epsilon \xi ))
		 \right)
		 +  \overline{\epsilon} 
		 (d  \mathfrak{f}^{(3)}   )_{ I(\epsilon\xi)+q_{\perp}(I(\epsilon \xi )) }
		 \left[ 
		 I(\epsilon h)+(dq_{\perp})_{I(\epsilon \xi )}[I(\epsilon  h)] 
		 \right]
		 +\mathcal{O}(\overline{\epsilon}^2)  \right]
\end{align*}
as well as
\begin{align*}
	& \epsilon^{-3} \frac{1}{2} I^{-1} Q \left[ \mathfrak{f}^{(3)}(I(\epsilon \overline{\xi})) \right] = \\
	&\epsilon^{-3}\frac{1}{2} I^{-1} Q \left[ 
		 \mathfrak{f}^{(3)}  \left(  
		 I(\epsilon\xi) 
		 \right)
		 +  \overline{\epsilon} 
		 (d  \mathfrak{f}^{(3)}   )_{ I(\epsilon\xi)  }
		 \left[ 
		 I(\epsilon h) 
		 \right]
		 +\mathcal{O}(\overline{\epsilon}^2)  \right].
\end{align*}
Consequently, we have
\begin{align*}
	 d\mathfrak{G}_{\epsilon}(\xi)[h]&=  
	 \epsilon^{-3}\frac{1}{2} I^{-1} Q \left[ 
		 (d  \mathfrak{f}^{(2)}   )_{ I(\epsilon\xi)+q_{\perp}(I(\epsilon \xi )) }
		 \left[ 
		 I(\epsilon h)+(dq_{\perp})_{I(\epsilon \xi )}[I(\epsilon  h)] 
		 \right]  \right] \\
		 &-\epsilon^{-3}\frac{1}{2} I^{-1} Q \left[ 
		 (d  \mathfrak{f}^{(2)}   )_{ I(\epsilon\xi)+\mathfrak{F}^{(2)}(I(\epsilon \xi )) }
		 \left[ 
		 I(\epsilon h)+(d \mathfrak{F}^{(2)})_{I(\epsilon \xi )}[I(\epsilon  h)] 
		 \right]  \right] \\
		 &+\epsilon^{-3}\frac{1}{2} I^{-1} Q \left[  
		 (d  \mathfrak{f}^{(3)}   )_{ I(\epsilon\xi)+q_{\perp}(I(\epsilon \xi )) }
		 \left[ 
		 I(\epsilon h)+(dq_{\perp})_{I(\epsilon \xi )}[I(\epsilon  h)] 
		 \right]   \right] \\
		 &-\epsilon^{-3}\frac{1}{2} I^{-1} Q \left[  
		 (d  \mathfrak{f}^{(3)}   )_{ I(\epsilon\xi)  }
		 \left[ 
		 I(\epsilon h) 
		 \right]   \right].
\end{align*}
Recall that, according to Lemma \ref{LemmaSolutionPequation} together with \eqref{HkNormWithQ}, we have
	\begin{align*} 
		\|q_{\perp}(I(\epsilon\xi)) \|_{\mathcal{H}^1_s}  &\lesssim \| I(\epsilon\xi) \|_{\mathcal{H}^1_s}^{2}\simeq \epsilon^2 |\xi|_{s+1}^2 \simeq  \epsilon^2 \|\xi \|_{\mathcal{Q}}^2, \\
		\|q_{\perp}(I(\epsilon\xi)) -   \mathfrak{F}^{(2)} (I(\epsilon\xi)) \|_{\mathcal{H}^1_s}  &\lesssim \| I(\epsilon\xi) \|_{\mathcal{H}^1_s}^{3} \simeq \epsilon^3 |\xi|_{s+1}^3 \simeq  \epsilon^3 \|\xi \|_{\mathcal{Q}}^3, \\
		\|d \left(q_{\perp}  -   \mathfrak{F}^{(2)}\right)_{ (I(\epsilon\xi))}[I(\epsilon h)] \|_{\mathcal{H}^1_s}  &\lesssim \| I(\epsilon\xi) \|_{\mathcal{H}^1_s}^{2} \| I(\epsilon h) \|_{\mathcal{H}^1_s} \simeq \epsilon^2 |\xi|_{s+1}^2 \epsilon | h |_{s+1} \simeq  \epsilon^3 \|\xi \|_{\mathcal{Q}}^2\| h \|_{\mathcal{Q}}, \\
		\|\left(d q_{\perp}   \right)_{ (I(\epsilon\xi))}[I(\epsilon h)] \|_{\mathcal{H}^1_s}  &\lesssim \| I(\epsilon\xi) \|_{\mathcal{H}^1_s}  \| I(\epsilon h) \|_{\mathcal{H}^1_s} \simeq \epsilon  |\xi|_{s+1} \epsilon | h |_{s+1} \simeq  \epsilon^2 \|\xi \|_{\mathcal{Q}} \| h \|_{\mathcal{Q}},
	\end{align*} 
for all $\xi \in l_{s+3}^2$ and $h \in l_{s+3}^2$. Now, we set 
\begin{align*}
X &= I(\epsilon\xi)+q_{\perp}(I(\epsilon \xi )) , \\
Y &=I(\epsilon\xi)+\mathfrak{F}^{(2)}(I(\epsilon \xi )), \\
	U &=I(\epsilon h)+(dq_{\perp})_{I(\epsilon \xi )}[I(\epsilon  h)], \\
	V&=I(\epsilon h)+(d \mathfrak{F}^{(2)})_{I(\epsilon \xi )}[I(\epsilon  h)] 
\end{align*} 
and Lemma \ref{LipschitzModel3} yields
\begin{align*}
	&\left \| (d  \mathfrak{f}^{(2)}   )_{X}
		 \left[ 
		 U
		 \right]   -  
		 (d  \mathfrak{f}^{(2)}   )_{ Y }
		 \left[ 
		 V
		 \right] \right \|_{\mathcal{H}^1_s}= \\
	&\left \| (d  \mathfrak{f}^{(2)}   )_{X}
		 \left[ 
		 U
		 \right]- (d  \mathfrak{f}^{(2)}   )_{Y}
		 \left[ 
		 U
		 \right]  +
		 (d  \mathfrak{f}^{(2)}   )_{Y}
		 \left[ 
		 U
		 \right]   -  
		 (d  \mathfrak{f}^{(2)}   )_{ Y }
		 \left[ 
		 V
		 \right]\right \|_{\mathcal{H}^1_s}  \leq \\
&\left \| (d  \mathfrak{f}^{(2)}   )_{X}
		 \left[ 
		 U
		 \right]- (d  \mathfrak{f}^{(2)}   )_{Y}
		 \left[ 
		 U
		 \right]\right \|_{\mathcal{H}^1_s}  +
		 \left \|
		 (d  \mathfrak{f}^{(2)}   )_{Y}
		 \left[ 
		 U
		 \right]   -  
		 (d  \mathfrak{f}^{(2)}   )_{ Y }
		 \left[ 
		 V
		 \right]\right \|_{\mathcal{H}^1_s}  \leq \\
&\left \| 
\left(
(d  \mathfrak{f}^{(2)}   )_{X} - (d  \mathfrak{f}^{(2)}   )_{Y}
\right)
		 \left[ 
		 U
		 \right]\right \|_{\mathcal{H}^1_s}  +
		 \left \| 
		 (d  \mathfrak{f}^{(2)}   )_{ Y }
		 \left[ 
		 U-V
		 \right]\right \|_{\mathcal{H}^1_s}  \lesssim \\
&\left \|  
X-Y  \right \|_{\mathcal{H}^1_s}  \left \|
		 U \right \|_{\mathcal{H}^1_s}  +
		\left \|  Y  \right \|_{\mathcal{H}^1_s} 
		 \left \|   
		 U-V \right \|_{\mathcal{H}^1_s}  \lesssim \\	 
&\epsilon^3 \| \xi\|_{\mathcal{Q}}^3 \cdot \epsilon  \| h \|_{\mathcal{Q}} +\epsilon \| \xi\|_{\mathcal{Q}} \cdot \epsilon^3 \| \xi\|_{\mathcal{Q}}^2 \| h \|_{\mathcal{Q}} \lesssim \epsilon^4 \| h \|_{\mathcal{Q}} ,
\end{align*}
for all $\xi \in l_{s+3}^2$ and $h \in l_{s+3}^2$. Similarly, we set
\begin{align*}
	Z&=I(\epsilon \xi), \\
	W&=I(\epsilon h),
\end{align*}
and Lemma \ref{LipschitzModel3} yields
\begin{align*}
	&\left \| (d  \mathfrak{f}^{(3)}   )_{X}
		 \left[ 
		 U
		 \right]   -  
		 (d  \mathfrak{f}^{(3)}   )_{ Z }
		 \left[ 
		 W
		 \right] \right \|_{\mathcal{H}^1_s}= \\
	&\left \| (d  \mathfrak{f}^{(3)}   )_{X}
		 \left[ 
		 U
		 \right]- (d  \mathfrak{f}^{(3)}   )_{Z}
		 \left[ 
		 U
		 \right]  +
		 (d  \mathfrak{f}^{(3)}   )_{Z}
		 \left[ 
		 U
		 \right]   -  
		 (d  \mathfrak{f}^{(3)}   )_{ Z }
		 \left[ 
		 W
		 \right]\right \|_{\mathcal{H}^1_s}  \leq \\
&\left \| (d  \mathfrak{f}^{(3)}   )_{X}
		 \left[ 
		 U
		 \right]- (d  \mathfrak{f}^{(3)}   )_{Z}
		 \left[ 
		 U
		 \right]\right \|_{\mathcal{H}^1_s}  +
		 \left \|
		 (d  \mathfrak{f}^{(3)}   )_{Z}
		 \left[ 
		 U
		 \right]   -  
		 (d  \mathfrak{f}^{(3)}   )_{ Z }
		 \left[ 
		 W
		 \right]\right \|_{\mathcal{H}^1_s}  \leq \\
&\left \| 
\left(
(d  \mathfrak{f}^{(3)}   )_{X} - (d  \mathfrak{f}^{(3)}   )_{Z}
\right)
		 \left[ 
		 U
		 \right]\right \|_{\mathcal{H}^1_s}  +
		 \left \| 
		 (d  \mathfrak{f}^{(3)}   )_{ Z }
		 \left[ 
		 U-W
		 \right]\right \|_{\mathcal{H}^1_s}  \lesssim \\
&\left \|  
X-Z \right \|_{\mathcal{H}^1_s}^2  \left \|
		 U \right \|_{\mathcal{H}^1_s}  +
		\left \|  Z  \right \|_{\mathcal{H}^1_s}^2 
		 \left \|   
		 U-W \right \|_{\mathcal{H}^1_s}  \lesssim \\	 
&\left(\epsilon^2 \| \xi\|_{\mathcal{Q}}^2\right)^2 \cdot \epsilon  \| h \|_{\mathcal{Q}} +\left(\epsilon \| \xi\|_{\mathcal{Q}} \right)^2 \cdot \epsilon^2 \| \xi\|_{\mathcal{Q}}  \| h \|_{\mathcal{Q}} \lesssim \epsilon^4 \| h \|_{\mathcal{Q}} ,
\end{align*}
for all $\xi \in l_{s+3}^2$ and $h \in l_{s+3}^2$. Putting all together, yields
\begin{align*}
	\|d_{\xi} \mathfrak{G}_{\epsilon}(\xi)[h] \|_{\mathcal{Q} } \lesssim \epsilon \| h \|_{\mathcal{Q}}, 
\end{align*}
for all $\xi \in l_{s+3}^2$ and $h \in l_{s+3}^2$, that completes the proof.
 \end{proof}
It remains to show that there exists a solution to \eqref{ModifiedNewQequationConfiguration}, that is
\begin{align*} 
	 \pm  \mathfrak{A} \xi +    \langle \mathfrak{f}^{(3)}  \rangle( \xi)  =  -\left( \mathfrak{F}_{\epsilon}( \xi)+   \mathfrak{G}_{\epsilon}(\xi) \pm  \frac{2\epsilon^2}{\ssomega^2-1}\mathfrak{R}_{\epsilon}(\xi,\ssomega)\right).
\end{align*}
that we rewrite as
\begin{align}\label{NewQequationConfigurationNo2}
	\textswab{M}_{\pm}(\xi)   =  \mathfrak{H}_{\epsilon}(\xi)   ,
\end{align}
where we introduced the \textit{modified operator}
\begin{align*}
	\textswab{M}_{\pm}(\xi) &= \pm  \mathfrak{A} \xi +    \langle \mathfrak{f}^{(3)}  \rangle( \xi) + \mathfrak{F}_{0}( \xi) 
\end{align*} 
 and set
\begin{align*}
	\mathfrak{H}_{\epsilon}(\xi):=\mathfrak{F}_{0}( \xi)-\mathfrak{F}_{\epsilon}( \xi)- \mathfrak{G}_{\epsilon}(\xi) \mp \frac{2\epsilon^2}{\ssomega^2-1}\mathfrak{R}_{\epsilon}(\xi,\ssomega).
\end{align*}
Note that Lemmata \ref{LemmaEstimateR}, \ref{LemmaEstimateG} and the smoothness\footnote{In Lemma \ref{LemmaDefinitionMathfrakF} we show that $\mathfrak{F}_{\epsilon}( \xi)$ is smooth with respect to $\epsilon$. As one can see in the proof of Lemma \ref{LemmaDefinitionMathfrakF}, $\mathfrak{F}_{\epsilon}( \xi)$ is in fact linear with respect to $\epsilon$. See \eqref{FeLinear}. } of $\mathfrak{F}_{\epsilon}( \xi)$ with respect to $\epsilon$ yield 
\begin{align*}  
 	\| \mathfrak{H}_{\epsilon}(\xi)\|_{\mathcal{Q}}  \lesssim \epsilon, \Hquad \| d_{\xi}\mathfrak{H}_{\epsilon}(\xi)\|_{\mathcal{Q}}  \lesssim \epsilon.
\end{align*}
The following result constitutes the main modification of Bambusi--Paleari's theorem.
\begin{lemma}[Solution to the $Q-$equation]\label{LemmaSolutionQEquation}
	Let $\xi_0 \in l_{s+3}^2$ be a non--degenerate  zero of the \textit{modified operator}
	\begin{align*}
		 \textswab{M}_{\pm}(x) &=   \pm \mathfrak{A} x +     \langle \mathfrak{f}^{(3)}  \rangle( x)+\mathfrak{F}_{0}( x) ,
	\end{align*}
 that is
	\begin{align*}
	 \textswab{M}_{\pm}(\xi_0)=0, \Hquad \ker \left( d  \textswab{M}_{\pm}(\xi_0)\right)=\{ 0 \}.
	\end{align*} 
	Then, there exists a positive $\epsilon_0$ and a Lipschitz map 
	\begin{align*}
		\xi : [0,\epsilon_0 ) \longrightarrow l_{s+3}^2, \Hquad 
		\epsilon \longmapsto \xi(\epsilon)
	\end{align*}
	that solves the $Q-$equation \eqref{NewQequationConfiguration} with the plus sign. 
	Furthermore, we have the estimate
	\begin{align*}
		|\xi(\epsilon)-\xi_0 |_{s+3} \lesssim \epsilon.
	\end{align*}
\end{lemma}
\begin{proof}
	The proof follows from the implicit function theorem and is similar to the one of Proposition 4.8 in Bambusi--Paleari \cite{MR1819863}. Let $\xi_0 \in l_{s+3}^2$ be a non--degenerate  zero of the modified operator $\widetilde{\mathcal{M}}_{\pm}$ and  define the map
	\begin{align*}
	\mathcal{G}:	\mathbb{R} \times l_{s+3}^2 \longrightarrow \mathcal{H}_{s}^1, \Hquad 
	(\epsilon,\xi) \longmapsto \mathcal{G}(\epsilon,\xi):= \textswab{M}_{\pm}(\xi)  - \mathfrak{H}_{\epsilon}(\xi)
	\end{align*}
	and note that  it is Lipschitz, differentiable at $\epsilon=0$ and it vanishes at $(\epsilon,\xi)=(0,\xi_0)$. It remains to show that its differential with respect to $\xi$ at $(0,\xi_0)$,
	\begin{align*}
	d\mathcal{G}(0,\xi_0):	l_{s+3}^2 \longrightarrow \mathcal{H}_{s}^1, \Hquad X \longmapsto d\mathcal{G}(0,\xi_0)(X)=d \textswab{M}_{\pm}(\xi_0)(X),
	\end{align*}
	 is an isomorphism. Equivalently, this means that, for all $Y \in \mathcal{H}_{s}^1$, there exists $X \in  l_{s+3}^2$ that solves the equation
	 \begin{align*}
	 	d \textswab{M}_{\pm}(\xi_0)(X) =Y.
	 \end{align*}
	 Now, the operator
	 \begin{align*}
	 	d \textswab{M}_{\pm}(\xi_0) =\pm \mathfrak{A} +     d\langle \mathfrak{f}^{(3)}  \rangle( \xi_0)+d\mathfrak{F}_{0}( \xi_0)
	 \end{align*}
	  is a Fredholm operator since it is the sum of a Fredholm and a compact operator due to the fact that $\mathfrak{f}^{(3)}$ and $\mathfrak{F}_{\epsilon}( \xi)$  are bounded on $l_{s+3}^2$, and that they are differentiable with bounded differential. Since the defect index of $d \textswab{M}_{\pm}(\xi_0)$ is $0$ from the non--degeneracy condition, it follows that it is an isomorphism, and thus we can apply the implicit function theorem. Note finally, that the range of $\epsilon$, defined by $\epsilon_0$, does not depend on $\ssomega$, since $\mathcal{G}$ depends continuously on $\ssomega$ and all the necessary bounds hold uniformly with respect to $\ssomega$. 
\end{proof} 
Finally, we prove Theorem \ref{TheoremModificationofBambusiPaleari}.
\begin{proof}[Proof of Theorem \ref{TheoremModificationofBambusiPaleari}]
Let $\ssomega \in \mathcal{W}_{\alpha}$ be fixed. Then, according to Lemmata \ref{LemmaSolutionPequation} and \ref{LemmaSolutionQEquation}, there exists $\epsilon_0>0$ so that the map
\begin{align*}
	[0,\epsilon_0) \ni \epsilon \longmapsto  \left(\epsilon I(\xi(\epsilon)),q_{\perp}\left(\epsilon I(\xi(\epsilon))\right) \right)
\end{align*}
solves both the P--equation \eqref{PEquation} and the Q--equation \eqref{QEquation}. Furthermore, pick $\ssomega_{\star}$ so that $\epsilon(\ssomega_{\star})=\epsilon_0$. Then, the function
\begin{align*}
	\epsilon^2(\ssomega):= \pm \frac{\ssomega^2-1}{2}
\end{align*}
solves \eqref{OmegaEquation} and the map
\begin{align*}
	\ssomega \longmapsto  \left(\epsilon I(\xi(\epsilon(\ssomega))),q_{\perp}\left(\epsilon I(\xi(\epsilon(\ssomega)))\right) \right)
\end{align*}
defines a family of solutions to \eqref{BambusiMainPDE2} labelled by $\ssomega \in \mathcal{W}_{\alpha} \cap [1,1+\ssomega_{\star}]$ or $\ssomega \in \mathcal{W}_{\alpha} \cap [1-\ssomega_{\star},1]$. Finally, the map $\epsilon \longmapsto \epsilon (\ssomega)$ is one--to--one and hence this family can be also parametrized by $\epsilon \in \mathcal{E}_{\gamma}=\epsilon (\mathcal{W}_{\alpha} \cap [1,1+\ssomega_{\star}])$ or $\epsilon \in \mathcal{E}_{\gamma}=\epsilon (\mathcal{W}_{\alpha} \cap [1-\ssomega_{\star},1])$  that completes the proof. 
\end{proof}

\subsection{The function $\mathfrak{F}_{\epsilon}(\xi)$} \label{se:fe}

For future reference, we compute 
\begin{itemize}
	\item $\mathfrak{F}_{0}( \xi)$ for general initial data (Lemma \ref{LemmaDefinitionMathfrakF}),
	\item $\mathfrak{F}_{\epsilon}( \xi)$ for small $\epsilon$ close to zero and 1--mode initial data (Lemma \ref{LemmaComputemathfrakFforsfmallepsilon}),
	\item $d\mathfrak{F}_{0}( \xi)$ at the 1--mode initial data (Lemma \ref{LemmaDifferentialMathFrakF0LongComputation}).
\end{itemize}
To begin with, we compute $\mathfrak{F}_{0}( \xi)$ for general initial data.
\begin{lemma}[Computation of $\mathfrak{F}_{0}( \xi)$ for general initial data]\label{LemmaDefinitionMathfrakF}
	Let $\xi= \{\xi^m : m \geq 0\} \in l_{s+3}^2$. Then, for all integers $m \geq 0$, we have
	\begin{align*}
		\left( \mathfrak{F}_{0}( \xi) \right)^m =& \frac{9}{4} \sum_{\kappa,\nu \geq 0} \overline{\mathfrak{C}}_{\kappa \nu m} \sum_{ \substack{i,j \geq 0 \\ \ssomega_{i}-\ssomega_{j} \neq \pm \ssomega_{\nu}  }  } \frac{\overline{\mathfrak{C}}_{ij\nu }}{\ssomega_{\nu}^2-(\ssomega_{i}-\ssomega_{j})^2   } \xi^i \xi^j \xi^{\kappa} \sum_{\pm} \mathds{1}(\ssomega_{i}-\ssomega_{j}\pm \ssomega_{\kappa} \pm \ssomega_{m}=0) \\
		+ &\frac{9}{4} \sum_{\kappa,\nu \geq 0} \overline{\mathfrak{C}}_{\kappa \nu m} \sum_{ \substack{i,j \geq 0 \\ \ssomega_{i}+\ssomega_{j} \neq \pm \ssomega_{\nu}  }  } \frac{\overline{\mathfrak{C}}_{ij\nu }}{\ssomega_{\nu}^2-(\ssomega_{i}+\ssomega_{j})^2   } \xi^i \xi^j \xi^{\kappa} \sum_{\pm} \mathds{1}(\ssomega_{i}+\ssomega_{j}\pm \ssomega_{\kappa} \pm \ssomega_{m}=0).
	\end{align*}	
In addition, the function $\mathfrak{F}_{\epsilon}$ is smooth with respect to $\epsilon$.	
\end{lemma}
\begin{proof}
	Let $\xi= \{\xi^m : m \geq 0\} \in l_{s+3}^2$ be any initial data, $\epsilon >0$, set $x=\epsilon \xi$ and pick any integer $m \geq 0$. Then, we compute
	\begin{align*}
		\left( \mathfrak{f}^{(2)} \left(\{u^k: k \geq 0\} \right) \right)^m &=
		-3 \sum_{i,j \geq 0} \overline{\mathfrak{C}}_{ijm} u^i u^j, \\
		\left(\mathfrak{f}^{(2)} \left(\pPhi^t (x) \right) \right)^m &=  
		-3 \sum_{i,j \geq 0} \overline{\mathfrak{C}}_{ijm} \left(\pPhi^t (x) \right)^i \left(\pPhi^t (x) \right)^j \\
	&=	-3 \sum_{i,j \geq 0} \overline{\mathfrak{C}}_{ijm} x^i x^j \cos(\ssomega_{i}t)\cos(\ssomega_{j}t) \\
	&=	- \frac{3}{2} \sum_{i,j \geq 0} \overline{\mathfrak{C}}_{ijm} x^i x^j \cos( (\ssomega_{i} - \ssomega_{j})t)
		- \frac{3}{2} \sum_{i,j \geq 0} \overline{\mathfrak{C}}_{ijm} x^i x^j \cos( (\ssomega_{i} + \ssomega_{j})t).
	\end{align*}
	Then, $(P\mathfrak{f}^{(2)}(\pPhi^t (x)) )^m$ is given by
	\begin{align*}
		- \frac{3}{2} \left[ \sum_{  \substack{ i,j \geq 0 \\ \ssomega_{i}-\ssomega_{j} \neq \pm \ssomega_{m}  }} \overline{\mathfrak{C}}_{ijm} x^i x^j \cos( (\ssomega_{i} - \ssomega_{j})t) +\sum_{  \substack{ i,j \geq 0 \\ \ssomega_{i}+\ssomega_{j} \neq \pm \ssomega_{m}  }} \overline{\mathfrak{C}}_{ijm} x^i x^j \cos( (\ssomega_{i} + \ssomega_{j})t) \right]
	\end{align*}
	and $(L_{\ssomega}^{-1} P\mathfrak{f}^{(2)}(\pPhi^t (x)) )^m$ reads
	\begin{align*}
		- \frac{3}{2} \left[ \sum_{  \substack{ i,j \geq 0 \\ \ssomega_{i}-\ssomega_{j} \neq \pm \ssomega_{m}  }} \frac{ \overline{\mathfrak{C}}_{ijm} }{ \lambda_{m,\ssomega_{i}-\ssomega_{j}} } x^i x^j \cos( (\ssomega_{i} - \ssomega_{j})t) +\sum_{  \substack{ i,j \geq 0 \\ \ssomega_{i}+\ssomega_{j} \neq \pm \ssomega_{m}  }} \frac{ \overline{\mathfrak{C}}_{ijm} }{ \lambda_{m,\ssomega_{i}+\ssomega_{j}} } x^i x^j \cos( (\ssomega_{i} + \ssomega_{j})t) \right],
	\end{align*}
	where we used the fact that the eigenvalues of $L_{\ssomega}$ are given by $\lambda_{ml}=\ssomega_{m}^2-l^2 \ssomega^2$. Hence, using the above together with the symmetries of the Fourier coefficients $\overline{\mathfrak{C}}_{\kappa \nu m} =\overline{\mathfrak{C}}_{\nu \kappa  m}  $ for all integers $\kappa,\nu,m \geq 0$, we deduce that $(\mathfrak{f}^{(2)} ( \pPhi^t (x)+L_{  1}^{-1} P \mathfrak{f}^{(2)} (\pPhi^t (x)  )  )   )^m$ is given by
	\begin{align*}
		& -3 \sum_{\kappa,\nu \geq 0} \overline{\mathfrak{C}}_{\kappa \nu m} \left(\pPhi^t (x)+L_{  1}^{-1} P \mathfrak{f}^{(2)} (\pPhi^t (x)  )  \right)^{\kappa}\left(\pPhi^t (x)+L_{  1}^{-1} P \mathfrak{f}^{(2)} (\pPhi^t (x)  )  \right)^{\nu}= \\
		& -3 \sum_{\kappa,\nu \geq 0} \overline{\mathfrak{C}}_{\kappa \nu m} 
		\left[ \left(\pPhi^t (x) \right)^{\kappa}+ \left(L_{  1}^{-1} P \mathfrak{f}^{(2)} (\pPhi^t (x)  )  \right)^{\kappa} \right] \left[ \left(\pPhi^t (x) \right)^{\nu}+ \left(L_{  1}^{-1} P \mathfrak{f}^{(2)} (\pPhi^t (x)  )  \right)^{\nu}  \right]= \\
		& -3 \sum_{\kappa,\nu \geq 0} \overline{\mathfrak{C}}_{\kappa \nu m} 
		 \left(\pPhi^t (x) \right)^{\kappa} \left(\pPhi^t (x) \right)^{\nu} -6  \sum_{\kappa,\nu \geq 0} \overline{\mathfrak{C}}_{\kappa \nu m} \left(\pPhi^t (x) \right)^{\kappa}   \left(L_{  1}^{-1} P \mathfrak{f}^{(2)} (\pPhi^t (x)  )  \right)^{\nu}  \\
		& -3  \sum_{\kappa,\nu \geq 0} \overline{\mathfrak{C}}_{\kappa \nu m}    \left(L_{  1}^{-1} P \mathfrak{f}^{(2)} (\pPhi^t (x)  )  \right)^{\kappa} \left(L_{  1}^{-1} P \mathfrak{f}^{(2)} (\pPhi^t (x)  )  \right)^{\nu} 
	\end{align*}
	and, by setting $x=\epsilon \xi$, we infer that $(\mathfrak{f}^{(2)} ( \pPhi^t (\epsilon \xi)+L_{  1}^{-1} P \mathfrak{f}^{(2)} (\pPhi^t (\epsilon \xi)  )  )   )^m$ equals to
	\begin{align} \label{eq:extrquad}
		\epsilon^2 \left(\mathfrak{f}^{(2)} \left(\pPhi^t (\xi) \right) \right)^m + \epsilon^3 \left( E(\xi)  \right)^m + \epsilon^4 \left( F(\xi)  \right)^m,
	\end{align}
	where 
	\begin{align*}
		\left( E(\xi)  \right)^m & =-6  \sum_{\kappa,\nu \geq 0} \overline{\mathfrak{C}}_{\kappa \nu m} \left(\pPhi^t (\xi) \right)^{\kappa}   \left(L_{  1}^{-1} P \mathfrak{f}^{(2)} (\pPhi^t (\xi)  )  \right)^{\nu}  , \\
		\left( F(\xi)  \right)^m & = -3  \sum_{\kappa,\nu \geq 0} \overline{\mathfrak{C}}_{\kappa \nu m}    \left(L_{  1}^{-1} P \mathfrak{f}^{(2)} (\pPhi^t (\xi)  )  \right)^{\kappa} \left(L_{  1}^{-1} P \mathfrak{f}^{(2)} (\pPhi^t (\xi)  )  \right)^{\nu} .
	\end{align*}
	In particular, we have
	\begin{align} \label{def:Eextrquad}
		\left( E(\xi)  \right)^m & =-6  \sum_{\kappa,\nu \geq 0} \overline{\mathfrak{C}}_{\kappa \nu m} \left(\pPhi^t (\xi) \right)^{\kappa}   \left(L_{  1}^{-1} P \mathfrak{f}^{(2)} (\pPhi^t (\xi)  )  \right)^{\nu} \\
		& = -6 \sum_{\kappa,\nu \geq 0} \overline{\mathfrak{C}}_{\kappa \nu m} \xi^{\kappa}\cos(  \ssomega_{\kappa}t)   \left(L_{  1}^{-1} P \mathfrak{f}^{(2)} (\pPhi^t (\xi)  )  \right)^{\nu} \\
		& = \left( E(\xi)  \right)_{-}^m + \left( E(\xi)  \right)_{+}^m,
	\end{align}
	where we set
	\begin{align*}
		 \left( E(\xi)  \right)_{-}^m & = 9\sum_{\kappa,\nu \geq 0} \overline{\mathfrak{C}}_{\kappa \nu m} \sum_{  \substack{ i,j \geq 0 \\ \ssomega_{i}-\ssomega_{j} \neq \pm \ssomega_{m}  }} \frac{ \overline{\mathfrak{C}}_{ijm} }{ \ssomega_{m}^2-(\ssomega_{i}-\ssomega_{j})^2 } \xi^i \xi^j\xi^{\kappa} \cos( (\ssomega_{i} - \ssomega_{j})t)\cos( \ssomega_{\kappa}  t), \\
		  \left( E(\xi)  \right)_{+}^m & = 9\sum_{\kappa,\nu \geq 0} \overline{\mathfrak{C}}_{\kappa \nu m} \sum_{  \substack{ i,j \geq 0 \\ \ssomega_{i}+\ssomega_{j} \neq \pm \ssomega_{m}  }} \frac{ \overline{\mathfrak{C}}_{ijm} }{ \ssomega_{m}^2-(\ssomega_{i}+\ssomega_{j})^2 } \xi^i \xi^j\xi^{\kappa} \cos( (\ssomega_{i} +\ssomega_{j})t)\cos( \ssomega_{\kappa}  t).
	\end{align*}
	Next, we firstly apply the linear flow and then the average in time to obtain 
	\begin{align}
	\left( \mathfrak{F}_{\epsilon}( \xi) \right)^m&:=  \epsilon^{-3} \left(
		 \langle  
		 \mathfrak{f}^{(2)}  \left( ( \cdot ) + L_{  1}^{-1} P \mathfrak{f}^{(2)} ( \cdot ) \right)   
		 \rangle(\epsilon \xi) \right)^m \nonumber \\
		 &=	  \frac{\epsilon^{-3}}{2\pi } \int_{0}^{2\pi }
		  \left( \pPhi^t \left( 
		 \mathfrak{f}^{(2)}  \left( \pPhi^t(\epsilon \xi ) + L_{  1}^{-1} P \mathfrak{f}^{(2)} \pPhi^t ( \epsilon \xi) \right)   
		  \right) \right)^m dt\nonumber  \\
		& = \frac{\epsilon^{-1}}{2\pi } \int_{0}^{2\pi } \left(\pPhi^t \left(\mathfrak{f}^{(2)} \left(\pPhi^t (\xi) \right) \right) \right)^m dt+
		\frac{1}{2\pi } \int_{0}^{2\pi } \left(\pPhi^t \left( E(\xi) \right) \right)^m dt +
\frac{\epsilon}{2\pi } \int_{0}^{2\pi } \left(\pPhi^t \left( F(\xi) \right) \right)^m dt\nonumber  \\
&= \epsilon^{-1} \langle \mathfrak{f}^{(2)} \rangle (\xi)  +\frac{1}{2\pi } \int_{0}^{2\pi } \left(\pPhi^t \left( E(\xi) \right) \right)^m dt +
\frac{\epsilon}{2\pi } \int_{0}^{2\pi } \left(\pPhi^t \left( F(\xi) \right) \right)^m dt\nonumber  \\
& =  \frac{1}{2\pi } \int_{0}^{2\pi } \left(\pPhi^t \left( E(\xi) \right) \right)^m dt +
\frac{\epsilon}{2\pi } \int_{0}^{2\pi } \left(\pPhi^t \left( F(\xi) \right) \right)^m dt,\label{FeLinear}
	\end{align}
	where we used the condition that $ \mathfrak{f}^{(2)} $ is non--resonant according to \eqref{ConditionNonResonantf2}. Then, the latter at $\epsilon=0$ boils down to
	\begin{align*}
		\mathfrak{F}_{0}( \xi) &= \frac{1}{2\pi } \int_{0}^{2\pi } \left(\pPhi^t \left( E(\xi) \right) \right)^m dt 
		 =  \frac{1}{2\pi } \int_{0}^{2\pi }  \left( E(\xi) \right)^m \cos(\ssomega_{m} t) dt \\
		& =  \frac{1}{2\pi } \int_{0}^{2\pi }  \left( E(\xi) \right)_{-}^m \cos(\ssomega_{m} t) dt +  \frac{1}{2\pi } \int_{0}^{2\pi }  \left( E(\xi) \right)_{+}^m \cos(\ssomega_{m} t) dt .
	\end{align*}
	Finally, we use the fact that
	\begin{align*}
		&  \int_{0}^{2\pi } \cos((\ssomega_{i}-\ssomega_{j})t) \cos(\ssomega_{\kappa}t)\cos(\ssomega_{m}t) dt = \\
		& \frac{1}{4} \sum_{\pm } \int_{0}^{2\pi } \cos((\ssomega_{i}-\ssomega_{j}\pm \ssomega_{\kappa} \pm \ssomega_{m})t) dt = 
		\frac{\pi }{2} \sum_{\pm } \mathds{1} (\ssomega_{i}-\ssomega_{j}\pm \ssomega_{\kappa} \pm \ssomega_{m} =0) 
	\end{align*}
	to compute
	\begin{align*}
		& \frac{1}{2\pi } \int_{0}^{2\pi }  \left( E(\xi) \right)_{-}^m \cos(\ssomega_{m} t) dt = \\
		&  \frac{9}{2\pi } \sum_{\kappa,\nu \geq 0} \overline{\mathfrak{C}}_{\kappa \nu m} \sum_{  \substack{ i,j \geq 0 \\ \ssomega_{i}-\ssomega_{j} \neq \pm \ssomega_{m}  }} \frac{ \overline{\mathfrak{C}}_{ijm} }{ \ssomega_{m}^2-(\ssomega_{i}-\ssomega_{j})^2 } \xi^i \xi^j\xi^{\kappa}  \int_{0}^{2\pi } \cos( (\ssomega_{i} - \ssomega_{j})t)\cos( \ssomega_{\kappa}  t) \cos(\ssomega_{m} t) dt = \\
		&  \frac{9}{4 }  \sum_{\kappa,\nu \geq 0} \overline{\mathfrak{C}}_{\kappa \nu m} \sum_{  \substack{ i,j \geq 0 \\ \ssomega_{i}-\ssomega_{j} \neq \pm \ssomega_{m}  }} \frac{ \overline{\mathfrak{C}}_{ijm} }{\ssomega_{m}^2-(\ssomega_{i}-\ssomega_{j})^2 } \xi^i \xi^j\xi^{\kappa}   \sum_{\pm } \mathds{1} (\ssomega_{i}-\ssomega_{j}\pm \ssomega_{\kappa} \pm \ssomega_{m} =0) .
	\end{align*}
	Similarly, we infer
	\begin{align*}
		& \frac{1}{2\pi } \int_{0}^{2\pi }  \left( E(\xi) \right)_{+}^m \cos(\ssomega_{m} t) dt = \\
		&  \frac{9}{4 }  \sum_{\kappa,\nu \geq 0} \overline{\mathfrak{C}}_{\kappa \nu m} \sum_{  \substack{ i,j \geq 0 \\ \ssomega_{i}+\ssomega_{j} \neq \pm \ssomega_{m}  }} \frac{ \overline{\mathfrak{C}}_{ijm} }{ \ssomega_{m}^2-(\ssomega_{i}+\ssomega_{j})^2 } \xi^i \xi^j\xi^{\kappa}   \sum_{\pm } \mathds{1} (\ssomega_{i}+\ssomega_{j}\pm \ssomega_{\kappa} \pm \ssomega_{m} =0) ,
	\end{align*}
	that completes the proof.
\end{proof}

Next, we compute $\mathfrak{F}_{\epsilon}( \xi)$ for small $\epsilon$ close to zero and 1--mode initial data. 
\begin{lemma}[Computation of $\mathfrak{F}_{\epsilon}(\xi)$ for small $\epsilon$ close to zero and 1--mode initial data]\label{LemmaComputemathfrakFforsfmallepsilon}
	Assume that $\ssomega_{n}=n+1$, for all integers $n \geq 0$. Let $\gamma  \geq 0$ be an integer, $\mathfrak{K}_{\gamma} \in \mathbb{R}$ and $\xi$ be the 1--mode initial data, that is
	\begin{align*}
		\xi^m :=\mathfrak{K}_{\gamma} \mathds{1}(m=\gamma), \Hquad m\geq 0.
	\end{align*}
	Then, for all integers $m\geq 0$, we have
	\begin{align*}
 	\left( \mathfrak{F}_{\epsilon}( \xi) \right)^m = \mathfrak{q}_{\gamma}    \mathfrak{K}_{\gamma} ^3\mathds{1}(m=\gamma)+ \mathcal{O}(\epsilon),
	\end{align*}
	where
	\begin{align*}
	 \mathfrak{q}_{\gamma}   :=	\frac{9}{4} \sum_{  \nu = 0}^{2\gamma}  \left(  \overline{\mathfrak{C}}_{\gamma \gamma \nu} \right)^2
  \left(\frac{2}{\ssomega_{\nu}^2 }  +  \frac{\mathds{1}(\ssomega_{\nu}^2 \neq (2\ssomega_{\gamma})^2) }{\ssomega_{\nu}^2-(2\ssomega_{\gamma})^2   } \right) 
	\end{align*}
\end{lemma}
\begin{proof}
	Let $\gamma  \geq 0$ be an integer, define $\xi$ as the 1--mode initial data, that is
	\begin{align*}
		\xi^m :=\mathfrak{K}_{\gamma} \mathds{1}(m=\gamma), \Hquad m\geq 0.
	\end{align*}
	and pick any integer $m\geq 0$. Then, we compute
	\begin{align*}
		\left(  \mathfrak{f}^{(2)} \left( \{ u^k: k\geq 0\}\right)  \right)^m &=
		-3\sum_{ i,j \geq 0} \overline{\mathfrak{C}}_{ijm}u^i u^j, \\
		\left(  \mathfrak{f}^{(2)}  \left( \pPhi^t(\xi) \right) \right)^m &=
		-3\sum_{ i,j \geq 0} \overline{\mathfrak{C}}_{ijm}\left(  \pPhi^t(\xi) \right)^i \left(  \pPhi^t(\xi) \right)^j 
		=	-3\sum_{ i,j \geq 0} \overline{\mathfrak{C}}_{ijm} \xi^i \xi^j \cos(\ssomega_{i}t) \cos(\ssomega_{j}t)   \\
	&=	 	-3 \mathfrak{K}_{\gamma} ^2  \overline{\mathfrak{C}}_{\gamma \gamma m}   \cos^2(\ssomega_{\gamma}t) 
	=	 	-\frac{3}{2}  \mathfrak{K}_{\gamma} ^2 \overline{\mathfrak{C}}_{\gamma \gamma m} \left[1+   \cos(2\ssomega_{\gamma}t) \right].
\end{align*}
Recall the definition of the eigenvalues $\ssomega_{i}:=i+2$, for all integers $i\geq 0$. Then, we have
\begin{align*}
	\ssomega_{m}  \neq 0 \Longleftrightarrow  m  \geq 0, \Hquad 
	\ssomega_{m}  \neq 2\ssomega_{\gamma} \Longleftrightarrow m \neq 2 \gamma +2.
\end{align*}
Hence, we infer
\begin{align*}	 
\left(P  \mathfrak{f}^{(2)}  \left(\pPhi^t(\xi) \right) \right)^m  &= -\frac{3}{2}   \mathfrak{K}_{\gamma} ^2 \overline{\mathfrak{C}}_{\gamma \gamma m} \left[1+ \mathds{1}(m \neq 2\gamma+2) \cos(2\ssomega_{\gamma}t) \right], \\	 
\left(L_{\ssomega}^{-1}P  \mathfrak{f}^{(2)}  \left(\pPhi^t(\xi) \right) \right)^m  &= -\frac{3}{2}  \mathfrak{K}_{\gamma} ^2 \overline{\mathfrak{C}}_{\gamma \gamma m} \left[ \frac{1}{\lambda_{m,0}}+ \mathds{1}(m \neq 2\gamma+2) \frac{1}{\lambda_{m,2\ssomega_{\gamma}}}\cos(2\ssomega_{\gamma}t) \right]  \\
&= -\frac{3}{2}  \mathfrak{K}_{\gamma} ^2 \overline{\mathfrak{C}}_{\gamma \gamma m} \left[ \frac{1}{\ssomega_{m}^2 }+  \frac{\mathds{1}(m \neq 2\gamma+2)}{\ssomega_{m}^2-(2\ssomega_{\gamma})^2 \ssomega^2 }\cos(2\ssomega_{\gamma}t) \right],
\end{align*}
 where we used the fact that the eigenvalues of $L_{\ssomega}$ are given by \eqref{EigenvaluesofLomega}, i.e. $\lambda_{ml}=\ssomega_{m}^2-l^2 \ssomega^2 $. In addition, we set $x =\epsilon \xi$ and $\left(	 \mathfrak{f}^{(2)}  \left(\pPhi^t(x)+L_{  1}^{-1}P  \mathfrak{f}^{(2)}  \left(\pPhi^t(x) \right) \right) \right)^m $ is given by 
\begin{align*} 
& -3\sum_{\kappa, \nu  \geq 0} \overline{\mathfrak{C}}_{\kappa \nu  m}\left(\pPhi^t(\epsilon \xi)+L_{  1}^{-1}P  \mathfrak{f}^{(2)}  \left(\pPhi^t(\epsilon \xi) \right) \right)^{\kappa} \left(\pPhi^t(\epsilon \xi)+L_{  1}^{-1}P  \mathfrak{f}^{(2)}  \left(\pPhi^t(\epsilon \xi) \right) \right)^{\nu }= \\
& -3\sum_{\kappa, \nu  \geq 0} \overline{\mathfrak{C}}_{\kappa \nu  m}\left[ \left(\pPhi^t(\epsilon \xi) \right)^{\kappa} + \left(L_{  1}^{-1}P  \mathfrak{f}^{(2)}  \left(\pPhi^t(\epsilon \xi) \right) \right)^{\kappa} \right] \left[ \left(\pPhi^t(\epsilon \xi) \right)^{\nu } + \left( L_{  1}^{-1}P  \mathfrak{f}^{(2)}  \left(\pPhi^t(\epsilon \xi) \right)  \right)^{\nu } \right]= \\
& -3\sum_{\kappa, \nu  \geq 0} \overline{\mathfrak{C}}_{\kappa \nu  m}\Big[ \left(\pPhi^t(\epsilon \xi) \right)^{\kappa}\left(\pPhi^t(\epsilon \xi) \right)^{\nu } +  
 \left(\pPhi^t(\epsilon \xi) \right)^{\kappa} \left(L_{  1}^{-1}P  \mathfrak{f}^{(2)}  \left(\pPhi^t(\epsilon \xi) \right) \right)^{\nu } \\
 &+ \left(\pPhi^t(\epsilon \xi) \right)^{\nu } \left(L_{  1}^{-1}P  \mathfrak{f}^{(2)}  \left(\pPhi^t(\epsilon \xi) \right) \right)^{\kappa} +
 \left(L_{  1}^{-1}P  \mathfrak{f}^{(2)}  \left(\pPhi^t(\epsilon \xi) \right) \right)^{\kappa}
 \left(L_{  1}^{-1}P  \mathfrak{f}^{(2)}  \left(\pPhi^t(\epsilon \xi) \right) \right)^{\nu }
\Big]= \\
& -3\sum_{\kappa, \nu  \geq 0} \overline{\mathfrak{C}}_{\kappa \nu  m} \left(\pPhi^t(\epsilon \xi) \right)^{\kappa}\left(\pPhi^t(\epsilon \xi) \right)^{\nu } -6\sum_{\kappa, \nu  \geq 0} \overline{\mathfrak{C}}_{\kappa \nu  m}  
 \left(\pPhi^t(\epsilon \xi) \right)^{\kappa} \left(L_{  1}^{-1}P  \mathfrak{f}^{(2)}  \left(\pPhi^t(\epsilon \xi) \right) \right)^{\nu } \\
 & -3\sum_{\kappa, \nu  \geq 0} \overline{\mathfrak{C}}_{\kappa \nu  m}
 \left(L_{  1}^{-1}P  \mathfrak{f}^{(2)}  \left(\pPhi^t(\epsilon \xi) \right) \right)^{\kappa}
 \left(L_{  1}^{-1}P  \mathfrak{f}^{(2)}  \left(\pPhi^t(\epsilon \xi) \right) \right)^{\nu } 
 = \\
& \epsilon^2 \left(  \mathfrak{f}^{(2)} \left( \pPhi^t( \xi) \right) \right)^m + \epsilon^3 \left(E(\xi) \right)^{m} + \epsilon^4 \left(F(\xi) \right)^{m},
\end{align*}
where we set
\begin{align*}
  \left(F(\xi)\right)^{m} &:=-3\sum_{\kappa, \nu  \geq 0} \overline{\mathfrak{C}}_{\kappa \nu  m}
 \left(L_{  1}^{-1}P  \mathfrak{f}^{(2)}  \left(\pPhi^t(\xi) \right) \right)^{\kappa}
 \left(L_{  1}^{-1}P  \mathfrak{f}^{(2)}  \left(\pPhi^t(\xi) \right) \right)^{\nu }, \\
 \left(E(\xi)\right)^{m} &:= -6\sum_{\kappa, \nu  \geq 0} \overline{\mathfrak{C}}_{\kappa \nu  m}  
 \left(\pPhi^t(\xi) \right)^{\kappa} \left(L_{  1}^{-1}P  \mathfrak{f}^{(2)}  \left(\pPhi^t(\xi) \right) \right)^{\nu } \\
 & = 9 \mathfrak{K}_{\gamma} ^2 \sum_{\kappa, \nu  \geq 0} \overline{\mathfrak{C}}_{\kappa \nu  m}  
\xi^{\kappa}\cos(\ssomega_{\kappa}t) \overline{\mathfrak{C}}_{\gamma \gamma \nu} \left[ \frac{1}{\ssomega_{\nu}^2 }+  \frac{\mathds{1}(\nu \neq 2\gamma+2)}{\ssomega_{\nu}^2-(2\ssomega_{\gamma})^2   }\cos(2\ssomega_{\gamma}t) \right] \\
& = 9 \mathfrak{K}_{\gamma} ^3 \sum_{  \nu  \geq 0} \overline{\mathfrak{C}}_{ \gamma \nu  m}   \overline{\mathfrak{C}}_{\gamma \gamma \nu} 
\left[ \frac{1}{\ssomega_{\nu}^2 } \cos(\ssomega_{\gamma}t)+  \frac{\mathds{1}(\nu \neq 2\gamma+2)}{\ssomega_{\nu}^2-(2\ssomega_{\gamma})^2   }\cos(2\ssomega_{\gamma}t) \cos(\ssomega_{\gamma}t) \right] \\
& = 9 \mathfrak{K}_{\gamma} ^3 \sum_{  \nu  \geq 0} \overline{\mathfrak{C}}_{ \gamma \nu  m}   \overline{\mathfrak{C}}_{\gamma \gamma \nu} 
\left[ \left(\frac{1}{\ssomega_{\nu}^2 }  +\frac{1}{2}  \frac{\mathds{1}(\nu \neq 2\gamma+2)}{\ssomega_{\nu}^2-(2\ssomega_{\gamma})^2   } \right)\cos \left( \ssomega _{\gamma }t \right)+\frac{1}{2} \frac{\mathds{1}(\nu \neq 2\gamma+2)}{\ssomega_{\nu}^2-(2\ssomega_{\gamma})^2   }\cos \left(3 \ssomega _{\gamma }t \right) \right].
\end{align*}
Now, we firstly apply the linear flow and then the average in time to obtain that 
\begin{align*}
	\mathfrak{F}_{\epsilon}( \xi) &:=  \epsilon^{-3}
		 \langle  
		 \mathfrak{f}^{(2)}  \left( ( \cdot ) + L_{1 }^{-1} P \mathfrak{f}^{(2)} ( \cdot ) \right)   
		 \rangle(\epsilon \xi) \\
		 &=  \frac{\epsilon^{-3}}{2\pi }  \int_{0}^{2\pi } \left(\pPhi^t  \left(	 \mathfrak{f}^{(2)}  \left(\pPhi^t(\epsilon \xi)+L_{1 }^{-1}P  \mathfrak{f}^{(2)}  \left(\pPhi^t(\epsilon \xi) \right) \right) \right) \right)^m   
  dt \\
	&= \frac{\epsilon^{-1}}{2\pi }  \int_{0}^{2\pi } \left( \pPhi^{t}\left(  \mathfrak{f}^{(2)} \left( \pPhi^t( \xi) \right) \right) \right)^m dt + \frac{1 }{2\pi }  \int_{0}^{2\pi } \left( \pPhi^{t} \left(E(\xi) \right) \right)^{m} dt +  \frac{\epsilon }{2\pi }  \int_{0}^{2\pi } \left( \pPhi^{t} \left(F(\xi) \right) \right)^{m} dt.
\end{align*}
On the one hand, due to \eqref{ConditionNonResonantf2}, we have
\begin{align*}
	\frac{1}{2\pi }  \int_{0}^{2\pi } \left( \pPhi^{t}\left(  \mathfrak{f}^{(2)} \left( \pPhi^t( \xi) \right) \right) \right)^m dt = \langle  \mathfrak{f}^{(2)}  \rangle(\xi) = 0.
\end{align*}
On the other hand, we use the fact that 
\begin{align*}
\frac{1}{2\pi}	 \int_{0}^{2\pi } \cos(a t)\cos(bt)  dt  =  \frac{1 }{2}    \mathds{1} (|a|=|b|), 
\end{align*}
for all integers $a,b \geq 0$ together with
\begin{align*}
	\ssomega _{m} =\ssomega _{\gamma } \Longleftrightarrow m=\gamma, \Hquad 
	\ssomega _{m} =3\ssomega _{\gamma } \Longleftrightarrow m=3 \gamma +4, 
\end{align*}
to compute
\begin{align*}
&	\frac{1 }{2\pi }  \int_{0}^{2\pi } \left( \pPhi^{t} \left(E(\xi) \right) \right)^{m} dt =
	\frac{1 }{2\pi }  \int_{0}^{2\pi }  \left(E(\xi) \right) ^{m} \cos(\ssomega_{m}t) dt  \\
	& =   9 \mathfrak{K}_{\gamma} ^3 \sum_{  \nu  \geq 0} \overline{\mathfrak{C}}_{ \gamma \nu  m}   \overline{\mathfrak{C}}_{\gamma \gamma \nu} 
\Bigg[ \left(\frac{1}{\ssomega_{\nu}^2 }  +\frac{1}{2}  \frac{\mathds{1}(\nu \neq 2\gamma+2)}{\ssomega_{\nu}^2-(2\ssomega_{\gamma})^2   } \right) \frac{1 }{2\pi }  \int_{0}^{2\pi }\cos \left( \ssomega _{\gamma }t \right)\cos \left( \ssomega _{m }t \right)dt \\
& +\frac{1}{2} \frac{\mathds{1}(\nu \neq 2\gamma+2)}{\ssomega_{\nu}^2-(2\ssomega_{\gamma})^2   } \frac{1 }{2\pi }  \int_{0}^{2\pi }\cos \left(3 \ssomega _{\gamma }t \right)\cos \left(\ssomega _{m }t \right)dt \Bigg] \\
& = 9 \mathfrak{K}_{\gamma} ^3 \sum_{  \nu  \geq 0} \overline{\mathfrak{C}}_{ \gamma \nu  m}   \overline{\mathfrak{C}}_{\gamma \gamma \nu} 
\Bigg[  \left(\frac{1}{2\ssomega_{\nu}^2 }  +\frac{1}{4}  \frac{\mathds{1}(\nu \neq 2\gamma+2)}{\ssomega_{\nu}^2-(2\ssomega_{\gamma})^2   } \right)   \mathds{1}(m=\gamma) \\
& +\frac{1}{4} \frac{\mathds{1}(\nu \neq 2\gamma+2)}{\ssomega_{\nu}^2-(2\ssomega_{\gamma})^2   }  \mathds{1}(m=3\gamma+4) \Bigg] \\
& = \frac{9}{4} \mathfrak{K}_{\gamma} ^3 \sum_{  \nu  \geq 0}  \left(  \overline{\mathfrak{C}}_{\gamma \gamma \nu} \right)^2
  \left(\frac{2}{\ssomega_{\nu}^2 }  +  \frac{\mathds{1}(\nu \neq 2\gamma+2)}{\ssomega_{\nu}^2-(2\ssomega_{\gamma})^2   } \right)   \mathds{1}(m=\gamma) \\
& + \frac{9}{4} \mathfrak{K}_{\gamma} ^3 \sum_{ \substack{ \nu  \geq 0 \\\nu \neq 2\gamma+2 }}  \frac{\overline{\mathfrak{C}}_{ \gamma ,\nu  ,3\gamma+4}   \overline{\mathfrak{C}}_{\gamma \gamma \nu} }{\ssomega_{\nu}^2-(2\ssomega_{\gamma})^2   }  \mathds{1}(m=3\gamma+4).
\end{align*}
Finally, we note that 
 \begin{align*}
 	\overline{\mathfrak{C}}_{ \gamma ,\nu  ,3\gamma+4}   \overline{\mathfrak{C}}_{\gamma \gamma \nu}=0,
 \end{align*}
for all integers $\gamma \geq 0$ and $\nu \geq 0$. This follows immediately from the fact that $\overline{\mathfrak{C}}_{ijm}=0$ for all integers $i,j,m \geq 0$ with $i+j<m$ due to Lemma \ref{LemmaClosedFormulasModel3YMCbarijm} below. Specifically, we have
\begin{itemize}
	\item $\nu > 2\gamma \Longrightarrow \overline{\mathfrak{C}}_{\gamma \gamma \nu}=0$,
	\item $0\leq \nu \leq 2\gamma \Longrightarrow \gamma+\nu \leq 3\gamma <3\gamma+4 \Longrightarrow \overline{\mathfrak{C}}_{ \gamma ,\nu  ,3\gamma+4}=0 $.
\end{itemize} 
Consequently, we conclude that
 \begin{align*}
 	\frac{1 }{2\pi }  \int_{0}^{2\pi } \left( \pPhi^{t} \left(E(\xi) \right) \right)^{m} dt  &=	\frac{9}{4} \mathfrak{K}_{\gamma} ^3 \sum_{  \nu = 0}^{2\gamma}  \left(  \overline{\mathfrak{C}}_{\gamma \gamma \nu} \right)^2
  \left(\frac{2}{\ssomega_{\nu}^2 }  +  \frac{1}{\ssomega_{\nu}^2-(2\ssomega_{\gamma})^2   } \right)   \mathds{1}(m=\gamma) ,
 \end{align*}
 that completes the proof.
\end{proof} 
Finally, we compute the differential of $\mathfrak{F}_{0}( \xi)$ at the 1--mode initial data.
\begin{lemma}[Differential of $\mathfrak{F}_{0}( \xi)$ at the 1--mode initial data]\label{LemmaDifferentialMathFrakF0LongComputation}
	Let $\gamma  \geq 0$ be an integer, $\mathfrak{K}_{\gamma} \in \mathbb{R}$ and $\xi$ be the 1--mode initial data, that is
	\begin{align*}
		\xi^m :=\mathfrak{K}_{\gamma} \mathds{1}(m=\gamma), \Hquad m\geq 0.
	\end{align*}
	Then, for all $h \in l_{s+3}^2$ and integers $m\geq 0$, we have that 
	\begin{align*}
		\left(	d  \mathfrak{F}_{0}( \xi)[h] \right)^m & =\mathfrak{K}_{\gamma}^{2} \left[  \mathfrak{a}_{\gamma m}   h^{m} +  \mathds{1}(0\leq m\leq 2\gamma) \mathfrak{b}_{\gamma m} h^{2\gamma-m}   \right],
	\end{align*}
	where
	\begin{align*}
		 \mathfrak{a}_{\gamma m}  &=  \frac{9}{2}   \sum_{ \substack{ \nu = 0  }}^{m+\gamma}    \frac{\left(\overline{\mathfrak{C}}_{\gamma \nu m} \right)^2 }{\ssomega_{\nu}^2-(\ssomega_{m}+\ssomega_{\gamma})^2  }  +\frac{9}{2}  \sum_{ \substack{ \nu = 0 \\ \nu \neq \pm(m-\gamma)-2  }}^{m+\gamma}   \frac{\left(\overline{\mathfrak{C}}_{m\gamma \nu } \right)^2}{\ssomega_{\nu}^2-(\ssomega_{m}-\ssomega_{\gamma})^2  }  
	+  \frac{9}{4}  \sum_{\substack{ \nu = 0  } }^{2\gamma}   
		 \frac{\overline{\mathfrak{C}}_{m \nu m}\overline{\mathfrak{C}}_{\gamma\gamma \nu }}{\ssomega_{\nu}^2 }   ,   \\
		\mathfrak{b}_{\gamma m}  &= \frac{9}{4}   \sum_{\substack{ \nu = 0  } }^{2\gamma}  
		 \frac{\overline{\mathfrak{C}}_{2\gamma-m, \nu, m} \overline{\mathfrak{C}}_{\gamma\gamma \nu }}{\ssomega_{\nu}^2-(2\ssomega_{\gamma})^2  }  +\frac{9}{2} \sum_{ \substack{ \nu = 0 \\ \nu \neq \pm(m-\gamma)-2  }}^{m+\gamma}   \frac{\overline{\mathfrak{C}}_{\gamma \nu m} \overline{\mathfrak{C}}_{2\gamma-m,\gamma ,\nu }}{\ssomega_{\nu}^2-(\ssomega_{2\gamma-m}-\ssomega_{\gamma})^2  } .
	\end{align*}
\end{lemma}
\begin{proof}
	Let $\gamma  \geq 0$ be an integer, $\mathfrak{K}_{\gamma} \in \mathbb{R}$ and $\xi$ to be the 1--mode initial data as above and pick any $h \in l_{s+3 }^2$, $\epsilon >0$ and integer $m\geq 0$. Then, we set $\hat{\xi}=\xi+\epsilon h$ and according to Lemma \ref{LemmaDefinitionMathfrakF}, we have
	\begin{align*}
		\left( \mathfrak{F}_{0}( \hat{\xi}) \right)^m =\left( \mathfrak{F}_{0-}( \hat{\xi})  \right)^m +\left( \mathfrak{F}_{0+}( \hat{\xi}) \right)^m,
	\end{align*}
	where
\begin{align*}
	\left(	\mathfrak{F}_{0-}( \hat{\xi}) \right)^m= & \frac{9}{4} \sum_{\kappa,\nu \geq 0} \overline{\mathfrak{C}}_{\kappa \nu m} \sum_{ \substack{i,j \geq 0 \\ \ssomega_{i}-\ssomega_{j} \neq \pm \ssomega_{\nu}  }  } \frac{\overline{\mathfrak{C}}_{ij\nu }}{\ssomega_{\nu}^2-(\ssomega_{i}-\ssomega_{j})^2  } \hat{\xi}^i \hat{\xi}^j \hat{\xi}^{\kappa} \sum_{\pm} \mathds{1}(\ssomega_{i}-\ssomega_{j}\pm \ssomega_{\kappa} \pm \ssomega_{m}=0), \\
		\left( \mathfrak{F}_{0+}( \hat{\xi})  \right)^m = &\frac{9}{4} \sum_{\kappa,\nu \geq 0} \overline{\mathfrak{C}}_{\kappa \nu m} \sum_{ \substack{i,j \geq 0 \\ \ssomega_{i}+\ssomega_{j} \neq \pm \ssomega_{\nu}  }  } \frac{\overline{\mathfrak{C}}_{ij\nu }}{\ssomega_{\nu}^2-(\ssomega_{i}+\ssomega_{j})^2  } \hat{\xi}^i \hat{\xi}^j \hat{\xi}^{\kappa} \sum_{\pm} \mathds{1}(\ssomega_{i}+\ssomega_{j}\pm \ssomega_{\kappa} \pm \ssomega_{m}=0).
	\end{align*} 
	We expand
	\begin{align*}
		\mathfrak{F}_{0-}( \hat{\xi}) &= \mathfrak{F}_{0-}(\xi) + \epsilon \cdot d \mathfrak{F}_{0-}(\xi)[h] + \mathcal{O}(\epsilon^2), \\
		\mathfrak{F}_{0+}( \hat{\xi}) &= \mathfrak{F}_{0+}(\xi) + \epsilon \cdot d \mathfrak{F}_{0+}(\xi)[h] + \mathcal{O}(\epsilon^2),
	\end{align*}
	where, using the definition of the 1--mode initial data, we obtain
	\begin{align*}
		\left( d \mathfrak{F}_{0-}(\xi)[h]  \right)^m &=
		\frac{9}{4} \sum_{\kappa,\nu \geq 0} \overline{\mathfrak{C}}_{\kappa \nu m} \sum_{ \substack{i,j \geq 0 \\ \ssomega_{i}-\ssomega_{j} \neq \pm \ssomega_{\nu}  }  } \frac{\overline{\mathfrak{C}}_{ij\nu }}{\ssomega_{\nu}^2-(\ssomega_{i}-\ssomega_{j})^2  }\cdot \\
		& \cdot  \left[
		  h^i  \xi^j \xi^\kappa 
		  + \xi^i h^j \xi^\kappa +
		  \xi^i \xi^j h^\kappa 
		\right]\sum_{\pm} \mathds{1}(\ssomega_{i}-\ssomega_{j}\pm \ssomega_{\kappa} \pm \ssomega_{m}=0) \\
		& =\mathfrak{K}_{\gamma}^2 \Bigg\{ \frac{9}{4} \sum_{ \nu \geq 0} \overline{\mathfrak{C}}_{\gamma \nu m} \sum_{ \substack{i \geq 0 \\ \ssomega_{i}-\ssomega_{\gamma} \neq \pm \ssomega_{\nu}  }  } \frac{\overline{\mathfrak{C}}_{i\gamma \nu }}{\ssomega_{\nu}^2-(\ssomega_{i}-\ssomega_{\gamma})^2  }  
		  h^i     \sum_{\pm} \mathds{1}(\ssomega_{i}-\ssomega_{\gamma}\pm \ssomega_{\gamma} \pm \ssomega_{m}=0) \\
		&+ \frac{9}{4} \sum_{ \nu \geq 0} \overline{\mathfrak{C}}_{\gamma \nu m} \sum_{ \substack{j \geq 0 \\ \ssomega_{\gamma}-\ssomega_{j} \neq \pm \ssomega_{\nu}  }  } \frac{\overline{\mathfrak{C}}_{\gamma j\nu }}{\ssomega_{\nu}^2-(\ssomega_{\gamma}-\ssomega_{j})^2  }   
		    h^j    \sum_{\pm} \mathds{1}(\ssomega_{\gamma}-\ssomega_{j}\pm \ssomega_{\gamma} \pm \ssomega_{m}=0)  \\
		&+ \frac{9}{4} \sum_{\kappa,\nu \geq 0} \overline{\mathfrak{C}}_{\kappa \nu m} 
		 \frac{\overline{\mathfrak{C}}_{\gamma\gamma \nu }}{\ssomega_{\nu}^2-(\ssomega_{\gamma}-\ssomega_{\gamma})^2  }   
		 h^\kappa   \sum_{\pm} \mathds{1}( \pm \ssomega_{\kappa} \pm \ssomega_{m}=0) \Bigg\}.
	\end{align*} 
Recall the definition of the eigenvalues $\ssomega_i:=i+2$ for all integers $i \geq 0$ and also recall that $m,i,j,\kappa,\nu,\gamma \geq 0$. Then, we have
\begin{align*}
	& \ssomega_{i}-\ssomega_{\gamma}
	\pm 
	\ssomega_{\gamma} 
	\pm \ssomega_{m}=0 \text{ and } \ssomega_{i}-\ssomega_{\gamma} \neq \pm \ssomega_{\nu} \Longleftrightarrow \\
	& \begin{dcases}
		 \ssomega_{i}-\ssomega_{\gamma}
	+ 
	\ssomega_{\gamma} 
	+ \ssomega_{m}=0 \text{ and } \ssomega_{i}-\ssomega_{\gamma} \neq \pm \ssomega_{\nu}   \\
	 \ssomega_{i}-\ssomega_{\gamma}
	+
	\ssomega_{\gamma} 
	- \ssomega_{m}=0 \text{ and } \ssomega_{i}-\ssomega_{\gamma} \neq \pm \ssomega_{\nu}   \\
	 \ssomega_{i}-\ssomega_{\gamma}
	-
	\ssomega_{\gamma} 
	+ \ssomega_{m}=0 \text{ and } \ssomega_{i}-\ssomega_{\gamma} \neq \pm \ssomega_{\nu}   \\
	 \ssomega_{i}-\ssomega_{\gamma}
	-
	\ssomega_{\gamma} 
	- \ssomega_{m}=0 \text{ and } \ssomega_{i}-\ssomega_{\gamma} \neq \pm \ssomega_{\nu}   
	\end{dcases} \Longleftrightarrow \\ 
	& \begin{dcases}
          i=m \text{ and }  \nu \neq  \pm(m-\gamma)-2      \\
           i=2 \gamma -m \text{ and }  m \leq 2\gamma \text{ and }    \nu  \neq \pm(m-\gamma) -2     \\
            i=2 \gamma +m+4 \text{ and }    \nu  \neq  2+m + \gamma    
	\end{dcases} 
\end{align*}
and
\begin{align*}
	& \ssomega_{\gamma}-\ssomega_{j}\pm \ssomega_{\gamma} \pm \ssomega_{m}=0 \text{ and } \ssomega_{\gamma}-\ssomega_{j} \neq \pm \ssomega_{\nu}  \Longleftrightarrow \\
	& \begin{dcases}
		\ssomega_{\gamma}-\ssomega_{j}
		+
		\ssomega_{\gamma} 
		+
		\ssomega_{m}=0 \text{ and } \ssomega_{\gamma}-\ssomega_{j} \neq \pm \ssomega_{\nu}  \\
		\ssomega_{\gamma}-\ssomega_{j}
		+
		\ssomega_{\gamma} 
		- 
		\ssomega_{m}=0 \text{ and } \ssomega_{\gamma}-\ssomega_{j} \neq \pm \ssomega_{\nu} \\
		\ssomega_{\gamma}-\ssomega_{j}
		- 
		\ssomega_{\gamma} 
		+ 
		\ssomega_{m}=0 \text{ and } \ssomega_{\gamma}-\ssomega_{j} \neq \pm \ssomega_{\nu}  \\
		\ssomega_{\gamma}-\ssomega_{j}
		-
		\ssomega_{\gamma} 
		- 
		\ssomega_{m}=0 \text{ and } \ssomega_{\gamma}-\ssomega_{j} \neq \pm \ssomega_{\nu} 
	\end{dcases} \Longleftrightarrow \\
	& \begin{dcases}
		j= 2 \gamma +m+4 \text{ and } \nu \neq  2+m+ \gamma  \\
	j= 2 \gamma -m \text{ and }  m \leq 2\gamma \text{ and } \nu \neq \pm(m-\gamma) -2  \\
		j=m \text{ and } \nu \neq \pm(m-\gamma) -2 
	\end{dcases}  
\end{align*}
and
\begin{align*}
	\pm \ssomega_{\kappa} \pm \ssomega_{m}=0 \Longleftrightarrow 
	\begin{dcases}
		m=-\kappa -4 \\
		m=\kappa
	\end{dcases} \Longleftrightarrow
	\kappa = m.
\end{align*}
Therefore, we infer
	\begin{align*}
		\mathfrak{K}_{\gamma}^{-2} \left( d \mathfrak{F}_{0-}(\xi)[h]  \right)^m
		&   =     \frac{9}{4} h^m \sum_{ \substack{ \nu = 0 \\ \nu \neq \pm(m-\gamma)-2  }}^{\infty}   \frac{\left(\overline{\mathfrak{C}}_{m\gamma \nu } \right)^2}{\ssomega_{\nu}^2-(\ssomega_{m}-\ssomega_{\gamma})^2  }  
		   \\
		  &   + \frac{9}{4}h^{2\gamma-m} \mathds{1}(0\leq m\leq 2\gamma) \sum_{ \substack{ \nu = 0 \\ \nu \neq \pm(m-\gamma)-2  }}^{\infty}   \frac{\overline{\mathfrak{C}}_{\gamma \nu m} \overline{\mathfrak{C}}_{2\gamma-m,\gamma ,\nu }}{\ssomega_{\nu}^2-(\ssomega_{2\gamma-m}-\ssomega_{\gamma})^2  }  
		   \\
		  &   + \frac{9}{4}h^{2\gamma+m+4}  \sum_{ \substack{ \nu = 0 \\ \nu \neq  2+m+\gamma  }}^{\infty}    \frac{\overline{\mathfrak{C}}_{\gamma \nu m}\overline{\mathfrak{C}}_{2\gamma+m+4,\gamma ,\nu }}{\ssomega_{\nu}^2-(\ssomega_{2\gamma+m+4}-\ssomega_{\gamma})^2  }  
		      \\
		&+ \frac{9}{4} h^{2\gamma+m+4}  \sum_{ \substack{\nu = 0 \\ \nu \neq 2+m+\gamma  }}^{\infty}  \frac{\overline{\mathfrak{C}}_{\gamma \nu m}\overline{\mathfrak{C}}_{\gamma ,2\gamma+m+4,\nu }}{\ssomega_{\nu}^2-(\ssomega_{\gamma}-\ssomega_{2\gamma+m+4})^2  }   
		 \\
		  & + \frac{9}{4}h^{2\gamma-m}   \mathds{1}(0\leq m\leq 2\gamma) \sum_{\substack{ \nu = 0 \\ \nu \neq \pm (m-\gamma)-2  }}^{\infty}  \frac{\overline{\mathfrak{C}}_{\gamma \nu m}\overline{\mathfrak{C}}_{\gamma ,2\gamma-m,\nu }}{\ssomega_{\nu}^2-(\ssomega_{\gamma}-\ssomega_{2\gamma-m})^2  }   
		     \\
		 &   + \frac{9}{4} h^m   \sum_{\substack{ \nu = 0 \\ \nu \neq \pm(m-\gamma)-2  }}^{\infty}   \frac{\left(\overline{\mathfrak{C}}_{\gamma \nu m}\right)^2}{\ssomega_{\nu}^2-(\ssomega_{\gamma}-\ssomega_{m})^2  }   
		       \\ 
		&+ \frac{9}{4} h^m  \sum_{ \nu =0}^{\infty}  
		 \frac{\overline{\mathfrak{C}}_{m \nu m}\overline{\mathfrak{C}}_{\gamma\gamma \nu }}{\ssomega_{\nu}^2 } .	
\end{align*} 
The latter can be written as
\begin{align*}
	&\mathfrak{K}_{\gamma}^{-2} \left( d \mathfrak{F}_{0-}(\xi)[h]  \right)^m  = 
	  h^m \left[
	\frac{9}{2}  \sum_{ \substack{ \nu = 0 \\ \nu \neq \pm(m-\gamma)-2  }}^{\infty}   \frac{\left(\overline{\mathfrak{C}}_{m\gamma \nu } \right)^2}{\ssomega_{\nu}^2-(\ssomega_{m}-\ssomega_{\gamma})^2  }  
	+  \frac{9}{4}   \sum_{ \nu =0}^{\infty} 
		 \frac{\overline{\mathfrak{C}}_{m \nu m}\overline{\mathfrak{C}}_{\gamma\gamma \nu }}{\ssomega_{\nu}^2 } 
	\right] \\\
	&+h^{2\gamma-m} \mathds{1}(0\leq m\leq 2\gamma) \left[
	 \frac{9}{2} \sum_{ \substack{ \nu = 0 \\ \nu \neq \pm(m-\gamma)-2  }}^{\infty}   \frac{\overline{\mathfrak{C}}_{\gamma \nu m} \overline{\mathfrak{C}}_{2\gamma-m,\gamma ,\nu }}{\ssomega_{\nu}^2-(\ssomega_{2\gamma-m}-\ssomega_{\gamma})^2  }   
	\right] \\
	&+ h^{2\gamma+m+4}\left[
	 \frac{9}{2}   \sum_{ \substack{ \nu = 0 \\ \nu \neq  2+m+\gamma  }}^{\infty}    \frac{\overline{\mathfrak{C}}_{\gamma \nu m}\overline{\mathfrak{C}}_{2\gamma+m+4,\gamma ,\nu }}{\ssomega_{\nu}^2-(\ssomega_{2\gamma+m+4}-\ssomega_{\gamma})^2  }      
	\right]  .
\end{align*}
Similarly, using the definition of the 1--mode initial data, we obtain
	\begin{align*}
		\left( d \mathfrak{F}_{0+}(\xi)[h]  \right)^m &=
		\frac{9}{4} \sum_{\kappa,\nu \geq 0} \overline{\mathfrak{C}}_{\kappa \nu m} \sum_{ \substack{i,j \geq 0 \\ \ssomega_{i}+\ssomega_{j} \neq \pm \ssomega_{\nu}  }  } \frac{\overline{\mathfrak{C}}_{ij\nu }}{\ssomega_{\nu}^2-(\ssomega_{i}+\ssomega_{j})^2  }\cdot \\
		& \cdot  \left[
		  h^i  \xi^j \xi^\kappa 
		  + \xi^i h^j \xi^\kappa +
		  \xi^i \xi^j h^\kappa 
		\right]\sum_{\pm} \mathds{1}(\ssomega_{i}+\ssomega_{j}\pm \ssomega_{\kappa} \pm \ssomega_{m}=0) \\
		& =\mathfrak{K}_{\gamma}^2 \Bigg\{ \frac{9}{4} \sum_{ \nu \geq 0} \overline{\mathfrak{C}}_{\gamma \nu m} \sum_{ \substack{i \geq 0 \\ \ssomega_{i}+\ssomega_{\gamma} \neq \pm \ssomega_{\nu}  }  } \frac{\overline{\mathfrak{C}}_{i\gamma \nu }}{\ssomega_{\nu}^2-(\ssomega_{i}+\ssomega_{\gamma})^2  }  
		  h^i     \sum_{\pm} \mathds{1}(\ssomega_{i}+\ssomega_{\gamma}\pm \ssomega_{\gamma} \pm \ssomega_{m}=0) \\
		&+ \frac{9}{4} \sum_{ \nu \geq 0} \overline{\mathfrak{C}}_{\gamma \nu m} \sum_{ \substack{j \geq 0 \\ \ssomega_{\gamma}+\ssomega_{j} \neq \pm \ssomega_{\nu}  }  } \frac{\overline{\mathfrak{C}}_{\gamma j\nu }}{\ssomega_{\nu}^2-(\ssomega_{\gamma}+\ssomega_{j})^2  }   
		    h^j    \sum_{\pm} \mathds{1}(\ssomega_{\gamma}+\ssomega_{j}\pm \ssomega_{\gamma} \pm \ssomega_{m}=0)  \\
		&+ \frac{9}{4} \sum_{\substack{ \kappa,\nu \geq 0  }}\overline{\mathfrak{C}}_{\kappa \nu m}  
		 \frac{ \mathds{1}(\ssomega_{\gamma}+\ssomega_{\gamma} \neq \pm \ssomega_{\nu} )\overline{\mathfrak{C}}_{\gamma\gamma \nu }}{\ssomega_{\nu}^2-(\ssomega_{\gamma}+\ssomega_{\gamma})^2  }   
		 h^\kappa   \sum_{\pm} \mathds{1}(\ssomega_{\gamma}+\ssomega_{\gamma} \pm \ssomega_{\kappa} \pm \ssomega_{m}=0) \Bigg\}.
	\end{align*} 
As before, recall the definition of the eigenvalues $\ssomega_i:=i+2$ for all integers $i \geq 0$ and also recall that $m,i,j,\kappa,\nu,\gamma \geq 0$. Then, we have
\begin{align*}
	& \ssomega_{i}+\ssomega_{\gamma}
	\pm 
	\ssomega_{\gamma} 
	\pm \ssomega_{m}=0 \text{ and } \ssomega_{i}+\ssomega_{\gamma} \neq \pm \ssomega_{\nu} \Longleftrightarrow \\
	& \begin{dcases}
		 \ssomega_{i}+\ssomega_{\gamma}
	+ 
	\ssomega_{\gamma} 
	+ \ssomega_{m}=0 \text{ and } \ssomega_{i}+\ssomega_{\gamma} \neq \pm \ssomega_{\nu}   \\
	 \ssomega_{i}+\ssomega_{\gamma}
	+
	\ssomega_{\gamma} 
	- \ssomega_{m}=0 \text{ and } \ssomega_{i}+\ssomega_{\gamma} \neq \pm \ssomega_{\nu}   \\
	 \ssomega_{i}+\ssomega_{\gamma}
	-
	\ssomega_{\gamma} 
	+ \ssomega_{m}=0 \text{ and } \ssomega_{i}+\ssomega_{\gamma} \neq \pm \ssomega_{\nu}   \\
	 \ssomega_{i}+\ssomega_{\gamma}
	-
	\ssomega_{\gamma} 
	- \ssomega_{m}=0 \text{ and } \ssomega_{i}+\ssomega_{\gamma} \neq \pm \ssomega_{\nu}   
	\end{dcases} \Longleftrightarrow \\ 
	& \begin{dcases}
          i=-4+m-2\gamma \text{ and } m \geq 4+2\gamma \text{ and } \nu \neq \pm(m-\gamma)-2      \\
           i=m \text{ and }    \nu \neq 2+m+\gamma  
	\end{dcases} 
\end{align*}
and
\begin{align*}
	& \ssomega_{\gamma}+\ssomega_{j}\pm \ssomega_{\gamma} \pm \ssomega_{m}=0 \text{ and } \ssomega_{\gamma}+\ssomega_{j} \neq \pm \ssomega_{\nu}  \Longleftrightarrow \\
	& \begin{dcases}
		\ssomega_{\gamma}+\ssomega_{j}
		+
		\ssomega_{\gamma} 
		+
		\ssomega_{m}=0 \text{ and } \ssomega_{\gamma}+\ssomega_{j} \neq \pm \ssomega_{\nu}  \\
		\ssomega_{\gamma}+\ssomega_{j}
		+
		\ssomega_{\gamma} 
		- 
		\ssomega_{m}=0 \text{ and } \ssomega_{\gamma}+\ssomega_{j} \neq \pm \ssomega_{\nu} \\
		\ssomega_{\gamma}+\ssomega_{j}
		- 
		\ssomega_{\gamma} 
		+ 
		\ssomega_{m}=0 \text{ and } \ssomega_{\gamma}+\ssomega_{j} \neq \pm \ssomega_{\nu}  \\
		\ssomega_{\gamma}+\ssomega_{j}
		-
		\ssomega_{\gamma} 
		- 
		\ssomega_{m}=0 \text{ and } \ssomega_{\gamma}+\ssomega_{j} \neq \pm \ssomega_{\nu} 
	\end{dcases} \Longleftrightarrow \\
	& \begin{dcases}
		j= -4+m-2\gamma  \text{ and }m\geq 4+2\gamma \text{ and } \nu \neq \pm(m-\gamma)-2  \\
	j= m   \text{ and } \nu \neq 2+m+\gamma  
	\end{dcases}  
\end{align*}
and
\begin{align*}
&	\ssomega_{\gamma}+\ssomega_{\gamma} \pm \ssomega_{\kappa} \pm \ssomega_{m}=0 \text{ and } \ssomega_{\gamma}+\ssomega_{\gamma} \neq \pm \ssomega_{\nu}   \Longleftrightarrow \\
& \begin{dcases}
		\ssomega_{\gamma}+\ssomega_{\gamma}
		+
		\ssomega_{\kappa} 
		+
		\ssomega_{m}=0 \text{ and } \ssomega_{\gamma}+\ssomega_{\gamma} \neq \pm \ssomega_{\nu}  \\
		\ssomega_{\gamma}+\ssomega_{\gamma}
		+
		\ssomega_{\kappa} 
		- 
		\ssomega_{m}=0 \text{ and } \ssomega_{\gamma}+\ssomega_{\gamma} \neq \pm \ssomega_{\nu} \\
		\ssomega_{\gamma}+\ssomega_{\gamma}
		- 
		\ssomega_{\kappa} 
		+ 
		\ssomega_{m}=0 \text{ and } \ssomega_{\gamma}+\ssomega_{\gamma} \neq \pm \ssomega_{\nu}  \\
		\ssomega_{\gamma}+\ssomega_{\gamma}
		-
		\ssomega_{\kappa} 
		- 
		\ssomega_{m}=0 \text{ and } \ssomega_{\gamma}+\ssomega_{\gamma} \neq \pm \ssomega_{\nu} 
	\end{dcases} \Longleftrightarrow \\
	& \begin{dcases}
		\kappa= -4+m-2\gamma  \text{ and } m\geq 4+2\gamma \text{ and } \nu \neq 2+2\gamma \\
	\kappa = 4+m+2\gamma   \text{ and } \nu \neq 2+2\gamma  \\
	\kappa = 2\gamma-m  \text{ and } m\leq 2\gamma \text{ and } \nu \neq 2+2\gamma
	\end{dcases}  
\end{align*}
Therefore, we infer
		\begin{align*}
		\mathfrak{K}_{\gamma}^{-2} \left( d \mathfrak{F}_{0+}(\xi)[h]  \right)^m 
		& =  \frac{9}{4} h^{-4+m-2\gamma}     \mathds{1}(m \geq 4+2\gamma)  \sum_{ \substack{ \nu = 0 \\ \nu \neq \pm(m-\gamma)-2 }}^{\infty}   \frac{\overline{\mathfrak{C}}_{\gamma \nu m}\overline{\mathfrak{C}}_{-4+m-2\gamma,\gamma, \nu }}{\ssomega_{\nu}^2-(\ssomega_{-4+m-2\gamma}+\ssomega_{\gamma})^2  }  
		 \\
		& + \frac{9}{4}  h^m   \sum_{ \substack{ \nu = 0 \\ \nu \neq 2+m+\gamma}}^{\infty}    \frac{\left(\overline{\mathfrak{C}}_{\gamma \nu m} \right)^2 }{\ssomega_{\nu}^2-(\ssomega_{m}+\ssomega_{\gamma})^2  }  
		  \\
		&+ \frac{9}{4}  h^{-4+m-2\gamma}  \mathds{1}(m \geq 4+2\gamma) \sum_{\substack{ \nu = 0 \\ \nu \neq \pm (m-\gamma)-2 } }^{\infty}   \frac{ \overline{\mathfrak{C}}_{\gamma \nu m}\overline{\mathfrak{C}}_{\gamma, -4+m-2\gamma,\nu }}{\ssomega_{\nu}^2-(\ssomega_{\gamma}+\ssomega_{-4+m-2\gamma})^2  }   
		     \\
		  &+ \frac{9}{4} h^m     \sum_{\substack{ \nu = 0 \\ \nu \neq 2+m+\gamma } }^{\infty}  \frac{\left( \overline{\mathfrak{C}}_{\gamma \nu m}\right)^2  }{\ssomega_{\nu}^2-(\ssomega_{\gamma}+\ssomega_{m})^2  }   
		     \\  
		&+ \frac{9}{4}h^{-4+m-2\gamma}  \mathds{1}(m \geq 4+2\gamma)  \sum_{\substack{ \nu = 0 \\ \nu \neq 2+2\gamma} }^{\infty}  
		 \frac{ \overline{\mathfrak{C}}_{-4+m-2\gamma, \nu, m}  \overline{\mathfrak{C}}_{\gamma\gamma \nu }}{\ssomega_{\nu}^2-( 2\ssomega_{\gamma})^2  }   
		 \\
		 &+ \frac{9}{4} h^{4+m+2\gamma} \sum_{\substack{ \nu = 0 \\ \nu \neq 2+2\gamma} }^{\infty}   
		 \frac{ \overline{\mathfrak{C}}_{4+m+2\gamma, \nu ,m} \overline{\mathfrak{C}}_{\gamma\gamma \nu }}{\ssomega_{\nu}^2-( 2\ssomega_{\gamma})^2  }   
		   \\
		 &+ \frac{9}{4}h^{2\gamma-m}  \mathds{1}(0\leq m\leq 2\gamma)  \sum_{\substack{ \nu = 0 \\ \nu \neq 2+2\gamma} }^{\infty}  
		 \frac{\overline{\mathfrak{C}}_{2\gamma-m, \nu, m} \overline{\mathfrak{C}}_{\gamma\gamma \nu }}{\ssomega_{\nu}^2-(2\ssomega_{\gamma})^2  }   .
	\end{align*} 
	The latter can be written as
\begin{align*}
	& \mathfrak{K}_{\gamma}^{-2} \left( d \mathfrak{F}_{0+}(\xi)[h]  \right)^m  = 
	  h^m \left[
	  \frac{9}{2}   \sum_{ \substack{ \nu = 0 \\ \nu \neq 2+m+\gamma}}^{\infty}    \frac{\left(\overline{\mathfrak{C}}_{\gamma \nu m} \right)^2 }{\ssomega_{\nu}^2-(\ssomega_{m}+\ssomega_{\gamma})^2  }   
	\right]+ h^{4+m+2\gamma}\left[
	\frac{9}{4}   \sum_{\substack{ \nu = 0 \\ \nu \neq 2+2\gamma} }^{\infty}   
		 \frac{ \overline{\mathfrak{C}}_{4+m+2\gamma, \nu ,m} \overline{\mathfrak{C}}_{\gamma\gamma \nu }}{\ssomega_{\nu}^2-( 2\ssomega_{\gamma})^2  } 
	\right] \\
	&+ h^{-4+m-2\gamma} \mathds{1}(m \geq 4+2\gamma) \Bigg[
	  \frac{9}{2}  \sum_{\substack{ \nu = 0 \\ \nu \neq \pm (m-\gamma)-2 } }^{\infty}   \frac{ \overline{\mathfrak{C}}_{\gamma \nu m}\overline{\mathfrak{C}}_{\gamma, -4+m-2\gamma,\nu }}{\ssomega_{\nu}^2-(\ssomega_{\gamma}+\ssomega_{-4+m-2\gamma})^2  }+ \frac{9}{4}   \sum_{\substack{ \nu = 0 \\ \nu \neq 2+2\gamma} }^{\infty}  
		 \frac{ \overline{\mathfrak{C}}_{-4+m-2\gamma, \nu, m}  \overline{\mathfrak{C}}_{\gamma\gamma \nu }}{\ssomega_{\nu}^2-( 2\ssomega_{\gamma})^2  }  
	\Bigg] \\
	&+h^{2\gamma-m} \mathds{1}(0\leq m\leq 2\gamma) \left[
	 \frac{9}{4}   \sum_{\substack{ \nu = 0 \\ \nu \neq 2+2\gamma} }^{\infty}  
		 \frac{\overline{\mathfrak{C}}_{2\gamma-m, \nu, m} \overline{\mathfrak{C}}_{\gamma\gamma \nu }}{\ssomega_{\nu}^2-(2\ssomega_{\gamma})^2  }  
	\right]  .
\end{align*}
Putting all together, yields that
\begin{align*}
	\mathfrak{K}_{\gamma}^{-2} \left( d \mathfrak{F}_{0}(\xi)[h]  \right)^m = \mathfrak{K}_{\gamma}^{-2} \left[ \left( d \mathfrak{F}_{0-}(\xi)[h]  \right)^m + \left( d \mathfrak{F}_{0+}(\xi)[h]  \right)^m  \right]
\end{align*}
equals to
\begin{align*}
	&  h^{m} \left[
	\frac{9}{2}   \sum_{ \substack{ \nu = 0 \\ \nu \neq 2+m+\gamma}}^{\infty}    \frac{\left(\overline{\mathfrak{C}}_{\gamma \nu m} \right)^2 }{\ssomega_{\nu}^2-(\ssomega_{m}+\ssomega_{\gamma})^2  }  +\frac{9}{2}  \sum_{ \substack{ \nu = 0 \\ \nu \neq \pm(m-\gamma)-2  }}^{\infty}   \frac{\left(\overline{\mathfrak{C}}_{m\gamma \nu } \right)^2}{\ssomega_{\nu}^2-(\ssomega_{m}-\ssomega_{\gamma})^2  }  
	+  \frac{9}{4}  \sum_{ \nu =0}^{\infty}  
		 \frac{\overline{\mathfrak{C}}_{m \nu m}\overline{\mathfrak{C}}_{\gamma\gamma \nu }}{\ssomega_{\nu}^2 } 
	\right] \\
	&+ h^{2\gamma-m} \mathds{1}(0\leq m\leq 2\gamma) \left[
	 \frac{9}{4}   \sum_{\substack{ \nu = 0 \\ \nu \neq 2+2\gamma} }^{\infty}  
		 \frac{\overline{\mathfrak{C}}_{2\gamma-m, \nu, m} \overline{\mathfrak{C}}_{\gamma\gamma \nu }}{\ssomega_{\nu}^2-(2\ssomega_{\gamma})^2  }  +\frac{9}{2} \sum_{ \substack{ \nu = 0 \\ \nu \neq \pm(m-\gamma)-2  }}^{\infty}   \frac{\overline{\mathfrak{C}}_{\gamma \nu m} \overline{\mathfrak{C}}_{2\gamma-m,\gamma ,\nu }}{\ssomega_{\nu}^2-(\ssomega_{2\gamma-m}-\ssomega_{\gamma})^2  }
	\right] \\ 
	&+ h^{4+m+2\gamma} \left[
	\frac{9}{4}   \sum_{\substack{ \nu = 0 \\ \nu \neq 2+2\gamma} }^{\infty}   
		 \frac{ \overline{\mathfrak{C}}_{4+m+2\gamma, \nu ,m} \overline{\mathfrak{C}}_{\gamma\gamma \nu }}{\ssomega_{\nu}^2-( 2\ssomega_{\gamma})^2  } +\frac{9}{2}   \sum_{ \substack{ \nu = 0 \\ \nu \neq  2+m+\gamma  }}^{\infty}    \frac{\overline{\mathfrak{C}}_{\gamma \nu m}\overline{\mathfrak{C}}_{2\gamma+m+4,\gamma ,\nu }}{\ssomega_{\nu}^2-(\ssomega_{2\gamma+m+4}-\ssomega_{\gamma})^2  }   
	\right] \\
	&+ h^{-4+m-2\gamma} \mathds{1}(m \geq 4+2\gamma) \left[
	 \frac{9}{2}  \sum_{\substack{ \nu = 0 \\ \nu \neq \pm (m-\gamma)-2 } }^{\infty}   \frac{ \overline{\mathfrak{C}}_{\gamma \nu m}\overline{\mathfrak{C}}_{\gamma, -4+m-2\gamma,\nu }}{\ssomega_{\nu}^2-(\ssomega_{\gamma}+\ssomega_{-4+m-2\gamma})^2  }+ \frac{9}{4}   \sum_{\substack{ \nu = 0 \\ \nu \neq 2+2\gamma} }^{\infty}  
		 \frac{ \overline{\mathfrak{C}}_{-4+m-2\gamma, \nu, m}  \overline{\mathfrak{C}}_{\gamma\gamma \nu }}{\ssomega_{\nu}^2-( 2\ssomega_{\gamma})^2  } 
	\right].
\end{align*}
Finally, we simplify the formula above and show that 
\begin{align}
	& \overline{\mathfrak{C}}_{4+m+2\gamma, \nu ,m} \overline{\mathfrak{C}}_{\gamma\gamma \nu }=\overline{\mathfrak{C}}_{\gamma \nu m}\overline{\mathfrak{C}}_{2\gamma+m+4,\gamma ,\nu }=0, \label{FourierCancel1} \\
	& \overline{\mathfrak{C}}_{\gamma \nu m}\overline{\mathfrak{C}}_{\gamma, -4+m-2\gamma,\nu }= \overline{\mathfrak{C}}_{-4+m-2\gamma, \nu, m}  \overline{\mathfrak{C}}_{\gamma\gamma \nu }=0, \label{FourierCancel2}
\end{align}
for all $\gamma,\nu, m\geq 0$ and $\gamma,\nu \geq 0$, $m\geq 2\gamma+4$ respectively. These follow immediately from the fact that $\overline{\mathfrak{C}}_{ijm}=0$ for all integers $i,j,m \geq 0$ with $i+j<m$ (Lemma \ref{LemmaClosedFormulasModel3YMCbarijm}). Specifically, we have
\begin{itemize}
	\item $\nu > 2\gamma \Longrightarrow \overline{\mathfrak{C}}_{\gamma \gamma \nu}=0$,
	\item $0\leq \nu \leq 2\gamma \Longrightarrow  m + \nu \leq m+ 2\gamma   <  4+m+2\gamma \Longrightarrow \overline{\mathfrak{C}}_{4+m+2\gamma, \nu ,m}=0 $,
\end{itemize}
as well as
\begin{itemize}
	\item $\nu > \gamma+m \Longrightarrow \overline{\mathfrak{C}}_{\gamma \nu m}=0$,
	\item $0\leq \nu \leq \gamma+m \Longrightarrow \gamma+\nu \leq 2\gamma+m <2\gamma+m+4 \Longrightarrow \overline{\mathfrak{C}}_{2\gamma+m+4,\gamma ,\nu }=0 $,
\end{itemize}
that prove \eqref{FourierCancel1}. Similarly, we have
\begin{itemize}
	\item $m > \gamma+\nu  \Longrightarrow \overline{\mathfrak{C}}_{\gamma \nu m}=0$,
	\item $0\leq m \leq  \gamma+\nu  \Longrightarrow -4+m-2\gamma+\gamma =  -4+m-\gamma  \leq -4+\nu < \nu \Longrightarrow \overline{\mathfrak{C}}_{\gamma, -4+m-2\gamma,\nu }=0 $,
\end{itemize}
as well as
\begin{itemize}
	\item $\nu > 2\gamma \Longrightarrow \overline{\mathfrak{C}}_{\gamma \gamma \nu}=0$,
	\item $0\leq \nu \leq 2\gamma \Longrightarrow -4+m-2\gamma+\nu\leq -4+m <m \Longrightarrow \overline{\mathfrak{C}}_{-4+m-2\gamma,\nu,m}=0 $,
\end{itemize}
that prove \eqref{FourierCancel2}. Therefore, $\mathfrak{K}_{\gamma}^{-2} \left( d \mathfrak{F}_{0}(\xi)[h]  \right)^m $ boils down to
\begin{align*}
	&  h^{m} \left[
	\frac{9}{2}   \sum_{ \substack{ \nu = 0 \\ \nu \neq 2+m+\gamma}}^{\infty}    \frac{\left(\overline{\mathfrak{C}}_{\gamma \nu m} \right)^2 }{\ssomega_{\nu}^2-(\ssomega_{m}+\ssomega_{\gamma})^2  }  +\frac{9}{2}  \sum_{ \substack{ \nu = 0 \\ \nu \neq \pm(m-\gamma)-2  }}^{\infty}   \frac{\left(\overline{\mathfrak{C}}_{m\gamma \nu } \right)^2}{\ssomega_{\nu}^2-(\ssomega_{m}-\ssomega_{\gamma})^2  }  
	+  \frac{9}{4}  \sum_{ \nu =0}^{\infty}  
		 \frac{\overline{\mathfrak{C}}_{m \nu m}\overline{\mathfrak{C}}_{\gamma\gamma \nu }}{\ssomega_{\nu}^2 } 
	\right] \\
	&+ h^{2\gamma-m} \mathds{1}(0\leq m\leq 2\gamma) \left[
	 \frac{9}{4}   \sum_{\substack{ \nu = 0 \\ \nu \neq 2+2\gamma} }^{\infty}  
		 \frac{\overline{\mathfrak{C}}_{2\gamma-m, \nu, m} \overline{\mathfrak{C}}_{\gamma\gamma \nu }}{\ssomega_{\nu}^2-(2\ssomega_{\gamma})^2  }  +\frac{9}{2} \sum_{ \substack{ \nu = 0 \\ \nu \neq \pm(m-\gamma)-2  }}^{\infty}   \frac{\overline{\mathfrak{C}}_{\gamma \nu m} \overline{\mathfrak{C}}_{2\gamma-m,\gamma ,\nu }}{\ssomega_{\nu}^2-(\ssomega_{2\gamma-m}-\ssomega_{\gamma})^2  }
	\right]. 
\end{align*}
Using the same argument as above (Lemma \ref{LemmaClosedFormulasModel3YMCbarijm}), we also infer 
\begin{itemize}
	\item $\nu > 2\gamma \Longrightarrow \overline{\mathfrak{C}}_{\gamma \gamma \nu}=0$,
	\item $\nu > \gamma+m \Longrightarrow \overline{\mathfrak{C}}_{\gamma \nu m}=0$.
\end{itemize}
Hence, the latter reduces to
\begin{align*}
	&  h^{m} \left[
	\frac{9}{2}   \sum_{ \substack{ \nu = 0  }}^{m+\gamma}    \frac{\left(\overline{\mathfrak{C}}_{\gamma \nu m} \right)^2 }{\ssomega_{\nu}^2-(\ssomega_{m}+\ssomega_{\gamma})^2  }  +\frac{9}{2}  \sum_{ \substack{ \nu = 0 \\ \nu \neq \pm(m-\gamma)-2  }}^{m+\gamma}   \frac{\left(\overline{\mathfrak{C}}_{m\gamma \nu } \right)^2}{\ssomega_{\nu}^2-(\ssomega_{m}-\ssomega_{\gamma})^2  }  
	+  \frac{9}{4}  \sum_{\substack{ \nu = 0  } }^{2\gamma}   
		 \frac{\overline{\mathfrak{C}}_{m \nu m}\overline{\mathfrak{C}}_{\gamma\gamma \nu }}{\ssomega_{\nu}^2 } 
	\right] \\
	&+ h^{2\gamma-m} \mathds{1}(0\leq m\leq 2\gamma) \left[
	 \frac{9}{4}   \sum_{\substack{ \nu = 0  } }^{2\gamma}  
		 \frac{\overline{\mathfrak{C}}_{2\gamma-m, \nu, m} \overline{\mathfrak{C}}_{\gamma\gamma \nu }}{\ssomega_{\nu}^2-(2\ssomega_{\gamma})^2  }  +\frac{9}{2} \sum_{ \substack{ \nu = 0 \\ \nu \neq \pm(m-\gamma)-2  }}^{m+\gamma}   \frac{\overline{\mathfrak{C}}_{\gamma \nu m} \overline{\mathfrak{C}}_{2\gamma-m,\gamma ,\nu }}{\ssomega_{\nu}^2-(\ssomega_{2\gamma-m}-\ssomega_{\gamma})^2  }
	\right] ,
\end{align*}
that completes the proof.
\end{proof}

 \section{The linear eigenvalue problems}\label{SectionEigenvalueproblems}
Next, we the study the linear eigenvalue problems where the linearized operators are given by \eqref{DefinitionLinearOperator}. As it turns out, in all three models, the associated eigenfunctions are given by Jacobi polynomials and this is a common feature with the Einstein-Klein-Gordon system in spherical symmetry \cite{Maliborski:2013jca}.  
 
  \subsection{Conformal cubic wave equation in spherical symmetry}
  Consider the Hilbert space
  \begin{align}\label{HilbertSpaceModel1}
  	L^2 ( [0,\pi ];\sin^2(x)dx )
  \end{align}
   associated with the inner product
  \begin{align*}
  	(f|g) :=\frac{2}{\pi } \int_{0}^{\pi } f(x) g(x) \sin^2 (x) dx.
  \end{align*}
For the conformal wave equation in spherical symmetry, the operator that governs the solutions to the linearized equation is given by 
\begin{align*}
	 Lu &:= - \frac{1}{\sin^2 (x) } \partial_{x} \left(
	\sin^2 (x) \partial_{x} u
	\right) +  u , 
\end{align*}
endowed with the domain $\mathcal{D}(L)$ defined by
\begin{align*}
	\mathcal{D}(L):=\{ u \in L^2((0,\pi); \sin^2(x)dx): \Hquad   Lu \in L^2((0,\pi); \sin^2(x)dx)\}.
\end{align*}
The operator $L$ is generated by the closed sesquilinear form $a$ defined on 
\begin{align*}
	H^1((0,\pi);\sin^2(x)dx) \times H^1((0,\pi);\sin^2(x)dx)
\end{align*}
that is given by
\begin{align*}
	a\left(u ,v\right):= \int_{0}^{\pi }\left(  \partial_{x}\overline{u} \cdot\partial_{x}v +   \overline{u}  \cdot v  \right)\sin^2(x) dx,
\end{align*}
where the bar denotes the complex conjugate, and
\begin{align*}
	 a(u,u) = \frac{2}{\pi}\|u \|_{H^1((0,\pi);\sin^2(x)dx)}^2.
\end{align*}
In particular, $L$ is self-adjoint on $\mathcal{D}(L)$. Now, the eigenvalue problem $Lu=\omega^2 u$ reads 
\begin{align*}
	 \partial_{x} \left(
	\sin^2 (x) \partial_{x} u
	\right) + (\omega^2-1)\sin^2 (x) u = 0
\end{align*}
and by setting 
\begin{align*}
	u(x)=v(y), \Hquad y=\cos(x)
\end{align*}
it becomes
\begin{align*}
	(1-y^2)v^{\prime \prime}(y) -3y v^{\prime }(y) + (\omega^2-1)  v(y)=0 
\end{align*}
that has non--trivial solutions if and only if the solutions given by the  Chebyshev polynomials of the second kind \cite{MR0372517}, that is $v(y)=U_{n}(y)$, associated with the eigenvalues
\begin{align*}
	  \omega_{n}^2  :=(n+1)^2,\Hquad n \geq 0.
\end{align*}
Hence, the solutions to the eigenvalue problem $Lu=\omega^2 u$ are given by
 \begin{align}\label{EigenfunctionsModel1WS}
	 e_{n}(x) := U_{n }(\cos(x)) 
	 ,\Hquad n \geq 0 .
\end{align}
In addition, the set $\{ e_{n}: n \geq 0 \}$ forms an orthonormal and complete basis for the weighted $L^2 $ in \eqref{HilbertSpaceModel1}. In fact, their orthogonality follows by the orthogonality of the Chebyshev polynomials of the second kind since
\begin{align*}
	\left( e_{n} | e_{m} \right) &=\frac{2}{\pi } \int_{0}^{\pi }   e_{n}(x)  e_{m}(x) \sin^2 (x) dx \\
	&=\frac{2}{\pi } \int_{0}^{\pi } U_{n}(y) U_{m}(y) (1-y^2) \frac{dy}{\sqrt{1-y^2}} \\
	&=\frac{2}{\pi } \int_{0}^{\pi } U_{n}(y) U_{m}(y) \sqrt{1-y^2} dy \\
	&= \mathds{1}(n=m),
\end{align*}
for any $n,m \geq 0$.

 \subsection{Conformal cubic wave equation out of spherical symmetry}
Consider the Hilbert space
\begin{align}\label{HilbertSpaceModel2}
	L^2 ( [0,\pi ];\sin^2(x)dx )
\end{align}
associated with the inner product
 \begin{align*}
 \langle	 f|g \rangle :=  \int_{0}^{\frac{\pi }{2}} f(x)g(x) \sin(2x) dx.
 \end{align*}
For the conformal cubic wave equation out of spherical symmetry, the operator that governs the solutions to the linearized equation is given by 
\begin{align*}
	 \mathsf{L}^{(\mu_1,\mu_2)}u &:=-\partial_{x}^2 u - \left(\frac{\cos x}{\sin x} - \frac{\sin x}{\cos x} \right) \partial_{x} u + \left( 
	\frac{\mu_1^2}{\sin^2 x} + \frac{\mu_2^2}{\cos^2 x}+1
	\right)u \\
	&= -\frac{1}{\sin(2x)} \partial_{x} \left(\sin(2x) \partial_{x}u \right)+ \left( 
	\frac{\mu_1^2}{\sin^2 x} + \frac{\mu_2^2}{\cos^2 x}+1
	\right)u
\end{align*}
endowed with the domain $\mathcal{D}(\mathsf{L^{(\mu_1,\mu_2)}})$ defined by 
\begin{align*}
	\mathcal{D}(\mathsf{L^{(\mu_1,\mu_2)}}):=\{ u \in L^2((0,\pi/2); \sin(2x)dx): \Hquad   \mathsf{L}^{(\mu_1,\mu_2)}u  \in L^2((0,\pi/2); \sin(2x)dx)\}.
\end{align*}
The operator $\mathsf{L}^{(\mu_1,\mu_2)}$ is generated by the closed sesquilinear form $\mathsf{a}$ defined on 
\begin{align*}
	H^1((0,\pi/2);\sin(2x)dx) \times H^1((0,\pi/2);\sin(2x)dx)
\end{align*}
that is given by    
\begin{align*}
	\mathsf{a}\left(u ,v\right):= \int_{0}^{\pi }\left(  \partial_{x}\overline{u} \cdot\partial_{x}v + \left( 
	\frac{\mu_1^2}{\sin^2 x} + \frac{\mu_2^2}{\cos^2 x}+1
	\right)  \overline{u}  \cdot v  \right)\sin(2x) dx,
\end{align*}
where the bar denotes the complex conjugate. and Hardy's inequality yields
\begin{align*}
	 \mathsf{a}(u,u) \simeq \|u \|_{H^1((0,\pi/2);\sin(2x)dx)}^2.
\end{align*}
In particular, $\mathsf{L^{(\mu_1,\mu_2)}}$ is self-adjoint on $\mathcal{D}(\mathsf{L^{(\mu_1,\mu_2)}})$.  Now, the eigenvalue problem $ \mathsf{L}^{(\mu_1,\mu_2)}u=\somega^2 u$ reads 
\begin{align*}
	\partial_{x}^2 u + \left(\frac{\cos x}{\sin x} - \frac{\sin x}{\cos x} \right) \partial_{x} u - \left( 
	\frac{\mu_1^2}{\sin^2 x} + \frac{\mu_2^2}{\cos^2 x}+1
	-\omega^2  \right)u=0.
\end{align*}
and by setting
\begin{align*}
	u(x)=v(y), \Hquad v(y)=(1-y)^{\frac{\mu_1}{2}}(1+y)^{\frac{\mu_2}{2}}w(y), \Hquad  y=\cos(2x)
\end{align*}
it becomes 
\begin{align*}
	(1-y^2) w^{\prime \prime}(y) + \left[(\mu_2-\mu_1)-(2+\mu_1+\mu_2)y \right]  w^{\prime} (y)+
	\frac{1}{4}\left[
	 \omega^2  - (1+\mu_1 +\mu_2)^2
	\right]w(y)=0
\end{align*}
that has non--trivial solutions if and only if the solutions are given by the Jacobi polynomial with parameters $(\mu_1,\mu_2)$ and degree $n \geq 0$, that is $w(y)=P_{n}^{ (\mu_1, \mu_2)}(y)$, associated with the eigenvalues
 associated with the eigenvalues
\begin{align*}
	\left(  \somega_{n}^{(\mu_{1},\mu_{2})} \right)  ^2:=(2n+1+\mu_1+\mu_2)^2,\Hquad n \geq 0  
\end{align*}
Hence, the solutions to the eigenvalue problem $\mathsf{L}^{(\mu_1,\mu_2)}u=\somega^2 u$ are given by
\begin{align}\label{EigenfunctionsModel2WH}
	\mathsf{e}_{n}^{(\mu_{1},\mu_{2})} (x)=\mathsf{N}_{n}^{(\mu_{1},\mu_{2})}  (1-\cos(2x))^{\frac{\mu_1}{2}}(1+\cos(2x))^{\frac{\mu_2}{2}}P_{n}^{ ( \mu_1, \mu_2)}(\cos(2x)), \Hquad n \geq 0,
\end{align}
where the normalization constant reads
\begin{align}\label{NormalizationConstantModel2CH}
	\mathsf{N}_{n}^{(\mu_{1},\mu_{2})}  :=\sqrt{\frac{
	\somega_{n}^{(\mu_{1},\mu_{2})} }{2^{\mu_1 +\mu_2}} \frac{\Gamma(n+1)\Gamma(n+\mu_1+\mu_2+1)}{\Gamma(n+\mu_1+1)\Gamma(n+\mu_2+1)} }, \Hquad  n \geq 0.
\end{align}
In addition, the set $\{ \mathsf{e}_{n}^{(\mu_{1},\mu_{2})}: n \geq 0 \}$ forms an orthonormal and complete basis for the weighted $L^2 $ in \eqref{HilbertSpaceModel2}. In fact, their orthogonality follows by the orthogonality of the Jacobi polynomials since 
\begin{align*}
	\langle	 \mathsf{e}_{n}^{(\mu_{1},\mu_{2})} | \mathsf{e}_{m}^{(\mu_{1},\mu_{2})} \rangle  &=  \int_{0}^{\frac{\pi}{2} } \mathsf{e}_{n}^{(\mu_{1},\mu_{2})}(x) \mathsf{e}_{m}^{(\mu_{1},\mu_{2})} (x) \sin(2x) dx \\
	&  =\frac{1}{2} \mathsf{N}_{n}\mathsf{N}_{m} \int_{-1}^{1 } 
	(1-y)^{\mu_1}(1+y)^{\mu_2}P_{n}^{ ( \mu_1, \mu_2)}(y)
	P_{m}^{ ( \mu_1, \mu_2)}(y) dy \\
	&  =   \mathsf{N}_{n} ^2 
	\frac{2^{\mu_1+\mu_2}}{2n+\mu_1+\mu_2+1} \frac{\Gamma(n+\mu_1+1)\Gamma(n+\mu_2+1)}{\Gamma(n+\mu_1+\mu_2+1) \Gamma(n+1)} \mathds{1}(n=m) \\
	&  =   \mathds{1}(n=m),
\end{align*}
for any $n,m \geq 0$.

  \subsection{Yang--Mills equation in spherical symmetry}
Consider the Hilbert space
\begin{align}\label{HilbertSpaceModel3}
	L^2 ( [0,\pi ];\sin^4(x)dx )
\end{align}
associated with the inner product
 \begin{align*}
  [f|g] := \int_{0}^{\pi } f(x) g(x)  \sin^4(x) dx.
 \end{align*}
For the Yang--Mills equation in spherical symmetry, the operator that governs the solutions to the linearized equation is given by 
\begin{align*}
	\mathfrak{L}u:=- \frac{1}{\sin^4 x}\partial_x \left( \sin^4 x \partial_x u \right)+4u
\end{align*}
endowed with the domain $\mathcal{D}(\mathfrak{L})$ defined by 
\begin{align*}
	\mathcal{D}(\mathfrak{L}):=\{ u \in L^2((0,\pi); \sin^4 x dx): \Hquad  \mathfrak{L} u \in L^2((0,\pi); \sin^4 xdx)\}.
\end{align*}
The operator $\mathfrak{L}$ is generated by the closed sesquilinear form $\mathfrak{a}$ defined on 
\begin{align*}
	H^1((0,\pi);\sin^4 xdx) \times H^1((0,\pi);\sin^4 x dx)
\end{align*}
that is given by    
\begin{align*}
\mathfrak{a}\left(u ,v\right):= \int_{0}^{\pi }\left(  \partial_{x}\overline{u} \cdot\partial_{x}v +   4\overline{u}  \cdot v  \right)\sin^4 (x) dx,
\end{align*}
where the bar denotes the complex conjugate, and 
\begin{align*}
	\mathfrak{a}(u,u) \simeq \|u \|_{H^1((0,\pi);\sin^4 xdx)}^2.
\end{align*}
In particular, $\mathfrak{L}$ is self-adjoint on $\mathcal{D}(\mathfrak{L})$. Now, the eigenvalue problem $\mathfrak{L}u=\ssomega^2 u$ reads 
\begin{align*}
	   \partial_{x}^2 u +\frac{4}{\tan(x)}\partial_{x}u+(\omega^2 -4)u = 0
\end{align*}
and by setting 
\begin{align*}
	u(x)=w(y), \Hquad  y=\cos(x) 
\end{align*}
it becomes
\begin{align*}
	\left(1-y^2\right) w^{\prime \prime}(y)-5 y w^{\prime }(y)+\left(\omega ^2-4\right) w(y) = 0
\end{align*}
that has non--trivial solutions if and only if the solutions given by the Jacobi polynomials with parameters $(3/2,3/2)$ and degree $n$, that is $w(y)=P_{n}^{(3/2,3/2)}(y)$, associated with the eigenvalues
\begin{align*}
	 \ssomega_{n} ^2:=(n+2)^2,\Hquad n \geq 0.
\end{align*}
Hence, the solutions to the eigenvalue problem $\mathfrak{L}u=\ssomega^2 u$ are given by
 \begin{align}\label{EigenfunctionsModel3YM}
	\mathfrak{e}_{n}(x) := \mathfrak{N}_{n}  P_{n}^{\left(\frac{3}{2},\frac{3}{2} \right)}(\cos(x)) ,\Hquad n \geq 0 ,
\end{align}
where the normalization constant reads
\begin{align}\label{NormalizationConstantModel3YM}
	\mathfrak{N}_{n}  :=\frac{ \sqrt{
	\ssomega_{n}\Gamma(1+n)\Gamma(4+n)}}{2\sqrt{2}\Gamma \left(\frac{5}{2}+n \right)}, \Hquad  n \geq 0.
\end{align}
In addition, the set $\{\mathfrak{e}_{n}: n \geq 0 \}$ forms an orthonormal and complete basis for the weighted $L^2$ in \eqref{HilbertSpaceModel3}. In fact, their orthogonality follows by the orthogonality of the Jacobi polynomials since 
\begin{align*}
	\left[ \mathfrak{e}_{n}  |\mathfrak{e}_{m} \right] &= \int_{0}^{\pi } \mathfrak{e}_{n}(x) \mathfrak{e}_{m}(x)  \sin^4(x) dx \\ 
	& =\mathfrak{N}_{n} \mathfrak{N}_{m} \int_{-1}^{1 } P_{n}^{\left(\frac{3}{2},\frac{3}{2} \right)}(y) P_{m}^{\left(\frac{3}{2},\frac{3}{2} \right)}(y) \left( 1-y^2\right)^{2} \frac{dy}{\sqrt{1-y^2}} \\
	& =\mathfrak{N}_{n} \mathfrak{N}_{m} \int_{-1}^{1 } P_{n}^{\left(\frac{3}{2},\frac{3}{2} \right)}(y) P_{m}^{\left(\frac{3}{2},\frac{3}{2} \right)}(y) \left( 1-y^2\right)^{\frac{3}{2}} dy \\
	&= \mathfrak{N}_{n}  ^2 \frac{2^{4}}{2n+4} \frac{ \left(\Gamma\left(n+\frac{5}{2} \right) \right)^2 }{\Gamma\left(n+4 \right)\Gamma\left(n+1 \right)}\mathds{1}(n=m)  \\  
	&=\mathds{1}(n=m) ,  
\end{align*}
for all $n,m\geq 0$.

\section{The PDEs in Fourier space} \label{SectionPDEsinFourierSpace}

In this section, we express the partial differential equations \eqref{AllModelsinOnePDE},
\begin{align*} 
	\left( \partial_{t}^2 + \mathbf{L} \right) u = \mathbf{f}(x,u), \Hquad (t,x)\in \mathbb{R} \times I,
\end{align*}
in the Fourier space to obtain infinite dimensional systems of coupled, non-linear harmonic oscillators and we provide basic estimates for the non-linearities. Here, the non--linearities are given by \eqref{Definitionnon-linearily}, namely\begin{align*}
	\mathbf{f}(x,u) & = 
	\begin{dcases}
	 - u^3	, & \text{ for CW}  \\
	  - u^3	, & \text{ for WH} \\
	   -3 u^2 -\sin^2(x) u^3 ,& \text{ for YM}
	\end{dcases} 
\end{align*}
Let $u(t,\cdot)$ be a solution to any of the three models, 
\begin{itemize}
	\item CW: \eqref{Model1WSphericalSymmetry}--\eqref{DefinitionLinearOperatorModel1WSphericalSymmetry},
	 \item WH:  \eqref{Model1CWOUTSphericalSymmetryHOPFanstaz}--\eqref{DefinitionLinearOperatorModel1CWOUTSphericalSymmetryHOPF} together with the ansatz \eqref{AnsatzHof},
	 \item YM: \eqref{Model2YM}--\eqref{DefinitionLinearOperatorModel2YM},
\end{itemize} 
and recall that the set of the associated eigenfunctions, 
\begin{itemize}
\item CW: $\{  e_{n}: n\geq 0\}$ given by \eqref{EigenfunctionsModel1WS},
 \item WH: $\{ \mathsf{e}_{n}^{(\mu_1,\mu_2)}: n\geq 0\}$ given by \eqref{EigenfunctionsModel2WH},
 \item YM: $\{ \mathfrak{e}_{n} : n\geq 0\}$ given by \eqref{EigenfunctionsModel3YM},
\end{itemize}
forms an orthonormal and complete basis of the Hilbert spaces
\begin{itemize}
	\item CW: $L^2 ( [0,\pi ];\sin^2(x)dx )$,
	\item WH: $L^2 ( [0,\pi/2 ];\sin(2x)dx )$,
	  \item YM: $L^2 ( [0,\pi ];\sin^4(x)dx )$.
\end{itemize}
Then, we expand $u(t,\cdot)$ in terms of the eigenfunctions and substitute the expression into \eqref{AllModelsinOnePDE} to find infinite systems of non--linear harmonic oscillators.
\subsection{Conformal cubic wave equation in spherical symmetry}
For the conformal cubic wave equation in spherical symmetry, we expand 
\begin{align*}
	u(t,x) = \sum_{n=0}^{\infty}  u^{n}(t)  e_{n}(x),
\end{align*}
and use the expansion
\begin{align}\label{ExpnasionsFourierModel1CS}
	 e_{i}  e_{j}  e_{k} =  \sum_{m=0}^{\infty}  C_{ijkm}  e_{m}
\end{align}
 to find the following infinite system of non--linear harmonic oscillators 
\begin{align}\label{PDEFourierSpaceModel1WS}
	& \ddot{u}^{m}(t) + \left(A u(t)\right)^m   = \left( f( \{ u^j(t): j \geq 0 \} ) \right)^m, \Hquad m \geq 0
\end{align}
 where the dots denote derivatives with respect to time and 
\begin{align}
	\left(A u(t)\right)^m &:= \omega_{m}^2  u^m(t),\nonumber \\ 
	  \left( f( \{ u^j(t): j \geq 0 \} ) \right)^m &:= - \sum_{i,j,k=0}^{\infty}  C_{ijkm} 
	 u^{i}(t) u^{j}(t)  u^{k}(t)\label{DefinitionNonLinearityfModel1} .
\end{align}

\subsection{Conformal cubic wave equation out of spherical symmetry}
For the conformal cubic wave equation out of spherical symmetry, we expand  
\begin{align*}
	u(t,x) = \sum_{n=0}^{\infty} u^{n}(t) \mathsf{e}_{n}^{(\mu_1,\mu_2)}(x),
\end{align*}
 and use the expansion
\begin{align}\label{ExpnasionsFourierModel2CH}
	\mathsf{e}_{i}^{(\mu_1,\mu_2)} \mathsf{e}_{j}^{(\mu_1,\mu_2)} \mathsf{e}_{k}^{(\mu_1,\mu_2)} =  \sum_{m=0}^{\infty} {}\mathsf{C}_{ijkm}^{(\mu_1,\mu_2)} \mathsf{e}_{m}^{(\mu_1,\mu_2)}
\end{align}
 to find the following infinite system of non--linear harmonic oscillators 
\begin{align}\label{PDEFourierSpaceModel2CH}
	& \ddot{u}^{m}(t) + \left(\mathsf{A} u(t)\right)^m  = \left( \mathsf{f}( \{ u^j(t): j \geq 0 \} ) \right)^m  , \Hquad m \geq 0
\end{align}
  where the dots denote derivatives with respect to time and
\begin{align}
	\left(\mathsf{A} u(t)\right)^m   &:= \left(\somega_n^{(\mu_1,\mu_2)} \right)^2 u^m(t),\nonumber  \\ 
	  \left( \mathsf{f}( \{ u^j(t): j \geq 0 \} ) \right)^m &:= - \sum_{i,j,k=0}^{\infty}  \mathsf{C}_{ijkm}^{(\mu_1,\mu_2)} 
	u^{i}(t)u^{j}(t) u^{k}(t) .\label{DefinitionNonLinearityfModel2}
\end{align}

\subsection{Yang--Mills equation in spherical symmetry}\label{SubsetionYMSysteminFourier}
For the Yang--Mills equation in spherical symmetry, we expand  \begin{align*}
	u(t,x) = \sum_{n=0}^{\infty} u^{n}(t) \mathfrak{e}_{n} (x),
\end{align*}
and use the expansions
\begin{align}\label{ExpnasionsFourierModel3YM}
 \mathfrak{e}_{i}(x)\mathfrak{e}_{j}(x)   =  \sum_{m=0}^{\infty} \overline{\mathfrak{C}}_{ijm} \mathfrak{e}_{m}(x), \Hquad  \sin^2(x)\mathfrak{e}_{i}(x)\mathfrak{e}_{j} (x)\mathfrak{e}_{k}(x) =  \sum_{m=0}^{\infty} \mathfrak{C}_{ijkm} \mathfrak{e}_{m}(x)
\end{align}
 to find the following infinite system of non--linear harmonic oscillators 
\begin{align}\label{PDEFourierSpaceModel3YM}
	& \ddot{u}^{m}(t) + \left(\mathfrak{A} u(t)\right)^m =\left( \mathfrak{f}( \{ u^j(t): j \geq 0 \} ) \right)^m  , \Hquad m \geq 0
\end{align}
  where the dots denote derivatives with respect to time and
\begin{align*}
	 \left(\mathfrak{A} u(t)\right)^m &:= \ssomega_{m}^2 u^m(t), \\ 
	 \left( \mathfrak{f}( \{ u^j(t): j \geq 0 \} ) \right)^m  &:=   
  \left( \mathfrak{f}^{(2)}( \{ u^j(t): j \geq 0 \} ) \right)^m + \left( \mathfrak{f}^{(3)}( \{ u^j(t): j \geq 0 \} ) \right)^m ,
\end{align*}  
with
\begin{align}
	 \left( \mathfrak{f}^{(2)}( \{ u^j(t): j \geq 0 \} ) \right)^m  &:=-3\sum_{i,j=0}^{\infty}\overline{\mathfrak{C}}_{ijm}  u^{i}(t)u^{j}(t)  ,\label{DefinitionNonLinearityfModel3a} \\
 \left( \mathfrak{f}^{(3)}( \{ u^j(t): j \geq 0 \} ) \right)^m  &:=- \sum_{i,j,k=0}^{\infty}  \mathfrak{C}_{ijkm} u^{i}(t)u^{j}(t)u^{k}(t).\label{DefinitionNonLinearityfModel3b}
\end{align}
 
 \subsection{Lipschitz bounds}
Recall Section \ref{SectionMethodBambusiPaleari} where we define the Banach space 
	\begin{align*}
	\mathcal{H}^{k}_s   \subseteq H^{k}([0,2\pi];l_s^2)
\end{align*} 
consisting of spacetime functions of the form
\begin{align*}
	q(t)= \sum_{j=0}^{\infty} q^{j}(t) e_{j}=
	 \sum_{j=0}^{\infty} \left(\sum_{l=0}^{\infty}q^{lj} \cos(lt) \right) e_{j}
\end{align*} 
such that the norm
\begin{align*}
	\| q\|_{\mathcal{H}^{k}_s}^2 := \sum_{j=0}^{\infty} j^{2s} \left(2 |q^{0j}|^2  + \sum_{l=1}^{\infty} |q^{lj}|^2 (1+l^2)^k  \right)
\end{align*}
is finite. Note that
\begin{align*}
	\| q\|_{\mathcal{H}^{k}_s}^2 = \frac{1}{\pi } \int_{0}^{2\pi } \sum_{\lambda=0}^k |q^{(\lambda)}(t) |_{s}^2 dt,
\end{align*}
where $q^{(\lambda)}(t)$ denotes the $\lambda-$th derivative of $q(t)$ with respect to $t$. Next, we show that the non--linear terms we consider satisfy the following Lipschitz bounds and we begin by considering the conformal cubic wave equation in spherical symmetry.

\begin{lemma}[Lipschitz bounds for the CW]\label{LipschitzModel1}
	Let $f$ be given by \eqref{DefinitionNonLinearityfModel1}. Then, for all integers $k \geq 0$ and $s\ge 2$, there exists a positive constant (depending only on $k$ and $s$) such that 
 	\begin{align*}  
 		\|  f (u)- f (v) \|_{\mathcal{H}^{k}_s  } &\lesssim  \left(\|u \|_{\mathcal{H}^{k}_s  }^2 +\|v\|_{\mathcal{H}^{k}_s  }^2 \right) \|u -v\|_{\mathcal{H}^{k}_s  }, \\
 		\left\|df(u)[h] -df(v)[h] \right \|_{\mathcal{H}^{k}_s} & \lesssim \left( \left \| u  \right \|_{\mathcal{H}^{k}_s } +\left \| v \right \|_{\mathcal{H}^{k}_s } \right)\left \| h \right \|_{\mathcal{H}^{k}_s } \left \| u -v \right \|_{\mathcal{H}^{k}_s },
 	\end{align*}
 	for all $u,v,h \in \mathcal{H}^{k}_s  $ with $\|u\|_{\mathcal{H}^{k}_s  } \leq \epsilon$, $\|v\|_{\mathcal{H}^{k}_s  } \leq \epsilon$ and  $\|h\|_{\mathcal{H}^{k}_s  } \leq \epsilon$.
\end{lemma}
\begin{remark}[Regularity of the initial data for the CW model]
	 As stated above, for the CW model, we require $s\in  \mathbb{N}$ with $s \ge 2$. This means that the space of initial data  $ \mathcal{Q}  \simeq l_{s+1}^2$ is at least $l_{3}^2$ (Theorem \ref{OrigivalversionTheoremBambusi}). 
\end{remark}
\begin{proof}
	Let $s \geq 2$ be an integer and pick any $u=\{u^j: j \geq 0\} \in l^2_s$. We also denote by $u$ the corresponding function in the physical space, that is
 	\begin{align*}
 		u(x):= \sum_{j =0}^{\infty} u_j e_j(x)
 	\end{align*}
 	and recall that the operator in physical space is given by
 	\begin{align*}
 		L u:= 	- \frac{1}{\sin^2 (x) } \partial_{x} \left(
 		\sin^2 (x) \partial_{x} u
 		\right) +  u
 	\end{align*}
 	On the one hand, for any integer $s \geq 1$, we define the Sobolev space $H^s_{CW} $ space for spherically symmetric functions,
 	\begin{align*}
 		\| u \|^2_{H^s_{\text{CW}}} := \int_0^\pi u L^s u  \sin^2 x  dx   
 		=\sum_{j=0}^{+\infty} \omega_{j}^{2s} |u_j|^2 \simeq | u |_{s},
 	\end{align*}
 	since $\omega_{j} \simeq j$. On the other hand, note that, with a slight abuse of notation (we denote by $u$ the original variable as well as the spherically symmetric version of it), we have that the Sobolev space above is equivalent to the standard Sobolev on $\mathbb{S}^3$, defined by
 	\begin{align*}
 		\| u \|^2_{H^s(\mathbb{S}^3)} = \int_{\mathbb{S}^3} u (- \Delta^s_{\mathbb{S}^3}u)d \text{vol}_{\mathbb{S}^3} + \|u \|_{L^2(\mathbb{S}^3)}^2, 
 	\end{align*}
 	that is
 	\begin{align*}
 		\| u \|^2_{H^s_{CW} } \simeq \| u \|^2_{H^s(\mathbb{S}^3)}.
 	\end{align*}
 	Here, $\Delta_{\mathbb{S}^3}$ stands for the standard Laplacian on $\mathbb{S}^3$ for the round metric and the standard volume form $d \text{vol}_{\mathbb{S}^3}$. This equivalence yields that $H^s_{CW} $ is an algebra, since $H^s (\mathbb{S}^3)$ is an algebra provided that $s \geq 3/2$. Then, we pick an integer $s \geq 2$, $u,v \in l_s^2$ and the algebra property, the triangular inequality together with Plancherel's theorem yield
\begin{align*}
	\left|f(u) \right|_s 	&= \left \| u^3 \right \|_{H^s_{CW} } \lesssim \left \| u \right \|_{H^s_{CW} }^3 =   \left  | u \right  |_{s }^3 , \\
	\left|f(u) -f(v) \right|_s 	&= \left \| u^3 -v^3 \right \|_{H^s_{CW} } =\left \| (u -v)(u^2+uv+v^2) \right \|_{H^s_{CW} } \\
	&  \lesssim \left \| u -v \right \|_{H^s_{CW} } \left(\left \| u^2 \right \|_{H^s_{CW} } +\left \| u v \right \|_{H^s_{CW} }+\left \| v^2 \right \|_{H^s_{CW} }   \right) \\
	& \lesssim \left \| u -v \right \|_{H^s_{CW} } \left(\left \| u \right \|_{H^s_{CW} }^2 +\left \| u  \right \|_{H^s_{CW} }\left \|  v \right \|_{H^s_{CW} }+\left \| v \right \|_{H^s_{CW} }^2   \right) \\
	& \lesssim \left \| u -v \right \|_{H^s_{CW} } \left(\left \| u \right \|_{H^s_{CW} }^2  +\left \| v \right \|_{H^s_{CW} }^2   \right) \\
	& = \left | u -v \right |_{s } \left(\left | u \right |_{s }^2  +\left | v \right |_{s}^2   \right), \\
	\left|df(u)[h] -df(v)[h] \right|_s 	&= \left \|df(u)[h] -df(v)[h]  \right \|_{H^s_{CW} } = \left \|d(u^3)[h] -d(v^3)[h]  \right \|_{H^s_{CW} } \\
	& = \left \|3u^2 h -3v^2 h  \right \|_{H^s_{CW} }  \lesssim \left \| u^2 -v^2  \right \|_{H^s_{CW} } \left \| h \right \|_{H^s_{CW} } \\
	& \lesssim \left \| (u-v)(u+v)  \right \|_{H^s_{CW} } \left \| h \right \|_{H^s_{CW} } \\
	&  \lesssim \left \| u -v \right \|_{H^s_{CW} } \left \| u + v \right \|_{H^s_{CW} }\left \| h \right \|_{H^s_{CW} } \\
	&  \lesssim \left \| u -v \right \|_{H^s_{CW} } \left( \left \| u  \right \|_{H^s_{CW} } +\left \| v \right \|_{H^s_{CW} } \right)\left \| h \right \|_{H^s_{CW} } \\
	&  \lesssim \left | u -v \right |_{s } \left( \left | u  \right |_{s } +\left | v \right |_{s } \right)\left | h \right |_{s }.
\end{align*}
This proves the claim for $k=0$, Finally, we present the proof for $k=1$. In this case, Plancherel's theorem yields
\begin{align*}
	\left \|  f(u) \right\|_{\mathcal{H}^1_s} 	&= \left \|  f(u) \right\|_{H^1_{t} l_s^2} =\left \|  f(u) \right\|_{L^2_{t} l_s^2} +\left \| \partial_{t} f(u) \right\|_{L^2_{t} l_s^2}  =\left \|  f(u) \right\|_{L^2_{t} H^s_x} +\left \| \partial_{t} f(u) \right\|_{L^2_{t} H^s_x}.
\end{align*}
Furthermore, the algebra property, Holder's inequality together with the embedding $  H^1  \hookrightarrow  L^{\infty}$ yield
\begin{align*}
	\left \|  f(u) \right\|_{L^2_{t} H^s_x} &=\left \|  u^3 \right\|_{L^2_{t} H^s_x} = \left \|  \left \|  u^3 \right\|_{H^s_x} \right\|_{L^2_{t}}  \lesssim \left \|  \left \|  u \right\|_{H^s_x}^3 \right\|_{L^2_{t}}
	 \leq 
	\left \|  \left \|  u \right\|_{H^s_x} ^2 \right\|_{L^{\infty}_{t}} \left \|  \left \|  u \right\|_{H^s_x} \right\|_{L^2_{t}} \\
	&=\left \|  \left \|  u \right\|_{H^s_x}  \right\|_{L^{\infty}_{t}}^2 \left \|  \left \|  u \right\|_{H^s_x} \right\|_{L^2_{t}}  \lesssim
	\left \|  \left \|  u \right\|_{H^s_x} \right\|_{H^{1}_{t}} ^2 \left \|  \left \|  u \right\|_{H^s_x} \right\|_{L^2_{t}}\\
	& = \|u \|_{H^1_t H^s_x}^2  \|u \|_{H^0_t H^s_x} \leq \|u \|_{H^1_t H^s_x}^3, \\
	\left \|\partial_{t}  f(u) \right\|_{L^2_{t} H^s_x} &=\left \| 3 u^2 \partial_{t} u \right\|_{L^2_{t} H^s_x} \simeq  \left \|  \left \|  u^2 \partial_{t} u\right\|_{H^s_x} \right\|_{L^2_{t}}  \lesssim \left \|  \left \|  u \right\|_{H^s_x}^2 \left \|  \partial_{t}u \right\|_{H^s_x}  \right\|_{L^2_{t}} \\
	& \leq \left \|  \left \|  u \right\|_{H^s_x}^2   \right\|_{L^{\infty}_{t}} \left \|   \left \|  \partial_{t}u \right\|_{H^s_x}  \right\|_{L^2_{t}}
	= \left \|  \left \|  u \right\|_{H^s_x}   \right\|_{L^{\infty}_{t}} ^2 \left \|   \left \|  \partial_{t}u \right\|_{H^s_x}  \right\|_{L^2_{t}} \\
	& \lesssim \left \|  \left \|  u \right\|_{H^s_x}   \right\|_{H^{1}_{t}} ^2 \left \|   \left \|   u \right\|_{H^s_x}  \right\|_{H^1_{t}} \leq    \left \|  u    \right\|_{H^{1}_{t} H^s_x } ^3
\end{align*}
and hence
\begin{align*}
	\left \|  f(u) \right\|_{\mathcal{H}^1_s} 	\lesssim  \left \|  u    \right\|_{\mathcal{H}^1_s} ^3.
\end{align*}
All the other bounds follow similarly.
\end{proof} 

Next, we consider the conformal cubic wave equation out of spherical symmetry.

\begin{lemma}[Lipschitz bounds for the CH]\label{LipschitzModel2}
	Let $\mathsf{f}$ be given by \eqref{DefinitionNonLinearityfModel2}. Then, for all integers $k \geq 0$ and $s\ge 2$, there exists a positive constant (depending only on $k$ and $s$) such that 
 	\begin{align*}  
 		\|  \mathsf{f} (u)- \mathsf{f} (v) \|_{\mathcal{H}^{k}_s  } &\lesssim  \left(\|u \|_{\mathcal{H}^{k}_s  }^2 +\|v\|_{\mathcal{H}^{k}_s  }^2 \right) \|u -v\|_{\mathcal{H}^{k}_s  }, \\
 		\left\|d\mathsf{f}(u)[h] -d\mathsf{f}(v)[h] \right \|_{\mathcal{H}^{k}_s} & \lesssim \left( \left \| u  \right \|_{\mathcal{H}^{k}_s } +\left \| v \right \|_{\mathcal{H}^{k}_s } \right)\left \| h \right \|_{\mathcal{H}^{k}_s } \left \| u -v \right \|_{\mathcal{H}^{k}_s },
 	\end{align*}
 	for all $u,v,h \in \mathcal{H}^{k}_s  $ with $\|u\|_{\mathcal{H}^{k}_s  } \leq \epsilon$, $\|v\|_{\mathcal{H}^{k}_s  } \leq \epsilon$ and  $\|h\|_{\mathcal{H}^{k}_s  } \leq \epsilon$.
\end{lemma}
\begin{remark}[Regularity of the initial data for the CH model]
	 As stated above, for the CH model, we require $s\in  \mathbb{N}$ with $s \ge 2$. This means that the space of initial data  $ \mathcal{Q}  \simeq l_{s+1}^2$ is at least $l_{3}^2$ (Theorem \ref{OrigivalversionTheoremBambusi}). 
\end{remark}
\begin{proof}
The proof coincides with the one of Lemma \ref{LipschitzModel1}.
\end{proof}

Finally, we consider the Yang--Mills equation in spherical symmetry.

\begin{lemma}[Lipschitz bounds for the YM]\label{LipschitzModel3}
	Let $\mathfrak{f}^{(2)}$ and $\mathfrak{f}^{(3)}$ be given by \eqref{DefinitionNonLinearityfModel3a} and \eqref{DefinitionNonLinearityfModel3b} respectivel. Then, for all integers $k \geq 0$ and $s\ge 3$, there exists a positive constant (depending only on $k$ and $s$) such that 
 	\begin{align*} 
 		\|  \mathfrak{f}^{(2)} (u)- \mathfrak{f}^{(2)} (v) \|_{\mathcal{H}^{k}_s  } &\lesssim \left(\|u \|_{\mathcal{H}^{k}_s  }+\|v\|_{\mathcal{H}^{k}_s  }\right) \|u -v\|_{\mathcal{H}^{k}_s  }, \\
 		\left\|d\mathfrak{f}^{(2)}(u)[h] -d\mathfrak{f}^{(2)}(v)[h] \right \|_{\mathcal{H}^{k}_s} & \lesssim \left \| h \right \|_{\mathcal{H}^{k}_s } \left \| u -v \right \|_{\mathcal{H}^{k}_s }, \\
 		\|  \mathfrak{f}^{(3)} (u)- \mathfrak{f}^{(3)} (v) \|_{\mathcal{H}^{k}_s  } &\lesssim  \left(\|u \|_{\mathcal{H}^{k}_s  }^2 +\|v\|_{\mathcal{H}^{k}_s  }^2 \right) \|u -v\|_{\mathcal{H}^{k}_s  }, \\
 		\left\|d\mathfrak{f}^{(2)}(u)[h] -d\mathfrak{f}^{(2)}(v)[h] \right \|_{\mathcal{H}^{k}_s} & \lesssim \left( \left \| u  \right \|_{\mathcal{H}^{k}_s } +\left \| v \right \|_{\mathcal{H}^{k}_s } \right)\left \| h \right \|_{\mathcal{H}^{k}_s } \left \| u -v \right \|_{\mathcal{H}^{k}_s },
 	\end{align*}
 	for all $u,v,h \in \mathcal{H}^{k}_s  $ with $\|u\|_{\mathcal{H}^{k}_s  } \leq \epsilon$, $\|v\|_{\mathcal{H}^{k}_s  } \leq \epsilon$ and  $\|h\|_{\mathcal{H}^{k}_s  } \leq \epsilon$.
\end{lemma}
\begin{remark}[Regularity of the initial data for the YM model]
	 As stated above, for the YM model, we require $s \in  \mathbb{N}$ with $s \ge 3$. This means that the space of initial data  $  l_{s+3}^2$ is at least $l_{6}^2$ (Lemma \ref{LemmaEstimateR}, Theorem \ref{TheoremModificationofBambusiPaleari}). 
\end{remark}
\begin{proof}
Let $s \geq 3  $ be an integer and pick any $u =\{u^j: j \geq 0 \}\in l^2_s$. We also denote by $u$ the corresponding function in the physical space, that is
 	\begin{align*}
 		u(x):= \sum_{j = 0}^{\infty} u_j \mathfrak{e}_j(x)
 	\end{align*}
 	and recall that the operator in physical space is given by
 	\begin{align*}
 		\mathfrak{L} u:= 	- \frac{1}{\sin^4 (x) } \partial_{x} \left(
 		\sin^4 (x) \partial_{x} u\right) + 4 u.
 	\end{align*}
 	In the following, we claim that the operator 
 	\begin{align*}
 		\Delta_{\text{YM}} u := \frac{1}{\sin^4 (x) } \partial_{x} \left( \sin^4 (x) \partial_{x} u \right)
 	\end{align*}
 	coincides with the Laplace--Beltrami operator $\Delta_{\mathbb{S}^5} u$ on the sphere $\mathbb{S}^5 \hookrightarrow \mathbb{R}^6$ restricted to a class of symmetric functions. Indeed, we endow $\mathbb{S}^5$ with the round metric and consider the standard Eulerian coordinates $\left(x_1=x, x_2, x_3,x_4,x_5 \right) \in (0, \pi)^4 \times (0, 2 \pi)$, so that  $y =( y^1, y^2,y^3,y^4,y^5, y^6) \in \mathbb{S}^5$ with 
 	\begin{eqnarray*}
 		y^1&=& \cos x_1, \\
 		y^2&=& \sin x_1 \cos x_2, \\
 		y^3&=& \sin x_1 \sin x_2 \cos x_3, \\
 		y^4&=& \sin x_1 \sin x_2 \sin x_3 \cos x_4, \\
 		y^5&=& \sin x_1 \sin x_2 \sin x_3 \sin x_4   \cos x_5, \\ 
 		y^6&=& \sin x_1 \sin x_2\sin x_3\sin x_4 \sin x_5. 
 	\end{eqnarray*}
 	The metric element in these coordinates is given by the standard round metric on $\mathbb{S}^5$. Then, for a function $u$ defined on $\mathbb{S}^5$ that is invariant under all rotations around the $y^6-$axis, the operator $\Delta_{\text{YM}} u$ coincides with $\Delta_{\mathbb{S}^5} u$. We denote such functions on $\mathbb{S}^5$ below as spherically-symmetric. On the one hand, for any integer $s \geq 3$, we define the Sobolev space $H^s_{\text{YM}} $ for spherically symmetric functions as follows
 	\begin{align*}
 		\| u \|^2_{H^k_{\text{YM}} } := \int_0^\pi u \mathfrak{L}^k u  \sin^4 x  dx = \sum_{j=0}^{\infty} \ssomega_{j}^{2k} |u_j|^2 \simeq | u |^2_{k},
 	\end{align*} 
 	since $\ssomega_j \simeq j$.  On the other hand, note that, with a slight abuse of notation (we denote by $u$ the original variable as well as the spherically symmetric version of it), we have that the Sobolev space above is equivalent to the standrad Sobolev space on $\mathbb{S}^5$, defined by
 	\begin{align*}
 		\| u \|_{H^s\left(\mathbb{S}^5 \right)}^2= \int_{\mathbb{S}^5} u(- \Delta^s_{\mathbb{S}^5}u) d \text{vol}_{\mathbb{S}^5} + \| u \|^2_{L^2(\mathbb{S}^5)}, 
 	\end{align*}
 	that is
 	\begin{align*}
 		\| q \|_{H^s_{\text{YM}} }^2	\simeq \| q \|_{H^s (\mathbb{S}^5)}^2.
 	\end{align*} 
 	Here, $\Delta_{\mathbb{S}^5}$ stands for the standard Laplacian on $\mathbb{S}^5$ for the round metric and the standard volume form $dvol_{\mathbb{S}^5}$. This equivalence, yields that $H^s_{\text{YM}} $ is an algebra, since $H^s (\mathbb{S}^5)$ is an algebra provided that $s \geq 5/2$. We pick an integer $s \geq 3$ and the rest of the proof coincides with the one of Lemma \ref{LipschitzModel1}.
\end{proof}

 \section{The Fourier coefficients} 	\label{SectionFourierConstants}
This section is devoted to the study of the Fourier coefficients, as defined by \eqref{ExpnasionsFourierModel1CS}, \eqref{ExpnasionsFourierModel2CH} and  \eqref{ExpnasionsFourierModel3YM}. Since the eigenfunctions are given by Jacobi polynomials and since the Fourier coefficients involve products of the eigenfunctions, these are a priori complicated integrals, depending on the indices of the eigenfunctions. Nonetheless, we will derive here explicit closed formulas for the various Fourier coefficients on resonant indices.
\subsection{Conformal cubic wave equation in spherical symmetry}
In this case, the Fourier coefficients are given by \eqref{ExpnasionsFourierModel1CS}, that is
	\begin{align*}
		C_{ijkm}: &= (e_i e_j e_k|e_m) \\
		 & =\frac{2}{\pi } \int_{0}^{\pi }  e_{i}(x)e_{j}(x)e_{k}(x)e_{m}(x) \sin^2 (x)  dx \\
		 &=\frac{2}{\pi} \int_{-1}^{1} U_{i}(y)U_{j}(y)U_{k}(y)U_{m}(y) \sqrt{1-y^2}  d y,
	\end{align*}
where we used the definition of the Chebyshev polynomials of the second kind,
	\begin{align*}
		e_{n}(x) :=U_{n}(\cos(x))=\frac{\sin(\omega_{n}x)}{\sin(x)}, 
	\end{align*}
for all $ n \in \{i,j,k,m\}$	. Next, we call \textit{resonant} a quadruple $(i,j,k,m)$ of indices such that
\begin{align*}
	\omega_{i} \pm \omega_{j} \pm \omega_{k} \pm \omega_{m} =0
\end{align*}
and study the Fourier coefficients on resonant indices.

\subsubsection{Vanishing Fourier coefficients}
Firstly, we show that the Fourier coefficients vanish on some resonant indices.
\begin{lemma}[Vanishing Fourier coefficients on resonant indices]\label{LemmaVanishingFourierModel1CW}
For any integers $i,j,k,m \geq 0$, we have
\begin{align*}
	  \omega_{i} + \omega_{j} + \omega_{k} - \omega_{m}  &= 0
	  \Longrightarrow 
	  C_{ijkm}=0     ,\\
	  \omega_{i} + \omega_{j} - \omega_{k} + \omega_{m}  &= 0
	  \Longrightarrow 
	  C_{ijkm}=0     ,\\
	  \omega_{i} - \omega_{j} + \omega_{k} + \omega_{m}  &= 0
	  \Longrightarrow 
	  C_{ijkm}=0     ,\\
	   -\omega_{i} + \omega_{j} + \omega_{k} + \omega_{m}  &= 0
	  \Longrightarrow 
	  C_{ijkm}=0     .
\end{align*}
\end{lemma}

 \begin{proof}
 	Let $i,j,k,m$ be positive integers such that $\omega_{i} + \omega_{j} + \omega_{k} - \omega_{m} = 0$. Then, $m=2+i+j+k$ and according to the computation above, we have
\begin{align*}
		C_{ijkm} = \int_{-1}^{1} R_{N}(y) U_{m}(y) \sqrt{1-y^2}  d y	,
\end{align*} 
where
\begin{align*}
	R_{N}(y):= \frac{2}{\pi}  U_{i}(y)U_{j}(y)U_{k}(y)
\end{align*}
is a polynomial of degree $N=i+j+k<m$ and hence the Fourier coefficient vanishes since $U_{m}(y)$ forms an orthonormal and complete basis with respect to the weight $\sqrt{1-y^2} $. The other results now follow immediately using the symmetries of the Fourier coefficients with respect to $i,j,k,m$. 
 \end{proof}
 \subsubsection{Non--vanishing Fourier coefficients}
Secondly, we study the non--vanishing Fourier coefficients on resonant indices. In order to deal with these constants, one needs a computationally efficient formula. In the spherically symmetric case, where the basis consists of the Chebyshev polynomials, there exists the addition formula \cite{MR0372517}
\begin{align}\label{AdditionTheoremChebyshev}
U_{p}(y)U_{q}(y) =
\sum_{\substack{  r= |q-p| \\ \text{step 2} } }^{p+q} U_{r}(y) =\sum_{s=0}^{\min(p,q)} U_{|q-p|+2s}(y),
\end{align}	
for all $p, q \geq 0$. Consequently, one obtains the closed formula
\begin{align}\label{ClosedformulaFourierModel1CW}
		C_{ijkm}  & = \sum_{\substack{r = |j-i| \\ \text{step }2} }^{j+i}
	\sum_{\substack{ s= |m-k| \\ \text{step }2} }^{m+k}
       \mathds{1} \left( r = s \right),
\end{align}	
for any integers $i,j,k,m \geq 0$. Indeed, we implement the addition formula above together with the orthogonality property for the Chebyshev polynomials with respect to the weight $\sqrt{1-y^2}$ to obtain
	\begin{align*}
		C_{ijkm} &  = \frac{2}{\pi} \int_{-1}^{1} U_{i}(y)U_{j}(y)U_{k}(y)U_{m}(y) \sqrt{1-y^2}  d y  \\
		 &=\frac{2}{\pi}\sum_{\substack{  r= |j-i| \\ \text{step 2} } }^{i+j} 
		 \sum_{\substack{  s= |m-k| \\ \text{step 2} } }^{m+k} \int_{-1}^{1}U_{r}(y) U_{s}(y)\sqrt{1-y^2}  d y  \\
		  & = \sum_{\substack{r =|j-i| \\ \text{step }2} }^{j+i}
	\sum_{\substack{ s=|m-k| \\ \text{step }2} }^{m+k}
       \mathds{1} \left( r = s \right).
	\end{align*}	 
Next, we use \eqref{ClosedformulaFourierModel1CW} to derive closed formulas for the non-vanishing Fourier coefficients.
\begin{lemma}[Non--vanishing Fourier coefficients on resonant indices: closed formulas]\label{LemmaClosedformulaFourierModel1CW}
For any integers $i,j,k,m \geq 0$, we have
\begin{align*}	  
	  & \omega_{i} + \omega_{j} - \omega_{k} - \omega_{m} = 0
	  \Longrightarrow 
	  C_{ijkm}= \omega_{\min \{i,j,k,m\}}    ,\\
	  & \omega_{i} - \omega_{j} + \omega_{k} - \omega_{m} = 0
	  \Longrightarrow 
	  C_{ijkm}= \omega_{\min \{i,j,k,m\}}    ,\\
	   & \omega_{i} - \omega_{j} - \omega_{k} + \omega_{m} = 0
	  \Longrightarrow 
	  C_{ijkm}= \omega_{\min \{i,j,k,m\}}  .  
\end{align*}
\end{lemma}

\begin{proof}
	Let $i,j,k,m \geq 0$ be integers such that $\omega_{i} + \omega_{j} - \omega_{k} - \omega_{m} = 0$ or equivalently, $m=i+j-k$ and assume for simplicity that $i\leq j\leq k \leq m$. Then,  \eqref{ClosedformulaFourierModel1CW} yields
	\begin{align*} 
		C_{ijkm}  & = \sum_{\substack{r = j-i \\ \text{step }2} }^{j+i}
	\sum_{\substack{ s= i+j-2 k\\ \text{step }2} }^{i + j}
       \mathds{1} \left( r = s \right) 
        = \sum_{\substack{p = 0 \\ \text{step }2} }^{2i}
	\sum_{\substack{ q= 0\\ \text{step }2} }^{2k}
       \mathds{1} \left( i+j-p = i+j-q \right) \\
       & = \sum_{\substack{p = 0 \\ \text{step }2} }^{2i}
	\sum_{\substack{ q= 0\\ \text{step }2} }^{2k}
       \mathds{1} \left( p =  q \right) 
        = \sum_{ t = 0   }^{i}
	\sum_{ \tau= 0  }^{k}
       \mathds{1} \left( t =  \tau \right)    = \sum_{ t = 0   }^{i}
	\left( \sum_{ \tau= 0  }^{i}
       \mathds{1} \left( t =  \tau \right)+
	\sum_{ \tau= i+1  }^{k}
       \mathds{1} \left( t =  \tau \right) \right) \\
       &=\sum_{ t = 0   }^{i}
	  \sum_{ \tau= 0  }^{i}
       \mathds{1} \left( t =  \tau \right)  = i+1 = \omega_{i}.
\end{align*}	
All the other results follow immediately by the symmetries of the Fourier coefficients with respect to $i,j,k$ and $m$, that completes the proof. 
\end{proof} 
\subsection{Conformal cubic wave equation out of spherical symmetry}
In this case, the Fourier coefficients are given by \eqref{ExpnasionsFourierModel2CH}, that is
 \begin{align*}
		\mathsf{C}_{ijkm}^{(\mu_1,\mu_2)}: &= \left\langle \mathsf{e}_i^{(\mu_1,\mu_2)} \mathsf{e}_j^{(\mu_1,\mu_2)}\mathsf{e}_k^{(\mu_1,\mu_2)} \Big| \mathsf{e}_m^{(\mu_1,\mu_2)}\right\rangle  \\
		&  =\int_{0}^{\frac{\pi }{2}} \mathsf{e}_{i}^{(\mu_1,\mu_2)}(x) \mathsf{e}_{j}^{(\mu_1,\mu_2)}(x) \mathsf{e}_{k}^{(\mu_1,\mu_2)} (x) \mathsf{e}_{m}^{(\mu_1,\mu_2)} (x) \sin(2x) dx,  \\
		& =
		\frac{1}{2}\prod_{\lambda_{1} \in \{i,j,k,m\}} \mathsf{N}_{\lambda_{1}}^{(\mu_1,\mu_2)}  \int_{-1}^{1} (1-x)^{2\mu_1} (1+x)^{2\mu_2}
		\prod_{\lambda_{2} \in \{i,j,k,m\}} P_{\lambda_{2}}^{(\mu_1,\mu_2)}(x)  dx,
\end{align*} 
where the normalization constant $\mathsf{N}_{\lambda_{1}}^{(\mu_1,\mu_2)}$ is given by \eqref{NormalizationConstantModel2CH}. As before, we call \textit{resonant} a quadruple $(i,j,k,m)$ of indices such that
\begin{align*}
	\somega_{i}^{(\mu_1,\mu_2)} \pm \somega_{j}^{(\mu_1,\mu_2)} \pm \somega_{k}^{(\mu_1,\mu_2)} \pm \somega_{m}^{(\mu_1,\mu_2)} =0
\end{align*}
and study the Fourier coefficients on resonant indices.
\subsubsection{Vanishing Fourier coefficients}
Firstly, we show that the Fourier coefficients vanish on some resonant indices.

\begin{lemma}[Vanishing Fourier coefficients on resonant indices]\label{LemmaVanisginFourierModel2CH}
For any integers $\mu_1,\mu_2 \geq 0$ and $i,j,k,m \geq 0$, we have
\begin{align*}
	  \somega_{i}^{(\mu_1,\mu_2)} + \somega_{j}^{(\mu_1,\mu_2)} + \somega_{k}^{(\mu_1,\mu_2)} - \somega_{m}^{(\mu_1,\mu_2)}  & = 0
	  \Longrightarrow 
	  \mathsf{C}_{ijkm}^{(\mu_1,\mu_2)}=0     ,\\
	   \somega_{i}^{(\mu_1,\mu_2)} + \somega_{j}^{(\mu_1,\mu_2)} - \somega_{k}^{(\mu_1,\mu_2)} + \somega_{m}^{(\mu_1,\mu_2)}  &= 0
	  \Longrightarrow 
	  \mathsf{C}_{ijkm}^{(\mu_1,\mu_2)}=0     ,\\
	   \somega_{i}^{(\mu_1,\mu_2)} - \somega_{j}^{(\mu_1,\mu_2)} + \somega_{k}^{(\mu_1,\mu_2)} + \somega_{m}^{(\mu_1,\mu_2)}  &= 0
	  \Longrightarrow 
	  \mathsf{C}_{ijkm}^{(\mu_1,\mu_2)}=0     ,\\
	 - \somega_{i}^{(\mu_1,\mu_2)} + \somega_{j}^{(\mu_1,\mu_2)} + \somega_{k}^{(\mu_1,\mu_2)} + \somega_{m}^{(\mu_1,\mu_2)} & = 0
	  \Longrightarrow 
	  \mathsf{C}_{ijkm}^{(\mu_1,\mu_2)}=0 .   
\end{align*}
\end{lemma}
\begin{proof}
Let $\mu_1,\mu_2 \geq 0$ and $i,j,k,m \geq 0$ be integers such that $\somega_{i}^{(\mu_1,\mu_2)} + \somega_{j}^{(\mu_1,\mu_2)} + \somega_{k}^{(\mu_1,\mu_2)} - \somega_{m}^{(\mu_1,\mu_2)} = 0$. Then, $m=i+j+k+\mu_1+\mu_2+1$ and according to the computation above, we have
	\begin{align*} 
	\mathsf{C}_{ijkm}^{(\mu_1,\mu_2)}= \int_{-1}^{1} \mathsf{R}_{N}^{(\mu_1,\mu_2)}(x)P_{m}^{(\mu_1,\mu_2)}(x) (1-x)^{\mu_1} (1+x)^{\mu_2}dx,
\end{align*}
where
\begin{align*}
	\mathsf{R}_{N}^{(\mu_1,\mu_2)}(x):=\frac{1}{2}\prod_{\lambda_{1} \in \{i,j,k,m\}} \mathsf{N}_{\lambda_{1}}^{(\mu_1,\mu_2)}  (1-x)^{\mu_1}
	 (1+x)^{\mu_2} \prod_{\lambda_{2} \in \{i,j,k\}} P_{\lambda_{2}}^{(\mu_1,\mu_2)}(x) 
\end{align*}
is a polynomial of degree $N=i+j+k+\mu_1+\mu_2<m$ and hence the Fourier coefficient vanishes since $P_{m}^{(\mu_1,\mu_2)}(x)$ forms an orthonormal and complete basis with respect to the weight $(1-x)^{\mu_1} (1+x)^{\mu_2}$. The other results now follow immediately using the symmetries of the Fourier coefficients with respect to $i,j,k,m$.  
\end{proof}

\subsubsection{Non--vanishing Fourier coefficients}\label{SectionAsymptoticsAppendix}
Next, we study the non--vanishing Fourier coefficients.  In principle, in order to deal with these constants, one needs a computationally efficient formula as in the spherically symmetric case. However, out of spherically symmetry, where the basis consists of the Jacobi polynomials, although there exists the addition formula   
\begin{align*}
	P_{p}^{(\mu_1,\mu_2)}(x)P_{q}^{(\mu_1,\mu_2)}(x)=\sum_{r=|p-q|}^{p+q} L(p,q,r)P_{r}^{(\mu_1,\mu_2)}(x), 
\end{align*}
for all $p,q \geq 0$, similar to \eqref{AdditionTheoremChebyshev}, the linearization coefficients $L(p,q,r)$ remain unknown in closed form for generic values of $\mu_1$ and $\mu_2$ and hence a closed formula similar to \eqref{ClosedformulaFourierModel1CW} is not available in general. We note that Rahman\footnote{
According to \cite{MR3527696}, there was a minor typo in Rahman's published result; in the linearization coefficient in \cite{MR634149} , the term $(-\alpha-\beta-2m)$ should be replaced by the Pochhammer symbol $(-\alpha-\beta-2m)_k$. The corrected linearization formula is given in \cite{MR3527696}.}
(page 919 in \cite{MR634149}) was able to prove that the linearization coefficients of Jacobi polynomials can be represented as a well--posed hypergeometric function ${}_9 F_8 (1)$. On the other hand, in the special case where $\mu_1=\mu_2 $, the Jacobi polynomials reduce to Gegenbauer polynomials for which the linearization coefficients are given by well-posed and closed formulas \cite{MR2723248, MR1820610, MR0372517}. Hence, we restrict ourselves to the case 
\begin{align*}
	\mu_1=\mu_2 :=\mu   
\end{align*} 
and also denote by $\gamma$ the index referring to the fixed choice of the 1--mode initial data. Then, we derive closed formulas for the non--vanishing Fourier coefficients for a resonant pair of indices $(i,j,k,m)$ by using the formula (equation (20) in \cite{MR1820610}) for the Gegenbauer polynomials,
\begin{align}\label{CombinationOflinearizationANDconnection}
	 \left( C_{m}^{\left(\mu +\frac{1}{2} \right)}(x) \right)^2 &=  \sum_{\lambda=0}^{m} L_{\mu m}(\lambda) C_{2\lambda}^{\left(2 \mu +\frac{1}{2} \right)}(x),
\end{align}
where the coefficients are given by
\begin{align}\label{DefinitionLCombinationOflinearizationANDconnection}
	L_{\mu m}(\lambda):= \frac{(2 \mu +1)_m}{\Gamma (m+1)} \frac{\left(\frac{1}{2}\right)_{\lambda } \left(\frac{1}{2}\right)_{m-\lambda } \left(\mu +\frac{1}{2}\right)_{\lambda } (\lambda +2 \mu +1)_m}{\Gamma (m-\lambda +1) (\mu +1)_{\lambda } \left(2 \mu +\frac{1}{2}\right)_{2 \lambda } \left(2 \lambda +2 \mu +\frac{3}{2}\right)_{m-\lambda }},
\end{align}
valid of any real $x\in [-1,1]$ and integers $\mu,m,\lambda \geq 0$ and
\begin{align*}
	(a)_n := \frac{\Gamma(a+n)}{\Gamma(a)},
\end{align*}
stands for the Pochhammer's symbol defined for any $a \in \mathbb{R}$ with $a \notin \{0,-1,-2,\dots\} $ and $n \in \mathbb{N}$. Notice that \eqref{CombinationOflinearizationANDconnection} is a combination of a linearization and a connection formula for Gegenbauer polynomials. Specifically, we establish the following result.

\begin{lemma}[Non--vanishing Fourier coefficients on resonant indices: closed formulas]\label{LemmaUniformAsymptoticFormulasModel2CH}
Let $\gamma \geq 0$ and $m \geq \gamma$ be any integers. Then, we have
\begin{align*}
	\mathsf{C}_{\gamma \gamma mm}^{(\mu,\mu)}    
	    & =  \frac{1}{2}      \sum_{\lambda=0}^{\gamma} \mathsf{M}_{ \gamma}^{(\mu)} (\lambda) \mathsf{M}_{ m}^{(\mu)} (\lambda) \xi_{\lambda}(\mu),
\end{align*}
where
\begin{align*}
	\xi_{ \lambda}(\mu):=\frac{\pi  2^{1-4 \mu } \Gamma (2 \lambda +4 \mu +1)}{(4 \lambda +4 \mu +1) \Gamma (2 \lambda +1) \left( \Gamma \left(2 \mu +\frac{1}{2}\right) \right)^2}
\end{align*}
and
\begin{align*}
	\mathsf{M}_{ m}^{(\mu)} (\lambda) &=\frac{1}{2 \pi ^{3/2}} \frac{(4 \lambda +4 \mu +1) \Gamma \left(\lambda +\frac{1}{2}\right) \Gamma \left(2 \mu +\frac{1}{2}\right) \Gamma \left(\lambda +\mu +\frac{1}{2}\right)}{\Gamma (\lambda +\mu +1) \Gamma (\lambda +2 \mu +1)} \cdot \\
	& \cdot \frac{(2 \mu +2 m+1) \Gamma \left(m-\lambda +\frac{1}{2}\right) \Gamma (m+\lambda +2 \mu +1)}{\Gamma (m-\lambda +1) \Gamma \left(m+\lambda +2 \mu +\frac{3}{2}\right)}.
\end{align*}
\end{lemma} 
\begin{proof} 
	Let $\gamma \geq 1$ be a fixed integer and pick any integer $m \geq \gamma$. Then, by the definition of the Fourier coefficient, the fact that the Jacobi polynomials with equal parameters can be written in terms of the Gegenbauer polynomials (equation 18.7.1 in \cite{MR2723248}),
\begin{align*}
	P_{m}^{(\mu,\mu)}(x) 
	=w_{m}^{(\mu)}C_{m}^{\left(\mu +\frac{1}{2} \right)}(x), \Hquad
	w_{m}^{(\mu)}:=\frac{\Gamma (2 \mu +1) \Gamma (m+\mu +1)}{\Gamma (\mu +1) \Gamma (m+2 \mu +1)},
\end{align*}  
together with \eqref{CombinationOflinearizationANDconnection}--\eqref{DefinitionLCombinationOflinearizationANDconnection}, we have
\begin{align*}
	 \mathsf{C}_{\gamma \gamma mm}^{(\mu,\mu)}  
	  & = \frac{1}{2} \left(  \mathsf{N}_{\gamma}^{(\mu,\mu)}  \right)^2 \left(  \mathsf{N}_{m}^{(\mu,\mu)} \right)^2 \int_{-1}^{1} (1-x)^{2\mu} (1+x)^{2\mu} \left(  P_{\gamma}^{(\mu,\mu)}(x) \right)^2  \left(  P_{m}^{(\mu,\mu)}(x) \right)^2 dx \\
	   & = \frac{1}{2} \left( w_{\gamma}^{(\mu)} \mathsf{N}_{\gamma}^{(\mu,\mu)}\right)^2 \left( w_{m}^{(\mu)} \mathsf{N}_{m}^{(\mu,\mu)}  \right)^2 \int_{-1}^{1} (1-x^2)^{2\mu}  \left(  C_{\gamma}^{\left(\mu +\frac{1}{2} \right)}(x) \right)^2  \left( C_{m}^{\left(\mu +\frac{1}{2} \right)}(x) \right)^2 dx \\
	   & = \frac{1}{2} \left( w_{\gamma}^{(\mu)} \mathsf{N}_{\gamma}^{(\mu,\mu)}\right)^2 \left( w_{m}^{(\mu)} \mathsf{N}_{m}^{(\mu,\mu)}  \right)^2 \sum_{\nu=0}^{\gamma} L_{\mu \gamma}(\nu)\sum_{\lambda=0}^{m} L_{\mu m}(\lambda)  \int_{-1}^{1} (1-x^2)^{2\mu} C_{2\nu}^{\left(2 \mu +\frac{1}{2} \right)}(x) C_{2\lambda}^{\left(2 \mu +\frac{1}{2} \right)}(x)dx.
\end{align*}
Now, the orthogonality of the Gegenbauer polynomials,
\begin{align*}
	 \int_{-1}^{1}C_{2\nu}^{\left(2 \mu +\frac{1}{2} \right)}(x) C_{2\lambda}^{\left(2 \mu +\frac{1}{2} \right)}(x) (1-x^2)^{2\mu} dx= \xi_{ \lambda}(\mu) \mathds{1}(\nu=\lambda),
\end{align*}
where
\begin{align*}
	\xi_{ \lambda}(\mu):=\frac{\pi  2^{1-4 \mu } \Gamma (2 \lambda +4 \mu +1)}{(4 \lambda +4 \mu +1) \Gamma (2 \lambda +1) \left( \Gamma \left(2 \mu +\frac{1}{2}\right) \right)^2},
\end{align*}
together with the fact that $0\leq  \gamma \leq m$, yields
\begin{align*}
	\mathsf{C}_{\gamma \gamma mm}^{(\mu,\mu)}  
	  & =  \frac{1}{2} \left( w_{\gamma}^{(\mu)} \mathsf{N}_{\gamma}^{(\mu,\mu)}\right)^2 \left( w_{m}^{(\mu)} \mathsf{N}_{m}^{(\mu,\mu)}  \right)^2 \sum_{\nu=0}^{\gamma} L_{\mu \gamma}(\nu)\sum_{\lambda=0}^{m} L_{\mu m}(\lambda) \xi_{ \lambda}(\mu) \mathds{1}(\nu=\lambda) \\
	    & =  \frac{1}{2} \left( w_{\gamma}^{(\mu)} \mathsf{N}_{\gamma}^{(\mu,\mu)}\right)^2 \left( w_{m}^{(\mu)} \mathsf{N}_{m}^{(\mu,\mu)}  \right)^2 \sum_{\nu=0}^{\gamma} L_{\mu \gamma}(\nu)\sum_{\lambda=0}^{\gamma} L_{\mu m}(\lambda) \xi_{ \lambda}(\mu) \mathds{1}(\nu=\lambda) \\
	     & =  \frac{1}{2} \left( w_{\gamma}^{(\mu)} \mathsf{N}_{\gamma}^{(\mu,\mu)}\right)^2 \left( w_{m}^{(\mu)} \mathsf{N}_{m}^{(\mu,\mu)}  \right)^2   \sum_{\lambda=0}^{\gamma} L_{\mu \gamma}(\lambda)L_{\mu m}(\lambda) \xi_{\lambda}(\mu)  \\ 
	    & =  \frac{1}{2}      \sum_{\lambda=0}^{\gamma} \mathsf{M}_{ \gamma}^{(\mu)} (\lambda) \mathsf{M}_{ m}^{(\mu)} (\lambda) \xi_{\lambda}(\mu),
\end{align*}
where we set
\begin{align*} 
	\mathsf{M}_{ m}^{(\mu)} (\lambda) &:=\left( w_{m}^{(\mu)} \mathsf{N}_{m}^{(\mu,\mu)}  \right)^2   L_{\mu m}(\lambda)  .
\end{align*}
Finally, a direct computation yields that 
\begin{align*} 
	\mathsf{M}_{ m}^{(\mu)} (\lambda) &=\frac{1}{2 \pi ^{3/2}} \frac{(4 \lambda +4 \mu +1) \Gamma \left(\lambda +\frac{1}{2}\right) \Gamma \left(2 \mu +\frac{1}{2}\right) \Gamma \left(\lambda +\mu +\frac{1}{2}\right)}{\Gamma (\lambda +\mu +1) \Gamma (\lambda +2 \mu +1)} \cdot \\
	& \cdot \frac{(2 \mu +2 m+1) \Gamma \left(m-\lambda +\frac{1}{2}\right) \Gamma (m+\lambda +2 \mu +1)}{\Gamma (m-\lambda +1) \Gamma \left(m+\lambda +2 \mu +\frac{3}{2}\right)}
\end{align*}
that completes the proof. 
\end{proof}
Next, we show that the closed formulas we derived above are in fact monotone with respect to $m$.
\begin{lemma}[Monotonicity of $\mathsf{M}_{m}^{(\mu)}(\lambda)$]\label{LemmaMonotonicityA}
	Let $\gamma \geq 0$, $m \geq \gamma$ and $0\leq  \lambda \leq \gamma$ be any integers. Then, the function $\mathsf{M}_{m}^{(\mu)}(\lambda)$ defined in Lemma \ref{LemmaUniformAsymptoticFormulasModel2CH} is decreasing with respect to $m$.
\end{lemma}

\begin{proof}
Let $\gamma \geq 0$, $m \geq \gamma$ and $0\leq  \lambda \leq \gamma$ be any integers. The claim follows immediately by computing the difference  $\mathsf{M}_{ m+1}^{(\mu)} (\lambda)-\mathsf{M}_{ m}^{(\mu)} (\lambda)$. Indeed,  the identity for the ration of two Gamma functions,  $\Gamma(x+1)=x\Gamma(x)$, valid for all $x\in \mathbb{R}$, yields that $\mathsf{M}_{ m+1}^{(\mu)} (\lambda)-\mathsf{M}_{ m}^{(\mu)} (\lambda)$ equals to
\begin{align*}
 	-\frac{(4 \lambda +4 \mu +1) \Gamma \left(\lambda +\frac{1}{2}\right) \Gamma \left(2 \mu +\frac{1}{2}\right) \Gamma \left(\lambda +\mu +\frac{3}{2}\right) \Gamma \left(m-\lambda +\frac{1}{2}\right) \Gamma (m+\lambda +2 \mu +1)}{\pi ^{3/2} \Gamma (\lambda +\mu ) \Gamma (\lambda +2 \mu +1) \Gamma (m-\lambda +2) \Gamma \left(m+\lambda +2 \mu +\frac{5}{2}\right)}
\end{align*}
that is strictly negative for all $m  $, $\lambda  $ and $\mu$, and hence $\mathsf{M}_{ m}^{(\mu)} (\lambda)$ is decreasing with respect to $m$, for all $\lambda$ and $\mu$, that completes the proof.
\end{proof}

\begin{remark}[Closed formulas for $\mathsf{C}_{\gamma \gamma mm}^{(\mu,\mu)}$ for small values of $\gamma$]
	Finally, we note that one can use Lemma \ref{LemmaUniformAsymptoticFormulasModel2CH} to find closed formulas for the Fourier coefficients provided that $\gamma$ is sufficiently small. For example, for $\gamma \in \{0,1\}$, we find that 
\begin{align*}
	\mathsf{C}_{00 mm}^{(\mu,\mu)} &=\frac{1}{2 \pi }(2 \mu +1)\left(\frac{\Gamma \left(\mu +\frac{1}{2}\right)}{\Gamma (\mu +1)}\right)^2 \frac{ (2 \mu +2 m+1) \Gamma \left(m+\frac{1}{2}\right) \Gamma (m+2 \mu +1)}{ \Gamma (m+1) \Gamma \left(m+2 \mu +\frac{3}{2}\right)}, \\
	\mathsf{C}_{11 mm}^{(\mu,\mu)} &= \frac{1}{8 \pi }(\mu +1) (2 \mu +1) (2 \mu +3) \left(\frac{\Gamma \left(\mu +\frac{1}{2}\right)}{\Gamma (\mu +2)}\right)^2 \cdot \\
	& \cdot \frac{(2 \mu +2 m+1) (-\mu +2 m (2 \mu +m+1)-1) \Gamma \left(m-\frac{1}{2}\right) \Gamma (m+2 \mu +1)}{\Gamma (m+1) \Gamma \left(m+2 \mu +\frac{5}{2}\right)}.
\end{align*}	
Figure \ref{Lemma74PicRef} illustrates the Fourier coefficients $\mathsf{C}_{\gamma\gamma mm}^{(\mu,\mu)} $ for $\mu=30$ and $\gamma \in \{0,1,2\}$ respectively  as $m $ varies within $\{1,2,\dots,50\}$.
\begin{figure}[h!]
\vspace{0.5cm}
    \centering
    \includegraphics[width=0.7\textwidth]{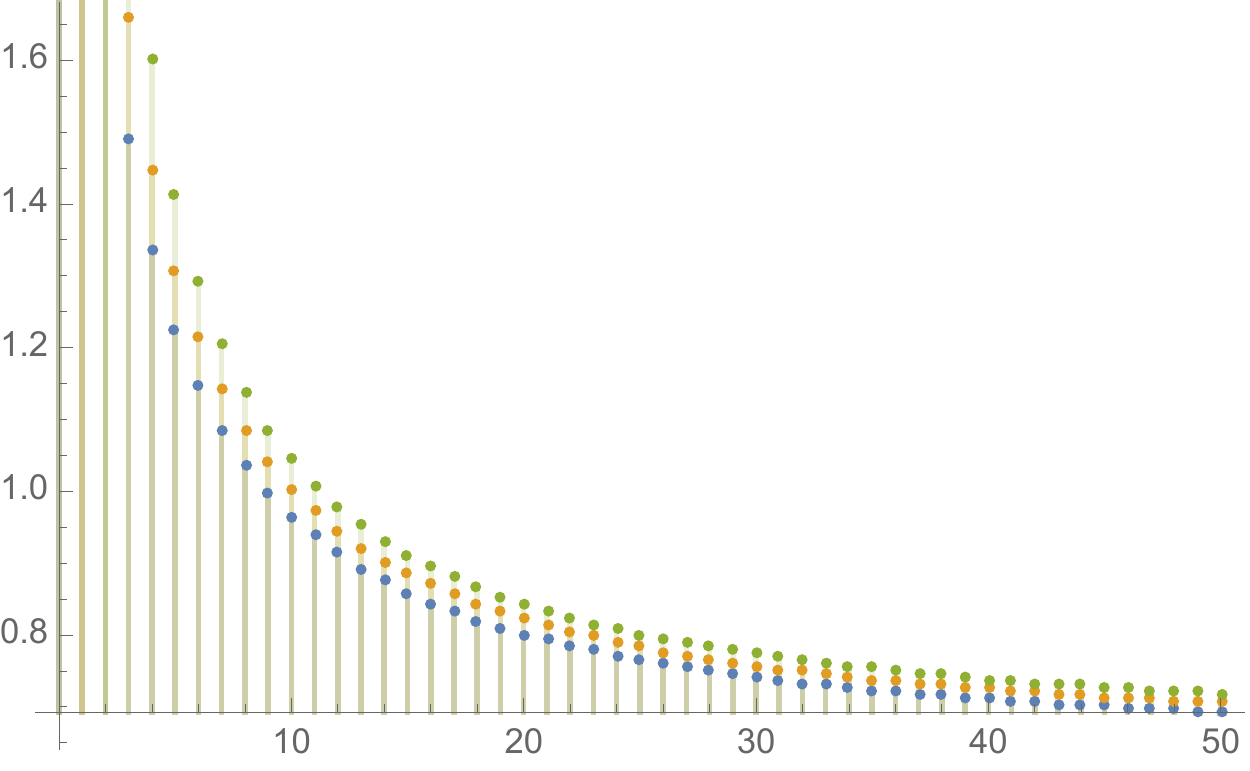}
    \caption{The Fourier coefficients $\mathsf{C}_{\gamma\gamma mm}^{(30,30)} $ for $\gamma =0$ (blue), $\gamma=1$ (orange) and $\gamma=2$ (green) as $m $ varies within the interval $[1,50]$. They are decreasing for $m\geq 2\gamma+1$.}
    \label{Lemma74PicRef}
\end{figure}
\end{remark}

\subsection{Yang--Mills equation in spherical symmetry}\label{SectionYangMills}
 In this case, the Fourier coefficients are given by \eqref{ExpnasionsFourierModel3YM}, that is
 \begin{align*}
		\overline{\mathfrak{C}}_{ijm} &: = \left[  \mathfrak{e}_i \mathfrak{e}_j \Big| \mathfrak{e}_m \right]  =  \int_{0}^{\pi }   \mathfrak{e}_{i}(x)\mathfrak{e}_{j}(x)\mathfrak{e}_{m}(x) \sin^4 (x)  dx \\
		 &=\prod_{\lambda_{1} \in \{i,j,m\}}  \mathfrak{N}_{\lambda_{1}}   \int_{-1}^{1 } \prod_{\lambda_{2} \in \{i,j,m\}} P_{\lambda_{2}}^{\left(\frac{3}{2},\frac{3}{2} \right)}(y)  \left(1-y^2 \right)^2 \frac{dy}{\sqrt{1-y^2}},\\
		\mathfrak{C}_{ijkm} &: = \left[ \sin^2 \cdot \mathfrak{e}_i \mathfrak{e}_j \mathfrak{e}_k    \Big| \mathfrak{e}_m \right]
		 =  \int_{0}^{\pi }  \mathfrak{e}_{i}(x)\mathfrak{e}_{j}(x)\mathfrak{e}_{k}(x)\mathfrak{e}_{m}(x)  \sin^6 (x)  dx \\
		 &= \prod_{\lambda_{1} \in \{i,j,k,m\}} \mathfrak{N}_{\lambda_{1}}   \int_{-1}^{1 } \prod_{\lambda_{2} \in \{i,j,k,m\}} P_{\lambda_{2}}^{\left(\frac{3}{2},\frac{3}{2} \right)}(y)  \left(1-y^2 \right)^3 \frac{dy}{\sqrt{1-y^2}},
	\end{align*}
where the normalization constant $\mathfrak{N}_{\lambda_{1}} $ is given by \eqref{NormalizationConstantModel3YM}. As before, we call \textit{resonant} a quadruple $(i,j,k,m)$ of indices such that 
\begin{align*}
	\ssomega_{i}  \pm \ssomega_{j} \pm \ssomega_{k}  \pm \ssomega_{m}  =0
\end{align*}
and study the Fourier coefficients on resonant indices.
\subsubsection{Vanishing Fourier coefficients}
Firstly, we show that the Fourier coefficients vanish on some resonant indices.
\begin{lemma}[Vanishing Fourier coefficients on resonant indices]\label{LemmaVanisginFourierModel3YM}
For any integers $i,j,k,m \geq 0$, we have
\begin{align*} 
	 \ssomega_{i} + \ssomega_{j}  - \ssomega_{m}   & = 0
	  \Longrightarrow  
	 \overline{\mathfrak{C}}_{ijm}=0     ,\\
	 \ssomega_{i} - \ssomega_{j}+  \ssomega_{m}  &  = 0
	  \Longrightarrow  
	   \overline{\mathfrak{C}}_{ijm}=0     ,\\
	 - \ssomega_{i} + \ssomega_{j}  + \ssomega_{m}  &  = 0 
	  \Longrightarrow  
	   \overline{\mathfrak{C}}_{ijm}=0 
\end{align*}
and
\begin{align*} 
	  \ssomega_{i} + \ssomega_{j} + \ssomega_{k} - \ssomega_{m} & = 0
	  \Longrightarrow 
	  \mathfrak{C}_{ijkm}=0     ,\\
	    \ssomega_{i} + \ssomega_{j} - \ssomega_{k} + \ssomega_{m} & = 0
	  \Longrightarrow 
	  \mathfrak{C}_{ijkm}=0     ,\\
	   \ssomega_{i} - \ssomega_{j} + \ssomega_{k} + \ssomega_{m} & = 0
	  \Longrightarrow 
	 \mathfrak{C}_{ijkm}=0     ,\\
	  - \ssomega_{i} + \ssomega_{j} + \ssomega_{k} + \ssomega_{m} & = 0
	  \Longrightarrow 
	 \mathfrak{C}_{ijkm} =0     .
\end{align*}
\end{lemma}
\begin{proof}
Let $i,j,k,m \geq 0$ be integers such that $\ssomega_{i} + \ssomega_{j} + \ssomega_{k} - \ssomega_{m}  = 0$. Then, $m=2 + i + j$ and according to the computation above, we have
\begin{align*}
		\overline{\mathfrak{C}}_{ijm} = \int_{-1}^{1} \overline{\mathfrak{R}}_{N}(y) P_{m}^{\left(\frac{3}{2},\frac{3}{2} \right)}(y) (1-y^2)^{\frac{3}{2}}  d y	,
\end{align*} 
where
\begin{align*}
	\overline{\mathfrak{R}}_{N}(y):= \prod_{\lambda_{1} \in \{i,j,m\}}  \mathfrak{N}_{\lambda_{1}} \prod_{\lambda_{2} \in \{i,j\}}     P_{\lambda_{2}}^{\left(\frac{3}{2},\frac{3}{2} \right)}(y)  
\end{align*}
is a polynomial of degree $N=i+j<m$ and hence the Fourier coefficient vanishes since $P_{m}^{\left(3/2,3/2 \right)}(y)$ forms an orthonormal and complete basis with respect to the weight $ (1-y^2)^{3/2}$. The other results now follow immediately using the symmetries of the Fourier coefficients with respect to $i,j,m$. The second result follows similarly. Let $i,j,k,m \geq 0$ be integers such that $\ssomega_{i} + \ssomega_{j} + \ssomega_{k} - \ssomega_{m} = 0$. Then, $m=i+j+k+4$ and according to the definition of the Fourier coefficients above, we have
\begin{align*}
		\mathfrak{C}_{ijkm} = \int_{-1}^{1} \mathfrak{R}_{N}(y) P_{m}^{\left(\frac{3}{2},\frac{3}{2} \right)}(y) (1-y^2)^{\frac{3}{2}}  d y,
\end{align*} 
where
\begin{align*}
	\mathfrak{R}_{N}(y):= (1-y^2) \prod_{\lambda_{1} \in \{i,j,k,m\}} \mathfrak{N}_{\lambda_{1}} \prod_{\lambda_{1} \in \{i,j,k\}}  P_{\lambda_{2}}^{\left(\frac{3}{2},\frac{3}{2} \right)}(y) 
\end{align*}
is a polynomial of degree $N=i+j+k+2<m$ and hence the Fourier coefficient vanishes since $P_{m}^{\left(3/2,3/2 \right)}(y)$ forms an orthonormal and complete basis with respect to the weight $ (1-y^2)^{3/2}$. The other results now follow immediately using the symmetries of the Fourier coefficients with respect to $i,j,k,m$.
\end{proof}

\subsubsection{Non--vanishing Fourier coefficients}
Next, we study the non--vanishing Fourier coefficients. In order to deal with these constants, one needs a computationally efficient formula as in the two previous cases. In the spherically symmetric case we consider here, the basis consists of the Jacobi polynomials with equal parameters and these are weighted Gegenbauer polynomials (equation 18.7.1 in \cite{MR2723248})
\begin{align*} 
	C_{n}^{(2)}(x):= \frac{(4)_{n}}{\left(\frac{5}{2} \right)_n}P_{n}^{\left(\frac{3}{2},\frac{3}{2}\right)}(x) =\frac{\Gamma \left(\frac{5}{2} \right)}{\Gamma \left(4 \right)} \frac{\Gamma \left(n+4 \right)}{\Gamma \left(n+\frac{5}{2} \right)}P_{n}^{\left(\frac{3}{2},\frac{3}{2}\right)}(x) = \frac{\sqrt{\pi }}{8} \frac{\Gamma \left(n+4 \right)}{\Gamma \left(n+\frac{5}{2} \right)}P_{n}^{\left(\frac{3}{2},\frac{3}{2}\right)}(x),
\end{align*}
for $n \in \{\gamma,m\}$ and $x\in [-1,1]$. The latter, together with the definition of the normalization constant $\mathfrak{N}_{n}$ from \eqref{NormalizationConstantModel3YM}, expresses the normalized Jacobi polynomials in terms of the normalized Gegenbauer polynomials as follows
 	\begin{align*} 
 	 \mathfrak{N}_{n} P_{n}^{\left(\frac{3}{2},\frac{3}{2} \right)}(x) =
 		\mathfrak{N}_{n} \frac{8}{\sqrt{\pi }} \frac{\Gamma \left(n+\frac{5}{2} \right)}{\Gamma \left(n+4 \right) } C_{n}^{(2)}(x)=
  \mathfrak{w}_{n} C_{n}^{(2)}(x) ,
 	\end{align*}
 	for all $n \in \{\gamma,m\}$, where the new normalization constant reads
 	\begin{align*}
 		\mathfrak{w}_{n}:=\sqrt{\frac{8}{\pi }}	\frac{1}{\sqrt{(n +1) (n +3)}}.
 	\end{align*} 
 	Consequently, the Fourier coefficients can be written in terms of the Gegenbauer polynomials as follows
\begin{align*}
		\overline{ \mathfrak{C}}_{ijm} &= \mathfrak{w}_{i}\mathfrak{w}_{j}\mathfrak{w}_{m} \int_{-1}^{1 } C_{i}^{(2)}(y) C_{j}^{(2)}(y) C_{m}^{(2)}(y) \left(1-y^2 \right)^{\frac{3}{2}} dy, \\
		\mathfrak{C}_{ijkm} & =\mathfrak{w}_{i}\mathfrak{w}_{j}\mathfrak{w}_{k}\mathfrak{w}_{m} \int_{-1}^{1 }C_{i}^{(2)}(y) C_{j}^{(2)}(y)C_{k}^{(2)}(y) C_{m}^{(2)}(y) \left(1-y^2 \right)^{\frac{5}{2}}dy.
		\end{align*} 
Then, we derive closed formulas for the non--vanishing Fourier coefficients for a resonant quadruple $(i,j,k,m)$ as follows: 	
\begin{itemize}
	\item Whenever a Gegenbauer polynomial has an order $\mu$ such that the weight $(1-y^2)^{\mu-1/2}$ (with respect to which it forms an orthonormal basis) does not coincide with the weights $(1-y^2)^{p}$ with $p\in \{3/2,5/2\}$ (that define the integrals above), we use the connection formula (18.18.16	in \cite{MR2723248})
	\begin{align}
 C^{(\mu)}_{n}\left(y\right) &=\sum_{\ell=0}^{\left\lfloor n/2\right\rfloor}\alpha_{n\mu\lambda}(\ell)C^{(\lambda)}%
_{n-2\ell}\left(y\right),\label{ConnectionFormula} 
	\end{align}   
	where the coefficients are given by
	\begin{align}
		\alpha_{n\mu\lambda}(\ell):=&\frac%
{\lambda+n-2\ell}{\lambda}\frac{{\left(\mu\right)_{n-\ell}}}{{\left(\lambda+1%
\right)_{n-\ell}}}\frac{{\left(\mu-\lambda\right)_{\ell}}}{\ell!}  \label{DefinitionAdditionCoefficient}.
	\end{align}
	In particular, we are going to use this only for $\lambda=\mu+1$. In this case, the latter is equivalent to the recurrence relation (18.9.7 in \cite{MR2723248})
	\[C^{(\mu)}_{n}\left(x\right)=\frac{\mu}{n+\mu}\left(C^{(\mu+1)}_{n}%
\left(x\right)-C^{(\mu+1)}_{n-2}\left(x\right)\right),\]
valid for all integers $\mu\geq 0$ and $n\geq 2$.
	\item Whenever a Gegenbauer polynomial is multiplied by itself, we use the addition formula (equation (19) in \cite{MR1820610} and 18.18.22 in \cite{MR2723248})
	\begin{align}
 \left( C^{(\lambda)}_{n}\left(y\right) \right)^2 &=\sum_{\ell=0}^{%
n}\beta_{n\lambda}(\ell)C^{(%
\lambda)}_{2\ell}\left(y\right),\label{AdditionFormulaYM}
	\end{align}
	where the coefficients are given by
	\begin{align}\label{DefinitionConnectionCoefficient2}
		\beta_{n\lambda}(\ell):=&\frac{1}{n!}\binom{n}{\ell} \frac{(2\ell)! (\lambda)_{\ell} (\lambda)_{n-\ell} (2\ell+2\lambda)_{n-\ell}}{\ell! (\ell+\lambda)_{\ell} (2\ell+\lambda+1)_{n-\ell}} .
	\end{align}
	\item Whenever two Gegenbauer polynomials of different degrees but of same order are multiplied, we use the addition formula (18.18.22 in \cite{MR2723248})
	\begin{align}
 C^{(\lambda)}_{m}\left(y\right)C^{(\lambda)}_{n}\left(y\right) &=\sum_{\ell=0}^{%
\min(m,n)} \zeta_{mn \lambda}(\ell) C^{(%
\lambda)}_{m+n-2\ell}\left(y\right), \label{LinearizationFormulaYM}
	\end{align}
	where the coefficients are given by
	\begin{align}
		\zeta_{mn \lambda}(\ell):=&\frac{(m+n+\lambda-2\ell)(m+n-2\ell)!}{(m+n+\lambda-\ell)\ell!\,(m-%
\ell)!\,(n-\ell)!}\*\frac{{\left(\lambda\right)_{\ell}}{\left(\lambda\right)_{%
m-\ell}}{\left(\lambda\right)_{n-\ell}}{\left(2\lambda\right)_{m+n-\ell}}}{{%
\left(\lambda\right)_{m+n-\ell}}{\left(2\lambda\right)_{m+n-2\ell}}}.\label{DefinitionConnectionCoefficient3}
	\end{align}
	\item Whenever a product of a monomial with a Gegenbauer polynomial is integrated, we use the formula (equation 18.17.37 in \cite{MR2723248})  
\begin{align}\label{DifficultIntegralGegenbauerPolynomialsYM}
	\int_{0}^{1}x^{z%
-1}C^{(\lambda)}_{n}\left(x\right)(1-x^{2})^{\lambda-\frac{1}{2}} dx=\frac{\pi\,2^{1-2\lambda-z}\Gamma\left(n+2\lambda\right)\Gamma%
\left(z\right)}{n!\Gamma\left(\lambda\right)\Gamma\left(\frac{1}{2}+\frac{1}{2%
}n+\lambda+\frac{1}{2}z\right)\Gamma\left(\frac{1}{2}+\frac{1}{2}z-\frac{1}{2}%
n\right)},
\end{align}
valid of all integers $\lambda \geq 0$ and real numbers $z>0$. Notice that this is the Mellin transform of the function $C^{(\lambda)}_{n}\left(x\right)(1-x^{2})^{\lambda-\frac{1}{2}}$ restricted to $[0,1]$.
\end{itemize}		
In particular, we will only need to study
\begin{itemize}
	\item $\overline{\mathfrak{C}}_{\gamma,\gamma,2\tau}$, for all integers , $\tau \in \{0,1,\dots,\gamma\}$,
	\item  $\overline{\mathfrak{C}}_{\gamma, 2\tau +m-\gamma  , m}$, for all integers $\tau \in \{0,1,\dots,\gamma\}$, $m\geq 2\gamma+1$, 
	\item   $\overline{\mathfrak{C}}_{m ,2\tau , m}$, for all integers $\tau \in \{0,1,\dots,\gamma\}$, $m\geq 2\gamma+1$, 
	\item $ \mathfrak{C}_{\gamma \gamma mm}$ for all integers $m \geq 2\gamma+1$.
\end{itemize}
To begin with, we focus on $\overline{\mathfrak{C}}_{ijm}$ for any integers $i,j,m \geq 0$ and establish the follow result. 
\begin{lemma}[Non--vanishing Fourier coefficients $\overline{\mathfrak{C}}_{ijm}$ on resonant indices: closed formula]\label{LemmaClosedFormulasModel3YMCbarijm}
For any integers $i,j,m \geq 0$, we have
		\begin{align*} 
		\overline{\mathfrak{C}}_{ijm} &=\frac{(i+j-m+2) (i-j+m+2) (-i+j+m+2) (i+j+m+6)}{4 \sqrt{2 \pi  (i+1) (i+3) (j+1) (j+3) (m+1) (m+3)}}\cdot \\
		& \cdot \mathds{1}\left( | j-m| \leq i\leq j+m \right) \mathds{1}\left( | i-m| \leq j\leq i+m \right) \mathds{1}\left( | i-j| \leq m\leq i+j \right)  \\
		&\cdot \mathds{1}(j+m-i\in 2\mathbb{N} \cup \{0\}) \mathds{1}(i+m-j\in 2\mathbb{N} \cup \{0\}) \mathds{1}(i+j-m\in 2\mathbb{N} \cup \{0\}).
		\end{align*}     
\end{lemma}
\begin{proof}
The result follows immediately from \eqref{LinearizationFormulaYM}--\eqref{DefinitionConnectionCoefficient3} together with the orthogonality of the Gegenbauer polynomials,
\begin{align*}
	\int_{-1}^{1 }C^{(
2)}_{n}\left(y\right) C_{m}^{(2)}(y) \left(1-y^2 \right)^{\frac{3}{2}} 
dy =\frac{\pi }{8} (m+1) (m+3) \mathds{1}(m=n),
\end{align*}
for all integers $m,n \geq 0$. Indeed, for any integers $i,j,m \geq 0$, we have
\begin{align*}
	\overline{ \mathfrak{C}}_{ijm} &= \mathfrak{w}_{i}\mathfrak{w}_{j}\mathfrak{w}_{m} \int_{-1}^{1 } C_{i}^{(2)}(y) C_{j}^{(2)}(y) C_{m}^{(2)}(y) \left(1-y^2 \right)^{\frac{3}{2}} dy \\
	& =\mathfrak{w}_{i}\mathfrak{w}_{j}\mathfrak{w}_{m} \sum_{\ell=0}^{
\min(i,j)} \zeta_{ij2}(\ell)  \int_{-1}^{1 }C^{(
2)}_{i+j-2\ell}\left(y\right) C_{m}^{(2)}(y) \left(1-y^2 \right)^{\frac{3}{2}} 
dy \\
& =  \frac{\pi }{8} (m+1) (m+3) \mathfrak{w}_{i}\mathfrak{w}_{j}\mathfrak{w}_{m}\sum_{\ell=0}^{
\min(i,j)} \zeta_{ij2}(\ell) \mathds{1}(2\ell=i+j-m) .
\end{align*}
On the one hand, for all integers $i$, $j$ and $m$ such that $i+j-m \notin 2\mathbb{N} \cup \{0\}$, the Fourier coefficient vanishes. Furthermore, we have
\begin{align*}
	0\leq \frac{i+j-m}{2}\leq \min(i,j) \Longleftrightarrow |i-j|\leq m\leq i+j.
\end{align*}
Consequently, for all integers $i$, $j$ and $m$ such that the condition $|i-j|\leq m\leq i+j$ is not fulfilled, the Fourier coefficient vanishes. On the other hand, for all $i$, $j$ and $m$ such that both $i+j-m \in 2\mathbb{N} \cup \{0\}$ and  $|i-j|\leq m\leq i+j$ hold true, we compute
\begin{align*} 
\overline{ \mathfrak{C}}_{ijm} &=  \frac{\pi }{8} (m+1) (m+3) \mathfrak{w}_{i}\mathfrak{w}_{j}\mathfrak{w}_{m}   \sum_{\ell=0}^{
\min(i,j)} \zeta_{ij2}(\ell) \mathds{1}(2\ell=i+j-m)\\
& =  \frac{\pi }{8} (m+1) (m+3) \mathfrak{w}_{i}\mathfrak{w}_{j}\mathfrak{w}_{m}   \zeta_{ij2} \left(\frac{i+j-m}{2} \right) \sum_{\ell=0}^{
\min(i,j)} \mathds{1}(2\ell=i+j-m) \\
&= \frac{(i+j-m+2) (i-j+m+2) (-i+j+m+2) (i+j+m+6)}{4 \sqrt{2 \pi  (i+1) (i+3) (j+1) (j+3) (m+1) (m+3)}}  \sum _{l=0}^{\min (i,j)} \mathds{1}(i+j-m=2 \ell) \\
&= \frac{(i+j-m+2) (i-j+m+2) (-i+j+m+2) (i+j+m+6)}{4 \sqrt{2 \pi  (i+1) (i+3) (j+1) (j+3) (m+1) (m+3)}},
\end{align*}
where we used the fact that
\begin{align*}
	\sum _{l=0}^{\min (i,j)} \mathds{1}(i+j-m=2 \ell) =1,
\end{align*}
for all such $i$, $j$ and $m$. Finally, using the symmetries of the Fourier coefficient with respect to $i$, $j$ and $m$ completes the proof.
\end{proof}
 
Next, we apply the previous result to obtain closed formulas for the Fourier coefficient $\overline{\mathfrak{C}}_{ijm}$ on the particular resonant indices we are interested in. Specifically, we establish the following result.

\begin{lemma}[Non--vanishing Fourier coefficients $\overline{\mathfrak{C}}_{ijm}$ on particular resonant indices: closed formulas]\label{LemmaClosedFormulasModel3YMCbarijmSpecificijm}
Let $\gamma,\tau,m$ be integers such that $\gamma \geq 0$, $\tau \in \{0,1,\dots, \gamma\}$ and $m\geq 2\gamma+1$. Then, we have
	\begin{align*}
	\overline{\mathfrak{C}}_{\gamma,\gamma,2\tau} &=
	2 \sqrt{\frac{2}{\pi }} \frac{(\tau +1)^2 (\gamma -\tau +1) (\gamma +\tau +3)}{(\gamma +1) (\gamma +3) \sqrt{4 \tau  (\tau +2)+3}}, \\
	\overline{\mathfrak{C}}_{m ,2\tau , m} & =2 \sqrt{\frac{2}{\pi }} \frac{ (\tau +1)^2 (m-\tau +1) (m+\tau +3)}{(m+1) (m+3) \sqrt{4 \tau  (\tau +2)+3}}, \\
	\overline{\mathfrak{C}}_{\gamma, 2\tau +m-\gamma  , m}&=2 \sqrt{\frac{2}{\pi }}  \frac{(\tau +1) (\gamma -\tau +1) (m+\tau +3) (-\gamma +m+\tau +1)}{\sqrt{(\gamma +1) (\gamma +3) (m+1) (m+3) (-\gamma +m+2 \tau +1) (-\gamma +m+2 \tau +3)}}.
\end{align*}
\end{lemma}
\begin{proof}
	Let $\gamma,\tau,m$ be integers such that $\gamma \geq 0$, $\tau \in \{0,1,\dots, \gamma\}$ and $m\geq 2\gamma+1$. Firstly, notice that all the indices of the Fourier coefficients above satisfy all the conditions in the Booleans in Lemma \ref{LemmaClosedFormulasModel3YMCbarijm}. Then, the result follows immediately by Lemma \ref{LemmaClosedFormulasModel3YMCbarijm} by direct substitution. 
\end{proof}

Now, we focus on $ \mathfrak{C}_{\gamma \gamma mm}$ and derive the following result.
\begin{lemma}[Non--vanishing Fourier coefficients $\mathfrak{C}_{\gamma \gamma mm}$ on resonant indices: closed formulas]\label{LemmaClosedFormulasModel3YMCgammagammamm}
Let $\gamma \geq 0$ be a fixed integer. Then, for all $m \geq 2\gamma+1$, we have
		\begin{align*} 
		\mathfrak{C}_{\gamma \gamma mm}  =  \mathfrak{w}_{\gamma}^2 \mathfrak{w}_{m}^2  \sum_{\ell_2=0}^{\gamma}\sum_{\nu_2=0}^{\ell_2} \delta_{\gamma}(\ell_2,\nu_2)J_{m}(\ell_2,\nu_2)  
		\end{align*} 
		where
		 \begin{align*} 
		 \delta_{\gamma}(\ell_2,\nu_2)&:=  	\frac{( \ell_2 +1)^2 (-1)^{ \nu_2 } (\gamma - \ell_2 +1) (\gamma + \ell_2 +3) 2^{2 ( \ell_2 - \nu_2 )} \Gamma (2  \ell_2 - \nu_2 +2)}{(4  \ell_2  ( \ell_2 +2)+3) \Gamma ( \nu_2 +1) \Gamma (2  \ell_2 -2  \nu_2 +1)}, \\
		J_{m}(\ell_2,\nu_2) & =
		\frac{3 \sqrt{\pi } (5  \ell_2  (3 m (m+4)+1)+m (m+4) (4-15  \nu_2 )-5 ( \nu_2 -4)) \Gamma \left( \ell_2 - \nu_2 +\frac{1}{2}\right)}{8 \Gamma ( \ell_2 - \nu_2 +5)}   \\
		& +\sum_{\ell_1=2}^{\ell_2-\nu_2+1} \frac{\pi  \ell_1 \left(\ell_1+1\right)^2 \left(2 \ell_1-1\right) 4^{-\ell_2+\nu _2-2} \left(\ell_1-m-1\right) \left(\ell_1+m+3\right) \Gamma \left(2 \ell_2-2 \nu _2+1\right)}{\Gamma \left(-\ell_1+\ell_2-\nu _2+2\right) \Gamma \left(\ell_1+\ell_2-\nu _2+3\right)} \\
		& - \sum_{\ell_1=2}^{\ell_2-\nu_2}  \frac{\pi  \left(\ell_1+1\right)^2 \left(\ell_1+2\right) \left(2 \ell_1+5\right) 4^{-\ell_2+\nu _2-2} \left(\ell_1-m-1\right) \left(\ell_1+m+3\right) \Gamma \left(2 \ell_2-2 \nu _2+1\right)}{\Gamma \left(-\ell_1+\ell_2-\nu _2+1\right) \Gamma \left(\ell_1+\ell_2-\nu _2+4\right)}.
		\end{align*}   
\end{lemma}
\begin{proof}
Let $\gamma \geq 0$ be a fixed integer and pick any integer $m \geq 2\gamma+1$. Then, by the definition of the Fourier coefficients, we have 
 		\begin{align*} 
		\mathfrak{C}_{\gamma \gamma mm} & =\mathfrak{w}_{\gamma}^2 \mathfrak{w}_{m}^2  \int_{-1}^{1 } \left(C_{\gamma}^{(2)}(y)\right)^2 \left(C_{m}^{(2)}(y)\right)^2 \left(1-y^2 \right)^{\frac{5}{2}}dy.
		\end{align*} 
		On the one hand, we use the linearization formula \eqref{AdditionFormulaYM}--\eqref{DefinitionConnectionCoefficient2} together with the special cases $C_{0}^{(2)}(y)=1$ and $C_{2}^{(2)}(y)=12 y^2-2 $ to obtain
		\begin{align*} 
			\left(C_{m}^{(2)}(y)\right)^2 &= \sum_{\ell_1=0}^{m} \beta_{m2}(\ell_1)C_{2\ell_1}^{(2)}(y) \\
			&=\sum_{\ell_1=0}^{m} \frac{\left(\ell_1+1\right)^2 \left(-\ell_1+m+1\right) \left(\ell_1+m+3\right)}{4 \ell_1 \left(\ell_1+2\right)+3} C_{2\ell_1}^{(2)}(y), \\ 
			&= \frac{1}{3} (m+1) (m+3)  + \frac{4}{15} m (m+4) \left(12 y^2-2 \right)  \\
			&+ \sum_{\ell_1=2}^{m} \frac{\left(\ell_1+1\right)^2 \left(-\ell_1+m+1\right) \left(\ell_1+m+3\right)}{4 \ell_1 \left(\ell_1+2\right)+3} C_{2\ell_1}^{(2)}(y).   
		\end{align*}
		Furthermore, for all $\ell_1 \geq 2$, the connection formula \eqref{ConnectionFormula}--\eqref{DefinitionAdditionCoefficient} yields 
		\begin{align*} 
		C_{2\ell_1}^{(2)}(y) &=	\sum_{\nu_1=0}^{\ell_1}\alpha_{2\ell_1,2,3}(\nu_1) 
			C_{2(\ell_1-\nu_1)}^{(3)}(y)  \\
			& =\sum_{\nu_1=0}^{\ell_1} \frac{(2  \ell_1 -2  \nu_1 +3) (-1)_{ \nu_1 } \Gamma (2  \ell_1 - \nu_1 +2)}{3  \nu_1 ! (4)_{2  \ell_1 - \nu_1 }} C_{2(\ell_1-\nu_1)}^{(3)}(y) \\
			& = \frac{(2  \ell_1 +3) (2)_{2  \ell_1 }}{3 (4)_{2  \ell_1 }}C_{2\ell_1}^{(3)}(y) -\frac{(2  \ell_1 +1) (2)_{2  \ell_1 -1}}{3 (4)_{2  \ell_1 -1}} C_{2(\ell_1-1)}^{(3)}(y) \\
			& = \frac{1}{ \ell_1 +1} \left(C_{2\ell_1}^{(3)}(y) - C_{2(\ell_1-1)}^{(3)}(y) \right),
		\end{align*}
		since $(-1)_{ \nu_1 }=0$ for all $\nu_1 \geq 2$. Consequently, we have
		\begin{align*}
			\left(C_{m}^{(2)}(y)\right)^2 &=\frac{1}{3} (m+1) (m+3)  + \frac{4}{15} m (m+4) \left(12 y^2-2 \right)  \\
			&+ \sum_{\ell_1=2}^{m}\frac{ ( \ell_1 +1) ( \ell_1 -m-1) ( \ell_1 +m+3) }{4  \ell_1  ( \ell_1 +2)+3} \left( C_{2(\ell_1-1)}^{(3)}(y)-C_{2\ell_1}^{(3)}(y) \right).
		\end{align*} 
		On the other hand, the linearization formula \eqref{AdditionFormulaYM}--\eqref{DefinitionConnectionCoefficient2} together with the definition of the Gegenbauer polynomials, 
		\begin{align*}
			C_{2\ell_2}^{(2)}(y)= \sum_{\nu_2=0}^{\ell_2} d_{\ell_2}(\nu_2) y^{2(\ell_2-\nu_2)}, \Hquad
			 d_{\ell_2}(\nu_2) := \frac{(-1)^{ \nu_2 } \left(4^{ \ell_2 - \nu_2 } \Gamma (2  \ell_2 - \nu_2 +2)\right)}{\Gamma ( \nu_2 +1) \Gamma (2  \ell_2 -2  \nu_2 +1)} 
		\end{align*}
		yield
		\begin{align*}
			\left(C_{\gamma}^{(2)}(y)\right)^2 &= \sum_{\ell_2=0}^{\gamma} \beta_{\gamma2}(\ell_2)C_{2\ell_2}^{(2)}(y) =
			\sum_{\ell_2=0}^{\gamma} \beta_{\gamma2}(\ell_2)\sum_{\nu_2=0}^{\ell_2} d_{\ell_2}(\nu_2) y^{2(\ell_2-\nu_2)}  = \sum_{\ell_2=0}^{\gamma}\sum_{\nu_2=0}^{\ell_2} \delta_{\gamma}(\ell_2,\nu_2) y^{2(\ell_2-\nu_2)},
		\end{align*}
		where we set
		\begin{align*}
		\delta_{\gamma}(\ell_2,\nu_2):= \beta_{\gamma2}(\ell_2)d_{\ell_2}(\nu_2) = 	\frac{( \ell_2 +1)^2 (-1)^{ \nu_2 } (\gamma - \ell_2 +1) (\gamma + \ell_2 +3) 2^{2 ( \ell_2 - \nu_2 )} \Gamma (2  \ell_2 - \nu_2 +2)}{(4  \ell_2  ( \ell_2 +2)+3) \Gamma ( \nu_2 +1) \Gamma (2  \ell_2 -2  \nu_2 +1)}.
		\end{align*}
		Now, putting all together, we infer
		\begin{align*}
			\mathfrak{C}_{\gamma \gamma mm} & = \mathfrak{w}_{\gamma}^2 \mathfrak{w}_{m}^2  \int_{-1}^{1 } \left(C_{\gamma}^{(2)}(y)\right)^2 \left(C_{m}^{(2)}(y)\right)^2 \left(1-y^2 \right)^{\frac{5}{2}}dy  \\ 
			& =\mathfrak{w}_{\gamma}^2 \mathfrak{w}_{m}^2  \sum_{\ell_2=0}^{\gamma}\sum_{\nu_2=0}^{\ell_2} \delta_{\gamma}(\ell_2,\nu_2)\int_{-1}^{1 }  y^{2(\ell_2-\nu_2)} \Bigg[
			\frac{1}{3} (m+1) (m+3)  + \frac{4}{15} m (m+4) \left(12 y^2-2 \right)  \\
			&+ \sum_{\ell_1=2}^{m}\frac{ ( \ell_1 +1) ( \ell_1 -m-1) ( \ell_1 +m+3) }{4  \ell_1  ( \ell_1 +2)+3} \left( C_{2(\ell_1-1)}^{(3)}(y)-C_{2\ell_1}^{(3)}(y) \right)
			\Bigg]  \left(1-y^2 \right)^{\frac{5}{2}}  dy \\
			& =\mathfrak{w}_{\gamma}^2 \mathfrak{w}_{m}^2  \sum_{\ell_2=0}^{\gamma}\sum_{\nu_2=0}^{\ell_2} \delta_{\gamma}(\ell_2,\nu_2)J_{m}(\ell_2,\nu_2), 
		\end{align*}
		where
		\begin{align*}
			J_{m}(\ell_2,\nu_2)&:= 
			\frac{1}{3} (m+1) (m+3) \int_{-1}^{1 }  y^{2(\ell_2-\nu_2)} \left(1-y^2 \right)^{\frac{5}{2}}dy \\
			&+ \frac{4}{15} m (m+4) \int_{-1}^{1 }  y^{2(\ell_2-\nu_2)} \left(12 y^2-2 \right) \left(1-y^2 \right)^{\frac{5}{2}}dy \\
			& + \sum_{\ell_1=2}^{m}\frac{ ( \ell_1 +1) ( \ell_1 -m-1) ( \ell_1 +m+3) }{4  \ell_1  ( \ell_1 +2)+3} \int_{-1}^{1 }  y^{2(\ell_2-\nu_2)}   C_{2(\ell_1-1)}^{(3)}(y) \left(1-y^2 \right)^{\frac{5}{2}}  dy \\
			& - \sum_{\ell_1=2}^{m}\frac{ ( \ell_1 +1) ( \ell_1 -m-1) ( \ell_1 +m+3) }{4  \ell_1  ( \ell_1 +2)+3} \int_{-1}^{1 }  y^{2(\ell_2-\nu_2)} C_{2\ell_1}^{(3)}(y)   \left(1-y^2 \right)^{\frac{5}{2}}  dy.
		\end{align*}
		We compute
		\begin{align*}
			\int_{-1}^{1 }  y^{2(\ell_2-\nu_2)} \left(1-y^2 \right)^{\frac{5}{2}}dy & =\frac{15 \sqrt{\pi } \Gamma \left(\ell_2 -\nu_2 +\frac{1}{2}\right)}{8 \Gamma (\ell_2 -\nu_2 +4)}, \\
			\int_{-1}^{1 }  y^{2(\ell_2-\nu_2)} \left(12 y^2-2 \right) \left(1-y^2 \right)^{\frac{5}{2}}dy &=\frac{15 \sqrt{\pi } (5 \ell_2  -5 \nu_2  -1) \Gamma \left(\ell_2  -\nu_2  +\frac{1}{2}\right)}{4 \Gamma (\ell_2  -\nu_2  +5)}.
		\end{align*}
		On the one hand, for all $\ell_1 \geq 2$ with $\ell_1>\ell_2-\nu_2+1 $, we have $2(\ell_2-\nu_2)<2(\ell_1-1)$ and hence
\begin{align*}
	\int_{-1}^{1 }  y^{2(\ell_2-\nu_2)}   C_{2(\ell_1-1)}^{(3)}(y) \left(1-y^2 \right)^{\frac{5}{2}}  dy=0
\end{align*}
since the Gegenbauer polynomial in the integrand forms an orthonormal and complete basis with respect to the weight $(1-y^2)^{5/2}$. On the other hand, for all $2\leq \ell_1 \leq \ell_2-\nu_2+1$, the identity 
\begin{align*}
	C_{2\lambda}^{\left(3 \right)}(-y)=C_{2\lambda}^{\left(3 \right)}(y),
\end{align*}
valid for all real $y\in [-1,1]$ and integers $\lambda \geq 0$, yields 
\begin{align*}
	& \int_{-1}^{1 }  y^{2(\ell_2-\nu_2)}   C_{2(\ell_1-1)}^{(3)}(y) \left(1-y^2 \right)^{\frac{5}{2}}  dy =   2\int_{0}^{1 }  y^{2(\ell_2-\nu_2)}   C_{2(\ell_1-1)}^{(3)}(y) \left(1-y^2 \right)^{\frac{5}{2}}  dy  \\
	& = \frac{\pi  4^{-\ell_2+\nu _2-3} \Gamma \left(2 \ell_1+4\right) \Gamma \left(2 \ell_2-2 \nu _2+1\right)}{\Gamma \left(2 \ell_1-1\right) \Gamma \left(-\ell_1+\ell_2-\nu _2+2\right) \Gamma \left(\ell_1+\ell_2-\nu _2+3\right)},
\end{align*}
where we used \eqref{DifficultIntegralGegenbauerPolynomialsYM} to compute the last integral. In other words, we have
\begin{align*}
	& \sum_{\ell_1=2}^{m}\frac{ ( \ell_1 +1) ( \ell_1 -m-1) ( \ell_1 +m+3) }{4  \ell_1  ( \ell_1 +2)+3}  \int_{-1}^{1 }  y^{2(\ell_2-\nu_2)}   C_{2(\ell_1-1)}^{(3)}(y) \left(1-y^2 \right)^{\frac{5}{2}}  dy  = \\
	& \sum_{\ell_1=2}^{\ell_2-\nu_2+1}\frac{ ( \ell_1 +1) ( \ell_1 -m-1) ( \ell_1 +m+3) }{4  \ell_1  ( \ell_1 +2)+3}  \int_{-1}^{1 }  y^{2(\ell_2-\nu_2)}   C_{2(\ell_1-1)}^{(3)}(y) \left(1-y^2 \right)^{\frac{5}{2}}  dy  = \\
	& \sum_{\ell_1=2}^{\ell_2-\nu_2+1} \frac{\pi  \ell_1 \left(\ell_1+1\right)^2 \left(2 \ell_1-1\right) 4^{-\ell_2+\nu _2-2} \left(\ell_1-m-1\right) \left(\ell_1+m+3\right) \Gamma \left(2 \ell_2-2 \nu _2+1\right)}{\Gamma \left(-\ell_1+\ell_2-\nu _2+2\right) \Gamma \left(\ell_1+\ell_2-\nu _2+3\right)}.
\end{align*}
Similarly, on the one hand, for all $\ell_1 \geq 2$ with $\ell_1>\ell_2-\nu_2 $, we have $2(\ell_2-\nu_2)<2\ell_1$ and hence
\begin{align*}
	\int_{-1}^{1 }  y^{2(\ell_2-\nu_2)}   C_{2\ell_1}^{(3)}(y) \left(1-y^2 \right)^{\frac{5}{2}}  dy=0
\end{align*}
since the Gegenbauer polynomial in the integrand forms an orthonormal and complete basis with respect to the weight $(1-y^2)^{5/2}$. On the other hand, for all $2\leq \ell_1 \leq \ell_2-\nu_2 $, the identity 
\begin{align*}
	C_{2\lambda}^{\left(3 \right)}(-y)=C_{2\lambda}^{\left(3 \right)}(y),
\end{align*}
valid for all real $y\in [-1,1]$ and integers $\lambda \geq 0$, yields 
\begin{align*}
	& \int_{-1}^{1 }  y^{2(\ell_2-\nu_2)}   C_{2\ell_1}^{(3)}(y) \left(1-y^2 \right)^{\frac{5}{2}}  dy =   2\int_{0}^{1 }  y^{2(\ell_2-\nu_2)}   C_{2\ell_1}^{(3)}(y) \left(1-y^2 \right)^{\frac{5}{2}}  dy  \\
	& =  \frac{\pi  4^{-\ell_2+\nu _2-3} \Gamma \left(2 \ell_1+6\right) \Gamma \left(2 \ell_2-2 \nu _2+1\right)}{\Gamma \left(2 \ell_1+1\right) \Gamma \left(-\ell_1+\ell_2-\nu _2+1\right) \Gamma \left(\ell_1+\ell_2-\nu _2+4\right)},
\end{align*}
where we used once again \eqref{DifficultIntegralGegenbauerPolynomialsYM} to compute the last integral. In other words, we have
\begin{align*}
	& \sum_{\ell_1=2}^{m}\frac{ ( \ell_1 +1) ( \ell_1 -m-1) ( \ell_1 +m+3) }{4  \ell_1  ( \ell_1 +2)+3}  \int_{-1}^{1 }  y^{2(\ell_2-\nu_2)}   C_{2\ell_1}^{(3)}(y) \left(1-y^2 \right)^{\frac{5}{2}}  dy  = \\
	& \sum_{\ell_1=2}^{\ell_2-\nu_2}\frac{ ( \ell_1 +1) ( \ell_1 -m-1) ( \ell_1 +m+3) }{4  \ell_1  ( \ell_1 +2)+3}  \int_{-1}^{1 }  y^{2(\ell_2-\nu_2)}   C_{2\ell_1}^{(3)}(y) \left(1-y^2 \right)^{\frac{5}{2}}  dy  = \\
	& \sum_{\ell_1=2}^{\ell_2-\nu_2}  \frac{\pi  \left(\ell_1+1\right)^2 \left(\ell_1+2\right) \left(2 \ell_1+5\right) 4^{-\ell_2+\nu _2-2} \left(\ell_1-m-1\right) \left(\ell_1+m+3\right) \Gamma \left(2 \ell_2-2 \nu _2+1\right)}{\Gamma \left(-\ell_1+\ell_2-\nu _2+1\right) \Gamma \left(\ell_1+\ell_2-\nu _2+4\right)}.
\end{align*}
Putting all together, yields
		 \begin{align*} 
		J_{m}(\ell_2,\nu_2) & =
		\frac{3 \sqrt{\pi } (5  \ell_2  (3 m (m+4)+1)+m (m+4) (4-15  \nu_2 )-5 ( \nu_2 -4)) \Gamma \left( \ell_2 - \nu_2 +\frac{1}{2}\right)}{8 \Gamma ( \ell_2 - \nu_2 +5)}   \\
		& +\sum_{\ell_1=2}^{\ell_2-\nu_2+1} \frac{\pi  \ell_1 \left(\ell_1+1\right)^2 \left(2 \ell_1-1\right) 4^{-\ell_2+\nu _2-2} \left(\ell_1-m-1\right) \left(\ell_1+m+3\right) \Gamma \left(2 \ell_2-2 \nu _2+1\right)}{\Gamma \left(-\ell_1+\ell_2-\nu _2+2\right) \Gamma \left(\ell_1+\ell_2-\nu _2+3\right)} \\
		& - \sum_{\ell_1=2}^{\ell_2-\nu_2}  \frac{\pi  \left(\ell_1+1\right)^2 \left(\ell_1+2\right) \left(2 \ell_1+5\right) 4^{-\ell_2+\nu _2-2} \left(\ell_1-m-1\right) \left(\ell_1+m+3\right) \Gamma \left(2 \ell_2-2 \nu _2+1\right)}{\Gamma \left(-\ell_1+\ell_2-\nu _2+1\right) \Gamma \left(\ell_1+\ell_2-\nu _2+4\right)}
		\end{align*} 
		that completes the proof.
\end{proof}

\begin{remark}[Closed formulas for $\mathfrak{C}_{\gamma \gamma mm} $ for small values of $\gamma$]\label{MonotonicityModel3YM}
	Finally, we note that one can use Lemma \ref{LemmaClosedFormulasModel3YMCgammagammamm} to find closed formulas for the Fourier coefficients provided that $\gamma$ is sufficiently small. For example, for $\gamma \in \{0,1,2,3,4,5\}$, we find that   
	 \begin{align*}  
	 	\mathfrak{C}_{00 mm}&=	\frac{4 (m (m+4)+5)}{3 \pi  (m+1) (m+3)}, &
	 	\mathfrak{C}_{11 mm}&=	\frac{2 (m (m+4)+7)}{\pi  (m+1) (m+3)}, &
	 \mathfrak{C}_{22 mm}&=	\frac{8 (5 m (m+4)+49)}{15 \pi  (m+1) (m+3)}	  ,\\
	  \mathfrak{C}_{33 mm}&=	\frac{2 (5 m (m+4)+67)}{3 \pi  (m+1) (m+3)}	  ,&
	  \mathfrak{C}_{44 mm}&=	\frac{4 (5 m (m+4)+89)}{5 \pi  (m+1) (m+3)}	,&
	 	 \mathfrak{C}_{55 mm}&=	\frac{14 (m (m+4)+23)}{3 \pi  (m+1) (m+3)} , 
	 \end{align*}
    for all $m \geq 2\gamma+1$. Figure \ref{Lemma78PicRef} illustrates the Fourier coefficients $\mathfrak{C}_{\gamma\gamma mm} $ for $\gamma \in \{0,1,2\}$ respectively as $m $ varies within $\{1,2,\dots,50\}$.
\begin{figure}[h!]
\vspace{0.5cm}
    \centering
    \includegraphics[width=0.7\textwidth]{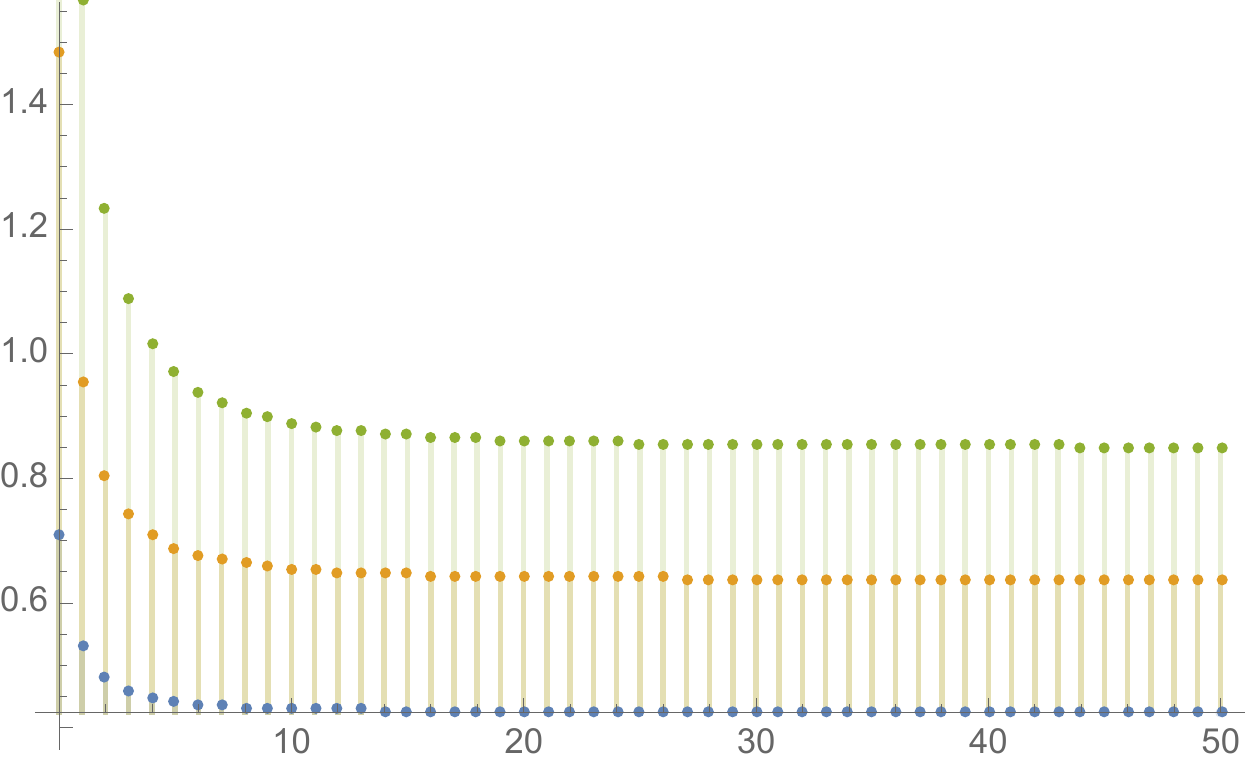}
    \caption{The Fourier coefficients $\mathfrak{C}_{\gamma\gamma mm} $ for $\gamma =0$ (blue), $\gamma=1$ (orange) and $\gamma=2$ (green) as $m $ varies within $\{1,2,\dots,50\}$. They are decreasing for $m\geq 2\gamma+1$.}
    \label{Lemma78PicRef}
\end{figure}

\end{remark}

\section{1--mode initial data} \label{Section1modeinitialdata}
In this section, we study the operators $\mathcal{M}$ and $ \textswab{M}_{\pm}$ (Section \ref{SectionMethodBambusiPaleari}) for 1--mode initial data. Specifically, we
\begin{itemize}
	\item verify that all the 1--modes are zeros of the operators $\mathcal{M}$ (for CW and CH) and $ \textswab{M}_{-}$ (for YM),
	\item compute the differentials $d\mathcal{M}$ and $d \textswab{M}_{-}$ at the 1--mode initial data.
\end{itemize}

\subsection{Conformal cubic wave equation in spherical symmetry}
Recall that the eigenfunctions $\{  e_{n}: n\geq 0\}$ are given by \eqref{EigenfunctionsModel1WS} and the PDE in the Fourier space from \eqref{PDEFourierSpaceModel1WS} reads 
\begin{align*} 
	& \ddot{u}^{m}(t) + \left(  A   u(t) \right)^m  =  \left( f \{ u^j(t): j \geq 0 \} ) \right)^m, \Hquad m \geq 0
\end{align*}
 where the dots denote derivatives with respect to time and
\begin{align*}
	\left(  A   u(t) \right)^m  &:= \omega_{m}^2  u^m(t), \\ 
	  \left( f( \{ u^j(t): j \geq 0 \} ) \right)^m &:= - \sum_{i,j,k=0}^{\infty}  C_{ijkm} 
	 u^{i}(t) u^{j}(t)  u^{k}(t) .
\end{align*}  
For any initial data 
\begin{align*}
u(0,\cdot)=\sum_{n=0}^{\infty} 	\xi^{n}e_{n} ,\Hquad \xi =\{ \xi ^{n} : n \geq 0 \},
\end{align*}
we denote by  
\begin{align*}
  \Phi ^{t} (\xi) = \left \{	  \xi^{n} \cos(\omega_{n}t) : n \geq 0   \right \}
\end{align*}
the linear flow, that is the solution to the linear problem 
\begin{align*}
\begin{dcases}
	  \ddot{u}^{n}(t) + \omega_{n}^2 u^{n}(t) =0,\Hquad  t\in \mathbb{R}
	 \\
	 u^{n}(0)=\xi^{n},  \Hquad  \dot{u}^{n}(0) =0 .
\end{dcases}
\end{align*} 
For this model, we aim towards implementing the original version of Bambusi--Paleari's theorem (Theorem \ref{OrigivalversionTheoremBambusi}) and define
\begin{align*} 
	\mathcal{M}\left(\xi \right) : =   A \xi + \langle f \rangle(\xi),   \Hquad 
\langle f \rangle(\xi):=\frac{1}{2\pi } \int_{0}^{2\pi } \Phi ^{t} \left[f\left(\Phi ^{t} (\xi) \right) 
	\right]  dt .
\end{align*} 
To begin with, we show that the 1--modes are zeros of the operator $\mathcal{M}$. 
\begin{lemma}[Zeros of the operator $\mathcal{M}$]\label{LemmaResonantSystemModel1CW}
Let $\xi = \{\xi^m: m\geq 0 \}$ be the rescaled 1--mode initial data,
\begin{align}\label{Definition1modesModel1CW}
	 \xi^m:=K_{\gamma}\mathds{1}(m=\gamma) ,\Hquad m\geq 0,
\end{align}
where
\begin{align}\label{DefinitionKappa1modesModel1CW}
	K_{\gamma}=\pm 2 \omega_{\gamma} \sqrt{\frac{2}{3 C_{\gamma\gamma\gamma\gamma}}} .
\end{align}
Then, we have $\mathcal{M}(\xi)=0$.
\end{lemma}
\begin{proof}
Let $\xi=\{\xi^m : m\geq 0\} $ be given by \eqref{Definition1modesModel1CW}--\eqref{DefinitionKappa1modesModel1CW} and pick any integer $m\geq 0$. Then, we compute
\begin{align*}
(  A \xi  )^m &= \omega_{m}^2 \xi ^m =    K_{\gamma}\omega_{\gamma }^2 \mathds{1}(m=\gamma), \\
	( \Phi ^{t} (\xi  ) )^m &=\xi ^m\cos(\omega_{m}t)=K_{\gamma}\cos(\omega_{\gamma}t) \mathds{1}(m=\gamma), \\
\left ( f\left(\Phi ^{t} (\xi ) \right) \right )^m & = 
	 - \sum_{i,j,k}  C_{ijkm} ( \Phi ^{t} (\xi  ) )^i 
	 ( \Phi ^{t} (\xi  ) )^j ( \Phi ^{t} (\xi  ) )^k 
	 = - K_{\gamma}^3  C_{\gamma \gamma \gamma m}   \cos^3(\omega_{\gamma}t) , \\
 \left( \Phi ^{t} \left[f\left(\Phi ^{t} (\xi ) \right) 
	\right] \right)^m &=  ( f\left(\Phi ^{t} (\xi ) \right) )^m  \cos (\omega_{m}t)   = - K_{\gamma}^3  C_{\gamma \gamma \gamma m}   \cos^3(\omega_{\gamma}t)\cos(\omega_{m}t), \\
  \left(\langle f \rangle(\xi) \right)^m &=  -  \frac{ C_{\gamma \gamma \gamma m} }{2\pi }K_{\gamma}^3 \int_{0}^{2\pi }  \cos^3(\omega_{\gamma}t)\cos(\omega_{m}t)    dt \\
& = -  \frac{ C_{\gamma \gamma \gamma m} }{2\pi } K_{\gamma}^3 \Bigg( \frac{3}{8} \int_{0}^{2\pi }\cos \left( ( \omega _m+ \omega _{\gamma })t\right)dt+\frac{1}{8} \int_{0}^{2\pi }\cos \left( (\omega _m+3 \omega _{\gamma })t\right) dt \\
&+\frac{3}{8}\int_{0}^{2\pi } \cos \left( ( \omega _m- \omega _{\gamma })t\right) dt +\frac{1}{8} \int_{0}^{2\pi }\cos \left( ( \omega _m-3  \omega _{\gamma })t dt \right) \Bigg)   \\
& = -  C_{\gamma \gamma \gamma m} K_{\gamma}^3  \Bigg(  \frac{3}{8}  \mathds{1}(m=\gamma) +\frac{1}{8}  \mathds{1}(m=3\gamma+2)  \Bigg) \\
& = -       \frac{3 C_{\gamma \gamma \gamma \gamma }  }{8} K_{\gamma}^3  \mathds{1}(m=\gamma )  ,\\
 ( \mathcal{M}\left( \xi  \right) )^m & = (  A  \xi  )^m +  \left(\langle f \rangle(\xi) \right)^m  = K_{\gamma} \left(\omega_{\gamma }^2 -       \frac{3 C_{\gamma \gamma \gamma \gamma }  }{8} K_{\gamma}^2 \right) \mathds{1}(m=\gamma)=0,
\end{align*}
where we used the fact that
\begin{align*}
	\omega_{m}+\omega_{\gamma} &\neq 0, \Hquad 
	\omega_{m}+3\omega_{\gamma} \neq 0, \\ 
	\omega_{m}-\omega_{\gamma} &= 0 \Longleftrightarrow m=\gamma, \\ 
	\omega_{m}-3\omega_{\gamma} &= 0 \Longleftrightarrow m=3 \gamma +2,
\end{align*}
as well as $C_{\gamma \gamma \gamma m} = 0$ for $m=3\gamma+2$ according to Lemma \ref{LemmaVanishingFourierModel1CW}.
\end{proof}

Next, we derive the differential of $\mathcal{M}$ at the rescaled 1--modes.
\begin{lemma}[Differential of $\mathcal{M}$ at the 1--modes]\label{LemmaDifferentialofMModel1CW}
	Let $\xi=\{\xi^m : m\geq 0\} $ be given by \eqref{Definition1modesModel1CW}--\eqref{DefinitionKappa1modesModel1CW}. Then, for all $h=\{ h^{j}: j\geq 0\} \in  l_{s+3}^{2} $, we have that 
\begin{align*}   
 C_{\gamma \gamma \gamma \gamma} (	d \mathcal{M}( \xi )[h] )^m & =	  
	\left[
	\left(\omega_{m}^2 C_{\gamma \gamma \gamma \gamma}-  2\omega_{\gamma}^2 C_{\gamma \gamma mm}  \right)h^m   -   \omega_{\gamma}^2 C_{\gamma, 2\gamma-m, \gamma,m}   h^{2\gamma-m}
	\right] \mathds{1} \left(0 \leq m \leq \gamma-1 \right)  \\
	&+
	\left[
	-2\omega_{\gamma}^2 C_{\gamma \gamma \gamma \gamma} h^{\gamma}
	\right]
	  \mathds{1} \left(  m = \gamma \right) \\
	&+   
	\left[
	\left(\omega_{m}^2 C_{\gamma \gamma \gamma \gamma}-  2\omega_{\gamma}^2 C_{\gamma \gamma mm}  \right)h^m   -  \omega_{\gamma}^2 C_{\gamma, 2\gamma-m, \gamma,m}   h^{2\gamma-m}
	\right] \mathds{1} \left(\gamma+1 \leq m \leq 2\gamma \right)  \\
	& + 
	\left[
	\left(\omega_{m}^2 C_{\gamma \gamma \gamma \gamma}-  2\omega_{\gamma}^2 C_{\gamma \gamma mm}  \right)h^m 
	\right] \mathds{1} \left(  m \geq 2\gamma +1 \right)  ,
\end{align*}
where $C_{ijkm}$ are given in closed formulas in Lemma \ref{LemmaClosedformulaFourierModel1CW}.
\end{lemma}
\begin{proof}
Let $\xi=\{\xi^m : m\geq 0\} $ be given by \eqref{Definition1modesModel1CW}--\eqref{DefinitionKappa1modesModel1CW}, $\epsilon >0$, $h=\{ h^{j}: j\geq 0\} \in  l_{s+3}^{2} $ and pick any integer $m\geq 0$. Then, using the symmetries of the Fourier coefficients, we compute 
\begin{align*}
 (  A (\xi +\epsilon h)  )^m &= \omega_{m}^2 (\xi^m +\epsilon h^m) =   (  A \xi  )^m  +\epsilon  \omega_{m}^2 h^m , \\
	( \Phi ^{t} (\xi +\epsilon h  ) )^m &=(\xi^m +\epsilon h^m)\cos(\omega_{m}t)=( \Phi ^{t} (\xi  ) )^m + \epsilon h^m \cos(\omega_{m}t), \\
 \left ( f\left(\Phi ^{t} (\xi +\epsilon h ) \right) \right )^m & = 
	 - \sum_{i,j,k}  C_{ijkm} ( \Phi ^{t} (\xi +\epsilon h  ) )^i 
	 ( \Phi ^{t} (\xi  ) )^j ( \Phi ^{t} (\xi +\epsilon h  ) )^k \\ 
& = \left ( f\left(\Phi ^{t} (\xi   ) \right) \right )^m   -3\epsilon K_{\gamma}^2   \sum_{ k}  C_{\gamma \gamma km} h^k   \cos^2(\omega_{\gamma}t)    \cos ( \omega_{k}t)   + \mathcal{O}(\epsilon^2) , \\
 \left( \Phi ^{t} \left[f\left(\Phi ^{t} (\xi+\epsilon h ) \right) 
	\right] \right)^m &=   
	 \left( \Phi ^{t} \left[f\left(\Phi ^{t} (\xi  ) \right) 
	\right] \right)^m  -3\epsilon K_{\gamma}^2   \sum_{ i}  C_{i\gamma \gamma m} h^i   \cos^2(\omega_{\gamma}t)    \cos ( \omega_{i}t) \cos(\omega_{m}t)   + \mathcal{O}(\epsilon^2)  , \\
  \left(\langle f \rangle(\xi+\epsilon h) \right)^m & =  \left(\langle f \rangle(\xi ) \right)^m - \frac{3\epsilon K_{\gamma}^2 }{2\pi }  \sum_{ i}  C_{i\gamma \gamma m} h^i \int_{0}^{2\pi }  \cos^2(\omega_{\gamma}t)    \cos ( \omega_{i}t) \cos(\omega_{m}t) dt  + \mathcal{O}(\epsilon^2).
 \end{align*}
Therefore, we infer
\begin{align*}
	& \left( d \langle f \rangle(\xi )  [h]\right)^m = - \frac{3  K_{\gamma}^2 }{2\pi }  \sum_{ i}  C_{i\gamma \gamma m} h^i \int_{0}^{2\pi }  \cos^2(\omega_{\gamma}t)    \cos ( \omega_{i}t) \cos(\omega_{m}t) dt \\
& =- \frac{3  K_{\gamma}^2 }{16\pi }  \sum_{ i}  C_{i\gamma \gamma m} h^i \sum_{\pm} \int_{0}^{2\pi } \cos(( \omega_{i}\pm\omega_{\gamma}\pm \omega_{\gamma}\pm \omega_{m})t) dt \\
& = - \frac{3  K_{\gamma}^2 }{8 }   \sum_{ i}  C_{i\gamma \gamma m} h^i \sum_{\pm} \mathds{1}(\omega_{i} \pm \omega_{\gamma}\pm \omega_{\gamma}  \pm \omega_{m}=0) \\
& =- \frac{3  K_{\gamma}^2 }{8 } \left[  \sum_{ i}  C_{i\gamma \gamma m} h^i \mathds{1}( i=m ) +\sum_{ i}  C_{i\gamma \gamma m} h^i \mathds{1}( i=m  ) 
+\sum_{ i}  C_{i\gamma \gamma m} h^i \mathds{1}( i=2\gamma -m\geq 0)  \right] \\
& =- \frac{3  K_{\gamma}^2 }{8 } \left[ 2    C_{m\gamma \gamma m} h^m    
+  C_{2\gamma -m,\gamma ,\gamma ,m} h^{2\gamma -m} \mathds{1}( 0\leq m \leq 2\gamma )  \right] \\
&  =  
- \frac{  \omega_{\gamma} ^2}{ C_{\gamma \gamma \gamma \gamma} }  \left[
2 C_{\gamma \gamma m m}     h^m  
+\mathds{1} \left(0 \leq m \leq 2\gamma \right)
 C_{\gamma,2\gamma-m,\gamma,m}   
	  h^{2\gamma-m }   \right] ,
\end{align*}
where we used the fact that $C_{ijkm}=0$ for $\omega_{i}\pm\omega_{j}\pm\omega_{k}\pm\omega_{m}=0$ with only 1 minus sign according to Lemma \ref{LemmaVanishingFourierModel1CW}, that is
\begin{align*} 
	\begin{dcases}
		-\omega_{i}+\omega_{\gamma}+ \omega_{\gamma}+ \omega_{m} =0, \\
		\omega_{i}-\omega_{\gamma}+ \omega_{\gamma}+ \omega_{m} =0, \\
		\omega_{i}+\omega_{\gamma}- \omega_{\gamma}+ \omega_{m} =0, \\
		\omega_{i}+\omega_{\gamma}+ \omega_{\gamma}- \omega_{m} =0,
	\end{dcases}
\end{align*}
so we are left with $\omega_{i}\pm\omega_{j}\pm\omega_{k}\pm\omega_{m}=0$ with only 2 minus signs and there are three such terms in total, that is
\begin{align*}
\begin{dcases}
	\omega_{i}+\omega_{\gamma}- \omega_{\gamma}- \omega_{m} =0, \\
	\omega_{i}-\omega_{\gamma}+ \omega_{\gamma}- \omega_{m} =0, \\
	\omega_{i}-\omega_{\gamma}- \omega_{\gamma}+ \omega_{m} =0, \\
	\end{dcases} \Longleftrightarrow
	\begin{dcases}
	i=m, \\
	i=m, \\
	i=2\gamma-m \text{ and } i \geq 0. \\
	\end{dcases}
\end{align*}
Finally, we obtain
\begin{align*}
	(	d \mathcal{M}( \xi )[h] )^m &=  \omega_{m}^2 h^m +\left( d \langle f \rangle(\xi )  [h] \right)^m\\
	&=\left[ \omega_{m}^2 -\frac{   2 \omega_{\gamma} ^2 C_{\gamma \gamma m m} }{ C_{\gamma \gamma \gamma \gamma} } 
   \right] h^m  -\mathds{1} \left(0 \leq m \leq 2\gamma \right) 
\frac{  \omega_{\gamma} ^2 C_{\gamma,2\gamma-m,\gamma,m}}{ C_{\gamma \gamma \gamma \gamma} }   
	  h^{2\gamma-m }   \\
	&=\left[ \omega_{m}^2 -\frac{   2 \omega_{\gamma} ^2 C_{\gamma \gamma m m} }{ C_{\gamma \gamma \gamma \gamma} } 
   \right] h^m  -\mathds{1} \left(0 \leq m \leq 2\gamma \right) 
\frac{  \omega_{\gamma} ^2 C_{\gamma,2\gamma-m,\gamma,m}}{ C_{\gamma \gamma \gamma \gamma} }   
	  h^{2\gamma-m }   \\
	&=  	  \mathds{1} \left(0 \leq m \leq \gamma-1 \right)
	\left[
	\left(\omega_{m}^2 - \frac{2\omega_{\gamma}^2 C_{\gamma \gamma mm}}{C_{\gamma \gamma \gamma \gamma}} \right)h^m   -  \frac{\omega_{\gamma}^2 C_{\gamma, 2\gamma-m, \gamma,m}}{C_{\gamma \gamma \gamma \gamma}}  h^{2\gamma-m}
	\right] \\
	&+
	 \mathds{1} \left(  m = \gamma \right)
	\left[
	-2\omega_{\gamma}^2 h^{\gamma}
	\right] \\
	&+  \mathds{1} \left(\gamma+1 \leq m \leq 2\gamma \right)
	\left[
	\left(\omega_{m}^2 - \frac{2\omega_{\gamma}^2 C_{\gamma \gamma mm}}{C_{\gamma \gamma \gamma \gamma}} \right)h^m   -  \frac{\omega_{\gamma}^2 C_{\gamma, 2\gamma-m, \gamma,m}}{C_{\gamma \gamma \gamma \gamma}}  h^{2\gamma-m}
	\right]  \\
	& +  \mathds{1} \left(  m \geq 2\gamma +1 \right)
	\left[
	\left(\omega_{m}^2 - \frac{2\omega_{\gamma}^2 C_{\gamma \gamma mm}}{C_{\gamma \gamma \gamma \gamma}} \right)h^m 
	\right] ,
 \end{align*} 
that completes the proof. 
\end{proof}

\subsection{Conformal cubic wave equation out of spherical symmetry}
Recall that the eigenfunctions $\{  \mathsf{e}_{n}^{(\mu_1,\mu_2)}: n\geq 0\}$ are given by \eqref{EigenfunctionsModel2WH} and the PDE in the Fourier space from \eqref{PDEFourierSpaceModel2CH} reads 
\begin{align*} 
	& \ddot{u}^{m}(t) + \left( \mathsf{A} u(t) \right)^m = \left( \mathsf{f}( \{ u^j(t): j \geq 0 \} ) \right)^m, \Hquad m \geq 0
\end{align*}
 where the dots denote derivatives with respect to time and
\begin{align*}
	 \left( \mathsf{A} u(t) \right)^m &:= \left(\somega_n^{(\mu_1,\mu_2)} \right)^2 u^m(t), \\ 
	 \left( \mathsf{f}( \{ u^j(t): j \geq 0 \} ) \right)^m&:= - \sum_{i,j,k=0}^{\infty}  \mathsf{C}_{ijkm}^{(\mu_1,\mu_2)}  
	u^{i}(t)u^{j}(t) u^{k}(t) .
\end{align*}  
For any initial data 
\begin{align*}
u(0,\cdot)=\sum_{n=0}^{\infty} 	\xi^{n}  \mathsf{e}_{n}^{(\mu_1,\mu_2)} ,\Hquad \xi =\{ \xi ^{n} : n \geq 0 \},
\end{align*}
we denote by  
\begin{align*}
  \mathsf{\Phi} ^{t} (\xi) = \left \{	  \xi^{n} \cos(\somega_{n}^{(\mu_1,\mu_2)}t) : n \geq 0   \right \}
\end{align*}
the linear flow, that is the solution to the linear problem 
\begin{align*}
\begin{dcases}
	\ddot{u}^{n}(t) + \left(\somega_{n}^{(\mu_1,\mu_2)} \right)^2 u^{n}(t) =0,\Hquad  t\in \mathbb{R}
	 \\
	 u^n(0)=\xi^n,  \Hquad  \dot{u}(0) =0 .
\end{dcases}
\end{align*} 
For this model, we aim towards implementing the original version of Bambusi--Paleari's theorem (Theorem \ref{OrigivalversionTheoremBambusi}) and define
\begin{align*} 
	 \mathcal{M} \left(\xi \right) : =    \mathsf{A}  \xi +\langle \mathsf{f} \rangle (\xi),  \Hquad 
	\langle \mathsf{f} \rangle (\xi) := \frac{1}{2\pi } \int_{0}^{2\pi }  \mathsf{\Phi}  ^{t} \left[f\left( \mathsf{\Phi}  ^{t} (\xi) \right) 
	\right]  dt .
\end{align*}
To begin with, we show that the 1--modes are zeros of the operator $\mathcal{M}$. 
\begin{lemma}[Zeros of the operator $\mathcal{M}$]\label{LemmaResonantSystemModel2CH}
Let $\xi = \{\xi^m: m\geq 0 \}$ be the rescaled 1--mode initial data,
\begin{align}\label{Definition1modesModel2CH}
	 \xi^m:= \mathsf{K}_{\gamma}^{(\mu_1,\mu_2)}\mathds{1}(m=\gamma) ,\Hquad m\geq 0,
\end{align}
where
\begin{align}\label{DefinitionKappa1modesModel2CH}
	\mathsf{K}_{\gamma}^{(\mu_1,\mu_2)}=\pm 2\somega_{\gamma}^{(\mu_1,\mu_2)} \sqrt{\frac{2 }{3 C_{\gamma\gamma\gamma\gamma}^{(\mu_1,\mu_2)}}}.
\end{align}
Then, we have $\mathcal{M}(\xi)=0$.
\end{lemma}
\begin{proof}
The proof is similar to the one of Lemma \ref{LemmaResonantSystemModel1CW}.
 \end{proof}
Next, we derive the differential of $\mathcal{M}$ at the rescaled 1--modes.
\begin{lemma}[Differential of $\mathcal{M}$ at the 1--modes]\label{LemmaDifferentialofMModel2CH}
	Let $\xi=\{\xi^m : m\geq 0\} $ be given by \eqref{Definition1modesModel2CH}--\eqref{DefinitionKappa1modesModel2CH}. Then, for all $h=\{ h^{j}: j\geq 0\} \in  l_{s+3}^{2} $, we have that 
\begin{align*}      
 \mathsf{C}_{\gamma \gamma \gamma \gamma}^{(\mu_1,\mu_2)} (	d \mathcal{M}( \xi )[h] )^m&  =	
	\left[
	\left(\somega_{m}^2 \mathsf{C}_{\gamma \gamma \gamma \gamma}^{(\mu_1,\mu_2)}-  2\somega_{\gamma}^2 \mathsf{C}_{\gamma \gamma mm}^{(\mu_1,\mu_2)}  \right)h^m   -   \somega_{\gamma}^2 \mathsf{C}_{\gamma, 2\gamma-m, \gamma,m}^{(\mu_1,\mu_2)}   h^{2\gamma-m}
	\right] \mathds{1} \left(0 \leq m \leq \gamma-1 \right)  \\
	&+
	\left[
	-2\somega_{\gamma}^2 \mathsf{C}_{\gamma \gamma \gamma \gamma}^{(\mu_1,\mu_2)} h^{\gamma}
	\right]  \mathds{1} \left(  m = \gamma \right) \\
	&+  
	\left[
	\left(\somega_{m}^2 \mathsf{C}_{\gamma \gamma \gamma \gamma}^{(\mu_1,\mu_2)} -  2\somega_{\gamma}^2 \mathsf{C}_{\gamma \gamma mm}^{(\mu_1,\mu_2)}  \right)h^m   -  \somega_{\gamma}^2 \mathsf{C}_{\gamma, 2\gamma-m, \gamma,m}^{(\mu_1,\mu_2)}    h^{2\gamma-m}
	\right]\mathds{1} \left(\gamma+1 \leq m \leq 2\gamma \right)   \\
	& +  
	\left[
	\left(\somega_{m}^2 \mathsf{C}_{\gamma \gamma \gamma \gamma}^{(\mu_1,\mu_2)} -  2\somega_{\gamma}^2 \mathsf{C}_{\gamma \gamma mm}^{(\mu_1,\mu_2)}  \right)h^m 
	\right] \mathds{1} \left(  m \geq 2\gamma +1 \right) ,
\end{align*}
where $\mathsf{C}_{\gamma \gamma m m}^{(\mu_1,\mu_2)}$ are given by closed formulas in Lemma \ref{LemmaUniformAsymptoticFormulasModel2CH}.
\end{lemma}
\begin{proof}
The proof is similar to the one of Lemma \ref{LemmaDifferentialofMModel1CW} due to the fact that  $\mathsf{C}_{ijkm}^{(\mu_1,\mu_2)}=0$ for $\somega_{i}^{(\mu_1,\mu_2)}\pm\somega_{j}^{(\mu_1,\mu_2)}\pm\somega_{k}^{(\mu_1,\mu_2)}\pm\somega_{m}^{(\mu_1,\mu_2)}=0$ with only 1 minus sign according to Lemma \ref{LemmaVanisginFourierModel2CH}, that is
\begin{align*} 
	\begin{dcases}
		-\somega_{i}^{(\mu_1,\mu_2)}+\somega_{\gamma}^{(\mu_1,\mu_2)}+ \somega_{\gamma}^{(\mu_1,\mu_2)}+ \somega_{m}^{(\mu_1,\mu_2)} =0, \\
		\somega_{i}^{(\mu_1,\mu_2)}-\somega_{\gamma}^{(\mu_1,\mu_2)}+ \somega_{\gamma}^{(\mu_1,\mu_2)}+  \somega_{m}^{(\mu_1,\mu_2)} =0, \\
		\somega_{i}^{(\mu_1,\mu_2)}+\somega_{\gamma}^{(\mu_1,\mu_2)}- \somega_{\gamma}^{(\mu_1,\mu_2)}+  \somega_{m}^{(\mu_1,\mu_2)} =0, \\
		\somega_{i}^{(\mu_1,\mu_2)}+\somega_{\gamma}^{(\mu_1,\mu_2)}+ \somega_{\gamma}^{(\mu_1,\mu_2)}-  \somega_{m}^{(\mu_1,\mu_2)} =0,
	\end{dcases}
\end{align*}
so we are left with $\somega_{i}^{(\mu_1,\mu_2)}\pm \somega_{j}^{(\mu_1,\mu_2)}\pm \somega_{k}^{(\mu_1,\mu_2)}\pm \somega_{m}^{(\mu_1,\mu_2)}=0$ with only 2 minus signs and there are again the same three such terms in total, that is
\begin{align*}
\begin{dcases}
	\somega_{i}^{(\mu_1,\mu_2)}+\somega_{\gamma}^{(\mu_1,\mu_2)}- \somega_{\gamma}^{(\mu_1,\mu_2)}- \somega_{m}^{(\mu_1,\mu_2)} =0, \\
	\somega_{i}^{(\mu_1,\mu_2)}-\somega_{\gamma}^{(\mu_1,\mu_2)}+ \somega_{\gamma}^{(\mu_1,\mu_2)}-\somega_{m}^{(\mu_1,\mu_2)} =0, \\
	\somega_{i}^{(\mu_1,\mu_2)}-\somega_{\gamma}^{(\mu_1,\mu_2)}- \somega_{\gamma}^{(\mu_1,\mu_2)}+\somega_{m}^{(\mu_1,\mu_2)} =0, \\
	\end{dcases} \Longleftrightarrow
	\begin{dcases}
	i=m, \\
	i=m, \\
	i=2\gamma-m \text{ and } i \geq 0, \\
	\end{dcases}
\end{align*}
that completes the proof.
\end{proof}

\subsection{Yang--Mills equation in spherical symmetry} 
Recall that the eigenfunctions $\{  \mathfrak{e}_{n}: n\geq 0\}$ are given by \eqref{EigenfunctionsModel3YM} and the PDE in the Fourier space from \eqref{PDEFourierSpaceModel3YM} reads 
\begin{align*} 
	& \ddot{u}^{m}(t) +\left( \mathfrak{A} u(t)\right)^m =\left(  \mathfrak{f}( \{ u^j(t): j \geq 0 \} ) \right)^m, \Hquad m \geq 0
\end{align*}
where the dots denote derivatives with respect to time and 
\begin{align*}
	\left( \mathfrak{A} u(t)\right)^m  :=  \ssomega_{n}^2 u^m(t), \Hquad 
	  \left(  \mathfrak{f}( u) \right)^m := \left(  \mathfrak{f}^{(2)}( u ) \right)^m+
	   \left(  \mathfrak{f}^{(3)}( u ) \right)^m,
\end{align*}  
with
\begin{align*}
\left(  \mathfrak{f}^{(2)}( \{ u^j(t): j \geq 0 \} ) \right)^m &:=	-3\sum_{i,j=0}^{\infty}\overline{\mathfrak{C}}_{ijm}  u^{i}(t)u^{j}(t)   , \\
\left(  \mathfrak{f}^{(3)}( \{ u^j(t): j \geq 0 \} ) \right)^m &:= - \sum_{i,j,k=0}^{\infty}  \mathfrak{C}_{ijkm} u^{i}(t)u^{j}(t)u^{k}(t) .
\end{align*}
For any initial data 
\begin{align*}
u(0,\cdot)=\sum_{n=0}^{\infty} 	\xi^{n}  \mathfrak{e}_{n}  ,\Hquad \xi =\{ \xi ^{n} : n \geq 0 \},
\end{align*}
we denote by  
\begin{align*}
  \pPhi ^{t} (\xi) = \left \{	  \xi^{n} \cos(\ssomega_{n} t) : n \geq 0   \right \}
\end{align*}
the linear flow, that is the solution to the linear problem 
\begin{align*}
\begin{dcases}
	\ddot{u}^{n}(t) +  \ssomega_{n}  ^2 u^{n}(t) =0,\Hquad  t\in \mathbb{R}
	 \\
	 u^n(0)=\xi^n,  \Hquad  \dot{u}^n(0) =0 .
\end{dcases}
\end{align*} 
As a starting point, we show that the original version of Bambusi--Paleari's theorem (Theorem \ref{OrigivalversionTheoremBambusi}) is not applicable.

\begin{lemma}[Non--resonant $\mathfrak{f}^{(2)}$]\label{LemmaNonResoantnf2}
	For all initial data $\xi \in l_s^2$, we have
	\begin{align*}
		\langle \mathfrak{f}^{(2)} \rangle(\xi):=\frac{1}{2\pi } \int_{0}^{2\pi }    \pPhi  ^{t} \left[ \mathfrak{f}^{(2)}\left(   \pPhi ^{t} (\xi) \right) 
	\right]  dt =0.
	\end{align*}
\end{lemma}
\begin{proof}
Let $\xi=\{\xi^m : m\geq 0\} \in l_s^2 $ be any initial data and pick an integer $m\geq 0$. We compute
\begin{align*}
(  \mathfrak{A} \xi  )^m &= \ssomega_{m}^2 \xi ^m  , \\
	(\pPhi ^{t} (\xi  ) )^m &=\xi ^m\cos(\ssomega_{m}t) , \\
\left ( \mathfrak{f}^{(2)}\left(\pPhi ^{t} (\xi ) \right) \right )^m & = 
	 -3\sum_{i,j=0}^{\infty}\overline{\mathfrak{C}}_{ijm}  ( \pPhi ^{t} (\xi  ) )^i ( \pPhi ^{t} (\xi  ) )^j    = 
	 -3\sum_{i,j=0}^{\infty}\overline{\mathfrak{C}}_{ijm} \xi^i \xi^j \cos(\ssomega_{i}t)\cos(\ssomega_{j}t)   \\
 \left( \pPhi ^{t} \left[\mathfrak{f}^{(2)}\left(\pPhi ^{t} (\xi ) \right) 
	\right] \right)^m &=  ( \mathfrak{f}\left(\pPhi ^{t} (\xi ) \right) )^m  \cos (\ssomega_{m}t)
	 = - 3\sum_{i,j=0}^{\infty}\overline{\mathfrak{C}}_{ijm} \xi^i \xi^j \prod_{\lambda \in \{i,j,m\}}\cos(\ssomega_{\lambda}t  ) , \\
\langle \mathfrak{f}^{(2)} \rangle(\xi)  &  =  
  - \frac{3}{2\pi }\sum_{i,j=0}^{\infty}\overline{\mathfrak{C}}_{ijm} \xi^i \xi^j  \int_{0}^{2\pi }\prod_{\lambda \in \{i,j,m\}}\cos(\ssomega_{\lambda}t  )  dt
  \\
 &= - \frac{3}{2\pi }\sum_{i,j=0}^{\infty}\overline{\mathfrak{C}}_{ijm} \xi^i \xi^j  \int_{0}^{2\pi }\sum_{ \substack{\pm   } } \frac{1}{4}\cos \left(\ssomega_{i} \pm \ssomega_{j} \pm \ssomega_{m} \right)  dt \\
  &= - \frac{3}{4 }\sum_{i,j=0}^{\infty}\overline{\mathfrak{C}}_{ijm} \xi^i \xi^j   \sum_{ \substack{\pm   } } \mathds{1} \left(\ssomega_{i} \pm \ssomega_{j} \pm \ssomega_{m}=0 \right)  .
\end{align*}
Now, notice that all the possible conditions are those with only 1 minus sign, namely 
\begin{align*}
	-\ssomega_{i} + \ssomega_{j} + \ssomega_{m}=0, \\
	+\ssomega_{i} - \ssomega_{j} +\ssomega_{m}=0, \\
	+\ssomega_{i} + \ssomega_{j} -\ssomega_{m}=0,  
\end{align*}
and according to Lemma \ref{LemmaVanisginFourierModel3YM} the corresponding Fourier coefficients vanish.  
\end{proof}
Consequently, for this model, we aim towards implementing the modified version of Bambusi--Paleari's theorem (Theorem \ref{TheoremModificationofBambusiPaleari}) and define
\begin{align*}
	 \textswab{M}_{\pm}(\xi) &= \pm  \mathfrak{A} \xi +     \langle \mathfrak{f}^{(3)}  \rangle( \xi)+\mathfrak{F}_{0}( \xi) , \Hquad 
	\langle \mathfrak{f}^{(3)} \rangle (\xi) := \frac{1}{2\pi } \int_{0}^{2\pi }  \pPhi ^{t} \left[\mathfrak{f}^{(3)}\left( \pPhi  ^{t} (\xi) \right) 
	\right]  dt , 
\end{align*}
where $\mathfrak{F}_{0}( \xi) $ is given for any initial data by Lemma \ref{LemmaDefinitionMathfrakF} and for the 1--mode initial data by Lemma \ref{LemmaComputemathfrakFforsfmallepsilon}. Also, recall the Diophantine condition $\ssomega \in \mathcal{W}_{\alpha}$ for some $0<\alpha< 1/ 3$ from Theorem \ref{TheoremModificationofBambusiPaleari}.  \\ \\
To begin with, we show that the 1--modes are zeros of the operator $ \textswab{M}_{-}$. 
\begin{lemma}[Zeros of the operator $ \textswab{M}_{-}$]\label{LemmaResonantSystemModel3YM}
Let $\gamma \in \{0,1,\dots, 5 \}$ and  
\begin{align*}
		\mathfrak{q}_{\gamma}  :=	\frac{9}{4} \sum_{  \nu = 0}^{2\gamma}  \left(  \overline{\mathfrak{C}}_{\gamma \gamma \nu} \right)^2
  \left(\frac{2}{\ssomega_{\nu}^2 }  +  \frac{1}{\ssomega_{\nu}^2-(2\ssomega_{\gamma})^2   } \right).
	\end{align*}
Then, we have that $8\mathfrak{q}_{\gamma} >3 \mathfrak{C}_{\gamma\gamma\gamma\gamma}$. Moreover, let $\xi = \{\xi^m: m\geq 0 \}$ to be the rescaled 1--mode initial data,  
\begin{align}\label{Definition1modesModel3YM}
	 \xi^m:=  \mathfrak{K}_{\gamma} \mathds{1}(m=\gamma) ,\Hquad m\geq 0,
\end{align}
where
\begin{align}\label{DefinitionKappa1modesModel3YM}
	 \mathfrak{K}_{\gamma}:=  \pm 2 \somega_{\gamma} \sqrt{ 
	  \frac{2}{8\mathfrak{q}_{\gamma} -3 \mathfrak{C}_{\gamma\gamma\gamma\gamma} }  }.
\end{align}
Then, we have that $ \textswab{M}_{-}(\xi)=0$.
\end{lemma}
\begin{proof}
Let $\gamma \in \{0,1,\dots, 5 \}$, define $\xi = \{\xi^m: m\geq 0 \}$ to be the rescaled 1--mode initial data given by \eqref{Definition1modesModel3YM}--\eqref{DefinitionKappa1modesModel3YM} and pick any integer $m\geq 0$. Firstly, we compute $-\mathfrak{A} \xi +    \langle \mathfrak{f}^{(3)}  \rangle( \xi)$.  We have
\begin{align*}
(  \mathfrak{A} \xi  )^m &= \ssomega_{m}^2 \xi ^m =   \mathfrak{K}_{\gamma}   \ssomega_{\gamma}^2 \mathds{1}(m=\gamma), \\
	( \pPhi ^{t} (\xi  ) )^m &=\xi ^m\cos(\ssomega_{m}t)= \mathfrak{K}_{\gamma}\cos(\ssomega_{\gamma}t) \mathds{1}(m=\gamma), \\
\left ( \mathfrak{f}^{(3)}\left(\pPhi ^{t} (\xi ) \right) \right )^m & = 
	 - \sum_{i,j,k}  \mathfrak{C}_{ijkm} ( \pPhi ^{t} (\xi  ) )^i 
	 ( \pPhi ^{t} (\xi  ) )^j ( \pPhi ^{t} (\xi  ) )^k 
	 = - \mathfrak{K}_{\gamma}^3   \mathfrak{C}_{\gamma \gamma \gamma m}   \cos^3(\ssomega_{\gamma}t) , \\
 \left( \pPhi ^{t} \left[\mathfrak{f}^{(3)}\left(\pPhi ^{t} (\xi ) \right) 
	\right] \right)^m &=  ( \mathfrak{f}^{(3)}\left(\pPhi ^{t} (\xi ) \right) )^m  \cos (\ssomega_{m}t)   = - \mathfrak{K}_{\gamma}^3   \mathfrak{C}_{\gamma \gamma \gamma m}   \cos^3(\ssomega_{\gamma}t)\cos(\ssomega_{m}t), \\
  \left(\langle \mathfrak{f}^{(3)} \rangle(\xi) \right)^m &=  -  \frac{ \mathfrak{C}_{\gamma \gamma \gamma m} }{2\pi }\mathfrak{K}_{\gamma}^3  \int_{0}^{2\pi }  \cos^3(\ssomega_{\gamma}t)\cos(\ssomega_{m}t)    dt \\
& = -  \frac{ \mathfrak{C}_{\gamma \gamma \gamma m} }{2\pi } \mathfrak{K}_{\gamma}^3  \Bigg( \frac{3}{8} \int_{0}^{2\pi }\cos \left( ( \ssomega _m+ \ssomega _{\gamma })t\right)dt+\frac{1}{8} \int_{0}^{2\pi }\cos \left( (\ssomega _m+3 \ssomega _{\gamma })t\right) dt \\
&+\frac{3}{8}\int_{0}^{2\pi } \cos \left( ( \ssomega _m- \ssomega _{\gamma })t\right) dt +\frac{1}{8} \int_{0}^{2\pi }\cos \left( ( \ssomega _m-3  \ssomega _{\gamma })t \right) dt \Bigg)   \\
& = -  \mathfrak{C}_{\gamma \gamma \gamma m} \mathfrak{K}_{\gamma}^3   \Bigg(  \frac{3}{8}  \mathds{1}(m=\gamma) +\frac{1}{8}  \mathds{1}(m=3\gamma+4)  \Bigg) \\
& = -       \frac{3 \mathfrak{C}_{\gamma \gamma \gamma \gamma }  }{8} \mathfrak{K}_{\gamma}^3   \mathds{1}(m=\gamma ) , \\
-(  \mathfrak{A} \xi  )^m +  \left(\langle \mathfrak{f}^{(3)} \rangle(\xi) \right)^m &= -\mathfrak{K}_{\gamma} \left(   \ssomega_{\gamma}^2 +  \frac{3 \mathfrak{C}_{\gamma \gamma \gamma \gamma }  }{8} \mathfrak{K}_{\gamma}^2  \right)\mathds{1}(m=\gamma ),
\end{align*}
where we used the fact that
\begin{align*}
	\ssomega_{m}+\ssomega_{\gamma} &\neq 0, \Hquad 
	\ssomega_{m}+3\ssomega_{\gamma} \neq 0, \\ 
	\ssomega_{m}-\ssomega_{\gamma} &= 0 \Longleftrightarrow m=\gamma, \\ 
	\ssomega_{m}-3\ssomega_{\gamma} &= 0 \Longleftrightarrow m=3\gamma+4,
\end{align*}
as well as $\mathfrak{C}_{\gamma \gamma \gamma m} = 0$ for $m=3\gamma+4$ according to Lemma \ref{LemmaVanisginFourierModel3YM}. Furthermore, we use the computation for $\mathfrak{F}_{0}( \xi) $ at the 1--mode initial data we derived in Lemma \ref{LemmaComputemathfrakFforsfmallepsilon}, that is
\begin{align*}
 	\left( \mathfrak{F}_{0}( \xi) \right)^m = \mathfrak{q}_{\gamma}   \mathfrak{K}_{\gamma} ^3  \mathds{1}(m=\gamma),
	\end{align*}
	with
	\begin{align*}
		\mathfrak{q}_{\gamma}  :=	\frac{9}{4} \sum_{  \nu = 0}^{2\gamma}  \left(  \overline{\mathfrak{C}}_{\gamma \gamma \nu} \right)^2
  \left(\frac{2}{\ssomega_{\nu}^2 }  +  \frac{1}{\ssomega_{\nu}^2-(2\ssomega_{\gamma})^2   } \right) ,
	\end{align*}
to conclude that
\begin{align*}
	 (  \textswab{M}_{-}\left( \xi  \right) )^m & = -(  \mathfrak{A}  \xi  )^m + \left(\langle \mathfrak{f}^{(3)} \rangle(\xi) \right)^m +\left( \mathfrak{F}_{0}( \xi) \right)^m  \\ 
	 & = -\mathfrak{K}_{\gamma} \left(   \ssomega_{\gamma}^2 +  \frac{3 \mathfrak{C}_{\gamma \gamma \gamma \gamma }  }{8} \mathfrak{K}_{\gamma}^2  -\mathfrak{q}_{\gamma}   \mathfrak{K}_{\gamma} ^2  \right)\mathds{1}(m=\gamma ) \\ 
	 & = -\frac{1}{8}\mathfrak{K}_{\gamma} \left(  8 \ssomega_{\gamma}^2 +  3 \mathfrak{C}_{\gamma \gamma \gamma \gamma }  \mathfrak{K}_{\gamma}^2  -8\mathfrak{q}_{\gamma}   \mathfrak{K}_{\gamma} ^2  \right)\mathds{1}(m=\gamma ) \\ 
	 &=0,
\end{align*}
provided that
\begin{align*}
	\mathfrak{K}_{\gamma}^2   
	 =  \frac{-8\ssomega_{\gamma}^2}{3 \mathfrak{C}_{\gamma\gamma\gamma\gamma}-8\mathfrak{q}_{\gamma} }.
\end{align*}
Finally, it remains to show that this choice is well defined, that is  $\mathfrak{K}_{\gamma}  \in \mathbb{R}$, for all $\gamma \in \{0,1,\dots, 10 \}$. To this end, we fix $\gamma \in \{0,1,\dots, 10 \}$ and use the definition of the Fourier coefficients to  compute each $\mathfrak{K}_{\gamma}  $ and verify that they are all real numbers. Figure \ref{Lemma86PicRef} illustrates the constants $\mathfrak{K}_{\gamma} $ as $\gamma$ varies within $\{0,1,\dots,10\}$. 
\begin{figure}[h!]
\vspace{0.5cm}
    \centering
    \includegraphics[width=0.7\textwidth]{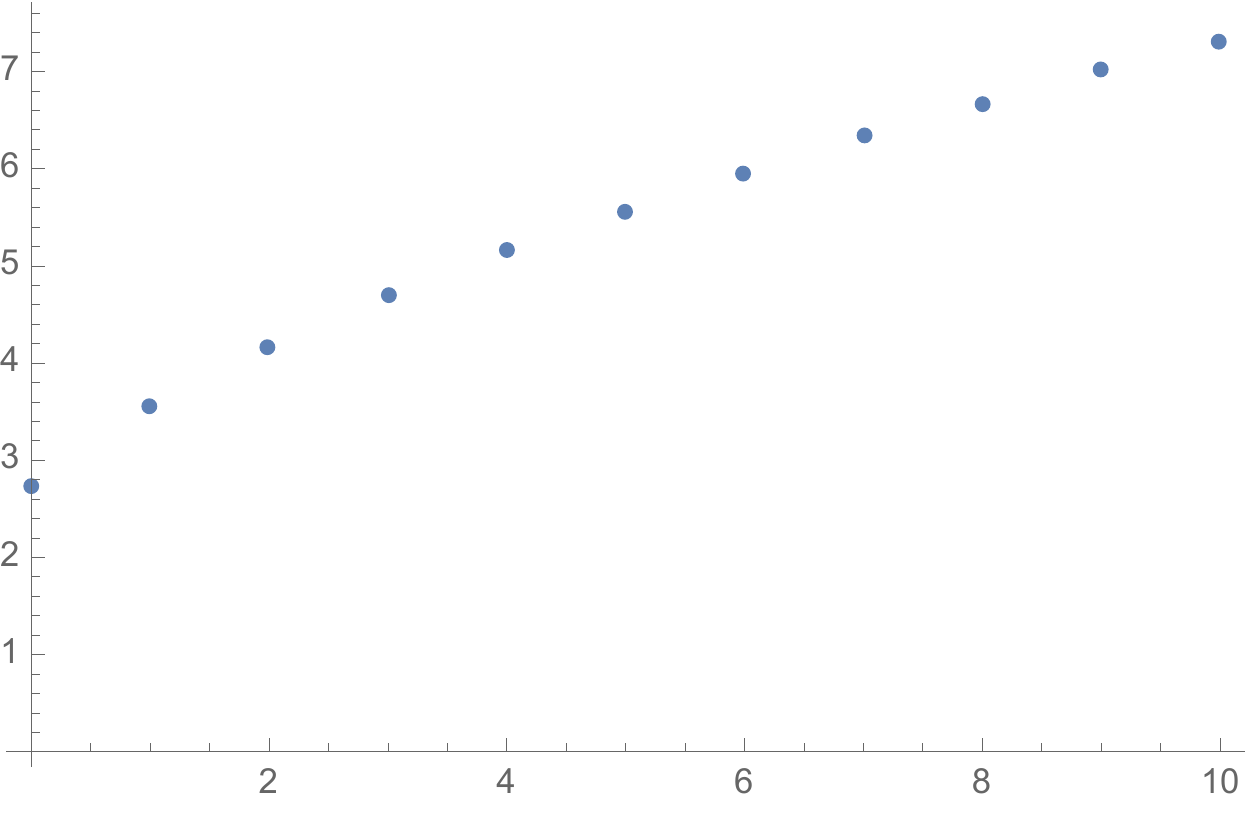}
    \caption{The constants $\mathfrak{K}_{\gamma}  $ as $\gamma$ varies within $\{0,1,\dots,10\}$. They are all real numbers.}
    \label{Lemma86PicRef}
\end{figure}  
\end{proof}
Next, we derive the differential of $ \textswab{M}_{-}$ at the rescaled 1--modes.
\begin{lemma}[Differential of $ \textswab{M}_{-}$ at the 1--modes]\label{LemmaDifferentialofMModel3YM}
	Let $\gamma \in \{0,1,\dots, 10 \}$ and let $\xi=\{\xi^m : m\geq 0\} $ be given by \eqref{Definition1modesModel3YM}--\eqref{DefinitionKappa1modesModel3YM}. Then, for all $h=\{ h^{j}: j\geq 0\} \in  l_{s+3}^{2} $, we have that  
\begin{align*}
 -\mathfrak{K}_{\gamma}^{-2} (	d  \textswab{M}_{-}( \xi )[h] )^m &=	\mathds{1}(0\leq m\leq \gamma-1)\Bigg[
 	h^m	\mathfrak{u}_{\gamma m}  
 	+ h^{2\gamma-m}  \mathfrak{v}_{\gamma m}   \Bigg] 
 	+\mathds{1}( m= \gamma)\Bigg[
 	h^{\gamma}	\left(  
 	\mathfrak{u}_{\gamma \gamma}  
  +\mathfrak{v}_{\gamma \gamma} )\right) 
 	  	\Bigg] \\
 	 &+\mathds{1}(\gamma+1\leq m\leq 2\gamma)\Bigg[
 	h^m	\mathfrak{u}_{\gamma m}   +
 	h^{2\gamma-m}  \mathfrak{v}_{\gamma m}   
 	\Bigg]  +\mathds{1}( m \geq  2\gamma+1)\Bigg[
  h^m	\mathfrak{u}_{\gamma m}  
  \Bigg],
\end{align*}
where 
	\begin{align*}
	\mathfrak{u}_{\gamma m}  :=    
 	\left(\frac{\ssomega_{m}}{  \mathfrak{K}_{\gamma}} \right)^{2}+\frac{3}{4} \mathfrak{C}_{\gamma \gamma mm} -  \mathfrak{a}_{\gamma m} 
 	 ,\Hquad 
	\mathfrak{v}_{\gamma m}  :=  
 	\frac{3}{8} \mathfrak{C}_{\gamma ,2\gamma-m,\gamma, m} -\mathfrak{b}_{\gamma m}  ,
\end{align*}
and $\overline{\mathfrak{C}}_{ijm}$ and $\mathfrak{C}_{\gamma \gamma mm} $ are given respectively by Lemmata \ref{LemmaClosedFormulasModel3YMCbarijm} and \ref{LemmaClosedFormulasModel3YMCgammagammamm} whereas $ \mathfrak{a}_{\gamma m} $ and $\mathfrak{b}_{\gamma m} $ are given by Lemma \ref{LemmaDifferentialMathFrakF0LongComputation}.  
\end{lemma}
\begin{proof}
Let $\gamma \in \{0,1,\dots, 10 \}$, $\xi=\{\xi^m : m\geq 0\} $ be given by \eqref{Definition1modesModel3YM}--\eqref{DefinitionKappa1modesModel3YM}, $h=\{ h^{j}: j\geq 0\} \in  l_{s+3}^{2} $ and pick any integer $m\geq 0$. Firstly, we use similar computations to the ones derived in Lemma \ref{LemmaDifferentialofMModel1CW}, to obtain 
\begin{align*}
	&  \left( d \langle \mathfrak{f}^{(3)} \rangle(\xi )  [h]\right)^m  = - \frac{3  \mathfrak{K}_{\gamma}^2  }{8 }   \sum_{ i}  \mathfrak{C}_{i\gamma \gamma m} h^i \sum_{\pm} \mathds{1}(\ssomega_{i} \pm \ssomega_{\gamma}\pm \ssomega_{\gamma}  \pm \ssomega_{m}=0) \\
& =- \frac{3  \mathfrak{K}_{\gamma}^2   }{8 } \left[  \sum_{ i}   \mathfrak{C}_{i\gamma \gamma m} h^i \mathds{1}( i=m ) +\sum_{ i}   \mathfrak{C}_{i\gamma \gamma m} h^i \mathds{1}( i=m  ) 
+\sum_{ i}   \mathfrak{C}_{i\gamma \gamma m} h^i \mathds{1}( i=2\gamma -m\geq 0)  \right] \\
& =- \frac{3  \mathfrak{K}_{\gamma}^2   }{8 } \left[ 2     \mathfrak{C}_{m\gamma \gamma m} h^m    
+   \mathfrak{C}_{2\gamma -m,\gamma ,\gamma ,m} h^{2\gamma -m} \mathds{1}( 0\leq m \leq 2\gamma )  \right]   ,
\end{align*}
where we used the fact that $\mathfrak{C}_{ijkm}=0$ for $\ssomega_{i}\pm\ssomega_{j}\pm\ssomega_{k}\pm\ssomega_{m}=0$ with only 1 minus sign according to Lemma \ref{LemmaVanisginFourierModel3YM}, that is
\begin{align*} 
	\begin{dcases}
		-\ssomega_{i}+\ssomega_{\gamma}+ \ssomega_{\gamma}+ \ssomega_{m} =0, \\
		\ssomega_{i}-\ssomega_{\gamma}+ \ssomega_{\gamma}+ \ssomega_{m} =0, \\
		\ssomega_{i}+\ssomega_{\gamma}- \ssomega_{\gamma}+ \ssomega_{m} =0, \\
		\ssomega_{i}+\ssomega_{\gamma}+ \ssomega_{\gamma}- \ssomega_{m} =0,
	\end{dcases}
\end{align*}
so we are left with $\ssomega_{i}\pm\ssomega_{j}\pm\ssomega_{k}\pm\ssomega_{m}=0$ with only 2 minus signs and there are three such terms in total, that is
\begin{align*}
\begin{dcases}
	\ssomega_{i}+\ssomega_{\gamma}- \ssomega_{\gamma}- \ssomega_{m} =0, \\
	\ssomega_{i}-\ssomega_{\gamma}+ \ssomega_{\gamma}- \ssomega_{m} =0, \\
	\ssomega_{i}-\ssomega_{\gamma}- \ssomega_{\gamma}+ \ssomega_{m} =0, \\
	\end{dcases} \Longleftrightarrow
	\begin{dcases}
	i=m, \\
	i=m, \\
	i=2\gamma-m \text{ and } i \geq 0. \\
	\end{dcases}
\end{align*}

Then, we infer
\begin{align*}
 &	-\left( d \mathfrak{A} \xi [h] \right)^m+     \left( d \langle \mathfrak{f}^{(3)}  \rangle( \xi)[h] \right)^m = - \ssomega_{m}^2 h^m +\left( d \langle \mathfrak{f}^{(3)}  \rangle(\xi )  [h] \right)^m\\ 
 & =  -\ssomega_{m}^2 h^m - \frac{3  \mathfrak{K}_{\gamma}^2   }{8 }  \left[ 2 \mathfrak{C}_{\gamma \gamma m m}     h^m  
+\mathds{1} \left(0 \leq m \leq 2\gamma \right)
 \mathfrak{C}_{\gamma,2\gamma-m,\gamma,m}   
	  h^{2\gamma-m }   \right] \\ 
	& =-\left[ \ssomega_{m}^2 + \frac{3}{4 }  \mathfrak{K}_{\gamma}^2    \mathfrak{C}_{\gamma \gamma mm} 
   \right] h^m  -\mathds{1} \left(0 \leq m \leq 2\gamma \right) 
 \frac{3  }{8 } \mathfrak{K}_{\gamma}^2    \mathfrak{C}_{\gamma,2\gamma-m,\gamma,m}
	  h^{2\gamma-m }   . 
\end{align*} 
Recall that the differential of $\mathfrak{F}_0$ at the 1--modes is given by Lemma \ref{LemmaDifferentialMathFrakF0LongComputation}, that is
\begin{align*}
		\left(	d  \mathfrak{F}_{0}( \xi)[h] \right)^m & =\mathfrak{K}_{\gamma}^{2}(\ssomega) \Big[  \mathfrak{a}_{\gamma m}   h^{m} +  \mathds{1}(0\leq m\leq 2\gamma) \mathfrak{b}_{\gamma m} h^{2\gamma-m}  \Big].
	\end{align*}
Putting all together yields that
 \begin{align*}
	(	d  \textswab{M}_{-}( \xi )[h] )^m &= - \left( d \mathfrak{A} \xi [h] \right)^m+     \left( d \langle \mathfrak{f}^{(3)}  \rangle( \xi)[h] \right)^m + \left(d\mathfrak{F}_{0}( \xi)[h] \right)^m 
\end{align*} 
  is given by
 \begin{align*}
 	&- h^m	\left[
 	\ssomega_{m}^2 + \mathfrak{K}_{\gamma}^2(\omega)\left(
 	\frac{3}{4} \mathfrak{C}_{\gamma \gamma mm} -  \mathfrak{a}_{\gamma m} 
 	\right)
 	\right]  -  
 	\mathds{1}(0\leq m \leq 2\gamma)  h^{2\gamma-m} \mathfrak{K}_{\gamma}^2(\omega)\left[
 	\frac{3}{8} \mathfrak{C}_{\gamma ,2\gamma-m,\gamma, m} -\mathfrak{b}_{\gamma m}  \right]  .
 \end{align*}
 Finally, we rewrite this as follows
\begin{align*}
 -\mathfrak{K}_{\gamma}^{-2} (	d  \textswab{M}_{-}( \xi )[h] )^m=	 \mathfrak{u}_{\gamma m}  h^m	 +  
 	\mathds{1}(0\leq m \leq 2\gamma) \mathfrak{v}_{\gamma m}  h^{2\gamma-m} , 
 	 \end{align*}
where we set
	\begin{align*}
	\mathfrak{u}_{\gamma m}  :=    
 	\left(\frac{\ssomega_{m}}{  \mathfrak{K}_{\gamma}} \right)^{2}+\frac{3}{4} \mathfrak{C}_{\gamma \gamma mm} -  \mathfrak{a}_{\gamma m} 
 	 ,\Hquad
	\mathfrak{v}_{\gamma m}  :=  
 	\frac{3}{8} \mathfrak{C}_{\gamma ,2\gamma-m,\gamma, m} -\mathfrak{b}_{\gamma m}.  
\end{align*}
or equivalently as
\begin{align*}
 	 -\mathfrak{K}_{\gamma}^{-2} (	d  \textswab{M}_{-}( \xi )[h] )^m &= \mathds{1}(0\leq m\leq \gamma-1)\Bigg[
 	h^m	\mathfrak{u}_{\gamma m}  
 	+ h^{2\gamma-m}  \mathfrak{v}_{\gamma m}   \Bigg] 
 	+\mathds{1}( m= \gamma)\Bigg[
 	h^{\gamma}	\left(  
 	\mathfrak{u}_{\gamma \gamma}  
  +\mathfrak{v}_{\gamma \gamma} )\right) 
 	  	\Bigg] \\
 	 &+\mathds{1}(\gamma+1\leq m\leq 2\gamma)\Bigg[
 	h^m	\mathfrak{u}_{\gamma m}   +
 	h^{2\gamma-m}  \mathfrak{v}_{\gamma m}   
 	\Bigg]  +\mathds{1}( m \geq  2\gamma+1)\Bigg[
  h^m	\mathfrak{u}_{\gamma m}  
  \Bigg],
\end{align*}
 that completes the proof. 
\end{proof}

\section{Non--degeneracy conditions for the 1--modes}

\label{SectionNonDegeneracyCondition}
In this section, we derive and establish the crucial non--degeneracy conditions for 1--mode initial data according to Theorem \ref{OrigivalversionTheoremBambusi} (for CW and CH) and Theorem \ref{TheoremModificationofBambusiPaleari} (for YM).   

\subsection{Conformal cubic wave equation in spherical symmetry}
 Firstly, we consider the conformal cubic wave equation in spherical symmetry and derive of the non-degeneracy condition for the 1--modes.	
		
\begin{lemma}[CW model: Derivation of the non-degeneracy condition for the 1--modes]\label{LemmaDerivationNonGegeneracyConditionCWModel1}
	Let $\gamma \geq 0$ be any integer and define $ \xi $ to be the rescaled 1--mode according to \eqref{Definition1modesModel1CW}--\eqref{DefinitionKappa1modesModel1CW}. Then, the non--degeneracy condition 
	\begin{align*}
		 \ker \left( d \mathcal{M}( \xi ) \right) = \{ 0 \}
	\end{align*}
	is equivalent to
	\begin{align}\label{NonDegeneracyConditionCWModel1}
		\begin{dcases} 
	\omega_{m}^2 C_{\gamma \gamma \gamma \gamma}-  2\omega_{\gamma}^2 C_{\gamma \gamma mm}   \neq 0, \text{ for all }   m \geq 2\gamma+1, \\
D_{\gamma n}  \neq 0 ,  \text{ for all }   n \in \{0,1,\dots, \gamma -1 \},
		\end{dcases}
	\end{align}
	where
	\begin{align*}
		D_{\gamma n}   :=
	 \left[ \omega_{n}^2 C_{\gamma \gamma \gamma \gamma}-  2\omega_{\gamma}^2 C_{\gamma \gamma nn} \right]\left[ \omega_{2\gamma-n}^2 C_{\gamma \gamma \gamma \gamma}-  2\omega_{\gamma}^2 C_{\gamma ,\gamma ,2\gamma-n,2\gamma-n}  \right]- \left[ \omega_{\gamma}^2 C_{\gamma, 2\gamma-n, \gamma,n} \right]^2.
	\end{align*}
\end{lemma}

\begin{proof}
Let $\gamma \geq 0$ be any integer and define $ \xi $ to be the rescaled 1--mode according to \eqref{Definition1modesModel1CW}--\eqref{DefinitionKappa1modesModel1CW}. Furthermore, pick any $h=\{h^j: j \geq 0\}\in  l_{s+3}^{2}  $ such that $d \mathcal{M}( \xi )[h]=0$ and fix any integer $m \geq 0 $. Then, according to Lemma \ref{LemmaDifferentialofMModel1CW}, we have that the system
\begin{align}
	&  \left(\omega_{m}^2 C_{\gamma \gamma \gamma \gamma}-  2\omega_{\gamma}^2 C_{\gamma \gamma mm}  \right)h^m   -   \omega_{\gamma}^2 C_{\gamma, 2\gamma-m, \gamma,m}   h^{2\gamma-m}=0, \Hquad \text {for } 0 \leq m \leq \gamma-1, \label{Model1CWEquation1}\\
	&  \left(\omega_{m}^2 C_{\gamma \gamma \gamma \gamma}-  2\omega_{\gamma}^2 C_{\gamma \gamma mm}  \right)h^m   -  \omega_{\gamma}^2 C_{\gamma, 2\gamma-m, \gamma,m}   h^{2\gamma-m}=0,\Hquad \text {for } \gamma+1 \leq m \leq 2\gamma,\label{Model1CWEquation2} \\
	& \left( \omega_{m}^2 C_{\gamma \gamma \gamma \gamma}-  2\omega_{\gamma}^2 C_{\gamma \gamma mm}  \right)h^m =0,\Hquad \text {for }   m \geq 2\gamma+1, \label{Model1CWEquation3}
\end{align}
coupled to
\begin{align}
	 \omega_{\gamma}^2 C_{\gamma \gamma \gamma \gamma} h^{\gamma} & = 0, 
\end{align}
has $h=\{h^i: i\geq 0\}=0$ as the unique solution. Firstly, \eqref{Model1CWEquation1} yields  $h^{\gamma}=0 $  due to the fact that $C_{\gamma \gamma \gamma \gamma} \neq 0 $ and $ \omega_{\gamma}  \neq 0$ for all $\gamma \geq 0$, whereas, for \eqref{Model1CWEquation3}, we must have
\begin{align*}
	 \omega_{m}^2 C_{\gamma \gamma \gamma \gamma}-  2\omega_{\gamma}^2 C_{\gamma \gamma mm}   \neq 0,
\end{align*}
for all $m \geq 2\gamma+1$. Next, we rearrange \eqref{Model1CWEquation1} and \eqref{Model1CWEquation2} by setting $m=n$ and $m=2\gamma-n$ respectively and obtain
\begin{align*}
	 \left(\omega_{n}^2 C_{\gamma \gamma \gamma \gamma}-  2\omega_{\gamma}^2 C_{\gamma \gamma nn}  \right)h^n   -   \omega_{\gamma}^2 C_{\gamma, 2\gamma-n, \gamma,n}   h^{2\gamma-n}&=0,    \\
	  \left(\omega_{2\gamma-n}^2 C_{\gamma \gamma \gamma \gamma}-  2\omega_{\gamma}^2 C_{\gamma ,\gamma ,2\gamma-n,2\gamma-n}  \right)h^{2\gamma-n}   -  \omega_{\gamma}^2 C_{\gamma, n, \gamma,2\gamma-n}   h^{n}&=0.
\end{align*}
for all $  n \in \{0,1,\dots, \gamma -1 \}$, that can be written in the matrix form
\begin{align*}
	\begin{bmatrix}
\omega_{n}^2 C_{\gamma \gamma \gamma \gamma}-  2\omega_{\gamma}^2 C_{\gamma \gamma nn} & -   \omega_{\gamma}^2 C_{\gamma, 2\gamma-n, \gamma,n} \\
  -  \omega_{\gamma}^2 C_{\gamma, n, \gamma,2\gamma-n}   & \omega_{2\gamma-n}^2 C_{\gamma \gamma \gamma \gamma}-  2\omega_{\gamma}^2 C_{\gamma ,\gamma ,2\gamma-n,2\gamma-n} 
\end{bmatrix}
	\begin{bmatrix}
h^n   \\
 h^{2\gamma-n}  
\end{bmatrix}=
	\begin{bmatrix}
0  \\
0  
\end{bmatrix},
\end{align*}
for all $  n \in \{0,1,\dots, \gamma -1 \}$. Observe that these are in total $\gamma$ $(2 \times 2)-$linear systems where the unknowns are $h^{m}$ for $m \in  \{0,1,\cdots,2\gamma\}\setminus \{\gamma\} $. Finally, these systems have only the trivial solution $h^{m}=0$ for all $m \in \{0,1,\cdots,2\gamma\}\setminus \{\gamma\} $ if and only if the determinants 
\begin{align*}
	 D_{\gamma n} &  =
	 \left[ \omega_{n}^2 C_{\gamma \gamma \gamma \gamma}-  2\omega_{\gamma}^2 C_{\gamma \gamma nn} \right]\left[ \omega_{2\gamma-n}^2 C_{\gamma \gamma \gamma \gamma}-  2\omega_{\gamma}^2 C_{\gamma ,\gamma ,2\gamma-n,2\gamma-n}  \right]- \left[ \omega_{\gamma}^2 C_{\gamma, 2\gamma-n, \gamma,n} \right]^2
\end{align*}
are non--zero for all $  n \in \{0,1,\dots, \gamma -1 \}$, that completes the proof.
\end{proof} 

Next, we establish the non--degeneracy condition for this model.

\begin{prop}[Non-degeneracy condition for the 1--modes and the CW model]
	Let $\gamma \geq 0$ be any integer. Then, the non--degeneracy condition  \eqref{NonDegeneracyConditionCWModel1} holds true. 
\end{prop}

\begin{proof}
	Let $\gamma \geq 0$ be any integer. Also, pick any integers $m \geq 2\gamma+1$ and $n \in \{0,1,\dots,\gamma-1\}$. Then, according to Lemma \ref{LemmaClosedformulaFourierModel1CW}, we have
\begin{align*}
	C_{ijkm}= \omega_{\min \{i,j,k,m\}} 
\end{align*}
provided that either
\begin{align*}	  
	  & \omega_{i} + \omega_{j} - \omega_{k} - \omega_{m} = 0, \\
	  & \omega_{i} - \omega_{j} + \omega_{k} - \omega_{m} = 0, \\
	   & \omega_{i} - \omega_{j} - \omega_{k} + \omega_{m} = 0.  
\end{align*}
One can easily show that all the indices $(i,j,k,m)$ of the Fourier coefficients that appear in Lemma \ref{LemmaDerivationNonGegeneracyConditionCWModel1} satisfy at least one of these conditions and hence we infer
\begin{align*}
	C_{\gamma\gamma\gamma\gamma} &= \omega_{\gamma}  , \\
	C_{\gamma\gamma nn} &=\omega_{n}   , \\
	C_{\gamma\gamma mm} &=  \omega_{\gamma}    , \\
	C_{\gamma,\gamma,2\gamma-n,2\gamma-n} &=   \omega_{\gamma}   , \\
	C_{\gamma,2\gamma-n,\gamma,n} &=  \omega_{n}   .
\end{align*}
Putting all together, yields
\begin{align*}    
	& \omega_{m}^2 C_{\gamma \gamma \gamma \gamma}-  2\omega_{\gamma}^2 C_{\gamma \gamma mm} = \omega_{\gamma} \left(\omega_{m}^2 -2 \omega_{\gamma}^2 \right) \geq \omega_{\gamma} \left(\omega_{2\gamma+1}^2 -2 \omega_{\gamma}^2 \right) =2\omega_{\gamma}^3  \geq 2,  \\
 & D_{\gamma n} = \omega_{n}\omega_{\gamma}^2  (n-3-4 \gamma ) (n-\gamma )^2  \neq 0 ,  	
\end{align*}
for all $m \geq 2\gamma+1$ and $n \in \{0,1,\dots,\gamma-1\}$, that competes the proof.
\end{proof}

\subsection{Conformal cubic wave equation out of spherical symmetry}
 Next, we consider the conformal cubic wave equation out of spherical symmetry and show that the non--degeneracy condition is a condition on the Fourier coefficients.		
		
\begin{lemma}[Derivation of the non-degeneracy condition for the 1--modes and the CH model]\label{LemmaDerivationNonGegeneracyConditionCHModel2}
	Let $\gamma$ and $\mu_1,\mu_2 $ be any integers and define $ \xi $ according to \eqref{Definition1modesModel2CH}--\eqref{DefinitionKappa1modesModel2CH}. Then, the non--degeneracy condition 
	\begin{align*}
		 \ker \left( d \mathcal{M}( \xi ) \right) = \{ 0 \}
	\end{align*}
	is equivalent to
	\begin{align}\label{NonDegeneracyConditionCHModel2}
		\begin{dcases} 
	\left(\somega_{m}^{(\mu_1,\mu_2)} \right)^2 \mathsf{C}_{\gamma \gamma \gamma \gamma}^{(\mu_1,\mu_2)}-  2 \left(\somega_{\gamma}^{(\mu_1,\mu_2)} \right)^2 \mathsf{C}_{\gamma \gamma mm}^{(\mu_1,\mu_2)}   \neq 0, \text{ for all }   m \geq 2\gamma+1, \\
\mathsf{D}_{\gamma n}^{(\mu_1,\mu_2)}  \neq 0 ,  \text{ for all }   n \in \{0,1,\dots, \gamma -1 \},
		\end{dcases}
	\end{align}
	where
	\begin{align*}
		\mathsf{D}_{\gamma n}^{(\mu_1,\mu_2)}  & :=
	 \left[ \left( \somega_{n}^{(\mu_1,\mu_2)} \right)^2 \mathsf{C}_{\gamma \gamma \gamma \gamma}^{(\mu_1,\mu_2)} -  2 \left(\somega_{\gamma}^{(\mu_1,\mu_2)} \right)^2 \mathsf{C}_{\gamma \gamma nn}^{(\mu_1,\mu_2)} \right]\cdot \\
	 & \cdot \left[ \big(\somega_{2\gamma-n}^{(\mu_1,\mu_2)} \big)^2 \mathsf{C}_{\gamma \gamma \gamma \gamma}^{(\mu_1,\mu_2)}  -  2 \left( \somega_{\gamma}^{(\mu_1,\mu_2)} \right)^2 \mathsf{C}_{\gamma ,\gamma ,2\gamma-n,2\gamma-n}^{(\mu_1,\mu_2)}  \right] 
	  - \left[ \left( \somega_{\gamma}^{(\mu_1,\mu_2)} \right)^2 \mathsf{C}_{\gamma, 2\gamma-n, \gamma,n}^{(\mu_1,\mu_2)} \right]^2.
	\end{align*}
\end{lemma}

\begin{proof}
The proof is similar to the one of Lemma \ref{LemmaDerivationNonGegeneracyConditionCWModel1}.
\end{proof}

Next, we establish the non--degeneracy condition for this model.

\begin{prop}[Non-degeneracy condition for the 1--modes and the CH model]
	Let $\gamma,\mu_1,\mu_2 \geq 0$ be any integers with $\gamma \in \{0,1,2,3,4,5\}$ and $\mu_1=\mu_2=:\mu$ where $\mu$ is either sufficiently small with $\mu \in \{0,1,2,3,4,5\}$ or sufficiently large. Then, the non--degeneracy condition \eqref{NonDegeneracyConditionCHModel2} holds true. 
\end{prop}

\begin{proof}
	Let $\gamma,\mu_1,\mu_2 \geq 0$ be any integers with $\gamma \in \{0,1,2,3,4,5\}$ and $\mu_1=\mu_2=:\mu$ where $\mu$ is either sufficiently small with $\mu \in \{0,1,2,3,4,5\}$ or sufficiently large. Also, pick any integer $m \geq 2\gamma+1$ and, according to Lemmata \ref{LemmaUniformAsymptoticFormulasModel2CH} and \ref{LemmaMonotonicityA}, we have that
\begin{align*}
	\mathsf{C}_{\gamma \gamma \gamma\gamma}^{(\mu,\mu)}  
	  & =  \frac{1}{2}      \sum_{\lambda=0}^{\gamma} \mathsf{M}_{ \gamma}^{(\mu)} (\lambda) \mathsf{M}_{ \gamma}^{(\mu)} (\lambda) \xi_{\lambda}(\mu), \\\mathsf{C}_{\gamma \gamma mm}^{(\mu,\mu)}  
	  & =  \frac{1}{2}      \sum_{\lambda=0}^{\gamma} \mathsf{M}_{ \gamma}^{(\mu)} (\lambda) \mathsf{M}_{ m}^{(\mu)} (\lambda) \xi_{\lambda}(\mu),
\end{align*}
where the function
\begin{align*} 
	\mathsf{M}_{ m}^{(\mu)} (\lambda) &=\frac{1}{2 \pi ^{3/2}} \frac{(4 \lambda +4 \mu +1) \Gamma \left(\lambda +\frac{1}{2}\right) \Gamma \left(2 \mu +\frac{1}{2}\right) \Gamma \left(\lambda +\mu +\frac{1}{2}\right)}{\Gamma (\lambda +\mu +1) \Gamma (\lambda +2 \mu +1)} \cdot \\
	& \cdot \frac{(2 \mu +2 m+1) \Gamma \left(m-\lambda +\frac{1}{2}\right) \Gamma (m+\lambda +2 \mu +1)}{\Gamma (m-\lambda +1) \Gamma \left(m+\lambda +2 \mu +\frac{3}{2}\right)}
\end{align*}	
is decreasing with respect to $m$. In addition, recall that the eigenvalues are given by
\begin{align*}
	\left(  \somega_{m}^{(\mu,\mu)} \right)  ^2:=(2m+1+2\mu)^2 
\end{align*}
and they are clearly increasing with respect to $m \geq 0$. In other words, the function
\begin{align*}
  \mathsf{P}_{m}^{(\mu)}(\lambda):=\frac{\mathsf{M}_{ m}^{(\mu)} (\lambda)}{\big( \somega_{m}^{(\mu,\mu)} \big)^2}  
\end{align*}
is decreasing with respect to $m$ for all $m \geq \gamma$ as a product of two positive and decreasing functions. In the following, we show that
\begin{align*}  
	\big( \somega_{m}^{(\mu,\mu)} \big)^2 \mathsf{C}_{\gamma \gamma \gamma \gamma}^{(\mu,\mu)}-   2 \left( \somega_{\gamma}^{(\mu,\mu)} \right)^2 \mathsf{C}_{\gamma \gamma mm}^{(\mu,\mu)}
\end{align*}
stays away from zero for all $m \geq 2\gamma+1$ and $\gamma \in \{0,1,2,3,4,5\}$ provided that $\mu$ is either sufficiently small, say $\mu \in \{1,2,\dots,10\}$ or sufficiently large. To this end, we use the monotonicity of $\mathsf{P}_{m}^{(\mu)}(\lambda)$ with respect to $m$ to infer
\begin{align*} 
&	\big( \somega_{m}^{(\mu,\mu)} \big)^2 \mathsf{C}_{\gamma \gamma \gamma \gamma}^{(\mu,\mu)}-   2 \left( \somega_{\gamma}^{(\mu,\mu)} \right)^2 \mathsf{C}_{\gamma \gamma mm}^{(\mu,\mu)} = \\
&	\big( \somega_{m}^{(\mu,\mu)} \big)^2\big( \somega_{\gamma}^{(\mu,\mu)} \big)^2 \left[\frac{\mathsf{C}_{\gamma \gamma \gamma \gamma}^{(\mu,\mu)}}{\big( \somega_{\gamma}^{(\mu,\mu)} \big)^2 } -2 \frac{\mathsf{C}_{\gamma \gamma mm }^{(\mu,\mu)}}{\big( \somega_{m}^{(\mu,\mu)} \big)^2 } \right]= \\
& \frac{1}{2} 	\big( \somega_{m}^{(\mu,\mu)} \big)^2\big( \somega_{\gamma}^{(\mu,\mu)} \big)^2 \left[     \sum_{\lambda=0}^{\gamma}\frac{ \mathsf{M}_{ \gamma}^{(\mu)} (\lambda)}{\big( \somega_{\gamma}^{(\mu,\mu)} \big)^2 } \mathsf{M}_{ \gamma}^{(\mu)} (\lambda) \xi_{\lambda}(\mu) -2 \sum_{\lambda=0}^{\gamma} \mathsf{M}_{ \gamma}^{(\mu)} (\lambda) \frac{ \mathsf{M}_{ m}^{(\mu)} (\lambda)}{\big( \somega_{m}^{(\mu,\mu)} \big)^2 }  \xi_{\lambda}(\mu)\right]= \\
& \frac{1}{2} 	\big( \somega_{m}^{(\mu,\mu)} \big)^2\big( \somega_{\gamma}^{(\mu,\mu)} \big)^2 \left[     \sum_{\lambda=0}^{\gamma}\mathsf{P}_{\gamma}^{(\mu)}(\lambda) \mathsf{M}_{ \gamma}^{(\mu)} (\lambda) \xi_{\lambda}(\mu) -2 \sum_{\lambda=0}^{\gamma} \mathsf{M}_{ \gamma}^{(\mu)} (\lambda) \mathsf{P}_{m}^{(\mu)}(\lambda) \xi_{\lambda}(\mu)\right]= \\
& \frac{1}{2} 	\big( \somega_{m}^{(\mu,\mu)} \big)^2\big( \somega_{\gamma}^{(\mu,\mu)} \big)^2  \sum_{\lambda=0}^{\gamma} \mathsf{M}_{ \gamma}^{(\mu)} (\lambda)\left[    \mathsf{P}_{\gamma}^{(\mu)}(\lambda)    -2    \mathsf{P}_{m}^{(\mu)}(\lambda)\right] \xi_{\lambda}(\mu) \geq  \\
& \frac{1}{2} 	\big( \somega_{m}^{(\mu,\mu)} \big)^2\big( \somega_{\gamma}^{(\mu,\mu)} \big)^2  \sum_{\lambda=0}^{\gamma} \mathsf{M}_{ \gamma}^{(\mu)} (\lambda)\left[    \mathsf{P}_{\gamma}^{(\mu)}(\lambda)    -2    \mathsf{P}_{2\gamma+1}^{(\mu)}(\lambda)\right] \xi_{\lambda}(\mu)  = 
	\big( \somega_{m}^{(\mu,\mu)} \big)^2 \mathsf{S}_{\gamma}^{(\mu)},
\end{align*}
where we set
\begin{align*}
	\mathsf{S}_{\gamma}^{(\mu)}:=\frac{1}{2} \big( \somega_{\gamma}^{(\mu,\mu)} \big)^2  \sum_{\lambda=0}^{\gamma} \mathsf{M}_{ \gamma}^{(\mu)} (\lambda)\left[    \mathsf{P}_{\gamma}^{(\mu)}(\lambda)    -2    \mathsf{P}_{2\gamma+1}^{(\mu)}(\lambda)\right] \xi_{\lambda}(\mu).
\end{align*}
On the one hand,  for all $\gamma \in \{0,1,2,3,4,5 \}$ and $\mu \in \{0,1,2,3,4,5\}$ we compute $\mathsf{S}_{\gamma}^{(\mu)}$ and verify that all  $\mathsf{S}_{\gamma}^{(\mu)}$ are strictly positive. On the other hand, for all $\gamma \in \{0,1,\dots,10 \}$ and sufficiently large $\mu$, we firstly compute $\mathsf{S}_{\gamma}^{(\mu)}$ in terms of $\mu$ and then derive its asymptotic expansion as $\mu \longrightarrow \infty$. For $\gamma=0$, we find
\begin{align*}
\mathsf{S}_{0}^{(\mu)}=	\frac{4^{\mu } (2 \mu +1) (10 \mu +7) \Gamma \left(\mu +\frac{1}{2}\right)^2 \Gamma \left(\mu +\frac{5}{2}\right)}{\pi  (2 \mu +3)^2 \Gamma (\mu +1) \Gamma \left(2 \mu +\frac{5}{2}\right)}  ,
\end{align*}
that is strictly positive, for all $\mu \geq 0$, and for $\gamma \in \{1,2,3,4,5\}$, we expand 
\begin{align*}
	\mathsf{S}_{\gamma}^{(\mu)}=  
		 \sigma_{\gamma} \mu ^{\frac{1}{2}}+ \mathcal{O} \left(\mu^{-\frac{1}{2}} \right), 
\end{align*} 
as $\mu \longrightarrow \infty$, for some strictly positive constants $\sigma_{\gamma}$. Figure \ref{Lemma94Pic3Ref} illustrates the constants $\sigma_{\gamma}$ as $\gamma $ varies within $ \{1,2,\dots,10\}$. 
\begin{figure}[h!]
\vspace{0.5cm}
    \centering
    \includegraphics[width=0.7\textwidth]{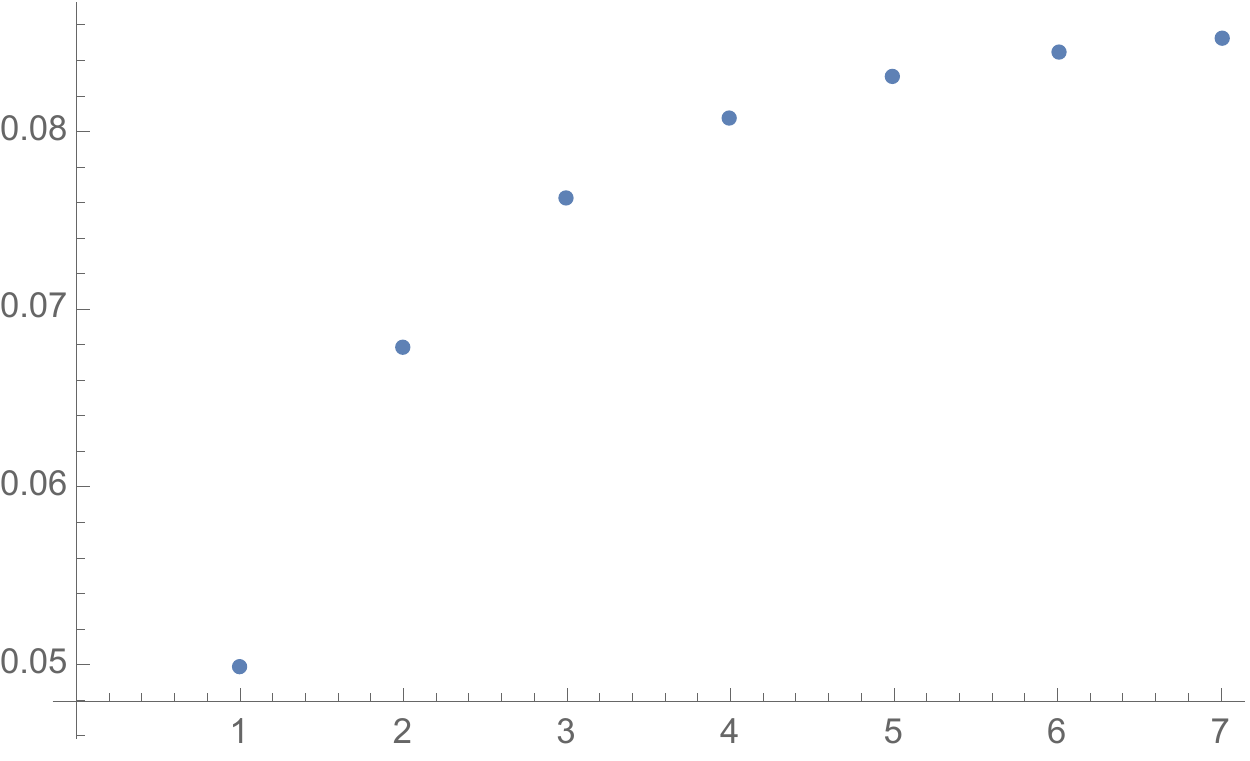}
    \caption{The constants $\sigma_{\gamma}$ as $\gamma $ varies within $ \{1,2,\dots,10\}$. They are all strictly positive. }
    \label{Lemma94Pic3Ref}
\end{figure}  
Consequently, in both cases, we can ensure that $\mathsf{S}_{\gamma}^{(\mu)} >0$ and hence we conclude that
\begin{align*} 
&	\big( \somega_{m}^{(\mu,\mu)} \big)^2 \mathsf{C}_{\gamma \gamma \gamma \gamma}^{(\mu,\mu)}-   2 \left( \somega_{\gamma}^{(\mu,\mu)} \right)^2 \mathsf{C}_{\gamma \gamma mm}^{(\mu,\mu)}  \geq  	\big( \somega_{m}^{(\mu,\mu)} \big)^2 \mathsf{S}_{\gamma}^{(\mu)} \geq \big( \somega_{2\gamma+1}^{(\mu,\mu)} \big)^2 \mathsf{S}_{\gamma}^{(\mu)}>0,
\end{align*}
for all $m \geq 2\gamma+1$. Finally, it remains to show that the determinants 
\begin{align*}
		\mathsf{D}_{\gamma n}^{(\mu,\mu)}  & :=
	 \left[ \left( \somega_{n}^{(\mu,\mu)} \right)^2 \mathsf{C}_{\gamma \gamma \gamma \gamma}^{(\mu,\mu)} -  2 \left(\somega_{\gamma}^{(\mu,\mu)} \right)^2 \mathsf{C}_{\gamma \gamma nn}^{(\mu,\mu)} \right]\cdot \\
	 & \cdot \left[ \big(\somega_{2\gamma-n}^{(\mu,\mu)} \big)^2 \mathsf{C}_{\gamma \gamma \gamma \gamma}^{(\mu,\mu)}  -  2 \left( \somega_{\gamma}^{(\mu,\mu)} \right)^2 \mathsf{C}_{\gamma ,\gamma ,2\gamma-n,2\gamma-n}^{(\mu,\mu)}  \right] 
	  - \left[ \left( \somega_{\gamma}^{(\mu,\mu)} \right)^2 \mathsf{C}_{\gamma, 2\gamma-n, \gamma,n}^{(\mu,\mu)} \right]^2.
	\end{align*}
are all non--zero for $ n \in \{0,1,\dots, \gamma -1 \}$. To this end, for all $\gamma \in \{1,2,3,4,5\}$ and $ n \in \{0,1,\dots,   \gamma -1 \}$, we firstly compute each of the Fourier coefficients in the determinants above in terms of $\mu$ and then either compute the determinants when $\mu$ is sufficiently small $\mu \in \{0,1,2,3,4,5\}$ or compute their limit as $\mu \longrightarrow \infty$. For example, for $\gamma=1$, we have $n=0$ and compute
\begin{align*}
	\mathsf{D} _{10}^{(\mu,\mu)} &= -\frac{3\ 16^{\mu -1}}{\pi ^2}(\mu +1) (2 \mu +3)^4 (2 \mu +5) (4 \mu +7)\cdot \\
	& \cdot \left(20 \mu ^4+328 \mu ^3+1029 \mu ^2+1155 \mu +435\right) \frac{\Gamma \left(\mu +\frac{1}{2}\right)^2 \Gamma \left(\mu +\frac{3}{2}\right)^3 \Gamma \left(\mu +\frac{3}{2}\right)}{\Gamma (\mu +2)^2 \Gamma \left(2 \mu +\frac{9}{2}\right)^2}
\end{align*}
and it is clearly strictly negative for all integers $\mu \geq 0$. For all $\gamma \in \{1,2,\dots,10\}$ and $n \in \{0,1,\dots,\gamma-1\}$, we find either $\mathsf{D} _{\gamma n}^{(\mu,\mu)} < 0$ when $\mu$ is sufficiently small and $\mathsf{D} _{\gamma n}^{(\mu,\mu)} \longrightarrow \pm \infty$ when $\mu \longrightarrow \infty$, that completes the proof.
\end{proof}

\subsection{Yang--Mills equation in spherical symmetry} 
 Finally, we consider the Yang--Mills equation in spherical symmetry and show that the non--degeneracy condition is a condition on the Fourier coefficients.		
		
\begin{lemma}[Derivation of the non-degeneracy condition for the 1--modes and the YM model]\label{LemmaDerivationNonGegeneracyConditionYMModel3}
	Let $\gamma \in \{0,1,2,3,4,5\}$ and define $ \xi $ according to \eqref{Definition1modesModel3YM}--\eqref{DefinitionKappa1modesModel3YM}. Then, the non--degeneracy condition 
	\begin{align*}
		 \ker \left( d  \textswab{M}_{-}( \xi ) \right) =\{ 0 \}
	\end{align*}
	is equivalent to
	\begin{align}\label{NonDegeneracyConditionCHModel3}
		\begin{dcases} 
	  \mathfrak{u}_{\gamma \gamma}   + \mathfrak{v}_{\gamma \gamma}  \neq 0,\\
	  \mathfrak{u}_{\gamma m}  \neq 0, \text{ for all }   m \geq 2\gamma+1  , \\
\mathfrak{D}_{\gamma n}    \neq 0 ,  \text{ for all }   n \in \{0,1,\dots, \gamma -1 \} ,
		\end{dcases}
	\end{align}
	where
	\begin{align*}
		\mathfrak{D}_{\gamma n}  & : =
	 \mathfrak{u}_{\gamma n} \mathfrak{u}_{\gamma ,2\gamma-n} - \mathfrak{v}_{\gamma n} \mathfrak{v}_{\gamma ,2\gamma-n}   
	\end{align*}
	and
\begin{align*}
	\mathfrak{u}_{\gamma m}  &:=    
 	  \left(\frac{\ssomega_{m}}{  \mathfrak{K}_{\gamma}} \right)^{2}+\frac{3}{4} \mathfrak{C}_{\gamma \gamma mm}  - \frac{9}{2}   \sum_{ \substack{ \nu = 0  }}^{m+\gamma}    \frac{\left(\overline{\mathfrak{C}}_{\gamma \nu m} \right)^2 }{\ssomega_{\nu}^2-(\ssomega_{m}+\ssomega_{\gamma})^2 }-  \frac{9}{4}  \sum_{\substack{ \nu = 0  } }^{2\gamma}   
		 \frac{\overline{\mathfrak{C}}_{m \nu m}\overline{\mathfrak{C}}_{\gamma\gamma \nu }}{\ssomega_{\nu}^2 }  \\
 	 & -\frac{9}{2}  \sum_{ \substack{ \nu = 0 \\ \nu \neq \pm(m-\gamma)-2  }}^{m+\gamma}   \frac{\left(\overline{\mathfrak{C}}_{m\gamma \nu } \right)^2}{\ssomega_{\nu}^2-(\ssomega_{m}-\ssomega_{\gamma})^2  }   , \\
	\mathfrak{v}_{\gamma m}  &:=   \frac{3}{8} \mathfrak{C}_{\gamma ,2\gamma-m,\gamma, m} -\frac{9}{4}   \sum_{\substack{ \nu = 0  } }^{2\gamma}  
		 \frac{\overline{\mathfrak{C}}_{2\gamma-m, \nu, m} \overline{\mathfrak{C}}_{\gamma\gamma \nu }}{\ssomega_{\nu}^2-(2\ssomega_{\gamma})^2  }  -\frac{9}{2} \sum_{ \substack{ \nu = 0 \\ \nu \neq \pm(m-\gamma)-2  }}^{m+\gamma}   \frac{\overline{\mathfrak{C}}_{\gamma \nu m} \overline{\mathfrak{C}}_{2\gamma-m,\gamma ,\nu }}{\ssomega_{\nu}^2-(\ssomega_{2\gamma-m}-\ssomega_{\gamma})^2  } .
\end{align*}

\end{lemma}

\begin{proof}
	Let $\gamma \in \{0,1,2,3,4,5\}$ and define $ \xi $ according to \eqref{Definition1modesModel3YM}--\eqref{DefinitionKappa1modesModel3YM}. Furthermore, assume that $\ker ( d  \textswab{M}_{-}( \xi ) ) =\{ 0 \}$ and pick any $h=\{h^j: j \geq 0\}\in  l_{s+3}^{2}  $ such that $d  \textswab{M}_{-}( \xi )[h]=0$. Also, fix any integer $m \geq 0 $. Then, according to Lemma \ref{LemmaDifferentialofMModel3YM}, we have that the system 
\begin{align}
	&\mathfrak{u}_{\gamma m}  h^m	
 	+\mathfrak{v}_{\gamma m}  h^{2\gamma-m}    =0, \Hquad \text{ for } 0\leq m \leq \gamma-1,\label{Eq11NDCModel3YM} \\
 	&\left( \mathfrak{u}_{\gamma \gamma}   + \mathfrak{v}_{\gamma \gamma}  \right)h^{\gamma}  
 	  =0, \Hquad \text{ for } m =\gamma,\label{Eq22NDCModel3YM}  \\
 	&\mathfrak{u}_{\gamma m}  h^m	 +
 \mathfrak{v}_{\gamma m} 	h^{2\gamma-m}    =0,  \Hquad \text{ for } \gamma+1\leq m \leq 2\gamma,\label{Eq33NDCModel3YM}  \\ 
 	&	\mathfrak{u}_{\gamma m} h^m   
=0, \Hquad \text{ for }   m \geq  2\gamma+1 , \label{Eq44NDCModel3YM} 
\end{align}
has $h=\{h^i: i\geq 0\}=0$ as the unique solution, where 
\begin{align*}
	\mathfrak{u}_{\gamma m}  &:=    
 	\left(\frac{\ssomega_{m}}{  \mathfrak{K}_{\gamma}} \right)^{2}+\frac{3}{4} \mathfrak{C}_{\gamma \gamma mm} -  \mathfrak{a}_{\gamma m} 
 	 ,\\
	\mathfrak{v}_{\gamma m}  &:=  
 	\frac{3}{8} \mathfrak{C}_{\gamma ,2\gamma-m,\gamma, m} -\mathfrak{b}_{\gamma m}  
\end{align*}
Furthermore, $ \mathfrak{a}_{\gamma m} $ and $\mathfrak{b}_{\gamma m} $ are given by Lemma \ref{LemmaDifferentialMathFrakF0LongComputation} and putting all together yields
\begin{align*}
	\mathfrak{u}_{\gamma m}  &:=    
 	  \left(\frac{\ssomega_{m}}{  \mathfrak{K}_{\gamma}} \right)^{2}+\frac{3}{4} \mathfrak{C}_{\gamma \gamma mm}  - \frac{9}{2}   \sum_{ \substack{ \nu = 0  }}^{m+\gamma}    \frac{\left(\overline{\mathfrak{C}}_{\gamma \nu m} \right)^2 }{\ssomega_{\nu}^2-(\ssomega_{m}+\ssomega_{\gamma})^2 \ssomega^2}-  \frac{9}{4}  \sum_{\substack{ \nu = 0  } }^{2\gamma}   
		 \frac{\overline{\mathfrak{C}}_{m \nu m}\overline{\mathfrak{C}}_{\gamma\gamma \nu }}{\ssomega_{\nu}^2 }  \\
 	 & -\frac{9}{2}  \sum_{ \substack{ \nu = 0 \\ \nu \neq \pm(m-\gamma)-2  }}^{m+\gamma}   \frac{\left(\overline{\mathfrak{C}}_{m\gamma \nu } \right)^2}{\ssomega_{\nu}^2-(\ssomega_{m}-\ssomega_{\gamma})^2  }   , \\
	\mathfrak{v}_{\gamma m}  &:=   \frac{3}{8} \mathfrak{C}_{\gamma ,2\gamma-m,\gamma, m} -\frac{9}{4}   \sum_{\substack{ \nu = 0  } }^{2\gamma}  
		 \frac{\overline{\mathfrak{C}}_{2\gamma-m, \nu, m} \overline{\mathfrak{C}}_{\gamma\gamma \nu }}{\ssomega_{\nu}^2-(2\ssomega_{\gamma})^2  }  -\frac{9}{2} \sum_{ \substack{ \nu = 0 \\ \nu \neq \pm(m-\gamma)-2  }}^{m+\gamma}   \frac{\overline{\mathfrak{C}}_{\gamma \nu m} \overline{\mathfrak{C}}_{2\gamma-m,\gamma ,\nu }}{\ssomega_{\nu}^2-(\ssomega_{2\gamma-m}-\ssomega_{\gamma})^2  } .
\end{align*}
Now, \eqref{Eq22NDCModel3YM} and \eqref{Eq44NDCModel3YM} yield
\begin{align*}
	\mathfrak{u}_{\gamma \gamma}   + \mathfrak{v}_{\gamma \gamma}  \neq 0
\end{align*}
and
\begin{align*}
	\mathfrak{u}_{\gamma m}  \neq 0
\end{align*}
for all $m\geq 2\gamma+1$. Next, we rearrange \eqref{Eq11NDCModel3YM} and \eqref{Eq33NDCModel3YM} by setting $m=n$ and $m=2\gamma-n$ respectively to obtain
\begin{align*}
	&\mathfrak{u}_{\gamma n}  h^n	
 	+\mathfrak{v}_{\gamma n}  h^{2\gamma-n}   =0,   \\
 	&\mathfrak{u}_{\gamma ,2\gamma-n}  h^{2\gamma-n}	 +
 \mathfrak{v}_{\gamma ,2\gamma-n} 	h^{n}      =0,     
 \end{align*}
 for all $n\in \{0,1,\dots,\gamma-1\}$, that can be written in the matrix form
\begin{align*}
	\begin{bmatrix}
\mathfrak{u}_{\gamma n} & \mathfrak{v}_{\gamma n}  \\
  \mathfrak{v}_{\gamma ,2\gamma-n}    &   \mathfrak{u}_{\gamma ,2\gamma-n} 
\end{bmatrix}
	\begin{bmatrix}
h^n   \\
 h^{2\gamma-n}  
\end{bmatrix}=
	\begin{bmatrix}
0  \\
0  
\end{bmatrix},
\end{align*}
for all $  n \in \{0,1,\dots, \gamma -1 \}$. Observe that these are in total $\gamma$ $(2 \times 2)-$linear systems where the unknowns are $h^{m}$ for $m \in  \{0,1,\cdots,2\gamma\}\setminus \{\gamma\} $. Finally, these systems have only the trivial solution $h^{m}=0$ for all $m \in \{0,1,\cdots,2\gamma\}\setminus \{\gamma\} $ if and only if the determinants 
\begin{align*}
	 \mathfrak{D}_{\gamma n}  &  =
	 \mathfrak{u}_{\gamma n} \mathfrak{u}_{\gamma ,2\gamma-n} - \mathfrak{v}_{\gamma n} \mathfrak{v}_{\gamma ,2\gamma-n} 
\end{align*}
are non--zero for all $  n \in \{0,1,\dots, \gamma -1 \}$, that completes the proof.
\end{proof}
Next, we establish the non--degeneracy condition for this model.

\begin{prop}[Non-degeneracy condition for the 1--modes and the YM model]
	Let $\gamma \in \{0,1,2,3,4,5\}$. Then, the non--degeneracy condition \eqref{NonDegeneracyConditionCHModel3} holds true. 
\end{prop}

\begin{proof}
	  Let $\gamma \in \{0,1,2,3,4,5\}$, $ \alpha = 1/ \sqrt{6}$ and pick any frequency $\ssomega \in \mathcal{W}_{\alpha}$ with $\ssomega<1$. Also, pick an integer $m \geq 2\gamma+1$ and define
\begin{align*}
	\mathfrak{u}_{\gamma m}  &:=    
 	  \left(\frac{\ssomega_{m}}{  \mathfrak{K}_{\gamma}} \right)^{2}+\frac{3}{4} \mathfrak{C}_{\gamma \gamma mm}  - \frac{9}{2}   \sum_{ \substack{ \nu = 0  }}^{m+\gamma}    \frac{\left(\overline{\mathfrak{C}}_{\gamma \nu m} \right)^2 }{\ssomega_{\nu}^2-(\ssomega_{m}+\ssomega_{\gamma})^2  }-  \frac{9}{4}  \sum_{\substack{ \nu = 0  } }^{2\gamma}   
		 \frac{\overline{\mathfrak{C}}_{m \nu m}\overline{\mathfrak{C}}_{\gamma\gamma \nu }}{\ssomega_{\nu}^2 }  \\
 	 & -\frac{9}{2}  \sum_{ \substack{ \nu = 0 \\ \nu \neq \pm(m-\gamma)-2  }}^{m+\gamma}   \frac{\left(\overline{\mathfrak{C}}_{m\gamma \nu } \right)^2}{\ssomega_{\nu}^2-(\ssomega_{m}-\ssomega_{\gamma})^2  }   , \\
	\mathfrak{v}_{\gamma m}  &:=   \frac{3}{8} \mathfrak{C}_{\gamma ,2\gamma-m,\gamma, m} -\frac{9}{4}   \sum_{\substack{ \nu = 0  } }^{2\gamma}  
		 \frac{\overline{\mathfrak{C}}_{2\gamma-m, \nu, m} \overline{\mathfrak{C}}_{\gamma\gamma \nu }}{\ssomega_{\nu}^2-(2\ssomega_{\gamma})^2  }  -\frac{9}{2} \sum_{ \substack{ \nu = 0 \\ \nu \neq \pm(m-\gamma)-2  }}^{m+\gamma}   \frac{\overline{\mathfrak{C}}_{\gamma \nu m} \overline{\mathfrak{C}}_{2\gamma-m,\gamma ,\nu }}{\ssomega_{\nu}^2-(\ssomega_{2\gamma-m}-\ssomega_{\gamma})^2  } .
\end{align*}
Firstly, we show that
\begin{align*}
	\mathfrak{u}_{\gamma \gamma}   + \mathfrak{v}_{\gamma \gamma}  \neq 0.
\end{align*}	
In this case, all the sums above are finite and in particular $\nu$ varies within $\{0,1,\dots,2\gamma\}$. Hence, we compute $\mathfrak{u}_{\gamma \gamma}   + \mathfrak{v}_{\gamma \gamma}$ for all $\gamma \in \{0,1,2,3,4,5\}$ and verify that its has a strictly positive lower bound. Figure \ref{Lemma96Pic1Ref}
\begin{figure}[h!]
\vspace{0.5cm}
    \centering
    \includegraphics[width=0.7\textwidth]{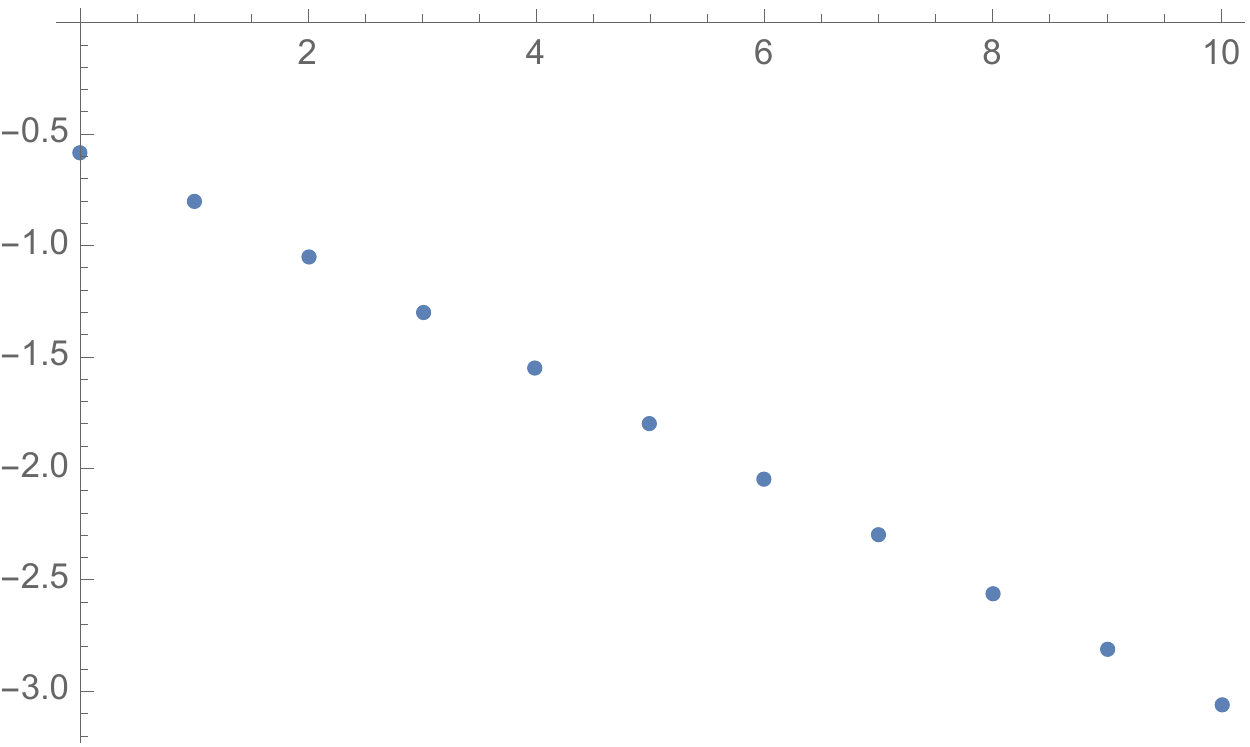}
    \caption{The constants $\mathfrak{u}_{\gamma \gamma}   + \mathfrak{v}_{\gamma \gamma}$ as $\gamma$ varies within $\{0,1,\dots,10 \}$. They decrease and stay away from zero.  }
    \label{Lemma96Pic1Ref}
\end{figure}  
illustrates the constants $\mathfrak{u}_{\gamma \gamma}   + \mathfrak{v}_{\gamma \gamma}$ as $\gamma$ varies within $\{0,1,2,3,4,5\}$. Secondly, we show that
\begin{align*}
	\mathfrak{u}_{\gamma m}  \neq 0,
\end{align*}	
for $ m \geq 2\gamma+1$. To this end, we note that $\overline{\mathfrak{C}}_{ijm}=0$ for all integers $i,j,m \geq 0$ with either $m>i+j$ or $|i-j|>m$ (Lemma \ref{LemmaClosedFormulasModel3YMCbarijm}). Specifically, for any integer $m\geq 2\gamma+1$, we focus on 
\begin{align*}
	   \sum_{ \substack{ \nu = 0  }}^{m+\gamma}    \frac{\left(\overline{\mathfrak{C}}_{\gamma \nu m} \right)^2 }{\ssomega_{\nu}^2-(\ssomega_{m}+\ssomega_{\gamma})^2   } , \Hquad 
	   \sum_{ \substack{ \nu = 0 \\ \nu \neq \pm(m-\gamma)-2  }}^{m+\gamma}   \frac{\left(\overline{\mathfrak{C}}_{m\gamma \nu } \right)^2}{\ssomega_{\nu}^2-(\ssomega_{m}-\ssomega_{\gamma})^2   } 
\end{align*}
and note that we must have 
\begin{align*}
m-\gamma	\leq \nu \leq m+\gamma 
\end{align*}
since
\begin{itemize}
	\item $ m-\gamma=|m-\gamma|>\nu   \Longrightarrow \overline{\mathfrak{C}}_{\gamma   \nu m}=0$,
	\item $ \nu>m+\gamma   \Longrightarrow \overline{\mathfrak{C}}_{\gamma   \nu m}=0$.
\end{itemize}
In addition, for all such $\nu$, the conditions $\nu \neq  \pm(m-\gamma)-2 $ are satisfied since
\begin{itemize} 
	\item $ \nu \geq m-\gamma \Longrightarrow (m-\gamma)-2 <m-\gamma \leq \nu  \Longrightarrow \nu \neq (m-\gamma)-2$, 
	\item $m\geq 2\gamma+1 \Longrightarrow -(m-\gamma)-2 <0 \Longrightarrow \nu \neq  -(m-\gamma)-2 $.
\end{itemize}
Consequently, for all $m\geq 2\gamma+1$, we have
\begin{align*}
	\mathfrak{u}_{\gamma m}  &:=    
 	\left(\frac{\ssomega_{m}}{  \mathfrak{K}_{\gamma}} \right)^{2}+\frac{3}{4} \mathfrak{C}_{\gamma \gamma mm} - \frac{9}{2}   \sum_{ \substack{ \nu = m-\gamma  }}^{m+\gamma}    \frac{\left(\overline{\mathfrak{C}}_{\gamma \nu m} \right)^2 }{\ssomega_{\nu}^2-(\ssomega_{m}+\ssomega_{\gamma})^2   }-  \frac{9}{4}  \sum_{\substack{ \nu = 0  } }^{2\gamma}   
		 \frac{\overline{\mathfrak{C}}_{m \nu m}\overline{\mathfrak{C}}_{\gamma\gamma \nu }}{\ssomega_{\nu}^2 }  \\
 	& -\frac{9}{2}  \sum_{ \substack{ \nu = m-\gamma  }}^{m+\gamma}   \frac{\left(\overline{\mathfrak{C}}_{m\gamma \nu } \right)^2}{\ssomega_{\nu}^2-(\ssomega_{m}-\ssomega_{\gamma})^2   }   
\end{align*}
and by setting $\nu= \sigma+m-\gamma  $ and $\nu=\sigma$ respectively we can rewrite the latter as follows
\begin{align*}
	\mathfrak{u}_{\gamma m}  &:=    
 	  \left(\frac{\ssomega_{m}}{  \mathfrak{K}_{\gamma}} \right)^{2}+\frac{3}{4} \mathfrak{C}_{\gamma \gamma mm} - \frac{9}{2}   \sum_{ \substack{ 
 	\sigma=0  }}^{2\gamma}    \frac{\left(\overline{\mathfrak{C}}_{\gamma, \sigma+m-\gamma  , m} \right)^2 }{\ssomega_{ \sigma+m-\gamma  }^2-(\ssomega_{m}+\ssomega_{\gamma})^2   }-  \frac{9}{4}  \sum_{\substack{ \sigma = 0  } }^{2\gamma}   
		 \frac{\overline{\mathfrak{C}}_{m \sigma m}\overline{\mathfrak{C}}_{\gamma\gamma \sigma }}{\ssomega_{\sigma}^2 }  \\
 	& -\frac{9}{2}  \sum_{ \substack{ \sigma = 0  }}^{2\gamma}   \frac{\left(\overline{\mathfrak{C}}_{m,\gamma, \sigma+m-\gamma   } \right)^2}{\ssomega_{ \sigma+m-\gamma  }^2-(\ssomega_{m}-\ssomega_{\gamma})^2   }.  
\end{align*}
Now, recall from Lemma \ref{LemmaResonantSystemModel3YM} that  
\begin{align*}
	\mathfrak{K}_{\gamma}^{-2}    
	& =\frac{3 \mathfrak{C}_{\gamma\gamma\gamma\gamma}-8\mathfrak{q}_{\gamma} }{-8\ssomega_{\gamma}^2}   = - \frac{ 3 }{8\ssomega_{\gamma}^2}   \mathfrak{C}_{\gamma\gamma\gamma\gamma} + \frac{ 1 }{\ssomega_{\gamma}^2}  \mathfrak{q}_{\gamma}   \\
	 & = - \frac{ 3 }{8\ssomega_{\gamma}^2}   \mathfrak{C}_{\gamma\gamma\gamma\gamma} +\frac{9}{4 \ssomega_{\gamma}^2}  \sum_{  \nu = 0}^{2\gamma}  \left(  \overline{\mathfrak{C}}_{\gamma \gamma \nu} \right)^2
  \left(\frac{2}{\ssomega_{\nu}^2 }  +  \frac{1}{\ssomega_{\nu}^2-(2\ssomega_{\gamma})^2    } \right) 
\end{align*} 
that yields
\begin{align*}
  \left(\frac{\ssomega_{m}}{  \mathfrak{K}_{\gamma}} \right)^{2} =
	- \frac{ 3\ssomega_{m}^2 }{8\ssomega_{\gamma}^2}   \mathfrak{C}_{\gamma\gamma\gamma\gamma} +\frac{9\ssomega_{m}^2}{4 \ssomega_{\gamma}^2}  \sum_{  \nu = 0}^{2\gamma}  \left(  \overline{\mathfrak{C}}_{\gamma \gamma \nu} \right)^2
  \left(\frac{2}{\ssomega_{\nu}^2 }  +  \frac{1}{\ssomega_{\nu}^2-(2\ssomega_{\gamma})^2    } \right) .
\end{align*}
Putting all together, we obtain
\begin{align*}
	\mathfrak{u}_{\gamma m}  &=    
 	 - \frac{ 3\ssomega_{m}^2 }{8\ssomega_{\gamma}^2}   \mathfrak{C}_{\gamma\gamma\gamma\gamma} +\frac{9\ssomega_{m}^2}{4 \ssomega_{\gamma}^2}  \sum_{  \sigma = 0}^{2\gamma}  \left(  \overline{\mathfrak{C}}_{\gamma \gamma \sigma} \right)^2
  \left(\frac{2}{\ssomega_{\sigma}^2 }  +  \frac{1}{\ssomega_{\sigma}^2-(2\ssomega_{\gamma})^2    } \right) \\
  & +\frac{3}{4} \mathfrak{C}_{\gamma \gamma mm} - \frac{9}{2}   \sum_{ \substack{ 
 	\sigma=0  }}^{2\gamma}    \frac{\left(\overline{\mathfrak{C}}_{\gamma, \sigma+m-\gamma  , m} \right)^2 }{\ssomega_{ \sigma+m-\gamma  }^2-(\ssomega_{m}+\ssomega_{\gamma})^2   }-  \frac{9}{4}  \sum_{\substack{ \sigma = 0  } }^{2\gamma}   
		 \frac{\overline{\mathfrak{C}}_{m \sigma m}\overline{\mathfrak{C}}_{\gamma\gamma \sigma }}{\ssomega_{\sigma}^2 }  \\
 	& -\frac{9}{2}  \sum_{ \substack{ \sigma = 0  }}^{2\gamma}   \frac{\left(\overline{\mathfrak{C}}_{m,\gamma, \sigma+m-\gamma   } \right)^2}{\ssomega_{ \sigma+m-\gamma  }^2-(\ssomega_{m}-\ssomega_{\gamma})^2   }.  
\end{align*}
In addition, we also note that $\overline{\mathfrak{C}}_{ijm}=0$ for all integers $i,j,m \geq 0$ with $i+j-m \notin 2 \mathbb{N}\cup \{0\}$ (Lemma \ref{LemmaClosedFormulasModel3YMCbarijm}). Specifically, we must have 
\begin{align*}
\sigma \in 2 \mathbb{N}\cup \{0\} 
\end{align*}
since
\begin{itemize}
	\item $ \sigma =\sigma+\gamma-\gamma \notin  2 \mathbb{N}\cup \{0\}  \Longrightarrow \overline{\mathfrak{C}}_{\gamma \gamma \sigma}=0$,
	\item $ \sigma =\gamma+\sigma+m-\gamma-m  \notin  2 \mathbb{N}\cup \{0\}  \Longrightarrow \overline{\mathfrak{C}}_{\gamma,\sigma+m-\gamma,m}=\overline{\mathfrak{C}}_{m,\gamma,\sigma+m-\gamma}=0$.
\end{itemize}
Therefore, by setting $\sigma =2\tau$, we arrive at	
\begin{align*}
	\mathfrak{u}_{\gamma m}  &=    
 	 - \frac{ 3\ssomega_{m}^2 }{8\ssomega_{\gamma}^2}   \mathfrak{C}_{\gamma\gamma\gamma\gamma} +\frac{9\ssomega_{m}^2}{2 \ssomega_{\gamma}^2}  \sum_{  \tau = 0}^{\gamma} 
 \frac{ \left(  \overline{\mathfrak{C}}_{\gamma ,\gamma , 2\tau} \right)^2}{\ssomega_{2\tau}^2 }  + \frac{9\ssomega_{m}^2}{4 \ssomega_{\gamma}^2}  \sum_{  \tau = 0}^{\gamma}  \frac{ \left(  \overline{\mathfrak{C}}_{\gamma ,\gamma , 2\tau} \right)^2}{\ssomega_{2\tau}^2-(2\ssomega_{\gamma})^2    }  \\
  & +\frac{3}{4} \mathfrak{C}_{\gamma \gamma mm} - \frac{9}{2}   \sum_{ \substack{ 
 \tau=0  }}^{\gamma}    \frac{\left(\overline{\mathfrak{C}}_{\gamma, 2\tau+m-\gamma  , m} \right)^2 }{\ssomega_{ 2\tau+m-\gamma  }^2-(\ssomega_{m}+\ssomega_{\gamma})^2   }-  \frac{9}{4}  \sum_{\substack{ \tau = 0  } }^{\gamma}   
		 \frac{\overline{\mathfrak{C}}_{m, 2\tau, m}\overline{\mathfrak{C}}_{\gamma,\gamma,2\tau }}{\ssomega_{2\tau}^2 }  \\
 	& -\frac{9}{2}  \sum_{ \substack{ \tau = 0  }}^{\gamma}   \frac{\left(\overline{\mathfrak{C}}_{m,\gamma, 2\tau+m-\gamma   } \right)^2}{\ssomega_{ 2\tau +m-\gamma  }^2-(\ssomega_{m}-\ssomega_{\gamma})^2   }.  
\end{align*}
Now, all the Fourier coefficients above are non--zero and, according to Lemma \ref{LemmaClosedFormulasModel3YMCbarijmSpecificijm}, we have
\begin{align*}
	\overline{\mathfrak{C}}_{\gamma,\gamma,2\tau} &=
	2 \sqrt{\frac{2}{\pi }} \frac{(\tau +1)^2 (\gamma -\tau +1) (\gamma +\tau +3)}{(\gamma +1) (\gamma +3) \sqrt{4 \tau  (\tau +2)+3}}, \\
	\overline{\mathfrak{C}}_{m ,2\tau , m} & =2 \sqrt{\frac{2}{\pi }} \frac{ (\tau +1)^2 (m-\tau +1) (m+\tau +3)}{(m+1) (m+3) \sqrt{4 \tau  (\tau +2)+3}}, \\
	\overline{\mathfrak{C}}_{\gamma, 2\tau +m-\gamma  , m}&=2 \sqrt{\frac{2}{\pi }}  \frac{(\tau +1) (\gamma -\tau +1) (m+\tau +3) (-\gamma +m+\tau +1)}{\sqrt{(\gamma +1) (\gamma +3) (m+1) (m+3) (-\gamma +m+2 \tau +1) (-\gamma +m+2 \tau +3)}}.
\end{align*}
The latter allows us to compute 
\begin{align*}
	\sum_{  \tau  = 0}^{\gamma}  \frac{\left(  \overline{\mathfrak{C}}_{\gamma ,\gamma, 2\tau  } \right)^2}{\ssomega_{2\tau  }^2 } &=\frac{(\gamma +2) (2 \gamma +3) (2 \gamma +5)}{15 \pi  (\gamma +1) (\gamma +3)}, \\
 \sum_{\substack{ \tau  = 0  } }^{\gamma}   
		 \frac{\overline{\mathfrak{C}}_{m ,2\tau , m}\overline{\mathfrak{C}}_{\gamma,\gamma, 2\tau }}{\ssomega_{2\tau }^2 }	&=  \frac{(\gamma +2) (-\gamma  (\gamma +4)+5 m (m+4)+15)}{15 \pi  (m+1) (m+3)},
\end{align*}	
for all integers $\gamma \geq 0$. Also, recall that $\mathfrak{C}_{\gamma\gamma mm}$ and $\mathfrak{C}_{\gamma\gamma\gamma\gamma}$ are given by closed formulas (Remark \ref{MonotonicityModel3YM}). Consequently, we rescale $\mathfrak{u}_{\gamma m} $, 
\begin{align*}
	\frac{\mathfrak{u}_{\gamma m} }{\ssomega_{m}^2} &=    
 	 - \frac{ 3  }{8\ssomega_{\gamma}^2}   \mathfrak{C}_{\gamma\gamma\gamma\gamma} +\frac{9 }{2 \ssomega_{\gamma}^2}  \sum_{  \tau = 0}^{\gamma} 
 \frac{ \left(  \overline{\mathfrak{C}}_{\gamma ,\gamma , 2\tau} \right)^2}{\ssomega_{2\tau}^2 }  + \frac{9 }{4 \ssomega_{\gamma}^2}  \sum_{  \tau = 0}^{\gamma}  \frac{ \left(  \overline{\mathfrak{C}}_{\gamma ,\gamma , 2\tau} \right)^2}{\ssomega_{2\tau}^2-(2\ssomega_{\gamma})^2    }  \\
  & +\frac{3}{4\ssomega_{m}^2} \mathfrak{C}_{\gamma \gamma mm} - \frac{9}{2\ssomega_{m}^2}   \sum_{ \substack{ 
 \tau=0  }}^{\gamma}    \frac{\left(\overline{\mathfrak{C}}_{\gamma, 2\tau+m-\gamma  , m} \right)^2 }{\ssomega_{ 2\tau+m-\gamma  }^2-(\ssomega_{m}+\ssomega_{\gamma})^2   }-  \frac{9}{4\ssomega_{m}^2}  \sum_{\substack{ \tau = 0  } }^{\gamma}   
		 \frac{\overline{\mathfrak{C}}_{m, 2\tau, m}\overline{\mathfrak{C}}_{\gamma,\gamma,2\tau }}{\ssomega_{2\tau}^2 }  \\
 	& -\frac{9}{2\ssomega_{m}^2}  \sum_{ \substack{ \tau = 0  }}^{\gamma}   \frac{\left(\overline{\mathfrak{C}}_{m,\gamma, 2\tau+m-\gamma   } \right)^2}{\ssomega_{ 2\tau +m-\gamma  }^2-(\ssomega_{m}-\ssomega_{\gamma})^2   }
\end{align*}
and obtain that
\begin{align}\label{SplitingofMathfrakU}
	\frac{\mathfrak{u}_{\gamma m} }{\ssomega_{m}^2} = \mathfrak{I}_{\gamma m}  + \mathfrak{E}_{\gamma m}  ,
\end{align}
where $\mathfrak{I}_{\gamma m} $ stands for the part that can be explicitly computed, 
\begin{align*}
	\mathfrak{I}_{\gamma m}  &:= - \frac{ 3  }{8\ssomega_{\gamma}^2}   \mathfrak{C}_{\gamma\gamma\gamma\gamma} +\frac{9 }{2 \ssomega_{\gamma}^2}  \sum_{  \tau = 0}^{\gamma} 
 \frac{ \left(  \overline{\mathfrak{C}}_{\gamma ,\gamma , 2\tau} \right)^2}{\ssomega_{2\tau}^2 }  + \frac{9 }{4 \ssomega_{\gamma}^2}  \sum_{  \tau = 0}^{\gamma}  \frac{ \left(  \overline{\mathfrak{C}}_{\gamma ,\gamma , 2\tau} \right)^2}{\ssomega_{2\tau}^2-(2\ssomega_{\gamma})^2    }  \\
  & +\frac{3}{4\ssomega_{m}^2} \mathfrak{C}_{\gamma \gamma mm} -  \frac{9}{4\ssomega_{m}^2}  \sum_{\substack{ \tau = 0  } }^{\gamma}   
		 \frac{\overline{\mathfrak{C}}_{m, 2\tau, m}\overline{\mathfrak{C}}_{\gamma,\gamma,2\tau }}{\ssomega_{2\tau}^2 },
\end{align*}
and $\mathfrak{E}_{\gamma m}$ stands for the part that cannot be explicitly computed, 
\begin{align*} 
	\mathfrak{E}_{\gamma m}   & :=    -  \frac{9}{2\ssomega_{m}^2}   \sum_{ \substack{ 
 \tau=0  }}^{\gamma}    \frac{\left(\overline{\mathfrak{C}}_{\gamma, 2\tau+m-\gamma  , m} \right)^2 }{\ssomega_{ 2\tau+m-\gamma  }^2-(\ssomega_{m}+\ssomega_{\gamma})^2  } 
 -\frac{9}{2\ssomega_{m}^2}  \sum_{ \substack{ \tau = 0  }}^{\gamma}   \frac{\left(\overline{\mathfrak{C}}_{m,\gamma, 2\tau+m-\gamma   } \right)^2}{\ssomega_{ 2\tau +m-\gamma  }^2-(\ssomega_{m}-\ssomega_{\gamma})^2  }.
\end{align*} 
Now, using the elementary inequalities 
\begin{align*}
	|\ssomega_{ 2\tau+m-\gamma  }^2-(\ssomega_{m}+\ssomega_{\gamma})^2|&=|4 (\gamma -\tau +1) (m+\tau +3)| \geq 4(m+3) , \\
	|\ssomega_{ 2\tau +m-\gamma  }^2-(\ssomega_{m}-\ssomega_{\gamma})^2|&=|4 (\tau +1) (-\gamma +m+\tau +1)| \geq 2(m+3),
\end{align*}
for all $0 \leq \tau \leq \gamma$ and $\gamma \geq 0$, we estimate
\begin{align*}
& \left|	\mathfrak{E}_{\gamma m}   \right| \leq   \frac{9}{2\ssomega_{m}^2}   \sum_{ \substack{ 
 \tau=0  }}^{\gamma}    \frac{\left( \overline{ \mathfrak{C}}_{\gamma, 2\tau+m-\gamma  , m} \right)^2 }{\left| \ssomega_{ 2\tau+m-\gamma  }^2-(\ssomega_{m}+\ssomega_{\gamma})^2     \right|}  +\frac{9}{2\ssomega_{m}^2}  \sum_{ \substack{ \tau = 0  }}^{\gamma}   \frac{\left(\overline{\mathfrak{C}}_{m,\gamma, 2\tau+m-\gamma   } \right)^2}{\left| \ssomega_{ 2\tau +m-\gamma  }^2-(\ssomega_{m}-\ssomega_{\gamma})^2     \right|} \\
 & \leq \frac{9  }{2\ssomega_{m}^2} \frac{1}{4(m+3)}  \sum_{ \substack{ 
 \tau=0  }}^{\gamma}  \left( \overline{ \mathfrak{C}}_{\gamma, 2\tau+m-\gamma  , m} \right)^2 + \frac{9  }{2\ssomega_{m}^2} \frac{1}{2(m+3)}\sum_{ \substack{ 
 \tau=0  }}^{\gamma}  \left( \overline{ \mathfrak{C}}_{\gamma, 2\tau+m-\gamma  , m} \right)^2 \\
 & = \frac{9  }{2\ssomega_{m}^2} \left(\frac{1}{4(m+3)}+\frac{1}{2(m+3)}\right) \sum_{ \substack{ 
 \tau=0  }}^{\gamma}  \left( \overline{ \mathfrak{C}}_{\gamma, 2\tau+m-\gamma  , m} \right)^2  \\
 & =\frac{9 (\gamma +2)}{70 \pi  (m+1) (m+2)^2 (m+3)^2}\Big[-3 \gamma ^4-24 \gamma ^3-40 \gamma ^2+32 \gamma +7 \gamma ^2 m^2+28 \gamma  m^2+35 m^2 \\
 &+28 \gamma ^2 m+112 \gamma  m+140 m+105 \Big] :=  \mathfrak{P}_{\gamma m},
\end{align*}
where we used the closed formula for $\overline{ \mathfrak{C}}_{\gamma, 2\tau+m-\gamma  , m} $ from above. Hence, for all $m \geq 2\gamma+1$, we obtain
\begin{align}\label{ProveRHSpositive}
	\frac{\mathfrak{u}_{\gamma m} }{\ssomega_{m}^2}   = \mathfrak{I}_{\gamma m}  +  \mathfrak{E}_{\gamma}  \geq  \mathfrak{I}_{\gamma m} -  \mathfrak{P}_{\gamma m}   :=\mathfrak{O}_{\gamma m}.
\end{align}
Finally, for each $\gamma \in \{0,1,2,3,4,5\}$, we use the closed formulas for $\mathfrak{C}_{\gamma\gamma mm}$ and $\mathfrak{C}_{\gamma\gamma\gamma\gamma}$ we derived in Remark \ref{MonotonicityModel3YM} to firstly explicitly compute $\mathfrak{I}_{\gamma m} $ in terms of $m$ and then explicitly compute $\mathfrak{O}_{\gamma m}$ in terms of $m$. Once the closed formula is derived, we then show that $\mathfrak{O}_{\gamma m}>0$ for all $m \geq 2\gamma+1$. For example, for $\gamma=0$, we find 
\begin{align*}
\mathfrak{I}_{0 m} =	\frac{m (m+4) (5 m (m+4)+29)+66}{12 \pi  (m+1) (m+2)^2 (m+3)}
\end{align*}
and hence 
\begin{align*}
	\mathfrak{O}_{0 m}=\frac{5 m^4+40 m^3+109 m^2+8 m-42}{12 \pi  (m+1) (m+2)^2 (m+3)},
\end{align*}
that is greater than $10^{-3}$ provided that $m \geq 1$. Similarly, for $\gamma=1$, we compute
\begin{align*}
	\mathfrak{I}_{1 m} =\frac{\left(m^2+2\right) (m (m+8)+18)}{4 \pi  (m+1) (m+2)^2 (m+3)}
\end{align*}
and hence 
\begin{align*}
	\mathfrak{O}_{1m}=\frac{m \left(m^4+11 m^3+44 m^2-32 m-348\right)}{4 \pi  (m+1) (m+2)^2 (m+3)^2},
\end{align*}
that is greater than $10^{-3}$ provided that $m \geq 3$.
For all the other cases with $\gamma \in \{2,3,4,5\}$, we find
\begin{align*}
	\mathfrak{O}_{2m} &= \frac{109 m^5+1199 m^4+4523 m^3-30347 m^2-132936 m+107244}{600 \pi  (m+1) (m+2)^2 (m+3)^2} ,\\
	\mathfrak{O}_{3m} &= \frac{43 m^5+473 m^4+1646 m^3-33554 m^2-129372 m+238248}{300 \pi  (m+1) (m+2)^2 (m+3)^2} ,\\
	\mathfrak{O}_{4m} &= \frac{83 m^5+913 m^4+2851 m^3-139159 m^2-515982 m+1611198}{700 \pi  (m+1) (m+2)^2 (m+3)^2} ,\\
	\mathfrak{O}_{5m} &=  \frac{17 m^5+187 m^4+505 m^3-53329 m^2-194760 m+905292}{168 \pi  (m+1) (m+2)^2 (m+3)^2}
\end{align*}
and the claim follows similarly\footnote{For $\gamma \in \{2,3\}$, the estimates $\mathfrak{O}_{\gamma m} \geq 10^{-3}$ hold true for all integers $m \geq 2\gamma+1$, whereas, for $\gamma \in \{4,5\}$, we have that $\mathfrak{O}_{\gamma m} \geq 10^{-3}$ provided that $m \geq 2\gamma+3$ instead of $m \geq 2\gamma+1$. In this case, we use the definition of the Fourier coefficients to explicitly compute $	 \mathfrak{u}_{\gamma m}  \ssomega_{m}^{-2} $ for $m\in \{2\gamma+1,2\gamma+2\}$ and verify that it is still strictly positive.}. Finally, it remains to show that the determinants
\begin{align*}
	\mathfrak{D}_{\gamma n} & : =
	 \mathfrak{u}_{\gamma n} \mathfrak{u}_{\gamma ,2\gamma-n} - \mathfrak{v}_{\gamma n} \mathfrak{v}_{\gamma ,2\gamma-n}  
\end{align*}
are all non--zero for $ n \in \{0,1,\dots, \gamma -1 \}$ that follows by a direct computation using the definition of the Fourier coefficients. Specifically, we  compute $\mathfrak{D}_{\gamma n}$ for all $\gamma \in \{0,1,2,3,4,5\}$  and $ n \in \{0,1,\dots, \gamma -1 \}$ and verify that they are all strictly negative, that completes the proof. Figure \ref{Lemma96Pic3Ref}
\begin{figure}[h!]
\vspace{0.5cm}
    \centering
    \includegraphics[width=0.7\textwidth]{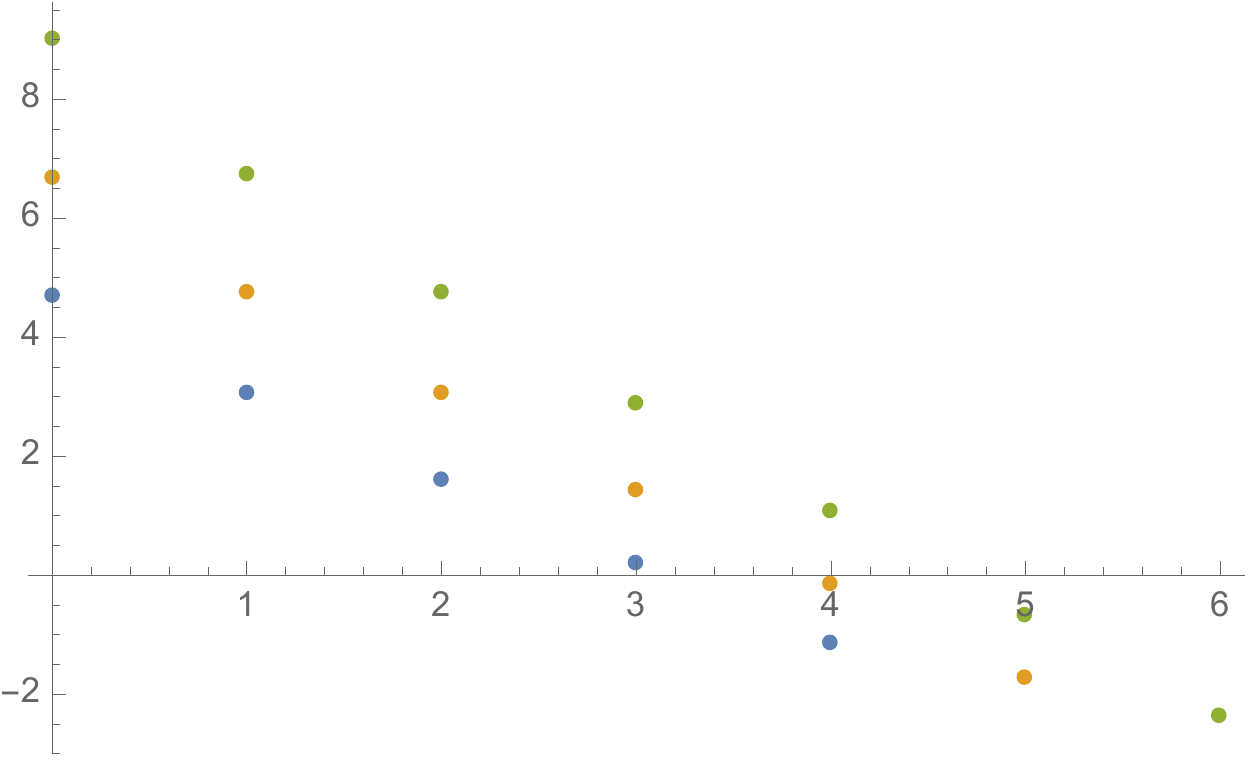}
    \caption{The determinants $\mathfrak{D}_{\gamma n}  $ for $\gamma =5$ (blue), $\gamma =6$ (orange) and $\gamma =7$ (green) as $n$ varies within $ \{0,1,\dots,\gamma-1\}$. They are all in fact non--zero.}
    \label{Lemma96Pic3Ref}
\end{figure}  
illustrates the determinants $\mathfrak{D}_{\gamma n}  $ for $\gamma \in \{5,6,7\}$ as $n$ varies within $ \{0,1,\dots,\gamma-1\}$.

\end{proof}		
  
\appendix

 \bibliographystyle{plain}
 \bibliography{V13_Non_linear_Periodic_waves}

\end{document}